\documentclass[11pt]{article}
\usepackage{todonotes}
\usepackage[margin=1in]{geometry}
\usepackage{amsmath}
\usepackage{amssymb}
\usepackage[toc,page]{appendix}
\usepackage{graphicx}
\usepackage{bbm}
\usepackage{latexsym}
\usepackage{amsthm}
\usepackage{epsfig}
\usepackage{xcolor}
\usepackage{graphicx}
\usepackage{bm} 
\usepackage{mathrsfs}
\usepackage{enumitem}
\usepackage{mathtools}
\usepackage{graphicx}
\usepackage{tikz}
\usetikzlibrary{arrows.meta, positioning}
  
\usepackage{pgfplots} 
\pgfplotsset{compat=1.17} 

\usepackage{caption}
\usepackage{subcaption}

\usepackage{mathrsfs}
\usepackage{hyperref}
\usepackage[most]{tcolorbox}
\hypersetup{
    colorlinks,
    linkcolor={blue!80!black},
    citecolor={green!50!black},
} 
\colorlet{linkequation}{blue}
\usepackage{mathtools}
\mathtoolsset{showonlyrefs}

\usepackage{authblk}

\renewcommand{\P}{\mathbb{P}}
\newcommand{\E}{\mathbb{E}}
\newcommand{\Var}{\text{Var}}
\newcommand{\Cov}{\text{Cov}}
\newcommand{\cN}{\mathcal{N}}

\newcommand{\Z}{\mathbb{Z}}
\newcommand{\R}{\mathbb{R}}

\renewcommand{\S}{\mathbb{S}}

\newcommand{\eps}{\varepsilon} 

\def\id{{\mathbf I}}

\renewcommand{\d}{\textup{d}}

\newcommand{\<}{\langle}
\renewcommand{\>}{\rangle}
\newcommand{\sign}{\operatorname{sign}}
\newcommand{\diag}{\text{diag}}

\newcommand{\op}{{\rm op}}

\def\sT{{\mathsf T}}

\def\bzero{{\boldsymbol 0}}

\DeclareMathOperator*{\argmin}{arg\,min}

\newtheorem{theorem}{Theorem}[section]
 
\newtheorem{lemma}[theorem]{Lemma}
\newtheorem{proposition}[theorem]{Proposition}
\newtheorem{corollary}[theorem]{Corollary}
\newtheorem{definition}[theorem]{Definition}

\newtheorem{assumption}[theorem]{Assumption}

\newtheorem*{theorem*}{Theorem}
\newtheorem*{assumption*}{Assumption}

\theoremstyle{definition}

\newtheoremstyle{myremark} 
    {\topsep}                    
    {\topsep}                    
    {\rm}                        
    {}                           
    {\bf}                        
    {.}                          
    {.5em}                       
    {}  

\theoremstyle{myremark}
\newtheorem{remark}{Remark}[section]

\DeclareSymbolFont{rsfs}{U}{rsfs}{m}{n}
\DeclareSymbolFontAlphabet{\mathscrsfs}{rsfs}

\def\bA{{\boldsymbol A}}
\def\bB{{\boldsymbol B}}
\def\bC{{\boldsymbol C}}
\def\bD{{\boldsymbol D}}

\def\bF{{\boldsymbol F}}
\def\bG{{\boldsymbol G}}

\def\bH{{\boldsymbol H}}
\def\bI{{\boldsymbol I}}

\def\bL{{\boldsymbol L}}
\def\bM{{\boldsymbol M}}

\def\bR{{\boldsymbol R}}
\def\bS{{\boldsymbol S}}
\def\bT{{\boldsymbol T}}
\def\bU{{\boldsymbol U}}
\def\bV{{\boldsymbol V}}
\def\bW{{\boldsymbol W}}
\def\bX{{\boldsymbol X}}
\def\bY{{\boldsymbol Y}}
\def\bZ{{\boldsymbol Z}}

\def\ba{{\boldsymbol a}}
\def\bb{{\boldsymbol b}}

\def\be{{\boldsymbol e}}
\def\boldf{{\boldsymbol f}}
\def\bg{{\boldsymbol g}}
\def\bh{{\boldsymbol h}}
\def\bi{{\boldsymbol i}}

\def\bk{{\boldsymbol k}}
\def\bt{{\boldsymbol t}}
\def\bu{{\boldsymbol u}}
\def\bv{{\boldsymbol v}}
\def\bw{{\boldsymbol w}}
\def\bx{{\boldsymbol x}}
\def\by{{\boldsymbol y}}
\def\bz{{\boldsymbol z}}

\def\bmu{{\boldsymbol \mu}}
\def\bbeta{{\boldsymbol \beta}}

\def\beps{{\boldsymbol \eps}}

\def\btheta{{\boldsymbol \theta}}

\def\bxi{{\boldsymbol \xi}}

\def\bDelta{{\boldsymbol \Delta}}
\def\bLambda{{\boldsymbol \Lambda}}

\def\bSigma{{\boldsymbol \Sigma}}
\def\bTheta{{\boldsymbol \Theta}}

\def\cR{\mathcal{R}}
\def\test{{\rm test}}

\def\de{{\rm d}}
\def\Tr{{\rm Tr}}

\def\de{{\rm d}}
\def\Unif{{\rm Unif}}

\def\RF{\mathsf{RF}}

\def\cV{{\mathcal V}}
\def\cG{{\mathcal G}}

\def\cP{{\mathcal P}}

\def\cT{{\mathcal T}}

\def\cL{{\mathcal L}}

\def\cE{{\mathcal E}}

\def\cI{{\mathcal I}}
\def\cV{{\mathcal V}}
\def\cG{{\mathcal G}}

\def\cP{{\mathcal P}}

\def\cT{{\mathcal T}}
\def\cH{{\mathcal H}}
\def\cA{{\mathcal A}}

\def\Unif{{\sf Unif}}

\def\proj{{\mathsf P}}

\def\RF{{\sf RF}}

\def\naturals{{\mathbb N}}

\def\proj{{\mathsf P}}

\def\Unif{{\sf Unif}}
\def\normal{\mathcal{N}}

\def\proj{{\mathsf P}}

\def\RF{{\sf RF}}

\def\naturals{{\mathbb N}}
\def\proj{{\mathsf P}}

\def\cE{{\mathcal E}}

\def\cI{{\mathcal I}}

\def\He{{\rm He}}

\def\de{{\rm d}}
\def\Unif{{\rm Unif}}

\def\cE{{\mathcal E}}
\def\bt{{\boldsymbol t}}

\def\bDelta{{\boldsymbol \Delta}}

\def\bA{{\boldsymbol A}}
\def\btheta{{\boldsymbol \theta}}
\def\bTheta{{\boldsymbol \Theta}}
\def\bLambda{{\boldsymbol \Lambda}}

\def\cM{{\mathcal M}}

\def\cT{{\mathcal T}}
\def\cV{{\mathcal V}}

\def\diag{{\rm diag}}
\def\bS{{\boldsymbol S}}

\def\bD{{\boldsymbol D}}

\def\bL{{\boldsymbol L}}

\def\bb{{\boldsymbol b}}

\def\bR{{\boldsymbol R}}

\def\bc{{\boldsymbol c}}
\def\bC{{\boldsymbol C}}

\def\boldf{\boldsymbol{f}}

\def\hbtheta{\hat \btheta}

\def\dom{{\rm dom}}

\def\br{{\boldsymbol r}}
\def\bGamma{{\boldsymbol \Gamma}}

\def\boldf{\boldsymbol{f}}

\def\hbtheta{\hat \btheta}

\def\bq{\boldsymbol{q}}

\def\bs{{\boldsymbol s}}

\def\bones{{\boldsymbol 1}}

\DeclareSymbolFont{Greekletters}{OT1}{iwona}{m}{n}
\DeclareSymbolFont{greekletters}{OML}{iwona}{m}{it}
\DeclareMathSymbol{\salpha}{\mathord}{greekletters}{"0B}
\DeclareMathSymbol{\sbeta}{\mathord}{greekletters}{"0C}
\DeclareMathSymbol{\sgamma}{\mathord}{greekletters}{"0D}
\DeclareMathSymbol{\sOmega}{\mathord}{Greekletters}{"0A}
\DeclareMathSymbol{\smu}{\mathord}{greekletters}{"16}
\DeclareMathSymbol{\svarepsilon}{\mathord}{greekletters}{"22}
\DeclareMathSymbol{\svarrho}{\mathord}{greekletters}{"25}
\DeclareMathSymbol{\svarphi}{\mathord}{greekletters}{"27}

\def\hcR{\widehat{\cR}}

\def\sK{\mathsf{K}}

\def\sG{\mathsf{G}}
\def\sCG{\mathsf{CG}}
\def\sPG{\mathsf{PG}}
\def\sC{\mathsf{C}}
\def\sc{\mathsf{c}}
\def\sr{\mathsf{r}}

\def\btau{{\boldsymbol \tau}}

\def\R{\mathbb{R}}
\def\E{\mathbb{E}}

\def\cE{\mathcal{E}}
\def\cL{\mathcal{L}}
\def\cN{\normal}
\def\cH{\mathcal{H}}

\def\1{\mathbf{1}}

\def\test{\mathrm{test}}

\def\op{\mathrm{op}}
\def\Sym{\mathrm{Sym}}
\def\Tr{\operatorname{Tr}}

\def\Prox{{\rm Prox}}
\def\CG{{\rm CG}}

\newtcolorbox{modelbox}[2][]{
    colback=black!5!white,  
    colframe=black!75!white,  
    fonttitle=\bfseries,
    coltitle=black,
    title=#2,
    breakable,
    enhanced,
    attach boxed title to top left={yshift=-2mm, xshift=3mm},
    boxed title style={
        colback=white,
        colframe=black!75!white,
    },
    #1
}

\title{When does Gaussian equivalence fail and how to fix it:\\
Non-universal behavior of random features with quadratic scaling
}

\author[1]{Garrett G. Wen}
\author[2,3]{\;Hong Hu}
\author[4]{\;Yue M. Lu}
\author[1]{\;Zhou Fan}
\author[1]{\;Theodor Misiakiewicz}

\affil[1]{\small Department of Statistics and Data Science, Yale University}
\affil[2]{\small 
Department of Electrical and System Engineering,  Washington University in Saint Louis}
\affil[3]{\small 
Department of Statistics and Data Science,  Washington University in Saint Louis}
\affil[4]{\small 
John A. Paulson School of Engineering and Applied Sciences, Harvard University}

\date{\today}

\allowdisplaybreaks

\begin{document}
\maketitle

\begin{abstract}
A major effort in modern high-dimensional statistics has been devoted to the analysis of linear predictors trained on nonlinear feature embeddings via empirical risk minimization (ERM). \textit{Gaussian equivalence theory} (GET) has emerged as a powerful universality principle in this context: it states that the behavior of high-dimensional, complex features can be captured by Gaussian surrogates, which are more amenable to analysis. Despite its remarkable successes, numerical experiments show that this equivalence can fail even for simple embeddings---such as polynomial maps---under general scaling regimes.

We investigate this breakdown in the setting of random feature (RF) models in the \textit{quadratic scaling regime}, where both the number of features and the sample size grow quadratically with the data dimension.  We show that when the target function depends on a low-dimensional projection of the data, such as generalized linear models, GET yields incorrect predictions. To capture the correct asymptotics, we introduce a \textit{Conditional Gaussian Equivalent} (CGE) model, which can be viewed as appending a low-dimensional non-Gaussian component to an otherwise high-dimensional Gaussian model. This hybrid model retains the tractability of the Gaussian framework and accurately describes RF models in the quadratic scaling regime. We derive sharp asymptotics for the training and test errors in this setting, which continue to agree with numerical simulations even when GET fails.

Our analysis combines general results on CLT for Wiener chaos expansions and a careful two-phase Lindeberg swapping argument. Beyond RF models and quadratic scaling, our work hints at a rich landscape of universality phenomena in high-dimensional ERM. 
\end{abstract}

\tableofcontents

\clearpage

\section{Introduction}

A central paradigm in statistical learning is \textit{Empirical Risk Minimization} (ERM), where predictors are constructed by minimizing a (possibly regularized) empirical risk on the training data. In recent decades, substantial effort has been devoted to characterizing the behavior of ERM-based procedures in the high-dimensional regime relevant to modern applications. Classical frameworks---such as consistency or uniform convergence---often fail to accurately describe estimators in these settings. In response, a large body of work has developed sharp theories that precisely capture the asymptotic performance of ERM, including the prediction risk and estimation errors, and their dependence on problem parameters \cite{bayati2011lasso,thrampoulidis2015regularized,donoho2016high,thrampoulidis2018precise,el2018impact,lei2018asymptotics,candes2020phase,liang2022precise,celentano2023lasso,montanari2025generalization,asgari2025local}. This program has clarified the high-dimensional behavior of widely used methods---including the Lasso, ridge regression, and logistic regression---while also guiding the design of new statistical procedures (e.g., debiasing methods for inference \cite{javanmard2014confidenceintervalshypothesistesting,van_de_Geer_2014,zhang2012confidenceintervalslowdimensionalparameters,sur2019likelihood,celentano2023lasso}) and elucidating a number of high-dimensional phenomena such as phase transitions in estimation \cite{donoho2009message,barbier2019optimal}, computational bottlenecks \cite{celentano2020estimation,celentano2022fundamental,montanari2024exceptional}, and benign overfitting \cite{hastie2022surprises,cheng2024dimension}. 

Consider the standard supervised learning setting in which we observe $n$ training samples $\{ (y_i,\bx_i) \}_{i \leq n}$ drawn i.i.d.~from an unknown distribution, where $\bx_i \in \R^d$ are covariate (input) vectors and $y_i \in \R$ are responses (or labels). The goal is to construct a predictor $\hat f$ with small test error from a parametric model class $\{ f(\bx;\btheta) : \btheta \in \bTheta\}$. For concreteness, we focus on the popular class of \textit{kernel} (or \textit{linear}) \textit{methods}. In these models, the inputs are embedded into a feature space through a (possibly stochastic) featurization map $\phi:\R^d \to \R^p$, and the predictor is taken to be linear in the features: 
\begin{equation}\label{eq:intro_linear_predictor}
f(\bx;\btheta) = \< \btheta , \phi (\bx ) \> , \qquad \btheta \in \R^p.
\end{equation}
The parameter $\btheta$ is estimated by minimizing the regularized empirical risk:
\begin{equation}\label{eq:intro_ERM}
\hbtheta = \argmin_{\btheta \in \R^p} \hcR_{n,p} (\btheta ), \qquad \quad \hcR_{n,p} (\btheta ) := \frac{1}{n} \sum_{i \in [n]} \ell (y_i, \< \btheta, \phi (\bx_i)\> )  + \frac{\lambda}{2} \| \btheta \|_2^2,
\end{equation}
where $\ell:\R \times \R \to \R_{\geq 0}$ is a loss function and $\lambda \geq 0$ is a regularization parameter. The performance of the resulting predictor is evaluated through its test error
\begin{equation}\label{eq:intro_test_error}
\cR_{\test} (\hbtheta) = \E_{y,\bx} [ \ell_{\test} (y, \< \hbtheta, \phi (\bx) \>)],
\end{equation}
for a test loss $\ell_{\test} : \R \times \R \to \R_{\geq 0}$ that may differ from the training loss. The method \eqref{eq:intro_ERM} covers many popular approaches, including classical linear and logistic regression, and kernel and random feature models.

We are interested in understanding the high-dimensional behavior of \eqref{eq:intro_ERM} when $n,d,p$ are all large. This analysis is challenging for two reasons. First, the estimator $\hbtheta$ is random and defined only implicitly as the solution  of a high-dimensional optimization problem. Second, the featurization map $\phi (\bx)$ may introduce complex dependencies among the feature coordinates. A fruitful approach addresses these difficulties by replacing the (potentially complicated) features $\phi(\bx_i)$ with a suitably constructed Gaussian model, and assuming that the asymptotic behavior of the ERM solution is preserved under this substitution. We refer to this approach as the \textit{Gaussian Equivalent Theory} (GET).  Specifically, assume the labels are generated according to $y_i = \eta ( f_* (\bx_i) ; \eps_i)$ where $\eta : \R^m \times \R \to \R$ is a link function, $f_*(\bx) \in \R^m$ is a latent signal (with $m = O(1)$) and $\eps_i$ is independent noise. GET posits that the joint distribution of $(f_* (\bx),\phi(\bx))$ can be approximated by a Gaussian vector $(f^\sG,\bz^\sG) \sim \normal (\bmu^\sG,\bSigma^\sG)$ with matching first two moments
\begin{equation}\label{eq:GET_model}
\bmu^\sG = \E_{\bx} \left[ \begin{pmatrix}
    f_*(\bx)\\
    \phi(\bx)
\end{pmatrix} \right]\in \R^{m+p},  \qquad\quad \bSigma^\sG = \Cov_\bx \left[ \begin{pmatrix}
    f_*(\bx)\\
    \phi(\bx)
\end{pmatrix} \right] \in \R^{(m+p)\times (m+p)},
\end{equation}
and that the asymptotic training and test errors of \eqref{eq:intro_ERM} remain unchanged when the original data $\{(f_*(\bz_i),\phi(\bx_i))\}_{i \leq n}$ are replaced by their Gaussian counterparts $\{ (f_i^\sG,\bz_i^\sG)\}_{i \leq n}$.

Importantly, Gaussian features often allow for exact asymptotic analyses of \eqref{eq:intro_ERM} through powerful tools such as Gaussian comparison inequalities \cite{gordon1985some,stojnic2013framework,oymak2013squared,thrampoulidis2015gaussianminmaxtheorempresence,thrampoulidis2015regularized,thrampoulidis2018precise}, Approximate Message Passing (AMP) \cite{donoho2009message,Bayati_2011,bayati2011lasso,Bayati_2015,donoho2016high,sur2019likelihood}, Dynamical Mean-Field Theory (DMFT) \cite{Zdeborov__2016,celentano2021high,bordelon2022selfconsistentdynamicalfieldtheory,gerbelot2024rigorous}, or Kac-Rice formula \cite{auffinger2011randommatricescomplexityspin,maillard2020landscape,asgari2025local}. 
Studying these Gaussian models has led to a number of remarkable insights, for instance in determining optimal losses and regularizers \cite{donoho2016high,celentano2022fundamental,aubin2020generalization},  debiasing statistical procedures
\cite{sur2019likelihood,celentano2023lasso}, or characterizing phase transitions \cite{donoho2009message,barbier2019optimal,sur2019likelihood,candes2020phase} and double descent phenomena  \cite{deng2022model,hastie2022surprises,montanari2025generalization,hastie2022surprises,zhou2023uniform}. Due to its effectiveness, GET has become a common modeling assumption or preprocessing step, explicitly or implicitly adopted in various recent works \cite{loureiro2021learning,cui2023bayes,defilippis2024dimension,atanasov2025scalingrenormalizationhighdimensionalregression,montanari2025dynamicaldecouplinggeneralizationoverfitting}.

This naturally leads to the question:
    \emph{Under what conditions on the model \eqref{eq:intro_ERM} does GET hold?} Early results establishing such universality were limited to feature maps with independent coordinates  \cite{korada2011applications,montanari2017universality,panahi2017universal}, which are far from the complex dependencies arising in practice. In a remarkable series of recent work, \cite{huUniversalityLawsHighdimensional2022,montanariUniversalityEmpiricalRisk2022,montanari2023universality} extended GET to significantly richer feature maps $\phi(\bx)$ with dependency structures. These results have now been established in several important models, including random features in the linear scaling regime \cite{Goldt_2020,meiGeneralizationErrorRandom2022,huUniversalityLawsHighdimensional2022}, certain neural tangent models \cite{montanariUniversalityEmpiricalRisk2022}, Gaussian mixtures with random labels \cite{gerace2024gaussian}, and kernel and random feature ridge regression on the sphere in the polynomial scaling \cite{xiao2022precise,meiGeneralizationErrorRandom2022a,hu2024asymptotics,misiakiewicz2024non}. Extensive numerical studies further demonstrate that Gaussian equivalent models provide remarkably accurate predictions even for complex feature maps, such as trained neural networks, and real data \cite{bordelon2020spectrum,jacot2020kernel,loureiro2021learning,goldt2022gaussian,wei2022more,ba2022high,defilippis2024dimension}.

Despite this success, GET can also fail to describe the quantitative behavior of \eqref{eq:intro_ERM} even in simple settings. Figure \ref{fig:quad_vs_non_quad_losses} illustrates this breakdown: whether GET holds depends jointly on the loss functions and the response $y$. For a ``random polynomial'' label $y_{\rm R}$, GET always holds. In contrast, for a single-index label $y_{\rm SI}$, GET fails as soon as either the training loss or the test loss is non-quadratic. Figure \ref{fig:marginal_pdf} further clarifies this phenomenon by plotting the marginal distribution of the fitted predictor $\<\hbtheta, \phi(\bx)\>$ on test data. For $y_{\rm R}$, the marginal is approximately Gaussian---a necessary condition for GET to hold. However, for $y_{\rm SI}$, the resulting marginal is non-Gaussian. 

In this paper, our goal is to investigate this non-universal behavior in Empirical Risk Minimization. In particular, we aim to clarify the precise scope of validity of GET and how to correct it when it fails. To this end, we focus on the \textit{Random Feature} (RF) model \cite{balcan2006kernels,Rahimi2007}
\begin{equation}\label{eq:intro_RF_featurization}
\phi_{\RF} ( \bx) := \sigma ( \bW \bx) = \big(\sigma (\<\bw_1,\bx\>), \ldots , \sigma (\<\bw_p,\bx\>) \big) \in \R^p, 
\end{equation}
where the weights $\bW = (\bw_j)_{j \in [p]} \in \R^{p \times d}$ are drawn i.i.d.~uniformly from the unit sphere and $\bx \sim \normal (0,\id_d)$. Although GET has been shown to hold for RF models in the linear scaling regime $n\asymp p \asymp d$ \cite{huUniversalityLawsHighdimensional2022,gerace2020generalisation,goldt2022gaussian,montanariUniversalityEmpiricalRisk2022}, we demonstrate that GET fails under more general scaling regimes as soon as the latent signal $f_*(\bx)$ presents some low-dimensional structure. In such settings, certain non-Gaussian components survive in the high-dimensional limit and are not captured by a purely Gaussian surrogate.

Our analysis focuses on the quadratic scaling regime $n \asymp p \asymp d^2$. We show that when the latent signal depends on a low-dimensional projection $\bW_{*}^\sT \bx \in \R^s$ (e.g., multi-index models), GET fails in this regime as soon as either the train or test loss is non-quadratic. 
To remedy this, we introduce a \textit{Conditional Gaussian Equivalent} (CGE) model that conditions on this low-dimensional signal subspace. We show that this model reduces to appending a small number of non-Gaussian features to an otherwise high-dimensional Gaussian model. In particular, this preserves the tractability of the Gaussian framework---e.g., it is amenable to CGMT \cite{oymak2013squared,thrampoulidis2015regularized}. We prove that CGE yields the correct asymptotic train and test errors for the RF model \eqref{eq:intro_RF_featurization} in the quadratic scaling regime for a large class of target functions. More broadly, we expect this CGE approach to extend to more general models and scalings, and we point to our companion papers \cite{wen2025asymptotics,wen2025empirical}.

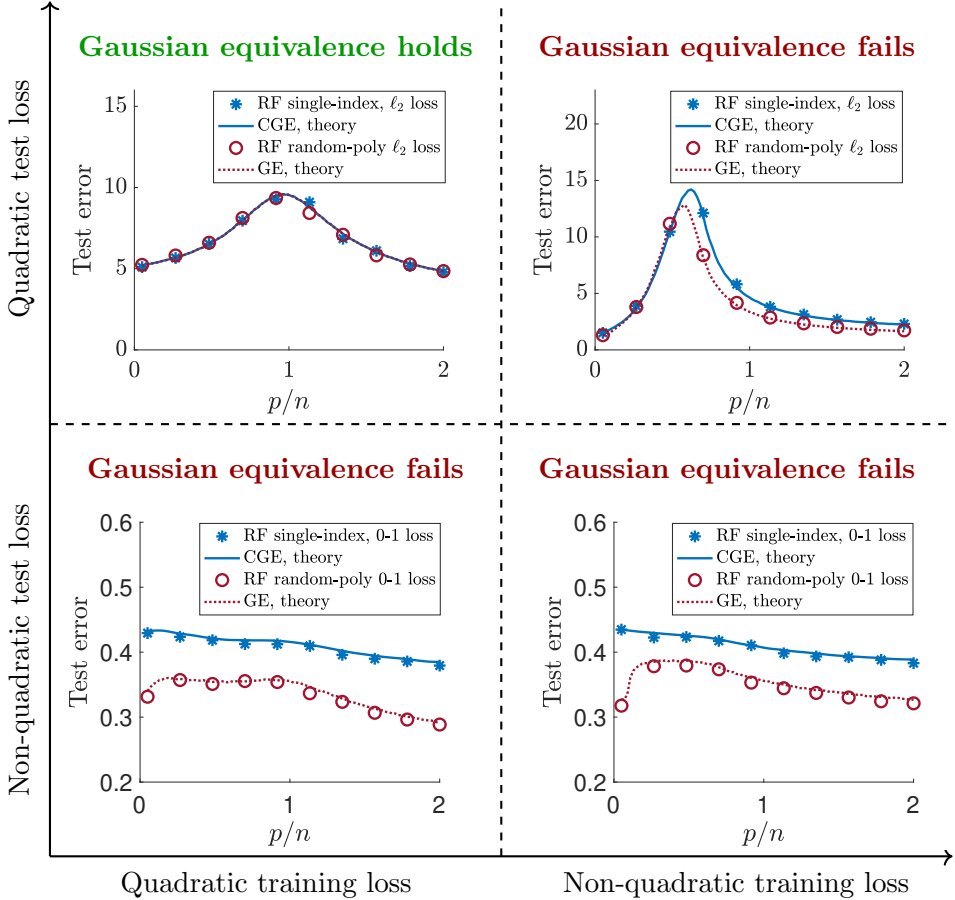
\begin{figure}[t!]

\centering
\begin{tikzpicture}[scale=2.5]

\draw[dashed, thick] (-2.4,0) -- (2.4,0) node[right] {};
\draw[dashed, thick] (0,-2.30) -- (0,2.25) node[above] {};
\draw[->, thick] (-2.4,-2.30) -- (2.4,-2.30) node[right] {};
\draw[->, thick] (-2.4,-2.30) -- (-2.4,2.25) node[right] {};

\node at (-1.25, -2.45) {Quadratic training loss};

\node at (1.25, -2.45) {Non-quadratic training loss};

\node[rotate=90] at (-2.55, 1.2) {Quadratic test loss};

\node[rotate=90] at (-2.55, -1.2) {Non-quadratic test loss};

\node at (1.2, 2.0) { \textbf{ \textcolor[rgb]{0.6, 0.0, 0.0}{Gaussian equivalence fails} } };

\node at (-1.2, 2.0) { \textbf{ \textcolor[rgb]{0.0, 0.6, 0.0}{Gaussian equivalence holds} } };

\node at (-1.2, -0.25) { \textbf{ \textcolor[rgb]{0.6, 0.0, 0.0}{Gaussian equivalence fails} } };

\node at (1.2, -0.25) { \textbf{ \textcolor[rgb]{0.6, 0.0, 0.0}{Gaussian equivalence fails} } };

\node at (1.25,0.85) [anchor=center] {
  \begin{subfigure}[t]{5.5cm}
    \centering
    \includegraphics[width=\linewidth]{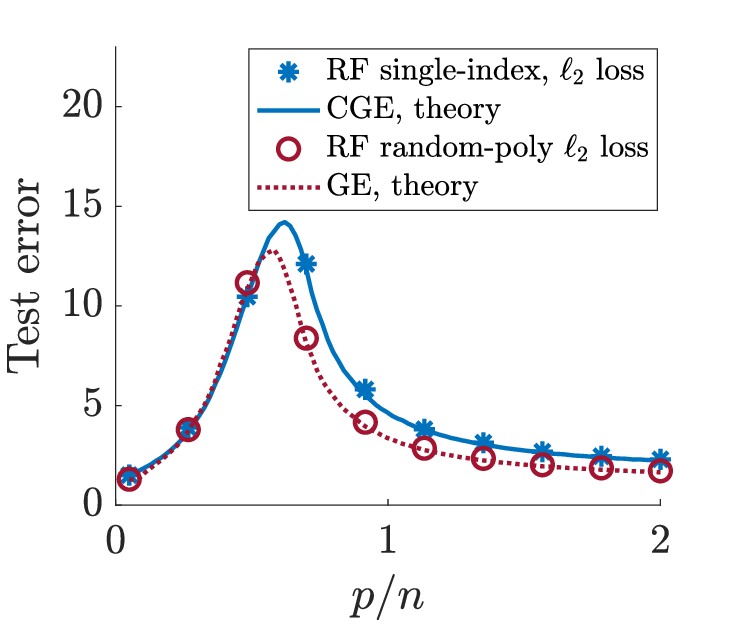}
    \caption*{}
  \end{subfigure}
};

\node at (-1.20,0.85) [anchor=center] {
  \begin{subfigure}[t]{5.5cm}
    \centering
    \includegraphics[width=\linewidth]{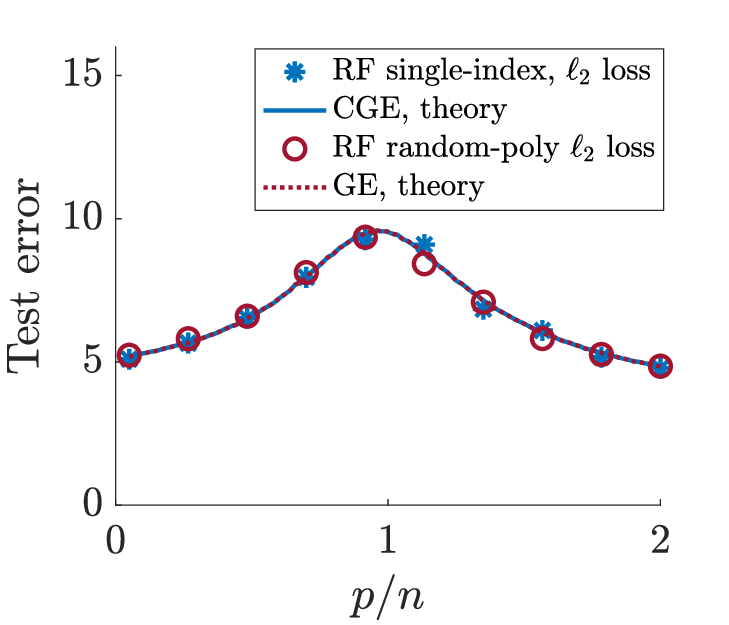}
    \caption*{}
  \end{subfigure}
};

\node at (-1.22,-1.45) [anchor=center] {
  \begin{subfigure}[t]{5.5cm}
    \centering
    \includegraphics[width=\linewidth]{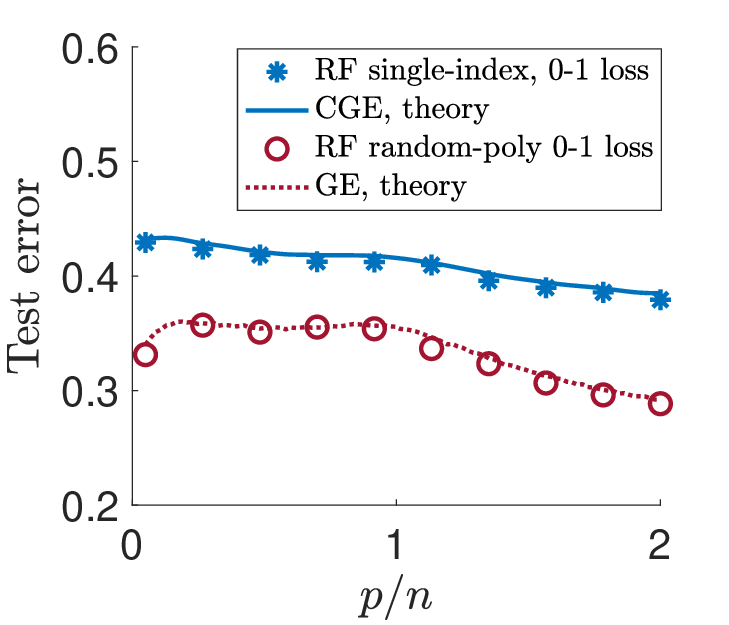}
    \caption*{}
  \end{subfigure}
};

\node at (1.3,-1.45) [anchor=center] {
  \begin{subfigure}[t]{5.5cm}
    \centering
    \includegraphics[width=\linewidth]{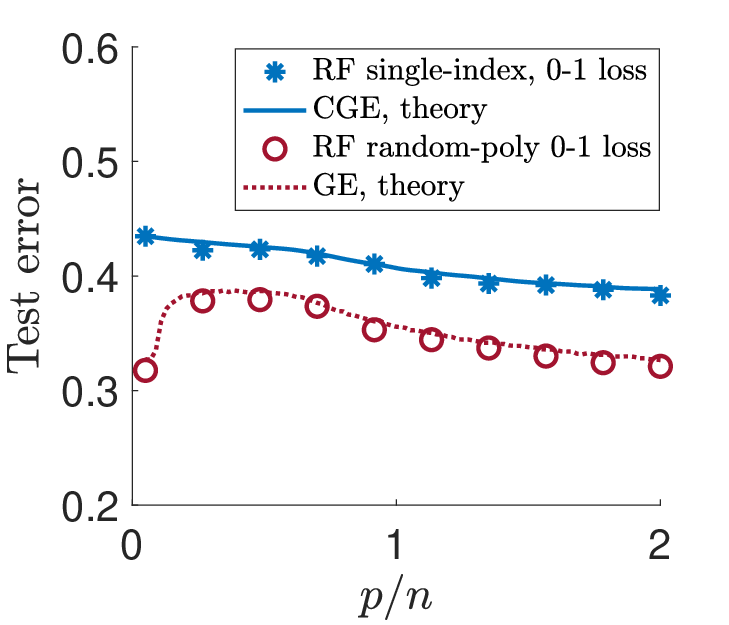}
    \caption*{}
  \end{subfigure}
};

\end{tikzpicture}
\caption{Universality and non-universality of the test error, with each quadrant representing a combination of either $\ell_2$ loss (quadratic) or hinge loss (non-quadratic) for the training loss and either $\ell_2$ loss (quadratic) or 0-1 loss (non-quadratic) for the test loss. 
We choose $d=50$, $n/d^2 = 0.5$, $\lambda = 10^{-3}$ and $\sigma(x) = \max\{0,x\}$, and consider two different responses: (i) a ``single-index'' response: $y_{\rm SI} =  1 + 2 \He_2(\< \bu_{*}, \bx\>) + \He_3(\<\bu_{*},  \bx\>)$; and (ii) a ``random-poly'' response: $y_{\rm R} =  1 + 2 \bbeta_{ 2}^\sT \bh_{2} (\bx) + \bbeta_{ 3}^\sT \bh_{3} (\bx)$, where $\bbeta_{ k} \sim \Unif \left( \S^{B_{d,k} - 1} \right)$. Both responses result in the same GE model predictions.
\label{fig:quad_vs_non_quad_losses}} 
\end{figure}

\subsection{An example: regression on quadratic polynomials}
\label{sec:intro_example}

Before summarizing our results, we illustrate how non-Gaussian behavior naturally arises in this setting. For simplicity, consider the example of regression on all degree-2 Hermite polynomials. The corresponding feature map is
\begin{equation}\label{eq:intro_Hermite_2_regression}
\phi_{\rm quad} (\bx) = \left\{ \He_2 (x_i) = \frac{1}{\sqrt{2}}(x_i^2 - 1) \; :\; i \in [d] \right\} \cup \left\{ x_i x_j \; : \; i<j \in [d] \right\} \in \R^{B_{d,2}},
\end{equation}
where $B_{d,2} = \frac{d(d+1)}{2}$.
Informally, this model can be viewed as the kernel limit\footnote{To obtain this precise kernel limit, take $\bw \sim \normal (0,\bI_d)$ and $\sigma (\bx;\bw) = \frac{1}{2}(\< \bx,\bw\>^2 - \| \bx\|_2^2 - \| \bw \|_2^2 +d)$.} ($p \to \infty$) of the RF model \eqref{eq:intro_RF_featurization} with activation $\sigma = \He_2$. The feature map \eqref{eq:intro_Hermite_2_regression} and linear predictor $f ( \bx; \btheta) = \< \btheta,\phi_{\rm quad} (\bx) \>$ can be written equivalently in matrix form as 
\begin{equation}
\phi_{\rm quad}(\bx) = \frac{1}{\sqrt{2}} ( \bx \bx^\sT - \id_d) \in \R^{d \times d}, \qquad f ( \bx; \btheta) = \frac{1}{\sqrt{2}} \big( \bx^\sT \bB \bx - \Tr(\bB) \big), 
\end{equation}
where we identify $\btheta : = \bB \in {\rm Sym}_d$ (the space of symmetric $d \times d$ matrices). 

Assume the response is generated by $y = \eta(f_*(\bx)) = \<\btheta_*, \phi_{\rm quad}(\bx)\>$ with $\| \btheta_* \|_2 = 1$ (noiseless labels) and we fit this data using ridge regression:
\begin{equation}
    \hat \btheta = \argmin_{\btheta \in \R^{B_{d,2}}} \left\{ \frac{1}{n}\sum_{i=1}^n \left(y_i - \<\btheta, \phi_{\rm quad} (\bx_i) \> \right)^2 + \frac{\lambda}{2} \| \btheta\|_2^2  \right\}.
\end{equation}
The solution can be expressed explicitly in terms of the resolvent of the feature matrix, and one can show that 
\begin{equation}
    \hat \btheta = \alpha \btheta_* + \hat \btheta_\perp, \qquad \| \hat \bB_\perp \|_\op = o_d(1),
\end{equation}
where $\hat \btheta_\perp$ denotes the component orthogonal to $\btheta_*$ and $\hat \bB_\perp $ its matrix representation.

 Consider two representative responses (let $\bB_*$ be the matrix representation of $\btheta_*$):
\begin{itemize}
\item[(1)] $y_{\rm R} = f_{\rm R} (\bx) = \< \btheta_*,\phi_{\rm quad}(\bx)\>$  with $\| \bB_* \|_\op = o_d(1)$ (here, $\eta (x) = x$): In this case,
\begin{equation}\label{eq:quad_example_fR_marginal_law}
{\rm Law} \left( f_{\rm R} (\bx), \< \hat \btheta , \phi_{\rm quad}(\bx)\> \right) \approx {\rm Law} \left(  G_1, \alpha G_1 + \| \hbtheta_\perp \|_2 G_2 \right),
\end{equation}
where $G_1,G_2 \sim \normal (0,1)$ independently. The signal and the fitted model are asymptotically jointly Gaussian, and one can indeed show that the GE model \eqref{eq:GET_model} correctly predicts the train and test errors.

\item[(2)] $y_{\rm SI} = \eta( \<\bu_*,\bx\>)= \He_2 (\< \bu_* ,\bx\>)$ with $\| \bu_* \|_2 =  1$ (here, $\eta (x) = \He_2 (x)$ and $\btheta_*:=\bB_* = \bu_*\bu_*^\sT$): One can show
\begin{equation}
{\rm Law} \left( \< \bu_* ,\bx\>, \< \hat \btheta , \phi_{\rm quad}(\bx)\> \right) \approx {\rm Law} \left( G_1, \alpha \frac{1}{\sqrt{2}}(G_1^2 - 1) + \| \hbtheta_\perp \|_2 G_2 \right).
\end{equation}
In this case, the joint law contains a non-Gaussian chi-squared component aligned with the rank-one structure of the target. Only means and covariances match those of the GE model \eqref{eq:GET_model}, but the full distribution does not. Consequently, a non-quadratic test loss produces different asymptotic test errors in the quadratic model and Gaussian surrogate.
\end{itemize}

We illustrate this non-Gaussian additive term in the marginal for the RF model in Figure \ref{fig:marginal_pdf}, where we plot the marginal distribution of $\< \hat \btheta, \phi_\RF (\bx)\>$ using the same two target functions as above. In Figure~\ref{fig:quad_vs_non_quad_losses}, we consider the four possible combinations of quadratic and non-quadratic train and test losses, and plot the resulting test error for each case and target function. For $y_{\rm R}$ without rank-$1$ component, the GET always yields the correct prediction. In contrast, for the single-index response $y_{\rm SI}$, GET fails as soon as either the train loss or the test loss is non-quadratic.

\begin{figure}[t]
\centering
\begin{subfigure}[b]{0.48\textwidth}
     \centering
     \includegraphics[width=\textwidth]{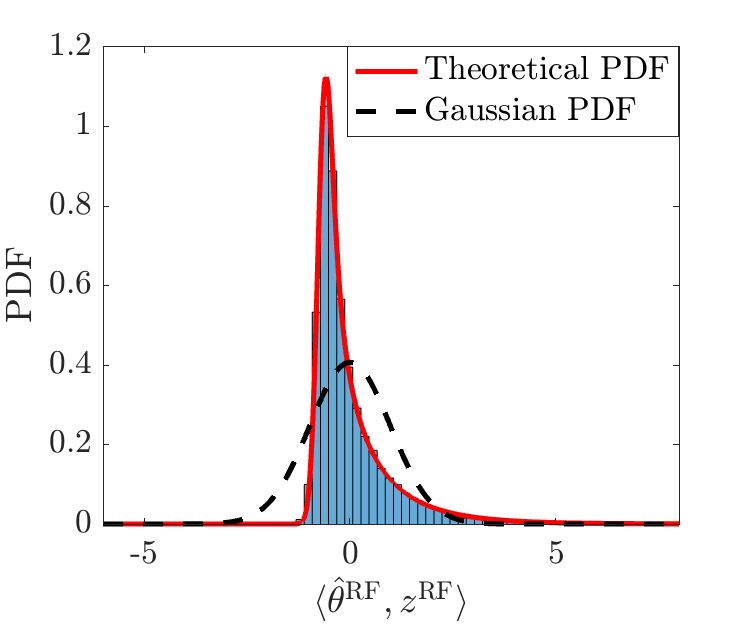}
     \caption{$y_{\rm SI} = \He_2(\bu_{*}^\sT \bx)$, $\;\| \bu_{*} \|_2 = 1$. \label{fig:marginal_chi_dist}}
 \end{subfigure}
 \hfill
 \begin{subfigure}[b]{0.48\textwidth}
     \centering
     \includegraphics[width=\textwidth]{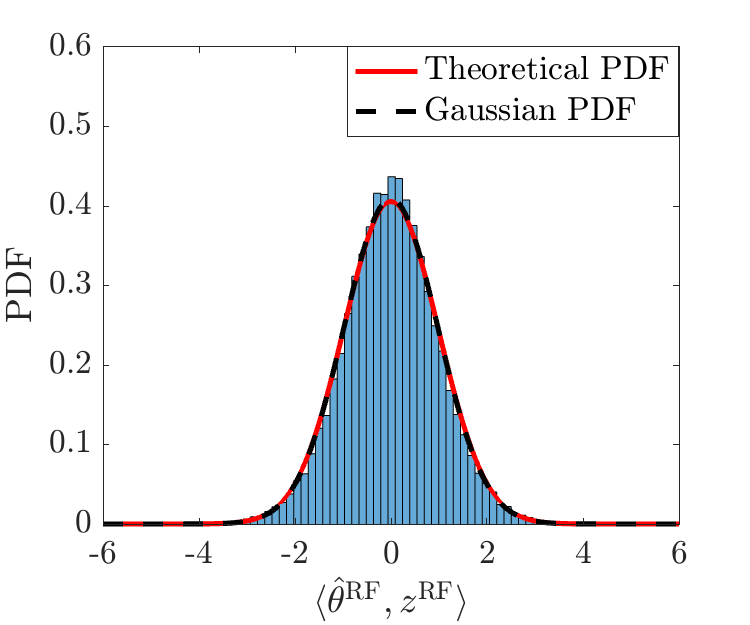}
     \caption{$y_{\rm R} = \bbeta_{*}^\sT \phi_{\rm quad} (\bx)$, $\;\bbeta_{*} \sim \Unif \left( \S^{B_{d,2} - 1} \right)$. \label{fig:marginal_gauss_dist}}
 \end{subfigure}
\caption{Marginal distributions of $\<\hbtheta,\phi_{\RF} (\bx)\>$ for two different responses $y_{\rm SI}$ and $y_{\rm R}$, where $\hbtheta$ is the solution of \eqref{eq:intro_ERM} with squared loss. The histogram corresponds to the empirical distribution; the curves correspond to the predictions from the CGE model (solid red line) and the GE model (black dashed line). 
We choose $d=50$, $n/d^2 = 1$, $p/d^2 = 0.5$, $\lambda = 10^{-3}$, and $\sigma(x) = x^2$. }
\label{fig:marginal_pdf}
\end{figure}

To address this issue, we show that a simple modification of the Gaussian model recovers the correct asymptotics for $y_{\rm SI}$ in all cases. We condition explicitly on the signal direction and define
\begin{equation}\label{eq:CGE-quadratic-example}
\bz^\sCG :=  \He_2 ( \<\bu_*,\bz\>) \btheta_* + (\id_{B_{d,2}} - \btheta_* \btheta_*^\sT) \bg,
\end{equation}
where $\bg \sim \normal (0,\id_{B_{d,2}})$ is independent of $\bz$. This is a Gaussian model augmented with a single non-Gaussian component (a centered chi-square term) in the direction $\btheta_*$. The main result of this paper shows that this \textit{Conditional Gaussian Equivalent} model yields the correct asymptotic train and test errors for a large class of losses and responses and for the (more complex) RF model \eqref{eq:RF_model_class}.

This construction extends to general polynomial scaling by conditioning on the appropriate higher-order Hermite chaos components (see Remark \ref{rmk:general-polynomial-scaling}); we postpone the presentation of this general Conditional Gaussian universality principle to a follow-up paper \cite{wen2025empirical}.

\subsection{Summary of main results}

Our main results characterize the high-dimensional behavior of the ERM solution \eqref{eq:intro_ERM} for the Random Feature (RF) model \eqref{eq:intro_RF_featurization} in the quadratic scaling regime $n \asymp p \asymp d^2$, in terms of a \textit{Conditional Gaussian Equivalent} (CGE) model. 

Consider responses of the form $y = \eta( f_*(\bx); \eps)$, where  $\eta : \R^m \times \R \to \R$ is a link function, $f_* (\bx) \in \R^m$ is a latent signal, and $\eps$ is independent noise. A key subtlety is that the same joint law $(y,\bx)$ may admit multiple equivalent decompositions $(\eta,f_*)$, while the Gaussian Equivalent (GE) model \eqref{eq:GET_model} depends on this choice. We resolve this ambiguity by fixing a canonical representation for the response that separates the different Wiener chaos contributions: 
\begin{equation}\label{eq:label_multi-index-chaos_intro}
f_*(\bx) = \{ \bxi_1 (\bx), \ldots, \bxi_{D'} (\bx) \},
\end{equation}
where each $\bxi_{k} = ( \xi_{ki} (\bx))_{i \in [s_{k}]}$ consists of coordinates $\xi_{ki} (\bx)$ that are pure degree-$k$ Hermite polynomials, with approximate Gaussian marginals ${\rm Law} (\xi_{ki} (\bx)) \approx {\rm Law} ( \normal (0,1) )$ (e.g., $f_{\rm R}$ in \eqref{eq:quad_example_fR_marginal_law} is an example of degree-$2$ chaos with approximate Gaussian marginal). The representation \eqref{eq:label_multi-index-chaos_intro} encompasses a wide class of responses: (1) Multi-index models $f_*(\bx) = \bxi_1 (\bx) = \bTheta_*^\sT \bx$ with $\bTheta_* \in \R^{d \times m}$; (2) Regression functions with random coefficients; (3) A broad class of deterministic regression functions (see discussion in Section \ref{sec:setting_CGE}). 

The CGE model conditions on the linear chaos components $\bxi_1 (\bx)$ in the response and compares the RF model to the GE model \eqref{eq:GET_model} \textit{conditional} on $\bxi_1 (\bx)$. For example, consider the case $\bxi_1 (\bx) = \<\bu_*,\bx\>$ with $\| \bu_*\|_2 =1$. In this setting, the CGE model takes the form 
\[
f^{\sCG} = \{\<\bu_*,\bx\>, \Tilde f^\sG \}, \qquad \bz^{\sCG} = \mu_1 \br_1 \< \bu_*, \bx\> + \mu_2 \br_2 \He_2 (\< \bu_*, \bx\>) + \Tilde \bz^\sG,
\]
where $(\Tilde f^\sG,\Tilde \bz^\sG)$ is independent of $\<\bu_*,\bx\>$ and
\[
\br_1 := ( \<\bw_j, \bu_* \>)_{j \in [p]},  \quad \br_2 := (\<\bw_j,\bu_* \>^2)_{j \in [p]}, \quad \mu_1 = \E[\sigma (G)G], \quad \mu_2 = \E[\sigma (G) \He_2 (G)],
\]
with $G \sim \normal (0,1)$. The pair $(\Tilde f^\sG, \Tilde \bz^\sG)$ is the GE model \eqref{eq:GET_model} associated to the remaining components $\{\bxi_2(\bx), \ldots, \bxi_{D'} (\bx) \}$ and $\sigma (\bW_{-\bu_*} \bx)$, where $\bW_{-\bu_*} = \bW (\id_d - \bu_*\bu_*^\sT)$. We postpone the full and explicit presentation of the CGE model to Section \ref{sec:setting_CGE}.

The main contributions of this paper are:
\begin{enumerate}
    \item[(1)] \textit{Conditional Gaussian Universality.} We prove that the CGE model correctly captures the asymptotic train and test errors of the ERM solution \eqref{eq:intro_ERM} in the quadratic scaling regime, for the RF model \eqref{eq:RF_model_class}, responses \eqref{eq:label_multi-index-chaos_intro}, and a broad class of loss and activation functions.

    \item[(2)] \textit{Exact asymptotics.} Using the convex Gaussian minimax theorem (CGMT) \cite{stojnic2013framework,oymak2013squared,thrampoulidis2015regularized}, we derive precise asymptotics for the train and test errors in the CGE model. These predictions are stated in terms of an explicit low-dimensional minimax problem that can be evaluated numerically. See Figure \ref{fig:quad_vs_non_quad_losses} for an illustration of these theoretical predictions. 

    \item[(3)] \textit{Applications.} We illustrate our framework by analyzing binary classification in the quadratic scaling regime with random features and single-index responses. We obtain sharp characterizations for (a) the decision boundary (given by two hyperplanes), (b) phase transition in the existence of interpolating solutions, and (c) double descent and benign overfitting. None of these phenomena are correctly captured by GET.
\end{enumerate}

Because of space constraints, we only summarize points (2) and (3) in this paper, and defer full derivations of the asymptotics as well as detailed applications to the companion paper~\cite{wen2025asymptotics}.

We are now ready to state what our results imply about the scope of validity of GET. In the RF model under the quadratic scaling regime, we find two key phenomena:
\begin{itemize}
    \item[(i)] \textit{Linear chaos in the response:} If the response contains no linear chaos component, then the CGE model collapses to the GE model. In particular, multi-index models always require a correction to GE, but even more complicated targets---e.g., $y= \< \bu_*,\bx\> g(\bx)$ for $g(\bx)$ a pure degree-$k$ Hermite polynomial---can also lead to non-universality.

    \item[(ii)] \textit{Quadratic versus non-quadratic losses:} When both the train and test losses are quadratic, the asymptotics of CGE and GE coincide even in the presence of linear chaos. However, changing either the train or the test loss to a non-quadratic function causes GET to break down.
\end{itemize}
Figure \ref{fig:quad_vs_non_quad_losses} illustrates this Gaussian/non-Gaussian universality for the single-index target $f_{\rm SI}$ and for the response $f_{\rm R}$ without linear chaos (in which case CGE and GE always agree), across both quadratic and non-quadratic training and test losses.

From a technical standpoint, establishing conditional universality in our setting presents significant challenges. In particular, our proof must contend with: (1) the non-sub-Gaussianity of $\phi_{\mathrm{RF}}$, (2) dependencies between Wiener chaos components of different orders; and (3) the presence of diverging singular values (spikes) in the random features arising in the quadratic scaling regime. We overcome these obstacles through a careful two-phase Lindeberg swapping argument and by leveraging a Malliavin-Stein argument to establish a quantitative ``partial'' CLT for Wiener chaos expansions---following the approach of the celebrated ``Fourth Moment Theorem'' \cite{nualart2005central,nourdin2009stein,Nourdin_2010,nourdin2012normal}.

The remainder of the paper is organized as follows.  Section \ref{sec:main_GE_CGE_RF} introduces our setting and describes precisely the GE and CGE models for RF features in the quadratic scaling regime. Section \ref{sec:main_results} states our main conditional Gaussian universality results for the train and test errors of the RF model, and presents several applications. Finally, Section \ref{sec:proof_outline} outlines the proof strategy for our main results. Technical details and full proofs are deferred to the appendices.

\subsection{Additional related work}

As mentioned in the introduction, several papers have established universality of ERM beyond features with independent entries. For a broad class of (potentially non-convex) loss functions, \cite{montanariUniversalityEmpiricalRisk2022} established ERM universality using a smooth free energy approximation, assuming sub-Gaussian data and pointwise normality. \cite{montanari2023universality} subsequently extended this to max-margin classifiers by leveraging duality and relating the problem back to a surrogate ERM, proving universality via   verifying the assumptions   of \cite{montanariUniversalityEmpiricalRisk2022}. Closely related to our approach, \cite{huUniversalityLawsHighdimensional2022} utilized the Lindeberg swapping technique \cite{chatterjee2006generalization,korada2011applications} for random feature models with  convex losses  and ridge penalty, relying on a careful localization of the ERM solution via a leave-one-out argument. 
While we follow the approach of \cite{huUniversalityLawsHighdimensional2022}, our setting introduces substantial new technical challenges which we detail in Section \ref{sec:proof_outline}. We further note that several works, including \cite{meiGeneralizationErrorRandom2022, hu2024asymptotics, xiao2022precise,meiGeneralizationErrorRandom2022a, misiakiewicz2024non,pandit2024universality}, have leveraged random matrix theory to show universality of kernel and random feature ridge regression for arbitrary responses and general polynomial scaling regimes. However, as noted in this paper, this universality stops holding as soon as one changes either the train or test loss to a non-quadratic loss.  

Several works have documented the breakdown of Gaussian universality in the context of mixture data. For Gaussian mixture inputs, \cite{pesce2023gaussian} showed that a covariance-matched single Gaussian surrogate fails to reproduce the asymptotic risk in generalized linear estimation if the teacher vector correlates with the data or if the mixture is heteroscedastic, though universality of the training error persists under square loss with vanishing regularization in underparametrized settings. On the other hand, \cite{dandi2023universalitylawsgaussianmixtures} proposed a conditional one-dimensional CLT to argue that in general mixtures, the ERM performance is governed solely by class-conditional means and covariances.
Scrutinizing the validity of this conditional CLT, \cite{mai2025breakdowngaussianuniversalityclassification} identified the  boundaries of this Gaussian-mixture universality for a specific  model.   They also showed that while ERM with square loss yields a classifier statistically indistinguishable from one trained on the equivalent Gaussian mixture, the resulting generalization error on the true distribution is different from the Gaussian mixture data.  Our work differs fundamentally from these studies in both the model setting and the source of non-universality. We exhibit a breakdown of Gaussian equivalence even with isotropic Gaussian inputs: specifically, beyond the linear scaling, the low-dimensional structure in the response preserves non-trivial Hermite components that are not captured by a purely Gaussian model.

From a technical standpoint, establishing the CLT condition differs from previous approaches \cite{goldt2022gaussian,huUniversalityLawsHighdimensional2022,montanari2023universality} which used (sub-)Gaussian specific tools. Instead, we leverage the Malliavin-Stein approach \cite{nualart2005central,nourdin2009stein,nourdin2012normal,azmoodeh2021malliavinsteinmethodsurveyrecent,Nourdin_2010}  to establish a quantitative ``partial'' CLT for Wiener chaos expansions. Malliavin calculus has proven to be a powerful probabilistic tool in the statistics and machine learning literature, as evidenced by   recent applications in \cite{peccati2013gammalimitsustatisticspoisson,schroder2024asymptoticslearningdeepstructured,celli2025entropicboundsconditionallygaussian,caponera2019asymptoticssphericalfunctionalautoregressions,neufeld2024solvingstochasticpartialdifferential, mirafzali2025malliavincalculusapproachscore,mirafzali2025malliavincalculusscorebaseddiffusion,pidstrigach2025conditioningdiffusionsusingmalliavin}.

\section{The Conditional Gaussian Equivalent model}
\label{sec:main_GE_CGE_RF}

We begin by describing our setting and the Conditional Gaussian Equivalent (CGE) model for random features in the quadratic polynomial scaling regime. In Section \ref{sec:setting_ERM}, we set up the basic notations and definitions for the empirical risk minimization problem we consider in this paper. Section \ref{sec:setting_GET} introduces the Wiener chaos expansion of the random features and presents the Gaussian Equivalent (GE) model. Section \ref{sec:setting_CGE} then develops the CGE model by modifying the GE model with additional non-Gaussian components.

\subsection{Empirical risk minimization with Random Features}
\label{sec:setting_ERM}

Fix an integer $m \in \naturals$. Throughout the paper, we assume the data $(y,\bx) \in \R \times \R^d$ is generated as
\begin{equation}\label{eq:data_distribution}
\bx \sim \normal (0,\id_d), \qquad\qquad y = \eta (f_*(\bx);\eps),
\end{equation}
for some $\eta : \R^m \times \R \to \R$, target function\footnote{Formally, we consider a sequence $f_{*,d} : \R^d \to \R^m$ indexed by $d$. We suppress the dependence on $d$ for notational simplicity.} $f_* :\R^d \to \R^m$ satisfying $\E [ \| f_*(\bx) \|_2^2] <\infty$, and noise $\eps \sim \nu$ independent of $\bx$. 

We consider our predictor to be a Random Feature (RF) model of the form
\begin{equation}\label{eq:RF_model_class}
h_\RF (\bx; \btheta) = \< \btheta, \phi_{\RF} (\bx) \> , \qquad\qquad \phi_{\RF} (\bx) = \sigma (\bW \bx) = \begin{pmatrix}
    \sigma (\<\bw_1,\bx\>) \\
    \vdots\\
    \sigma (\<\bw_p,\bx\>) 
\end{pmatrix} \in \R^p,
\end{equation}
where the activation $\sigma := \sigma_d : \R \to \R$ is allowed to depend on $d$, and the weight matrix $\bW = [ \bw_1, \ldots , \bw_p ]^\sT \in \R^{p \times d}$ is fixed independently of the data. We assume the weight vectors $ \bw_1, \ldots, \bw_p $ are drawn i.i.d.~uniformly from the unit sphere \begin{equation}
    \bw_{j} \overset{\text{i.i.d.}}{\sim} \Unif (\S^{d-1} ).
\end{equation} 
For convenience, we write $\bz^\RF := \phi_\RF (\bx)$, $f^\RF := f_* (\bx)$, and $y^\RF := \eta (f^\RF;\eps)$. Denote by $\E_{\bz^\RF,f^\RF}$ the expectation over $(\bz^\RF,f^\RF) = (\phi_\RF (\bx),f_*(\bx))$ when $\bx \sim \normal (0,\bI_d)$. 

We observe $n$ i.i.d.~samples $(y_i,\bx_i)_{i \in [n]}$ from \eqref{eq:data_distribution}. Define the feature matrix, signal matrix, and noise vector:
\begin{equation}
    \begin{aligned}
\bZ^\RF = [\bz^\RF_1,\ldots,\bz^\RF_n] := [ \phi_{\RF} (\bx_1), \ldots , \phi_{\RF} (\bx_n) ] \in& \R^{p \times n}, \\
\boldf^\RF = [f^\RF_1, \ldots, f^\RF_n] := (f_{*} (\bx_1 ), \ldots , f_* (\bx_n)) \in& \R^{m \times n},\\
\beps = (\eps_1, \ldots , \eps_n) \in& \R^n.
\end{aligned}
\end{equation}

Given a training loss $\ell:\R \times \R \to \R_{\geq 0}$ and regularization parameter $\lambda > 0$, we define the \textit{empirical risk with ridge penalty} as
\begin{equation}\label{eq:empirical_risk_def}
    \hcR_{n,p} (\btheta ; \bZ,\boldf,\beps) :=  \frac{1}{n} \sum_{i \in [n]} \ell (\eta(f_i;\eps_i), \< \btheta, \bz_i\> )  + \frac{\lambda}{2} \| \btheta \|_2^2.
\end{equation}
 For brevity, we might keep the dependence on $\beps$ implicit and write $\hcR_{n,p} (\btheta ; \bZ,\boldf) := \hcR_{n,p} (\btheta ; \bZ,\boldf,\beps).$

We fit the RF model \eqref{eq:RF_model_class} by minimizing this empirical risk: 
\begin{equation}\label{eq:ERM_setting_risk}
\begin{aligned}
\hbtheta^\RF =&~ \argmin_{\btheta \in \R^p} \hcR_{n,p} (\btheta ; \bZ^\RF,\boldf^\RF,\beps).
\end{aligned}
\end{equation}
The predictor is evaluated through the \textit{test error}:
\begin{equation}
\cR_{\test} (\hbtheta^\RF; \P_{\bz^\RF,f^\RF}) = \E_{\bz^\RF,f^\RF,\eps} [ \ell_{\test} (\eta (f^\RF;\eps), \< \hbtheta, \bz^\RF \>) | \bW],
\end{equation} 
for some test loss $\ell_{\test} : \R \times \R \to \R_{\geq 0}$, where $\P_{\bz^\RF,f^\RF}$ denotes the joint law of $(\bz^\RF,f^\RF)$ for a new test sample $\bx \sim \normal (0,\bI_d)$ with same weight matrix $\bW$.  
For example, one can set the test loss equal to the training loss $\ell_{\test} (y,\hat y) := \ell (y,\hat y)$.

\begin{remark}[Beyond ridge penalty]
    For simplicity, we focus here on the ridge penalty: it is the most natural regularizer\footnote{The $\ell_2$-norm $\| \btheta\|_2$ corresponds to the RKHS norm of $ \<\btheta,\phi (\cdot)\>$ in the associated Reproducing Kernel Hilbert space.} when considering featurized data $\phi(\bx)$ as it allows for tractability of \eqref{eq:ERM_setting_risk} via the kernel trick even when $p = \infty$. Our proofs extend to general strongly convex regularizers---cf.\ the analysis in \cite{huUniversalityLawsHighdimensional2022}---but we restrict the presentation and the proof to ridge penalty for clarity. We leave the exploration of non strongly-convex regularizers to future work.
\end{remark}

\subsection{Gaussian Equivalent model for Random Features}
\label{sec:setting_GET}

Before presenting the Conditional Gaussian model, we first describe the Gaussian Equivalent (GE) model for random features. Early work in which this equivalence was observed are \cite{Goldt_2020,meiGeneralizationErrorRandom2022}. It was proven for ridge regression under proportional scaling $n \asymp p \asymp d$ \cite{meiGeneralizationErrorRandom2022}, and more recently in the polynomial scaling $n \asymp d^{k_1}$ and $p \asymp d^{k_2}$ in \cite{hu2024asymptotics}. Subsequent work has established that the equivalence holds in the proportional scaling for a broad class of losses and regularizers \cite{gerace2020generalisation,goldt2022gaussian,huUniversalityLawsHighdimensional2022,montanariUniversalityEmpiricalRisk2022}.

\paragraph*{Wiener Chaos (Hermite) expansion.} We first introduce the Hermite expansions of the target $f_*$ and feature map $\phi_\RF$. Let $\{ \He_k \}_{k \geq 0} $ denote the normalized univariate Hermite polynomials, 
\[
\E [ \He_k (G) \He_{l} (G)] = \delta_{kl}, \qquad G \sim \normal (0,1),
\]
which form an orthonormal basis of $L^2 (\R, \normal(0,1))$. For a multi-index $\bk \in \Z_{\geq 0}^d$, define the multivariate Hermite polynomial in $\R^d$ by
\begin{equation}
\He_\bk (\bx) = \prod_{j \in [d]} \He_{k_j} (x_j),
\end{equation}
so that $\{ \He_\bk \}_{\bk \in \Z_{\geq 0}}$ is an orthonormal basis of $L^2(\R^d,\normal (0,\id_d))$. 

For each $k \geq 0$, let
\begin{equation}
\bh_k (\bx) = ( \He_{\bk} (\bx) )_{\| \bk \|_1 = k} \in \R^{B_{d,k}}, \qquad B_{d,k}  = {{d+k -1}\choose{k}}.
\end{equation}
The span of functions $\{\<\bu,\bh_k(\bx)\> : \bu \in \R^{B_{d,k}}\}$ defines the degree-$k$ Wiener chaos, denoted $\cH_k$. 

\paragraph*{The Random Feature model.}
Following these notations, the target function\footnote{As before, we suppress the dependence of $f_*$ and $\bbeta_{k}$ on $d$.} admits the Hermite expansion (equality in $L^2$)
\begin{equation}\label{eq:hermite_expansion_target}
    f_* (\bx) = \sum_{k =0}^\infty \bbeta_{k}^\sT \bh_k (\bx),\qquad \bbeta_{k} = \E_{\bx} [ \bh_k (\bx) f_*(\bx)^\sT]\in \R^{  B_{d,k}\times m}.
\end{equation}
Turning to the random feature map, let the activation $\sigma$ have Hermite expansion\footnote{We allow $\sigma$ to depend on $d$, and thus the coefficients $\mu_{k} := \mu_{d,k}$. We suppress this dependence on $d$ for simplicity.}
\begin{equation}\label{eq:sigma_Hermite_expansion}
    \sigma (u) = \sum_{k =0}^\infty \mu_{k} \He_k (u), \qquad \mu_{k} := \E [ \sigma(G) \He_k (G)].
\end{equation}
For any $\bw \in \S^{d-1}$ and $k \geq 0$, one has (see Appendix \ref{sec:TechnicalBackground} for details)
\begin{equation}\label{eq:Hermite_w_x_expansion}
    \He_k (\<\bw,\bx\>) = \bq_k(\bw)^\sT \bh_k (\bx), \qquad \quad\bq_k (\bw) = \left( \frac{\sqrt{k!}} {\sqrt{k_1!k_2!\cdots k_d!}}\prod_{j \in [d]} w_j^{k_j} \right)_{\| \bk \|_1 = k} \in \R^{B_{d,k}}.
\end{equation}
Defining
\begin{equation}
    \bV_k = \bq_k (\bW) = [ \bq_k (\bw_1), \ldots , \bq_k (\bw_p)]^\sT \in \R^{p \times B_{d,k}},
\end{equation}
we obtain the Hermite expansion of the feature map:
\begin{equation}
     \phi_{\RF} (\bx) = \sigma(\bW\bx) = \sum_{k =0}^\infty \mu_{k} \bV_k \bh_k (\bx).
\end{equation}
Note in particular that $\bV_1= \bW$ is the original weight matrix. In summary:

\begin{modelbox}{Random Feature (RF) Model}
\vspace{-0pt}
\begin{equation}\label{eq:Hermite_expansion_RF_model}
\begin{aligned}
\bz^\RF
   &:= \mu_{0} \bones_p + \sum_{k =1}^\infty \mu_{k} \bV_k \bh_k (\bx),
      \\
f^\RF
   &:=\bbeta_{0}^\sT  + \sum_{k =1}^\infty \bbeta_{k}^\sT \bh_k (\bx).
\end{aligned}
\end{equation}
\end{modelbox}

\paragraph*{The Gaussian Equivalent model.} Using the expansion \eqref{eq:Hermite_expansion_RF_model} and the orthogonality of multivariate Hermite polynomials, the Gaussian Equivalent model \eqref{eq:GET_model} for $(f_*(\bx),\phi_\RF(\bx))$ takes the form 
\begin{equation}
    (f^\sG,\bz^\sG) \sim \normal (\bmu^\sG, \bSigma^\sG),
\end{equation} 
with mean and covariance
\begin{equation}\label{eq:general_GET_RF}
    \bmu^\sG = \E_{\bx} \left[ \begin{pmatrix}
        f_*(\bx) \\ \phi_{\RF} (\bx)
    \end{pmatrix}\right] = \begin{pmatrix}
    \bbeta_{0}^\sT \\
    \mu_{0} \bones_p
\end{pmatrix}, \quad  \quad \bSigma^\sG =  \begin{pmatrix} \bSigma_{11}^\sG
    & (\bSigma_{21}^\sG)^\sT \\
    \bSigma_{21}^\sG &  \bSigma_{22}^\sG
\end{pmatrix},
\end{equation}
where
\begin{equation}\label{eq:covariance_matrix_GE_RF_general}
\begin{aligned}
    \bSigma_{11}^\sG = &~ \Cov_{\bx}( f_*, f_*) = \sum_{k \geq 1} \bbeta_{k}^\sT \bbeta_{k} \in \R^{m \times m}, \\
    \bSigma_{21}^\sG =&~ \Cov_{\bx} (\phi_{\RF} ,f_*)= \sum_{k \geq 1} \mu_{k} \bV_k \bbeta_{k}  \in \R^{p \times m}, \\
    \bSigma_{22}^\sG =&~ \Cov_{\bx} (\phi_{\RF} ,\phi_{\RF})=  \sum_{k \geq 1} \mu_{k}^2 \bV_k \bV_k^\sT \in \R^{p \times p}.
\end{aligned}
\end{equation}
Equivalently, the GE model \eqref{eq:general_GET_RF} can be written as a Gaussian linear model:
\begin{modelbox}{Gaussian Equivalent (GE) Model}
\vspace{-0pt}
\begin{equation}\label{eq:general_GET_RF_expanded}
\begin{aligned}
\bz^\sG
   &:= \mu_{0} \bones_p + \sum_{k =1}^\infty \mu_{k} \bV_k \bg_k,
      \\
f^\sG
   &:=\bbeta_{0}^\sT  + \sum_{k =1}^\infty \bbeta_{k}^\sT \bg_k.
\end{aligned}
\end{equation}
\end{modelbox}
\noindent
Here, $\bg_k \sim \normal (\bzero, \id_{B_k})$ are independent across $k$. In words, the GE model simply replaces the Hermite basis $\{ \He_\bk (\bx) \}_{\bk \in \Z_{\geq 0}^d}$ with independent standard Gaussian variables.

\paragraph*{Gaussian Equivalence Theory (GET):} GET states that the empirical and test risks under the RF model \eqref{eq:Hermite_expansion_RF_model} coincide asymptotically with those of the GE model \eqref{eq:general_GET_RF_expanded}:
\begin{equation}\label{eq:GET_test_train}
\begin{aligned}
    \hcR_{n,p} (\hbtheta ; \bZ^\RF,\boldf^\RF) =&~ \hcR_{n,p} (\hbtheta^\sG ; \bZ^\sG,\boldf^\sG) + o_{d,\P}(1), \\
     \cR_{\test} (\hbtheta; \P_{\bz^\RF,f^\RF}) =&~ \cR_{\test} (\hbtheta^\sG ; \P_{\bz^\sG,f^\sG} ) + o_{d,\P}(1),
    \end{aligned}
\end{equation}
where $\hbtheta^\sG$ is the empirical risk minimizer under the GE model
\[
\hbtheta^\sG = \argmin_{\btheta \in \R^p} \;\hcR_{n,p} (\btheta ; \bZ^\sG,\boldf^\sG),
\]
and $(\boldf^\sG,\bZ^\sG)= (f_i^\sG, \bz_i^\sG)_{i \in [n]}$ are i.i.d.~samples of \eqref{eq:general_GET_RF}.

\medskip

In the covariance of $\phi_{\RF}$ \eqref{eq:covariance_matrix_GE_RF_general}, we have the identity 
\[
\bV_k \bV_k^\sT = (\bW \bW^\sT)^{\odot k},
\]
proved in Appendix~\ref{sec:TechnicalBackground}. This allows the GE model \eqref{eq:general_GET_RF_expanded} to be further simplified under specific high-dimensional scalings.

\paragraph*{Simplification in the linear scaling.} In the proportional regime $p \asymp d \asymp n$, one has\footnote{Specifically, $\| \bV_2 \bV_2^\sT - d^{-1} \bones_p \bones_p^\sT - \id_p \|_\op =o_{d,\P}(1)$ (spike absorbed in $\mu_{0}$) and $\| \bV_k \bV_k^\sT - \id_p \|_\op = o_{d,\P}(1), k \geq 3$.} $\bV_k \bV_k^\sT = (\bW \bW^\sT)^{\odot k} \approx \bI_d$ for $k \geq 2$. Consequently, the covariance $\bSigma^{\sG}$ simplifies as
\begin{equation}\label{eq:simplification_covariance_linear}
 \sum_{k = 0}^\infty \mu_{k}^2 \bV_k \bV_k^\sT = ( \mu_{0}^2 + o_{d,\P}(1))   \bones_p \bones_p^\sT+ \mu_{1}^2 \bW \bW^\sT + \mu_{>1}^2 \id_p + \bDelta ,
\end{equation}
and $\|  \bV_k \bbeta_k  \|_F = o_{d,\P}(1)$ for $k \geq 2$, where $\mu_{>1}^2 = \sum_{k \geq 2} \mu_{k}^2$ and $\| \bDelta\|_\op = o_{d,\P}(1)$. Thus, the GE model  \eqref{eq:general_GET_RF_expanded} reduces to:
\begin{modelbox}{Gaussian Equivalent (GE) Model in the linear scaling $n\asymp p \asymp d$}
\vspace{-0pt}
\begin{equation}\label{eq:GE_model_linear_scaling}
\begin{aligned}
\bz^\sG
   &:= \mu_{0} \bones_p + \mu_{1} \bW \bg_1 + \mu_{>1}^2 \bg_* ,
      \\
f^\sG
   &:=\bbeta_{0}^\sT  +  \bbeta_{1}^\sT \bg_1 + \bC_{>1}^{1/2} \bG_*.
\end{aligned}
\end{equation}
\end{modelbox}
\noindent
Here, 
\[
\bg_1 \sim \normal (0,\id_d), \qquad \bg_* \sim \normal (0,\id_p), \qquad \bG_* \sim \normal (0, \id_m) \;\;\;\text{ independently},
\]
and $\bC_{>1} = \sum_{k \geq 2} \bbeta_{k}^\sT \bbeta_{k}$. 

\cite{huUniversalityLawsHighdimensional2022,montanariUniversalityEmpiricalRisk2022} established the Gaussian equivalence \eqref{eq:GET_test_train} with model \eqref{eq:GE_model_linear_scaling} for $f_*(\bx) = \bbeta_{0}^\sT  + \bbeta_{1}^\sT \bx$ in the proportional regime for a broad class of loss functions and regularizations\footnote{These works assume $\mu_{0}=0$; our analysis extends to $\mu_{0} \neq 0$, which introduces a diverging ``spike'' singular value in the feature matrix.}.

\paragraph*{Simplification in the quadratic scaling.} In this paper, we consider the quadratic polynomial scaling $n \asymp p \asymp d^2$. Analogously to \eqref{eq:simplification_covariance_linear}, the covariance $\bSigma^{\sG}$ simplifies to\footnote{Specifically, $\| \bV_3 \bV_3^\sT - 3\bW\bW^\sT/d - \id_p \|_\op = o_{d,\P}(1)$, $\| \bV_4 \bV_4^\sT - 3 \bones \bones^\sT/d^2 - \id_p \|_\op = o_{d,\P}(1)$, and $\| \bV_{k} \bV_{k}^\sT - \id_p \|_\op = o_{d,\P}(1)$ for $k \geq 5$ in the regime $p\asymp d^2$. We absorb the spikes in $\bV_3,\bV_4$ into $\mu_{0}$ and $\mu_{1}$. See Appendix \ref{sec:ConcentrationInequality} for details.}
\begin{equation}\label{eq:simplification_covariance_quad}
 \sum_{k = 0}^\infty \mu_{k}^2 \bV_k \bV_k^\sT = (\mu_{0}^2 + o_{d,\P}(1))\bones_p \bones_p^\sT + (\mu_{1}^2 +o_{d,\P}(1)) \bW \bW^\sT + \mu_{2}^2 \bV_2 \bV_2^\sT + \mu_{>2}^2 \id_p + \bDelta, 
\end{equation}
where $\mu_{>2}^2 = \sum_{k =3}^\infty \mu_{k}^2$ and $\| \bDelta\|_\op = o_{d,\P}(1)$, and we have $\| \bV_k \bbeta_{k}  \|_F = o_{d,\P}(1)$ for $k \geq 3$.
In this regime, the GE model \eqref{eq:general_GET_RF_expanded} reduces to:
\begin{modelbox}{Gaussian Equivalent (GE) Model in the quadratic scaling $n\asymp p \asymp d^2$}
\vspace{-0pt}
\begin{equation}\label{eq:GET_quadratic_scaling}
\begin{aligned}
\bz^\sG
   &:= \mu_{0} \bones_p + \mu_{1} \bW \bg_1 + \mu_{2} \bV_2 \bg_2 +\mu_{>2} \bg_* ,
      \\
f^\sG
   &:=\bbeta_{0}^\sT  +  \bbeta_{1} ^\sT\bg_1 + \bbeta_{2}^\sT \bg_2 +\bC_{>2}^{1/2} \bG_*.
\end{aligned}
\end{equation}
\end{modelbox}
\noindent
Here, 
\[
\bg_1 \sim \normal (0,\id_d), \quad \bg_2 \sim \normal (0,\id_{B_{d,2}}),\quad \bg_* \sim \normal (0,\id_p),  \quad \bG_* \sim \normal (0, \id_m) \;\;\;\text{ independently},
\]
and $\bC_{>2} = \sum_{k \geq 3} \bbeta_{k}^\sT \bbeta_{k} $.

We provide conditions on the coefficients $\{ \bbeta_{k} \}_{k \geq 1}$ under which the RF model \eqref{eq:Hermite_expansion_RF_model} is accurately captured by the GE model \eqref{eq:GET_quadratic_scaling} (see Assumption \ref{assumption:target} and Section \ref{sec:abstract_CLT}), 
so that Gaussian universality \eqref{eq:GET_test_train} continues to hold in this quadratic regime. However, we will see in the following section that this equivalence breaks as soon as $\bbeta_{1} \neq 0$.

\subsection{Conditional Gaussian Equivalent model for Random Features}
\label{sec:setting_CGE}

We now introduce the \textit{Conditional Gaussian Equivalent} (CGE) model.
As discussed in Section \ref{sec:intro_example}, in the quadratic scaling $n \asymp p \asymp d^2$, non-Gaussian behavior can arise due to the dependency on low-dimensional directions in the degree-$2$ Wiener chaos. The GE model fails to capture this effect. The CGE model corrects this by conditioning on the relevant low-dimensional subspace.

\paragraph*{Labels and target functions.} Recall that the responses are generated as $y=\eta(f_*(\bx);\eps)$ with $\eta:\R^m\times\R\to\R$ and $f_*:\R^d\to\R^m$. The decomposition of the pair $(\eta,f_*)$ is not unique: different choices can induce the same joint law of $(y,\bx)$, yet GE (and CGE) models may depend on this choice. For example, in the proportional scaling $n\asymp p\asymp d$ with $y=\bar\eta(\He_2(\<\bu_*,\bx\>);\eps)$, one can take either:
\begin{itemize}
    \item[(i)] $f_* (\bx) := \He_2 ( \< \bu_*,\bx\>)$ and $\eta := \bar \eta$; or

    \item[(ii)] $f_* (\bx) :=  \< \bu_*,\bx\>$ and $\eta (u;\eps) := \bar \eta (\He_2 (u);\eps)$.
\end{itemize}
In case (i), the marginal of $f_*(\bx)$ is non-Gaussian and GE fails; in case (ii), GE holds \cite{huUniversalityLawsHighdimensional2022}. 

To resolve this ambiguity and avoid such trivial failures, we will fix a specific representation for the response which isolates the different chaos contributions. Specifically, we take \footnote{Again, $\xi_{ki}$ (and therefore $\bbeta_{k,i}$) is a sequence indexed by $d$. We suppress this dependence for notational simplicity.}
\begin{equation}\label{eq:label_multi-index-chaos}
\begin{aligned}
    f_* (\bx) =&~ \{ \xi_{ki} \}_{k \in [D'], i \in [s_k]}, \qquad \xi_{ki} = \< \bbeta_{ki} ,  \bh_k (\bx)\>, \\
    y =&~ \eta ( \{ \xi_{ki} \}_{k \in [D'], i \in [s_k]};\eps),
\end{aligned}
\end{equation}
so that each coordinate $\xi_{ki}$ of $f_*$ belongs to a fixed Wiener chaos subspace $\cH_k$. In words, $y$ is a multi-index response whose indices are polynomials of $\bx$ with fixed chaos order.

Without loss of generality, we normalize $\| \bbeta_{ki} \|_2 = 1$. We impose the following ``genericity condition'' on the\footnote{Note that this is automatically verified for $k = 1$ as $\< \bbeta_{1i},\bx\> \sim \normal (0,1)$.} $\bbeta_{ki}$:
\begin{equation}\label{eq:fourth-moment-condition}
    \E [ \< \bbeta_{ki}, \bh_k (\bx) \>^4] = 3 + o_d (1).
\end{equation}
By the Fourth Moment Theorem for Wiener chaos (see Remark \ref{rmk:fourth-moment} below), \eqref{eq:fourth-moment-condition} implies $\xi_{ki} \overset{\de}{\Rightarrow} \normal(0,1)$ as $d \to \infty$. Hence, marginally $f_* (\bx)$ is asymptotically Gaussian with
\[
\E[\xi_{ki} \xi_{k'i'}] = \delta_{kk'} \< \bbeta_{ki},\bbeta_{ki'}\>,
\]
ruling out the aforementioned trivial non-Gaussianity. Note that condition \eqref{eq:fourth-moment-condition} is purely deterministic in $\{\bbeta_{k i}\}$. For $k=2$ (Hermite-2 as in Section \ref{sec:intro_example}), it amounts to the associated matrix representation $\bB$ of $\bbeta_{2i}$ satisfying $\|\bB\|_{\op}=o_d(1)$. For $k\ge 3$, it translates to vanishing Frobenius norms of certain tensor contractions of the coefficient tensor; see Section~\ref{sec:abstract_CLT} for the precise statements.

\begin{remark}[Fourth Moment Theorem]\label{rmk:fourth-moment}
If $X_d$ lies in a fixed Wiener chaos of order $k$ and is standardized by $\E[X_d] = 0$, $\Var (X_d) = 1$, then $X_d \overset{\de}{\Rightarrow} \normal (0,1)$ if and only if $\E[X_d^4] \to 3$. This is the Nualart-Peccati criterion; e.g., see \cite{nualart2005central,peccati2004gaussian} and the monograph \cite{nourdin2012normal}. In our setting, $X_d = \< \bbeta_{ki}, \bh_k (\bx) \>$ so \eqref{eq:fourth-moment-condition} yields asymptotic Gaussianity of each coordinate $\xi_{ki}$, which also implies the asymptotic Gaussianity of the joint distribution.
\end{remark}

\paragraph*{Models covered.} The representation \eqref{eq:label_multi-index-chaos} encompasses a large class of responses:
\begin{itemize}
    \item \textit{Multi-index models:} $y = \eta (\bTheta_*^\sT \bx; \eps)$ with $\bTheta_* \in \R^{d \times m}$ (including GLMs when $m=1$).

    \item \textit{Regression functions with random coefficients:} Let $m=1$ and  $y = f_* (\bx) + \eps$ with
    \[
    f_* (\bx) = \beta_0+\sum_{k = 1}^{D'} \<\bbeta_{k}, \bh_k (\bx)\>,
    \]
    where $\bbeta_{k}$ verify \eqref{eq:fourth-moment-condition}. E.g., it is satisfied by $\bbeta_k \sim \Unif (\S^{B_{d,k}-1})$ (see Appendix \ref{app:excess-kurtosis}).
    \item \textit{A broad class of deterministic regression functions.} Any $f_* \in L^2 ( \R^d,\normal (0,\id_d))$ can be written as $F (\{ \xi_{ki} \})$, where the spikes (low-rank components) are embedded among lower-order chaos coordinates. For instance, consider $f_*(\bx) = \bx^\sT \bB_* \bx - \Tr(\bB_*)$: then if $ \bB_* $ has a non-vanishing top eigenvalue, e.g., $\bB_* = \bu_* \bu_*^\sT + \bB_0$ with $\| \bB_0\|_\op = o_d(1)$, then  $f_*(\bx) = F( \< \bu_*,\bx\>, \bx^\sT \bB_0 \bx - \Tr(\bB_0))$. Here, \eqref{eq:label_multi-index-chaos} constrains $f_*$ to depend on at most a finite number of such chaos projections with Gaussian marginal.
\end{itemize}

\paragraph*{Conditioning on the degree-$1$ coordinates.}
Without loss of generality, take the first-order chaos coordinates as the first $s:=s_1$ coordinates of $\bx$ after an orthogonal change of basis, that is
\begin{equation}
\{ \xi_{1i}\}_{i \in [s]} = \bx_S = (x_1, \ldots, x_s).
\end{equation}
 
In the quadratic scaling $n\asymp p\asymp d^2$, our Conditional Gaussian universality result conditions on $\bx_S$ and compares the RF model to a GE model \emph{conditional on} these coordinates.

\begin{remark}[General polynomial scaling]\label{rmk:general-polynomial-scaling} For scalings $n\asymp p\asymp d^k$, $k >2$, one will need to condition on higher-order chaos in the response~\eqref{eq:label_multi-index-chaos}. We postpone this general construction to \cite{wen2025empirical}, which develops a general Conditional Gaussian universality principle for abstract polynomial chaos expansions.
\end{remark}

\paragraph*{The Conditional Gaussian Equivalent (CGE) model.} Conditioning on the $\bx_S$ component of $\phi_\RF (\bx)$ in \eqref{eq:Hermite_expansion_RF_model} leads to the following decomposition:
\begin{enumerate}
    \item \textit{First-order chaos $\bW\bx$:} Let $\bW_S \in \R^{p \times s}$ be the submatrix of $\bW$ corresponding to the coordinates of $\bx_S$, and let $\proj_{S,\perp} : \R^d \to \R^d$ denote the orthogonal projection onto the complement of that subspace. Then
    \begin{equation}
    \bW\bx = \bW_S \bx_S + \bW \proj_{S,\perp} \bx \overset{\de}{=} \bW_S \bx_S +  \bW \proj_{S,\perp} \bg_1,
    \end{equation}
    where $\bg_1 \sim \normal (0,\id_d)$ is independent of $\bx_S$.

    \item \textit{Second-order chaos $\bV_2\bh_2 (\bx)$:} Let $\bV_{2,S} \in \R^{p \times B_{s,2}}$ (with $B_{s,2} = s(s+1)/2$) denote the submatrix of $\bV_2$ corresponding to multi-indices supported within the first $s$ coordinates (in particular, $\bV_{2,S} = q_2 (\bW_S)$), and let $\proj_{S,\perp} :\R^{B_{d,2}} \to \R^{B_{d,2}}$ denote the orthogonal projection onto the complement subspace (with a slight overloading of notations). Then
    \begin{equation}
    \bV_2 \bh_2 (\bx) = \bV_{2,S} \bh_2 (\bx_S) + \bV_2 \proj_{S,\perp} \bh_2 (\bx),
    \end{equation}
    where $ \bh_2 (\bx_S) \in \R^{B_{s,2}}$ contains all degree-$2$ Hermite polynomials of $\bx_S$. The term $\proj_{S,\perp} \bh_2 (\bx)$ contains cross-terms $x_ix_j$, with $i \leq s <j$, and hence is not independent of $\bx_S$. However,
    \begin{equation}
        \E [(\proj_{S,\perp} \bh_2 (\bx)) (\proj_{S,\perp} \bh_2 (\bx))^\sT \vert \bx_S] = \proj_{S,\perp},
    \end{equation}
    and in the quadratic scaling, this term will behave as $\bV_2 \proj_{S,\perp} \bg_2$ where $\bg_2 \sim \normal (0,\id_{B_{d,2}})$ is independent of $(\bx_S,\bg_1)$.

    \item \textit{Higher-order chaos $\bV_k \bh_k (\bx), k\geq 3$:} Each term is a random low-dimensional projection of $\bh_k (\bx) \in \R^{\Theta(d^k)}$ onto the $p=\Theta(d^2)$-dimensional span of $\bV_k$. Similarly as in the GE model \eqref{eq:GET_quadratic_scaling}, these higher-order chaos components behave as an additive noise-term $\mu_{>2} \bg_*$, with $\bg_* \sim \normal (0, \id_p)$ independent of  $(\bx_S,\bg_1,\bg_2)$.

    \item \textit{Responses $y = \eta ( \{ \bx_S , (\xi_{ki})_{2 \leq k \leq D', i \in [s_k]} \};\eps)$:} Condition \eqref{eq:fourth-moment-condition} implies that the $\xi_{2i}$ are asymptotically independent of $\bx_S$, and can be replaced by $\{ \< \beta_{2i}, \bg_2\> \}_{i \in [s_2]}$. For $k \geq 3$, the $\{\xi_{ki}\}$ can be replaced by $\{ G_{ki}\} \sim \normal (0, \bC_{*,>2})$, independent of $(\bx_S,\bg_1,\bg_2,\bg_*)$, where $\bC_{*,>2}$ denotes the covariance of $(\xi_{ki})_{3 \leq k \leq D', i \in [s_k]}$, that is, $\E[G_{ki} G_{k'i'}] = \delta_{kk'} \< \bbeta_{ki},\bbeta_{ki'}\>$. 
\end{enumerate}
Combining these approximations yields the CGE model:
\begin{modelbox}{Conditional Gaussian Equivalent (CGE) Model in quadratic scaling $n\asymp p \asymp d^2$}
\vspace{-5pt}
\begin{equation}\label{eq:CGE_definition}
\begin{aligned}
\bz^\sCG
   &:= \mu_{0} \bones_p + \underbrace{\mu_{1} \bW_S \bx_S + \mu_{2} \bV_{2,S} \bh_2 (\bx_S)}_{\text{Part depending on the support $\bx_S$}} + \underbrace{\mu_{1} \bW \proj_{S,\perp}  \bg_1 + \mu_{2} \bV_{2} \proj_{S,\perp} \bg_2 +\mu_{>2} \bg_*}_{\text{Independent Gaussian part}} ,
      \\
f^\sCG
   &:= \big\{\bx_S, ( \<\bbeta_{2i}, \bg_2 \> )_{i \in [s_2]}, ( G_{ki} )_{3\leq k \leq D', i \in [s_k]} \big\} .
\end{aligned}
\end{equation}
\end{modelbox}
\noindent
Here,
\begin{equation}
    \bx_S \sim \normal (0,\id_s), \quad \bg_1 \sim \normal (0,\id_d), \quad \bg_2 \sim \normal (0,\id_{B_{d,2}}), \quad \bg_* \sim \normal (0,\id_p), \quad ( G_{ki} )_{\substack{3\leq k \leq D', \\i \in [s_k]}} \sim \normal (0, \bC_{*,>2}),
\end{equation}
independently. Thus, the CGE model is a high-dimensional Gaussian model augmented with $s(s+1)/2$ non-Gaussian components $\bh_2(\bx_S)$ capturing the low-dimensional dependence structure $\bx_S$ in the response. In the next section, we show that the asymptotic test and train errors of the RF model \eqref{eq:Hermite_expansion_RF_model} coincide with those of the CGE model \eqref{eq:CGE_definition} for a broad class of problems.

Let us illustrate the CGE model on two simple examples:

\paragraph*{Example 1: Generalized Linear Models (GLMs).} For $y = \eta ( \< \bu_*, \bx\> ; \eps)$, the associated CGE model is given by
\begin{equation}
\begin{aligned}
\bz^{\sCG} :=&~ \mu_{0} \bones_p + \mu_{1} \br_1 \<\bu_*,\bx\> +  \mu_{2} \br_2 \He_2 (\<\bu_*,\bx\>) \\
&~ + \mu_{1} \bW \left(\id_d - \bu_* \bu_*^\sT \right)  \bg_1 + \mu_{2} \bV_{2} \left(\id_{B_{d,2}} - \bq_2 (\bu_*) \bq_2 (\bu_*)^\sT\right) \bg_2 +\mu_{>2} \bg_*, \\
f^{\sCG} := &~ \< \bu_*,\bx\>,
\end{aligned}
\end{equation}
where
\[
\br_1 := \bW\bu_* =  ( \<\bw_j ,\bu_* \>)_{j \in [p]}, \qquad \br_2 := \bV_2 \bq_2 (\bu_*) =  ( \<\bw_j ,\bu_* \>^2)_{j \in [p]},
\]
and we recall that 
\[
\bq_2 ( \bu_*) = ( u_{*,i}^2 )_{i \in [d]} \oplus (\sqrt{2}u_{*,i}u_{*,j})_{ i<j \in [d]}.
\]
Thus, for GLMs in the quadratic scaling regime, the CGE feature $\bz^{\sCG}$ will contain a one-dimensional non-Gaussian component in direction $\br_2$ with (shifted) chi-squared marginal. 

\paragraph*{Example 2: No first-order Chaos components.} When $f_*$ has no degree-$1$ components (and the remaining components of degrees $\geq 2$ satisfy the genericity condition \eqref{eq:fourth-moment-condition}), the CGE model coincides with the GE model \eqref{eq:GET_quadratic_scaling}:
\begin{equation}
\begin{aligned}
    \bz^{\sCG} := &~\bz^{\sG} = \mu_{0} \bones_p + \mu_{1} \bW \bg_1 + \mu_{2} \bV_2 \bg_2 +\mu_{>2} \bg_*, \\
    f^{\sCG} := &~ f^{\sG} = \big\{( \<\bbeta_{2i}, \bg_2 \> )_{i \in [s_2]}, ( G_{ki} )_{3\leq k \leq D', i \in [s_k]} \big\} .
\end{aligned}
\end{equation}

\section{Main results}
\label{sec:main_results}

In this section, we state our main result on the \textit{Conditional Gaussian universality} of empirical risk minimization with random features in the quadratic scaling regime. Section~\ref{sec:assumptionsmain} introduces the assumptions under which our analysis applies. The  universality for the training and test errors (Theorems~\ref{thm:universality_train_quadratic} and~\ref{thm:universality_test_error}) are stated in Section~\ref{sec:universality-RF-CGE}. Additional discussions and numerical simulations are provided in Section~\ref{sec:discussion}.

\subsection{Assumptions}
\label{sec:assumptionsmain}

For the reader’s convenience, we recall the empirical risk minimization (ERM) problem under consideration:
\begin{equation}
\begin{aligned}
\hbtheta^\RF &= \argmin_{\btheta \in \R^p} \hcR_{n,p} (\btheta ; \bZ^\RF,\boldf^\RF,\beps), \\
\hcR_{n,p} (\btheta ; \bZ^\RF,\boldf^\RF,\beps) &:=  \frac{1}{n} \sum_{i \in [n]} \ell (\eta(f_i^\RF;\eps_i), \< \btheta, \bz_i^\RF\> )  + \frac{\lambda}{2} \| \btheta \|_2^2,
\end{aligned}
\end{equation}
where $f_i^\RF =  f_* (\bx_i)$ and $\bz_i^\RF = \phi_{\RF} (\bx_i) = \sigma ( \bW\bx_i)$ with 
\begin{equation}
\bx_i \overset{\text{i.i.d.}}{\sim} \normal (0,\id_d), \qquad \eps_i \overset{\text{i.i.d.}}{\sim} \nu \qquad \text{independently},
\end{equation}
and $\bW=[\bw_1,\ldots,\bw_p]^\sT \in \R^{p \times d}$ is fixed independently with $\bw_j \overset{\text{i.i.d.}}{\sim} \Unif (\S^{d-1})$.

We study this problem in the high-dimensional regime where $n,p,d\to\infty$. The activation function $\sigma:=\sigma_d$, its Hermite coefficients $\mu_k:=\mu_{d,k}$, the regularization parameter $\lambda:=\lambda_d$, and the target function $f_*:=f_{d,*}$ may depend on the ambient dimension $d$. For notational simplicity, we suppress this dependence. All constants introduced below are fixed independent of~$d$.

We now state the assumptions under which our main results hold.

\begin{assumption}[Quadratic scaling regime]\label{assumption:scaling}
    There exists a constant $\sC_0 >0$ such that 
    \begin{equation}
    \sC_0^{-1} \leq n/d^2 \leq \sC_0, \qquad \text{and} \qquad \sC_0^{-1} \leq p/d^2 \leq \sC_0.
    \end{equation}
\end{assumption}

\begin{assumption}[Loss function]\label{assumption:loss}
    There exists a constant $\sC_1>0$ such that the loss function $\ell : \R \times \R \to \R$ satisfies the following.
\begin{itemize}
    \item[(i)]  The loss function is nonnegative, $\ell(y,\hat y) \ge 0$ for all $y, \hat y \in \R$, and convex with respect to $\hat y$.


    \item[(ii)]  The loss $\ell (y,\hat y)$ is three times continuously differentiable with
    \begin{equation}
   \| \nabla \ell (y,\hat y) \|_F, \;\; \| \nabla^2 \ell(y,\hat y) \|_F ,\;\;    \| \nabla^3 \ell(y,\hat y) \|_F \leq \sC_1, \qquad \text{for all } y,\hat y \in \R.
    \end{equation}
\end{itemize}
\end{assumption}

We will further assume a mild growth condition on the loss function:

\begin{assumption}[Calibrated growth]\label{assumption:polynomial-growth} There exist  constants $\sc_2,\sr_2>0$ and $\sC_2 \geq 0$, such that either:
\begin{itemize}
    \item[(i)] \emph{(Regression.)} We have $\ell (y,\hat y) \geq \sc_2\,| \hat y - y |^{\sr_2}$ for all $y,\hat y \in \R$ with $|\hat y - y | \geq \sC_2$. 
    
    \item[(ii)] \emph{(Binary classification.)} We have $\ell (y,\hat y) \geq \sc_2\, ( - y\hat y)_+^{\sr_2}$ for all $y,\hat y \in \R$.
\end{itemize}
\end{assumption}

This assumption is satisfied by a number of standard losses, such as the squared loss $\ell(y,\hat y)=\tfrac12(\hat y-y)^2$ for regression (with $\sr_2=2$, $\sc_2=\tfrac12$, $\sC_2=0$) and logistic loss $\ell(y,\hat y)=\log(1+e^{-y\hat y})$ for binary classification (with $\sr_2 = 1$, $\sc_2 = 1$, $\sC_2 = 0$).
Note that Assumption~\ref{assumption:loss}(ii) requires $\ell$ to be Lipschitz. The case of the squared loss can be handled with minor adjustments to the proof. However, extending the analysis to general \emph{pseudo-Lipschitz} losses would require substantial modifications, and we will not pursue this extension in the current work.

\begin{assumption}[Regularization parameter]\label{assumption:lambda}
  There exists $\sC_5 > 0$ such that $\lambda \ge \log^{-\sC_5} d$.
\end{assumption}

\begin{remark}
Considerable attention has recently been devoted to \emph{interpolating} estimators ($\lambda=0^+$) and to the phenomenon of \emph{benign overfitting} \cite{zhang2021understanding,belkin2018understand,Bartlett_2020,montanari2025generalization}. Examining the proof, our main results continue to hold for nearly interpolating solutions with $\lambda = d^{-c}$, for sufficiently small $c>0$ depending only on the constants in our assumptions. Establishing universality for the exact interpolating estimator $\lambda=0^+$ would, however, require a different proof strategy. For example, \cite{montanariUniversalityEmpiricalRisk2022} prove universality for interpolating solutions, but their argument applies only to sub-Gaussian features $\phi(\bx)$ and does not extend to our quadratic scaling regime, where $\sigma(\bW\bx)$ is no longer sub-Gaussian.
\end{remark}

\begin{assumption}[Response]\label{assumption:target}  For some constants $\sc_3,\sC_3>0$ and fixed integers $D'$ and $(s,s_2,\ldots,s_{D'})$ not depending on $n,p,d$, the following holds: Let $m = s +s_2 + \ldots + s_{D'}$. There exist a link function $\eta : \R^m \times \R \to \R$ and a latent signal $f_* : \R^d \to \R^m$ such that the response satisfies
\[
y = \eta (f_*(\bx), \eps),
\]
with noise $\eps \sim \nu$ independent of $\bx$.
\begin{itemize}
    \item[(i)] \emph{(Link function.)} The link function satisfies for all $\boldf,\boldf' \in \R^m, \eps \in \R$:
    \[
     | \eta ( \boldf, \eps) | \leq \sC_3 ( 1 + \| \boldf \|_2^{\sC_3} + | \eps |^{\sC_3} ), \qquad | \eta ( \boldf, \eps) - \eta (\boldf', \eps) | \leq \sC_3  \|\boldf-\boldf'\|_2 .
    \]
    \item[(ii)] \emph{(Latent signal.)} The latent signal $f_*$ is of the form (see Section \ref{sec:setting_CGE} for a detailed discussion)
    \begin{equation}
    f_*(\bx) = \{ \bx_S, \bxi_2, \bxi_3,\ldots, \bxi_{D'} \} ,
\end{equation}
where $\bx_S = (x_1, \ldots ,x_s)$ and $\bxi_{k} = (\xi_{ki})_{i = 1}^{s_k}$ with $\xi_{ki} = \< \bbeta_{ki}, \bh_k (\bx)\>$, $\bbeta_{ki} \in \R^{B_{d,k}}.$ We assume that\footnote{See Appendix \ref{app:excess-kurtosis} for an explicit criterion in terms of vanishing Frobenius norms of tensor  contractions.}
\begin{equation}\label{eq:ass_bbeta_ki_4_th_moment}
\| \bbeta_{ki} \|_2 = 1, \qquad \E [ \<\bbeta_{ki},\bh_k (\bx) \>^4 ] \leq 3 + d^{-\sc_3} \qquad \text{for all } \;\;2 \leq k \leq D', i\in [s_k].
\end{equation}
\end{itemize}
\end{assumption}

\begin{assumption}[Label noise] \label{assumption:noise}
For some constants $\sC_4,\sr_4 >0$, the noise variable $\eps$ satisfies
\begin{equation}\P[|\eps| \geq t] \leq \sC_4 e^{-t^{\sr_4}} \text{ for all } t \geq 0.\end{equation}
\end{assumption}

In the binary classification setting of Assumption \ref{assumption:polynomial-growth}(ii), we will further assume the following condition:

\begin{assumption}[Label noise in classification]\label{assumption:bayes-error}  
In the setting of Assumption \ref{assumption:polynomial-growth}(ii), for any constant $\sC>0$, there exists a constant $\sc \in (0,1/2)$ such that
\[\sup_{\boldf \in \R^m:\|\boldf\|_2 \leq \sC} \P[\eta(\boldf,\eps)=+1] \geq \sc,
\qquad \sup_{\boldf \in \R^m:\|\boldf\|_2 \leq \sC} \P[\eta(\boldf,\eps)=-1] \geq \sc.\]
\end{assumption}

This assumption requires that there is non-vanishing label noise, and hence non-vanishing Bayes prediction error, for $y=\eta(\boldf,\eps)$ over bounded domains of $\boldf \in \R^m$.
Note that in this context of binary classification, our results formally apply to a model where $\eta$ is any Lipschitz approximation of a function taking value in $\{-1,+1\}$.

 \begin{assumption}[Activation function]\label{assumption:activation}
 For some constants $\sc_6 \in (0,1/2)$ and $\sC_6>0$ and a fixed integer $D$ not depending on $n,p,d$, the following holds: The activation $\sigma:\mathbb{R}\to\mathbb{R}$ is a polynomial of degree at most~$D$. Its Hermite coefficients $ \mu_{k}:=\E_{G \sim \normal(0,1)} \bigl[\sigma(G)\,\He_{k}(G)\bigr]$ satisfy
 \begin{equation}\label{eq:assumption_sigma}
     |\mu_{0}| ,|\mu_{2}|\ge \log^{-\sC_6}d, \qquad d^{-\sc_6} /\sC_6 \leq | \mu_{1}| \leq \sC_6 d^{-\sc_6},\qquad  \bigl|\mu_{k}\bigr|\le\log^{\sC_6} d \;\;\; \text{for all $k \leq D$}.
 \end{equation}
\end{assumption}

\begin{remark} 
The restriction that $\sigma$ be a polynomial can be relaxed to requiring sufficiently fast decay of its Hermite coefficients.  
The scaling condition $|\mu_1| \asymp d^{-\sc_6}$ with $\sc_6 \in (0,1/2)$ is a technical assumption introduced to control the spikes appearing in the second Lindeberg swapping step.  
Note that under this condition, the covariance of the features still has a rank-$d$ spike of the form $\mu_1^2 \bW\bW^\sT$ with diverging eigenvalues on the order of $\Theta(d^{1 - 2\sc_6})$, allowing the model to still learn the linear component of the target function.
\end{remark}

\subsection{Universality under the Conditional Gaussian Model}
\label{sec:universality-RF-CGE}

Denote the minimum of the empirical risk
\begin{equation}
\widehat{\mathcal R}_{n,p}^*(\bZ^\RF,\boldf^\RF)
: = \min_{\btheta}\widehat{\mathcal R}_{n,p}(\btheta;\bZ^\RF,\boldf^\RF), \qquad \widehat{\mathcal R}_{n,p}^*(\bZ^\sCG,\boldf^\sCG)
: = \min_{\btheta}\widehat{\mathcal R}_{n,p}(\btheta;\bZ^\sCG,\boldf^\sCG),
\end{equation}
where $(\bZ^\RF,\boldf^\RF)$ is the training data under the random feature (RF) model \eqref{eq:Hermite_expansion_RF_model}, and $(\bZ^{\sCG},\boldf^{\sCG})$ is the training data under the Conditional Gaussian Equivalent (CGE) model \eqref{eq:CGE_definition}.

We first state our universality result for the training error:

\begin{theorem}[Universality of Training Error]\label{thm:universality_train_quadratic}
    Suppose Assumptions \ref{assumption:scaling}--\ref{assumption:activation}
    hold. Then, there exist constants $c,d_0>0$ depending only on the constants in these assumptions, such that for all $d \geq d_0$ and any twice-differentiable function $\varphi : \R \to \R$ with $\| \varphi \|_\infty,  \| \varphi' \|_\infty, \| \varphi '' \|_\infty \leq 1$, 
    \[
    \left| \E  \Big[ \varphi \left( \hcR_{n,p}^* (\bZ^\RF,\boldf^\RF) \right) \Big] - \E  \Big[ \varphi \left( \hcR_{n,p}^* (\bZ^\sCG,\boldf^\sCG) \right) \Big] \right| \leq  \frac{1}{d^{c}}.
    \]
    In particular, for every $\eps \in (0,1)$ and  $\kappa \in \R$, 
    \[
    \begin{aligned}
    \P \left( \left|\hcR_{n,p}^* (\bZ^\RF,\boldf^\RF) - \kappa \right| \geq 2\eps \right) \leq \P \left( \left|\hcR_{n,p}^* (\bZ^\sCG,\boldf^\sCG) - \kappa \right| \geq \eps \right) + \frac{1}{\eps d^{c}},\\
    \P \left( \left|\hcR_{n,p}^* ( \bZ^\sCG,\boldf^\sCG) - \kappa \right| \geq 2\eps \right) \leq \P \left( \left|\hcR_{n,p}^* (\bZ^\RF,\boldf^\RF) - \kappa \right| \geq \eps \right) + \frac{1}{\eps d^{c}}.
    \end{aligned}
    \]
    Consequently, for all $\kappa \in \R$,
    \[
    \hcR_{n,p}^* (\bZ^\RF,\boldf^\RF) \overset{\P}{\longrightarrow} \kappa\quad \text{ if and only if } \quad\hcR_{n,p}^* ( \bZ^\sCG,\boldf^\sCG) \overset{\P}{\longrightarrow} \kappa.
    \]
\end{theorem}

We outline the proof of Theorem~\ref{thm:universality_train_quadratic} in Section~\ref{sec:proof_outline}.  
The argument builds on the Lindeberg interpolation method \cite{lindeberg1922neue,chatterjee2006generalization}, which has already been used in earlier works to establish universality of various statistical problems \cite{korada2011applications,montanari2017universality,panahi2017universal,oymak2018universality,huUniversalityLawsHighdimensional2022}.  
We interpolate the minimum of the empirical risk between the original and the CGE model by sequentially replacing one data point at a time with its CGE counterpart, and control the resulting change in the objective at each step. Compared with previous ERM universality analyses \cite{huUniversalityLawsHighdimensional2022,montanariUniversalityEmpiricalRisk2022}, our setting is considerably more challenging. The features are not sub-Gaussian, with dependent chaos of different orders, and the feature covariance exhibits (a diverging number of) spikes.  
To handle these issues, we introduce an intermediate model---the \textit{Partial Gaussian Equivalent} (PGE) model---and perform two phases of Lindeberg swapping. Each phase requires delicate estimates to control the contributions of the different chaos and spike directions.

Theorem~\ref{thm:universality_train_quadratic} establishes universality for the training error. Of course, the quantity of primary interest from a statistical learning standpoint is the test error at the empirical minimizer: 
\begin{equation}\label{eq:test_error_result}
\cR_{\test} (\hbtheta^\RF; \P_{\bz^\RF,f^\RF} ) = \E_{\bz^\RF,f^\RF,\eps} [ \ell_{\test} (\eta (f^\RF;\eps), \<\hbtheta^\RF,\bz^\RF\>) | \bW],
\end{equation}
where $\hbtheta^\RF$ minimizes $\widehat{\mathcal R}_{n,p}(\btheta;\bZ^\RF,\boldf^\RF)$. We will establish universality of this test error under the following additional conditions:

\begin{assumption}[Test loss]\label{ass:test_loss}
    There exists a constant $\sC_7 >0$ such that the test loss function $\ell_\test :\R \times \R \to \R$ is three times differentiable with pseudo-Lipschitz derivatives, that is,
    \[
    \| g (\bu ) - g(\bu') \|_F \leq \sC_7 (1 + \| \bu \|_2^{\sC_7} + \| \bu' \|_2^{\sC_7} ) \| \bu - \bu'\|_2,
    \]
    for all $\bu , \bu' \in \R^2$ and $g \in \{ \ell_\test, \nabla \ell_\test, \nabla^2 \ell_\test, \nabla^3 \ell_\test\}$.
\end{assumption}
 
\begin{assumption}[Local strong convexity]\label{assumption:lsc}
For any constant $\sK>0$, there exists $\sK'>0$ such that for all $|y|,|\hat y| \leq (\log d)^{\sK}$,
\begin{equation}
\partial_{\hat y\hat y}^2\,\ell(y,\hat y)\ \ge\ \frac{1}{(\log d)^{\sK'}}.
\end{equation}
\end{assumption}

\begin{remark} Universality of the training error is more robust than that of the test error: without strong convexity, small perturbations in the training data may leave the training error nearly unchanged while inducing large deviations in the minimizer $\hbtheta$.
In our setting, the empirical risk is not strongly convex along the spike directions of the feature covariance, so we impose Assumption~\ref{assumption:lsc} to guarantee sufficient curvature in these directions.
This condition is satisfied, for instance, by a mild perturbation of the logistic loss (obtained via an $\ell_\infty$-perturbation outside a compact neighborhood of the origin, decaying poly-logarithmically in $d$). We note that weaker conditions are possible (see, e.g., \cite[Theorem~3]{montanariUniversalityEmpiricalRisk2022}), but we adopt Assumption~\ref{assumption:lsc} for simplicity.
\end{remark}

To transfer universality from the training to the test error, we introduce a \emph{perturbed empirical risk}. Let $\cT_{K_{\Gamma}} : \R \to \R$ be a smooth approximation of the truncation function $T_{K_{\Gamma}} (x) = \sign (x) \min(|x|, \log^{K_{\Gamma}} d)$, where $K_{\Gamma}>0$ is a constant which will be set sufficiently large but depending only on the constants in the assumptions. Define the perturbed empirical risk by
\begin{align}
\label{equ:perturbed_objectiveERM}
\widehat{\mathcal R}_{n,p}(\btheta;\btau,\bZ,\boldf)
=\widehat{\mathcal R}_{n,p}(\btheta;\bZ,\boldf)
+\btau\cdot\bGamma^\bW (\btheta),
\end{align}
where $\btau = (\tau_1,\tau_2) \in \R_{\geq 0}\times \R$ and 
\begin{equation}
    \Gamma_1^\bW (\btheta) = \cT_{K_{\Gamma}} ( \| \bV_+^\sT \btheta \|_2^2), \qquad \Gamma_2^\bW (\btheta) = \cT_{K_{\Gamma}} (  L_{\bW} (\btheta) ).
\end{equation}
Here $\bV_+ = [\mu_0\bones_p, \mu_1\bW] \in \R^{p \times (d+1)}$ collects the spike directions, and 
\begin{equation}\label{eq:test_loss_perturbation_def}
    L_{\bW} (\btheta) := \cR_{\test} (\btheta; \P_{\bz^\sPG,f^\sPG} ) =  \E_{\bz^\sPG,f^{\sPG},\eps} [ \ell_{\test} (\eta(f^{\sPG};\eps) , \< \btheta, \bz^\sPG \>) | \bW].
\end{equation}
The precise definition of the PGE model $(\bz^\sPG,f^\sPG)$ is postponed to Section \ref{sec:notations_and_conventions}.

Denote the minima of these perturbed risks by
\begin{equation}
    \widehat{\mathcal R}_{n,p}^*(\btau,\bZ,\boldf)
: = \min_{\btheta}\widehat{\mathcal R}_{n,p}(\btheta;\btau,\bZ,\boldf). 
\end{equation}
Then, Theorem \ref{thm:universality_train_quadratic} holds uniformly for all $\btau$ with  $d^{-c'} \leq \tau_1 \leq 1/(\log d)^{K_\Gamma}$ and $|\tau_2| \leq \tau_1/(\log d)^{K_\Gamma}$ (see Theorem \ref{thm:universality_Perturbed_risk} in Appendix \ref{app:perturbation_stability_ERM}). To extend universality from the perturbed risk to the test error, we ``differentiate'' the limiting value $\hcR_{n,p}^*(\btau,\bZ,\boldf)$ with respect to $\tau_2$ at $\tau_2 = 0$; as $\tau_1 \to 0$, this derivative coincides with the test error \eqref{eq:test_error_result}. Formally, we assume the following:

\begin{assumption}[Perturbed objective around $\btau = 0$]\label{ass:test_error} 
There exists a constant $\rho>0$ such that for any constant truncation threshold $K_\Gamma>0$ sufficiently large, setting $\tau_1=(\log d)^{-K_\Gamma}$ and $\tau_2 \in \{\pm (\log d)^{-2K_\Gamma}\}$, we have as $d \to \infty$
\begin{equation}
\frac{
\widehat{\mathcal R}_{n,p}^*(\btau,\bZ^\sCG,\boldf^\sCG)
-\widehat{\mathcal R}_{n,p}^*((\tau_1,0),\bZ^\sCG,\boldf^\sCG)
}{\tau_2}
\overset{\mathbb P}{\longrightarrow}
\rho  .
\end{equation}
\end{assumption}

Note that this condition only needs to be verified for the CGE model, allowing the use of Gaussian-specific tools (e.g., Convex Gordon minimax theorem) to check it in concrete cases.  
Moreover, it suffices to verify the assumption for a simplified perturbed objective where $\cT_{K_\Gamma}$ is replaced by identity and $L_\bW (\btheta)$ by the CGE test error
$\cR_{\test}(\hbtheta;\P_{\bz^{\sCG},f^{\sCG}})$. 

\begin{theorem}[Universality of Test Error]\label{thm:universality_test_error}
Suppose Assumptions~\ref{assumption:scaling}--\ref{assumption:activation} and~\ref{ass:test_loss}--\ref{ass:test_error} hold. Denote by $\hat\btheta^\RF$ (resp.\ $\hat\btheta^{\sCG}$) the minimizer of the empirical risk~\eqref{equ:perturbed_objectiveERM} with $\btau=0$ for the random feature data $(\bZ^\RF,\boldf^\RF)$  (resp.\ CGE data $(\bZ^{\sCG},\boldsymbol f^{\sCG})$). Then, as $d\to\infty$, 
\[
\biggl|\,
\cR_{\test}\left(\hbtheta^\RF;\P_{\bz^\RF,f^\RF}\right)
-
\cR_{\test}\left(\hbtheta^\sCG;\P_{\bz^{\sCG},f^{\sCG}}\right)
\biggr|
\;\xrightarrow{\;\P\;}0.
\]
\end{theorem}

Under these assumptions, the limiting test error of the ERM with random features \eqref{eq:Hermite_expansion_RF_model} coincides with that of its Conditional Gaussian Equivalent counterpart \eqref{eq:CGE_definition}. The proof of Theorem \ref{thm:universality_test_error} is given in Appendix \ref{sec:TestError}.

\begin{figure}
    \centering
    \begin{subfigure}[b]{0.32\textwidth}
         \centering         \includegraphics[width=\textwidth]{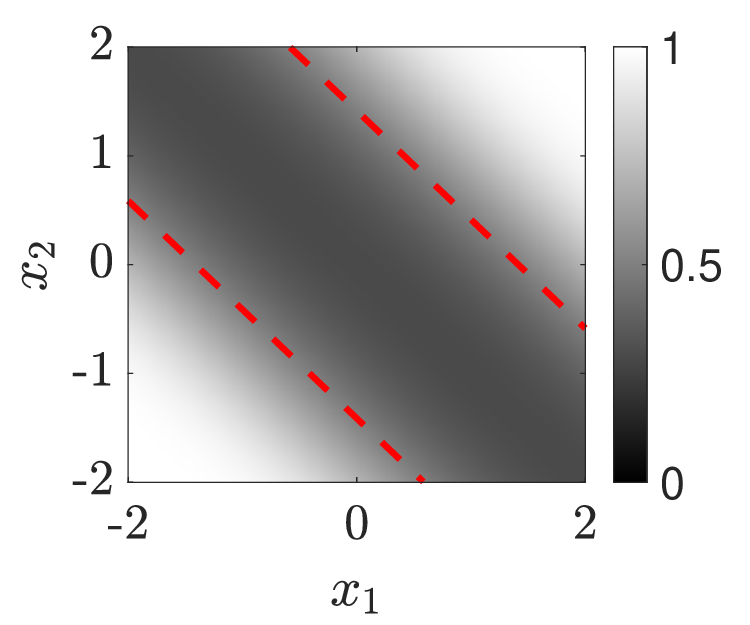}
         \caption{$p/d^2=0.25$}
     \end{subfigure}
     \hfill
     \begin{subfigure}[b]{0.32\textwidth}
         \centering         \includegraphics[width=\textwidth]{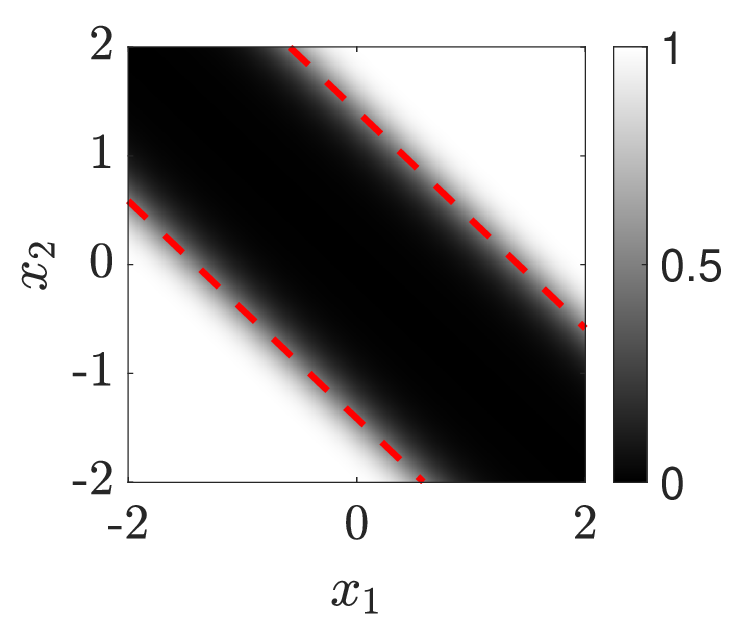}
         \caption{$p/d^2=1.5$}
     \end{subfigure}
     \hfill
     \begin{subfigure}[b]{0.32\textwidth}
         \centering         \includegraphics[width=\textwidth]{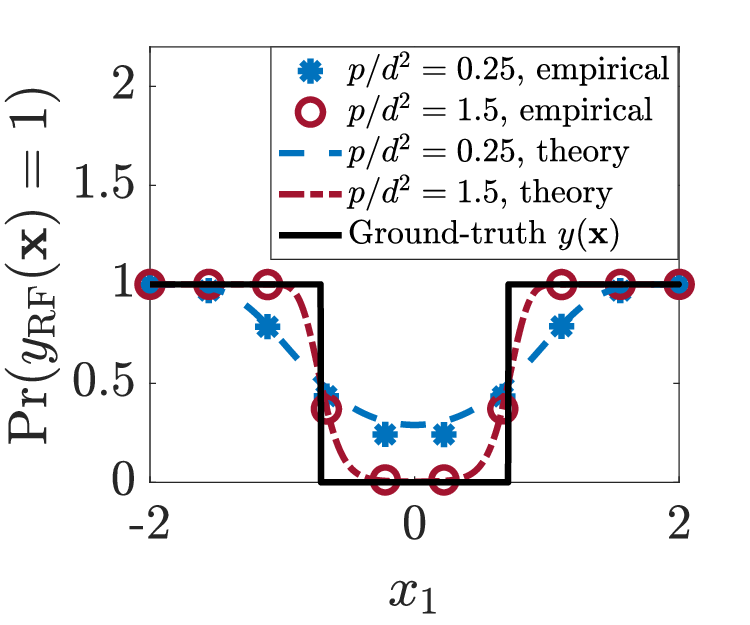}
         \caption{$\mathbb{P}(y_{\rm RF}(\bx) = +1)$}
     \end{subfigure}
     \caption{ (a) and (b): 2D diagram of $\P(y_{\RF}(\bx) = +1)$, where $y_{\RF}(\bx) = \sign(\phi_{\RF} (\bx))$ and $\bx \in \mathbb{R}^d$ is a new test sample, with $x_1 \text{ and } x_2$ being given and $(x_3, x_4, \cdots, x_d) \sim \mathcal{N}(0,\id_{d-3})$. We choose $d=30$, $\lambda = 10^{-3}$, and $y = \sign\left(\He_2\big( \frac{x_1 + x_2}{\sqrt{2}} \big)\right)$. We fix $n/d^2 = 2.5$. The red dash lines correspond to the ground truth boundaries between the two classes $y=+1$ and $y=-1$. (c): Theoretical predictions and empirical results of $\P(y_{\RF}(\bx) = +1)$ along the line $x_2 = x_1$ in the 2D diagram. The black solid line corresponds to the ground-truth label.}
     \label{fig:2d_class_boundary}
\end{figure}

\subsection{Discussion and numerical simulations}
\label{sec:discussion}

\begin{figure}[t]
\centering
\begin{subfigure}[b]{0.35\textwidth}
     \centering
     \includegraphics[width=\textwidth]{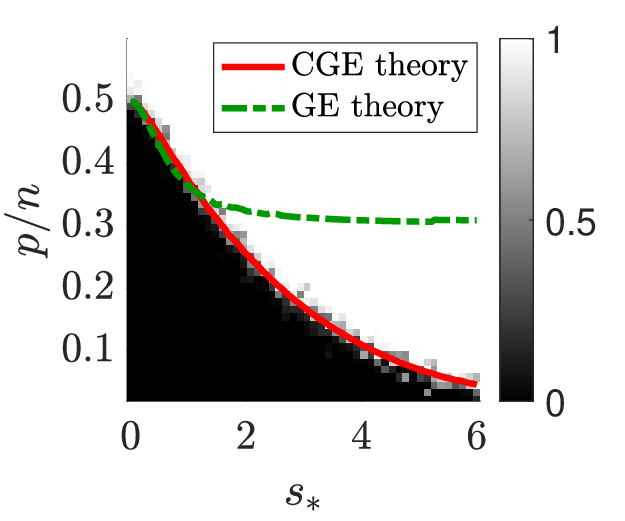}
 \end{subfigure}
 \hspace{4em}
 \begin{subfigure}[b]{0.35\textwidth}
     \centering
     \includegraphics[width=\textwidth]{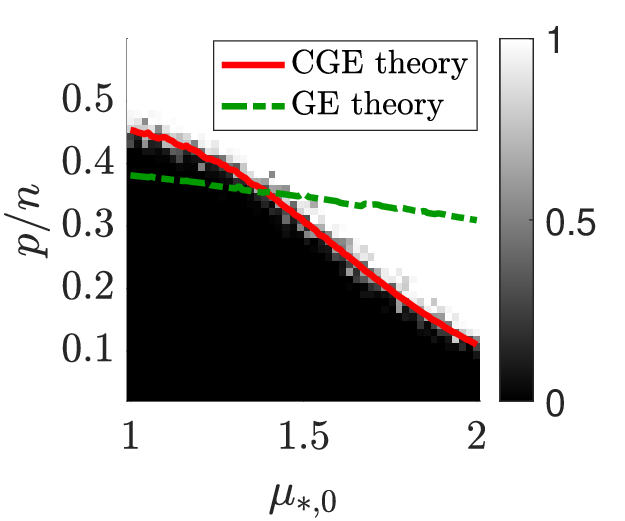}
 \end{subfigure}
\caption{Phase diagram for the existence of an interpolating RF model in binary classification in the quadratic scaling regime. Each pixel value represents the empirical probability that the RF model interpolates the training data $(y_i,\bx_i)_{i\in[n]}$, averaged over 60 independent trials. The red solid lines represent the asymptotic predictions for the phase transition boundary from the CGE model, and the green dashed curves are the asymptotic predictions from the GE model. We fix $d=50$ and $n/d^2 = 0.45$.
The target function is $f_*(\bx) = \sum_{k=0}^{3}\mu_{*,k} \He_k(\bu_{*}^\sT \bx)$
and the label is generated by the logistic model: $\P(y = 1 \mid \bx) = (1 + e^{-s_* f_*(\bx)})^{-1}$. 
Left panel: $\mu_{*,0}=2$, $\mu_{*,1}=1$, $\mu_{*,2}=2$ and $\mu_{*,3}=0.6$ ; Right panel: $s_* = 4$, $\mu_{*,1}=1$, $\mu_{*,2}=2$ and $\mu_{*,3}=0.6$.}
\label{fig:MLE_phasetransition}
\end{figure}

Below we briefly summarize several consequences of Theorems~\ref{thm:universality_train_quadratic} and \ref{thm:universality_test_error}. Due to space constraints, we postpone the explicit asymptotic formulas, their derivations, and full application details to the companion paper \cite{wen2025asymptotics}.

\paragraph*{Exact asymptotics for the CGE model.} Our main results imply that the asymptotic behavior of the RF model can be analyzed directly through the CGE model \eqref{eq:CGE_definition}. Crucially, this model  decomposes into the sum of a low-dimensional non-Gaussian component $\bz^\sCG_{\rm NonG}$ depending only on $\bx_S$, with $s + \frac{s(s+1)}{2}$ features given by $\{ \bx_S, \bh_2 (\bx_S)\}$, and an independent high-dimensional Gaussian model $\bz^\sCG_{\rm Gauss}$. Conditional on $\{\bx_{i,S}\}_{i \in [n]}$, we apply the Convex Gaussian Min–Max Theorem (CGMT) \cite{thrampoulidis2015regularized,thrampoulidis2018precise} to the high-dimensional Gaussian features $\bz^\sCG_{\rm Gauss}$, and combine this with uniform concentration for the low-dimensional component $\bz^\sCG_{\rm NonG}$. This yields an asymptotic characterization of both train and test errors in terms of $O(s^2)$ scalar parameters, determined as fixed points of an explicit deterministic minimax problem. Full formulas and derivations are provided in \cite{wen2025asymptotics}. Figures~\ref{fig:marginal_pdf}–\ref{fig:Double descent} illustrate the resulting theoretical predictions.

\medskip

We now apply this asymptotic characterization to several phenomena arising in binary classification with RF features in the quadratic scaling and single-index labels (generalized linear models).

\paragraph*{Non-linear classifiers.} Prior work in the linear scaling regime focused on linear classifiers and their high-dimensional limits \cite{huUniversalityLawsHighdimensional2022,montanari2023universality,sur2019likelihood,montanari2025generalization}. In the quadratic scaling regime, the learned classifier is generally quadratic, and our framework provides precise asymptotic descriptions of these nonlinear decision boundaries.
Figure \ref{fig:2d_class_boundary} gives an example with single-index model $y = \sign(\He_2((x_1 + x_2)/\sqrt{2}))$, whose true decision boundary consists of two parallel hyperplanes $(x_1 + x_2)/\sqrt{2} = \pm 1$. Additional examples for multi-index models are presented in \cite{wen2025asymptotics}.

\paragraph*{Existence of interpolating solution.} We next study the capacity of the RF model. Specifically, for $n$ samples $(y_i, \bx_i)_{i\in[n]}$, we are interested in characterizing the interpolation threshold $\psi_*$, that is, the smallest value such that when $p/n \ge \psi_*$, there exists $\btheta \in \R^p$ such that $y_i \cdot \btheta^\sT \bz_i^\RF \ge 0  $ for all $i \in [n]$ with high probability. Figure \ref{fig:MLE_phasetransition} provides phase diagrams for the existence of interpolating solutions under different over-parameterization ratios $p/n$, label SNR $s_*$, zeroth-order Hermite coefficient $\mu_{*,0}$, and single-index target.
A sharp interpolation phase transition emerges, and the transition boundaries align closely with our CGE predictions (red curves). In contrast, the GE model fails to correctly capture the phase transition boundary.

\paragraph*{Double descent and benign overfitting.} Figure~\ref{fig:MLE_phasetransition} shows that in the quadratic scaling regime, the RF model can interpolate purely random labels as soon as $p/n \ge 0.5$ ($s_* = 0$). While interpolation is often associated with overfitting, recent work shows that interpolating solutions can still generalize well, a phenomenon known as benign overfitting \cite{Bartlett_2020}. Using the same setting as Figure \ref{fig:MLE_phasetransition}, we show that the CGE model accurately captures benign overfitting in RF models.
Figure \ref{fig:Double descent} shows some numerical results on logistic regression with $\lambda\approx 0$.
In the left panel, as $p/n$ grows, the training loss decreases monotonically, while the test loss exhibits a ``double-descent'' phenomenon with a peak at $\psi_{\max}$. In the right panel, we fix $p/d^2 = 1$ and vary $n$.

\begin{figure}[t!]
\centering
\begin{subfigure}[b]{0.35\textwidth}
     \centering
     \includegraphics[width=\textwidth]{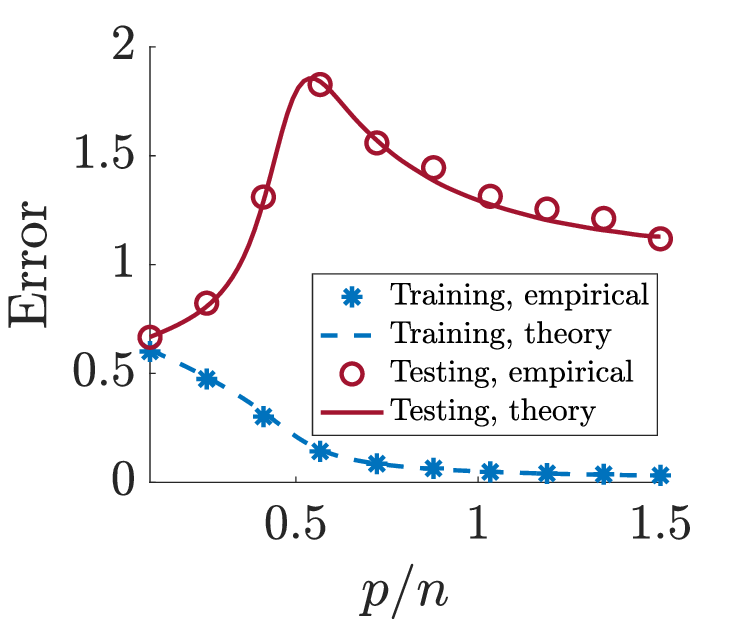}
 \end{subfigure}
 \hspace{4em}
 \begin{subfigure}[b]{0.35\textwidth}
     \centering
     \includegraphics[width=\textwidth]{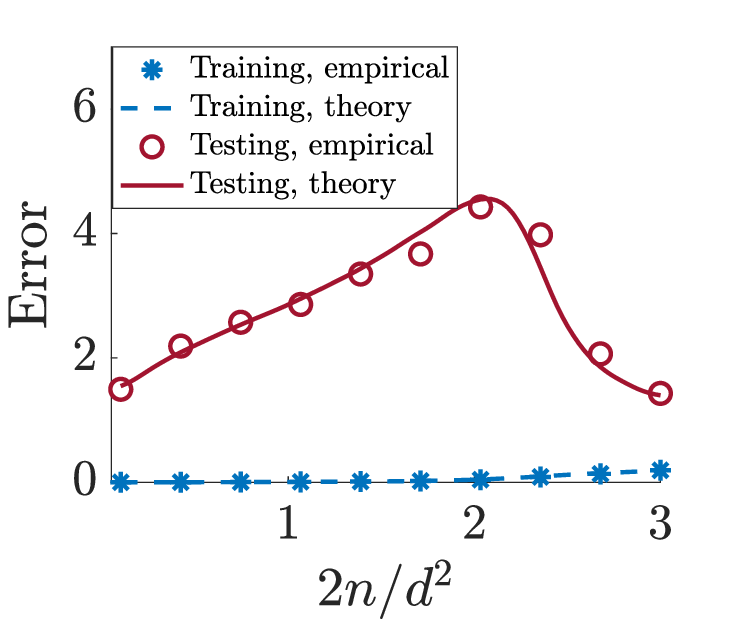}
 \end{subfigure} 
\caption{Benign overfitting and double descent phenomenon in binary classification. We choose $d=40$, $\lambda = 10^{-4}$, and the logistic function $\ell(y,z) = \log(1+e^{-yz})$ for both training and test loss. Left plot: Fix $n/d^2 = 0.75$. Right plot: Fix $p/d^2 = 1$. }
\label{fig:Double descent}
\end{figure}

\section{Proof outline for Theorem \ref{thm:universality_train_quadratic}}
\label{sec:proof_outline}

In this section, we provide the high-level roadmap for proving the training error universality (Theorem \ref{thm:universality_train_quadratic}). We summarize the main ideas and outline key technical steps. The detailed proofs are deferred to Appendices \ref{app:preliminaries}--\ref{sec:lindeberg-phase-ii}.

Our strategy is to prove universality directly for the \emph{perturbed empirical risk} $\widehat{\mathcal R}_{n,p}(\btheta; \btau, \bZ, \boldf)$ defined in~\eqref{equ:perturbed_objectiveERM}, and then transfer the result to the unperturbed case ($\tau_1=0$) using a stability argument (see Appendix~\ref{app:perturbation_stability_ERM} for details).  
Specifically, we will show (see Theorem~\ref{thm:universality_Perturbed_risk} for the precise statement) that there exist constants $c,C>0$, depending only on the constants in the assumptions, such that
\begin{equation}\label{eq:ERM_univeresality_goal}
\left| \E\left[ \varphi \left(\widehat{\mathcal R}_{n,p}^* (\btau,\bZ^\RF,\boldf^\RF) \right) \right] -  \E\left[ \varphi \left(\widehat{\mathcal R}_{n,p}^* ( \btau,\bZ^\sCG,\boldf^\sCG) \right) \right] \right| =  O_d( d^{-c} \tau_1^{-C}), 
\end{equation}
uniformly over all test functions $\varphi : \R \to \R$ with $\| \varphi\|_\infty, \| \varphi'\|_\infty, \| \varphi'' \|_\infty \leq 1$, where
\[
\widehat{\mathcal R}_{n,p}^*(\btau,\bZ,\boldf) := \min_{\btheta}\widehat{\mathcal R}_{n,p}(\btheta;\btau,\bZ,\boldf), \qquad \hbtheta = \argmin_{\btheta}\widehat{\mathcal R}_{n,p}(\btheta;\btau,\bZ,\boldf).
\]
There is by now a well-established approach for proving such universality results, which proceeds in three steps \cite{huUniversalityLawsHighdimensional2022,montanariUniversalityEmpiricalRisk2022}:
\begin{description}
    \item[Step 1.] (\textit{Uniform convergence of marginals over a constrained parameter set.}) Identify a subset of parameters $\bTheta_p \subseteq \R^p$ such that, uniformly over $\btheta \in \bTheta_p$, the marginal distributions of $(f,\<\bz,\btheta\>)$ coincide between the original and the CGE model: 
    \[
    \sup_{\btheta \in \bTheta_p} \left| \E_{\bz^\RF,f^\RF} \left[\psi \left(f^\RF, \<\bz^\RF,\btheta\>\right) \right] - \E_{\bz^{\sCG},f^{\sCG}} \big[\psi \big(f^{\sCG}, \<\bz^{\sCG},\btheta\>\big) \big]\right| = O_d (d^{-c'}\tau_1^{-C'}),
    \]
    uniformly over $1$-bounded $1$-Lipschitz functions $\psi : \R^m \times \R \to \R$.

        \item[Step 2.] (\textit{Unconstrained minimizers lie in $\bTheta_p$ with high probability.}) Show that both unconstrained minimizers $\hbtheta^\RF, \hbtheta^{\sCG}$ belong to $\bTheta_p$ with probability at least $1 - d^{-C}$.
    
    \item[Step 3.] (\textit{Universality of the constrained ERM.}) Using the above two steps, prove universality of the training error when the minimization is restricted to $\bTheta_p$:
 \[
\left|\E\left[ \varphi \left( \min_{\btheta \in \bTheta_p} \widehat{\mathcal R}_{n,p} (\btheta; \btau,\bZ^\RF,\boldf^\RF)\right) \right]  -  \E\left[ \varphi \left(\min_{\btheta \in \bTheta_p}  \widehat{\mathcal R}_{n,p} (\btheta; \btau,\bZ^\sCG,\boldf^\sCG)\right) \right] \right| = O_{d}(d^{-c''}\tau_1^{-C''}). 
\]
This is achieved via an interpolation argument---e.g., by swapping one data point at a time \cite{huUniversalityLawsHighdimensional2022} or via a continuous interpolation path \cite{korada2011applications,montanariUniversalityEmpiricalRisk2022}.

\end{description}
Combining these three steps yields the desired result~\eqref{eq:ERM_univeresality_goal}.

This strategy has been successfully applied to establish Gaussian universality for random feature models in the linear scaling $n \asymp p \asymp d$ \cite{huUniversalityLawsHighdimensional2022,montanariUniversalityEmpiricalRisk2022}. However, extending this approach to the quadratic scaling presents substantial difficulties. The feature map $\bz^\RF = \sigma (\bW\bx)$ is no longer a Lipschitz function in $\bx$ (indeed, $\| \bW\|_\op \asymp \sqrt{d}$), and its marginal distribution cannot be controlled as in \cite{goldt2022gaussian,huUniversalityLawsHighdimensional2022,montanariUniversalityEmpiricalRisk2022}. Instead, we consider its Hermite expansion
\begin{equation}\label{eq:hermite_exp_feature_intuition}
\bz^\RF = \mu_0 \bones_p + \mu_1 \bW \bx + \mu_2 \bV_2 \bh_2 (\bx) + \sum_{k \geq 3} \mu_k \bV_k \bh_k (\bx).
\end{equation}
The feature vector $\bz$ contains contributions from multiple Wiener chaos components that are mutually dependent and exhibit markedly distinct behaviors for $k \in \{0,1,2\}$, and $k\geq 3$---each requiring separate analyses to control. In particular, the feature covariance takes the form
\[
\E[(\bz^\RF)(\bz^\RF)^\sT] = \mu_0^2 \bones \bones_p^\sT + \mu_1^2 \bW \bW^\sT + \mu_2 \bV_2 \bV_2^\sT + \sum_{k \geq 3} \mu_k^2 \bV_k \bV_k^\sT,
\]
and exhibits one spike of size $\Theta(d^2)$ along $\bones_p$, and $d$ spikes of size $\Theta(d)$ along the columns of $\bW$. Our analysis therefore requires precise control of the projections $\btheta^\sT \bones_p$ and $\btheta^\sT \bW$ to prevent diverging terms in Step 3, which in turns necessitates a careful localization of the minimizers $\hbtheta$ along these spike directions in Step~2.

To address these challenges, we proceed as follows:
\begin{itemize}
    \item[(1)] To establish uniform convergence of the marginals over a constrained set, we prove a joint central limit theorem (CLT) for the different Wiener chaos components in the decomposition~\eqref{eq:hermite_exp_feature_intuition}, using a Malliavin--Stein argument \cite{nualart2005central,peccati2004gaussian,nourdin2009stein,nourdin2012normal} (see Section~\ref{sec:abstract_CLT}).  
    The constrained set $\bTheta_p$ is substantially more intricate than in the linear-scaling regime: the components corresponding to $k\in\{0,1,2\}$ and $k \ge 3$ impose different  constraints on~$\btheta$.  
    For instance, establishing the CLT for the order-$2$ chaos term requires showing
    \[
    \Big\| \sum_{j=1}^p \hat\theta_j\, \bw_{j,\setminus S}\bw_{j,\setminus S}^\sT \Big\|_{\op} = O_{d,\P}(d^{-c}).
    \]
        Directly localizing the minimizer $\hbtheta \in \bTheta_p$ with high probability is therefore challenging.  Instead, we localize the minimizers in two stages and introduce an intermediate model.

    \item[(2)]  We define a \emph{Partial Gaussian Equivalent} (PGE) model, serving as an intermediate step between the RF and CGE models.  
    Specifically, in the PGE model we replace the Wiener chaos components of order $k \ge 3$ in the feature-signal pair with their Gaussian equivalents.  
    Denote by $(\bz^\sPG, f^{\sPG})$ the data in this model.  
    The universality proof then proceeds in two phases: first between the RF and PGE models, and subsequently between the PGE and CGE models,
    \[
    \widehat{\mathcal R}_{n,p}^* ( \btau,\bZ^\RF,\boldf^\RF) \;\; \longrightarrow \;\;   \widehat{\mathcal R}_{n,p}^* ( \btau,\bZ^\sPG,\boldf^\sPG)\;\; \longrightarrow \;\;   \widehat{\mathcal R}_{n,p}^* ( \btau,\bZ^\sCG,\boldf^\sCG).
    \]
 In each phase, the proof follows the same three-step structure described above.
\end{itemize}

The remainder of this section is organized as follows.  
Section~\ref{sec:notations_and_conventions} introduces the Partial Gaussian Equivalent (PGE) model. Section~\ref{sec:abstract_CLT} defines the constrained parameter sets $\bTheta_p^{\sPG}$ and $\bTheta_p^{\sCG}$ associated with the two phases, and establishes the uniform CLT for the marginals.  
Section~\ref{sec:lindeberg-phase-is} and Section~\ref{sec:second_phase_step_2_and_3} then describe the first (RF to PGE) and second phase (PGE to CGE) phases of Lindeberg swapping.

\subsection{The Partial Gaussian Equivalent model}\label{sec:notations_and_conventions}

We define the PGE model below, as well as recall the RF and CGE models for clarity.

\paragraph*{Random Feature (RF) model:} Recall that $\bh_k (\bx) \in \R^{B_{d,k}}$ denotes an orthonormal basis of degree-$k$ Hermite polynomials, where $B_{d,k} \asymp d^k$, and let $\bV_k = (\bq_k(\bw_j))_{j \in [p]} \in \R^{p \times B_{d,k}}$ (see Section \ref{sec:setting_GET}). Under Assumptions \ref{assumption:target} and \ref{assumption:activation}, the RF model~\eqref{eq:Hermite_expansion_RF_model} can be written as
\begin{modelbox}{Random Feature (RF) Model}
\begin{equation}
    \begin{aligned}
\bz^\RF
   &:= \mu_{0} \bones_p + \mu_1 \bW \bx + \mu_2 \bV_2 \bh_2 (\bx) + \sum_{k =3}^{D} \mu_{k} \bV_k \bh_k (\bx),
      \\
f^\RF
   &:= \{ \bx_S, \bxi_2, \bxi_3,\ldots, \bxi_{D'} \} ,\qquad \qquad y^\RF := \eta (f^\RF;\eps).
\end{aligned}
\end{equation}
\end{modelbox}
Here, $\bx_S = (x_1, \ldots,x_s)$ and $\bxi_{k} = (\xi_{ki})_{i = 1}^{s_k}$ with $\xi_{ki} = \< \bbeta_{ki}, \bh_k (\bx)\>$.

\paragraph*{Partial Gaussian Equivalent (PGE) model:} The PGE model serves as an intermediate step in which we partially Gaussianize the feature-signal pair:  
we retain the degree-$1$ and degree-$2$ Hermite components and replace all higher-order chaos terms ($k \ge 3$) by their Gaussian equivalents.  
Following the same intuition as in the CGE construction (Section~\ref{sec:setting_CGE}):
\begin{enumerate}
    \item \textit{Higher-order chaos in the feature:}  
    The terms $\bV_k\bh_k(\bx)$ for $k \ge 3$ act as additive Gaussian noise in the feature.  
    We replace them by $\mu_{>2}\bg_*$, where $\bg_* \sim \normal(0,\id_p)$ is independent of $\bx$, and $\mu_{>2}^2 := \|\proj_{>2}\sigma\|_{L^2}^2 = \sum_{k=3}^D \mu_k^2.$

    \item \textit{Higher-order chaos in the response:}  
    Similarly, we replace $\{\bxi_k\}_{k=3}^{D'}$ by $\{\tilde\bxi_k\}_{k=3}^{D'}$, where
    $\tilde\bxi_k = (\tilde\xi_{ki})_{i=1}^{s_k}$ and
    $\tilde\xi_{ki} = \<\bbeta_{ki}, \bg_k\>$ with $\{ \bg_{k}\}_{k=3}^{D'}$, $\bg_k \sim \normal(0,\id_{B_{d,k}})$, mutually independent and independent of $\{\bx,\bg_*\}$.
\end{enumerate}

Combining these gives the feature-signal pair $(\bz^{\sPG}, f^{\sPG})$ under the PGE model:

\begin{modelbox}{Partial Gaussian Equivalent (PGE) Model}
\begin{equation}\label{equ:PGmodel1}
\begin{aligned}
\bz^{\sPG}
   &:=
      \underbrace{\mu_{0}\mathbf{1}_p
      +\mu_{1}\bW\bx
      +\mu_{2}\bV_2\bh_2(\bx)}_{\text{Same as RF model}}
      \;+\;
      \underbrace{\mu_{>2}\,\bg_*}_{\text{Gaussian substitute for } k \ge 3},
      \\
f^{\sPG}
   &:=
     \{\bx_{S}, \bxi_2, \tilde \bxi_3,\ldots,\tilde \bxi_{D'}\}, \qquad \qquad y^\sPG := \eta (f^\sPG;\eps).
\end{aligned}
\end{equation}
\end{modelbox}
We will denote
\[
\hbtheta^{\sPG} = \argmin \widehat{\cR}_{n,p} (\btheta; \tau,\bZ^{\sPG},\boldf^{\sPG}),
\]
where $(\bZ^{\sPG},\boldf^{\sPG})$ are $n$ i.i.d.~samples from the PGE model \eqref{equ:PGmodel1}.

\paragraph{Conditional Gaussian Equivalent (CGE) Model.}
To transition from the PGE model to the CGE model, we replace the components orthogonal to the support direction $\bx_S$ by Gaussian vectors.  For the degree-$1$ chaos,
\[
\bW\bx \;\; \longrightarrow \;\; \bW_S \bx_S + \bW \proj_{\perp S} \bg_1, \qquad \bW= [\bW_S, \bW_{\setminus S}],
\]
where $\bg_1 \sim \normal (0,\id_d)$ is independent of $\{\bx_S,\bg_*,\bg_3,\ldots,\bg_{D'}\}$.
For the degree-$2$ chaos,
\[
\begin{aligned}
\bV_2 \bh_2 (\bx) \;\; \longrightarrow &~\;\; \bV_{2,S} \bh_2 (\bx_S) + \bV_2 \proj_{S,\perp} \bg_2, \\
\bxi_2 = \{ \xi_{2i}\}_{i=1}^{s_2} = \{ \< \bbeta_{2,k}, \bh_2 (\bx) \>\}_{i=1}^{s_2} \;\; \longrightarrow&~ \;\; \tilde \bxi_2 = \{\tilde \xi_{2i}\}_{i=1}^{s_2} = \{ \< \bbeta_{2,k}, \bg_2 \>\}_{i=1}^{s_2},
\end{aligned}
\]
where $\bg_2 \sim \normal (0,\id_{B_{d,2}})$ is independent of $\{\bx_S,\bg_1, \bg_*,\bg_3,\ldots,\bg_{D'}\}$.

The corresponding feature-signal pair $(\bz^{\sCG}, f^{\sCG})$ under the CGE model is given by:

\begin{modelbox}{Conditional Gaussian Equivalent (CGE) Model}
\begin{equation}\label{equ:CGmodel1}
\begin{aligned}
\bz^{\sCG}
   &:=
      \underbrace{\mu_{0}\mathbf{1}_p
      +\mu_{1}\bW_S\bx_S
      +\mu_{2}\bV_{2,S}\bh_2(\bx_S)}_{\text{Same as RF model}}
      +
      \underbrace{\mu_1\bW \proj_{S,\perp}\bg_1 + \mu_2 \bV_2 \proj_{S,\perp} \bg_2 + \mu_{>2}\bg_*}_{\text{Independent Gaussian part}},
      \\
f^{\sCG}
   &:=
     \{\bx_{S}, \tilde \bxi_2,\tilde \bxi_3,\ldots,\tilde \bxi_{D'}\}, \qquad \qquad y^\sCG := \eta (f^\sCG;\eps).
\end{aligned}
\end{equation}
\end{modelbox}

\subsection{Uniform Central Limit Theorem on constrained parameter sets}
\label{sec:abstract_CLT}

To establish a uniform central limit theorem (CLT) for the marginal distributions over a constrained parameter set, we rely on a Malliavin–Stein argument developed in a series of works by \cite{nualart2005central,peccati2004gaussian,nourdin2009stein,nourdin2012normal}. This approach provides quantitative Gaussian approximations for random variables living in a fixed Wiener chaos. We briefly outline the main ideas below, deferring the technical background on Malliavin calculus to Appendix~\ref{sec:preliminaries_malliavin}.

\paragraph*{CLT on Wiener chaos \cite{nourdin2009stein}.} The Malliavin–Stein method combines two key ingredients: (i) a differential calculus on the Wiener space, allowing one to differentiate random variables with respect to the underlying Gaussian noise, and (ii) Stein’s characterization of the normal distribution via the identity $\E[\varphi ' (Z) -Z\varphi (Z)] = 0$ for all smooth $Z$. The method exploits an integration-by-parts formula on the Gaussian space to compare a given random variable $F$ with a standard normal $Z$. Quantitative control of this comparison can be obtained in terms of the Malliavin derivative $DF$, which `differentiates' $F$ with respect to the underlying Gaussian process. When $F$ belongs to a fixed Wiener chaos of order $k$, this analysis reduces to controlling the variance of $\|DF\|^2$, which, in turn, can be expressed through tensor contractions of the coefficients of $F$.

Let $\cH_k$ denote the $k$-th Wiener chaos, corresponding the degree-$k$ homogeneous polynomials of Gaussian variables. Any $F \in \cH_k$ can be represented in tensor form as 
\[
F = \<\bT, \bH_k (\bx)\>,
\]
where $\bT \in \Sym_k (\R^d)$ is a symmetric $k$-tensor, and $\bH_k (\bx)$  is the Hermite tensor of order $k$ (see Appendix~\ref{section:HermitePolynomials}). The isometry between $\cH_k$ and $\Sym_k(\R^d)$ is induced by the linear mapping $\iota : \R^{B_{d,k}} \to \Sym_k(\R^d)$ satisfying, for instance,
\[
\He_k (\<\bw,\bx\>)= \< \bq_k (\bw), \bh_k (\bx)\> = \< \iota(\bq_k (\bw)) , \bH_k (\bx) \> = \< \bw^{\otimes k} , \bH_k (\bx)\>.
\]
Building on this representation, the following result---often referred to as the \textit{Fourth Moment Theorem}---gives a remarkably sharp quantitative CLT within each Wiener chaos.

\begin{theorem}[Quantitative CLT on Wiener Chaos \cite{nualart2005central,nourdin2009stein}]\label{thm:fourth-moment-CLT}
    Fix $k \geq 1$. There exist universal constants $C_k,C_k'> 0$ such that the following holds: For any $F = \< \bT, \bH_k (\cdot) \> \in \cH_k$ with $\E[F^2] =1$, 
    \begin{equation}
        \sup_{\varphi : \R \to \R, \; {\rm Lip} (\varphi) \leq 1} \left| \E[ \varphi (F) ] - \E [ \varphi (Z)] \right| \leq C_k \sqrt{ \E [F^4] - 3 },
    \end{equation}
    where $Z \sim \normal(0,1)$. Moreover, the excess kurtosis  $\E [F^4] - 3$ is controlled by tensor contractions of the coefficients:
    \[
   \frac{1}{C_k'} \left\{ \max_{r=1, \ldots,k-1} \| \bT \otimes_r \bT \|_F^2 \right\}  \leq  \E [F^4] - 3  \leq C'_k  \left\{ \max_{r=1, \ldots,k-1} \| \bT \otimes_r \bT \|_F^2 \right\},
    \]
    where $\bT \otimes_r \bS$ denotes the contraction of $r$ indices between tensors $\bT,\bS \in \Sym_k(\R^d)$.
\end{theorem}

In particular, convergence to a Gaussian law occurs if and only if all nontrivial contractions $\bT \otimes_r \bT$ vanish asymptotically. Hence, verifying a CLT reduces to bounding these contraction norms. A multivariate CLT extension of Theorem \ref{thm:fourth-moment-CLT} also holds: convergence to a multivariate Gaussian with covariance $\bC = (c_{ij})$ occurs whenever each component $F_i$ satisfies $\E [F_i^4] - 3 \to 0$  and $\< F_i,F_j\> \to c_{ij}$. Appendix~\ref{sec:preliminaries_malliavin} provides additional background on Theorem \ref{thm:fourth-moment-CLT}.

In our setting, the relevant tensor coefficients take the form
\begin{equation}
\bT_{k,\bW} (\btheta) := \iota ( \bV_k^\sT \btheta) = \sum_{j = 1}^p \theta_j \bw_j^{\otimes k}.
\end{equation}
We adapt the proof of Theorem~\ref{thm:fourth-moment-CLT} to establish a partial Gaussian multivariate CLT tailored to our setting (Theorem \ref{theorem:RecallBoundedLipschitzFunctionSupErgetPartial} in Appendix \ref{sec:PartialGaussianEquivalence}). Our uniform convergence results then follow by restricting $\btheta$ to a parameter space in which the corresponding tensors $\bT_{k,\bW} (\btheta)$ have all nontrivial contractions vanishing.

\paragraph*{Uniform CLT for the first phase (RF to PGE).} In this first phase, we replace the higher-order chaos components ($k \ge 3$) by their Gaussian equivalents.  
Accordingly, we must show that the following chaos random variables
\[
\<\bV_3^\sT \btheta, \bh_3(\bx) \>,\; \ldots \;,\; \<\bV_D^\sT \btheta, \bh_4 (\bx) \>, \; \{ \< \bbeta_{3i}, \bh_3 (\bx) \>\}_{i = 1}^{s_3}, \; \ldots \; , \; \{ \< \bbeta_{D'i}, \bh_{D'} (\bx) \>\}_{i = 1}^{s_{D'}},
\]
have tensorized coefficients whose nontrivial contractions vanish asymptotically. Equation~\eqref{eq:ass_bbeta_ki_4_th_moment} in Assumption~\ref{assumption:target} ensures precisely this property for the degree-$k$ ($k \ge 3$) chaos components of the response, namely that they have vanishing excess kurtosis and satisfy a CLT.  
Hence, it remains to verify the CLT for the higher-order chaos terms appearing in the features.  For this purpose, we consider the contractions
\[
\bT_{k,\bW} (\btheta) \otimes_r \bT_{k,\bW} (\btheta) = \sum_{i,j = 1}^p \theta_i \theta_{j} \< \bw_i,\bw_j \>^r \bw_i^{\otimes (k-r)} \otimes \bw_j^{\otimes (k-r)}
\]
and show that
\[
\| \bT_{k,\bW} (\btheta) \otimes_r \bT_{k,\bW} (\btheta)\|_F =O_{d,\P}(d^{-c}),
\]
with high probability over $\bW$ for all $3 \leq k \leq D$ and $1 \leq r \leq k-1$, provided $\| \btheta \|_\infty = \Tilde O (d^{-1/4})$.

Thus, we consider the following constrained parameter set in the first phase RF to  PGE:
\begin{equation} \label{equation:Ustarthetap}
  \bTheta_{\bW}^{\sPG}(K)  := \left\{
      \btheta\in\R^{p} : 
      \begin{array}{l}
          \|\btheta\|_{\infty} \leq (\log d)^K d^{-1/4}, \quad \|\btheta\|_{2} \leq (\log d)^K, \\
          |\,\1^{\sT}\btheta| \leq (\log d)^K, \quad
           \bigl\|\mu_{1}\,\btheta^{\sT}\bW\bigr\|_{2} \leq (\log d)^K
      \end{array}
    \right\},
\end{equation}
which ensures a uniform CLT between the RF and PGE models for any fixed $K>0$.  
Later, we will set $K$ sufficiently large---depending only on the constants in the assumptions---so that $\hbtheta \in \bTheta_{\bW}^{\sPG}(K)$ with high probability.  
The constraints in~$\bTheta_{\bW}^{\sPG}(K)$ also control the projections of $\btheta$ onto the spiked directions of the feature covariance, which is essential for handling the Lindeberg swapping step (Step~3).  
For brevity, we may omit the dependence on $\bW$ and $K$ and simply write $\bTheta^{\sPG}$ when clear from context.

\begin{theorem}[Replacing higher-order chaos with isotropic gaussian]
\label{theorem:RecallBoundedLipschitzFunctionSupErgetisotropic}
Let $\mathcal{L}$ denote the class of $L$-Lipschitz functions $\varphi:\R^{m+1} \to \R$. Then for any constants $C,K>0$, there exists constants $c,C'>0$ such that with probability at least $1-d^{-C}$ over $\bW$,
\begin{align}
\sup_{\varphi \in \mathcal{L}}\;\;
\sup_{\btheta \in \bTheta_{\bW}^{\sPG}(K)} \Bigg| \E \Big[ \varphi\big(\bx_S,\bxi_2,\bxi_3,
\ldots,\bxi_{D'}, \btheta^\sT\bz^\RF\big) \Big] \nonumber 
- \E \Big[ \varphi\big(\bx_S, \bxi_2, \tilde \bxi_3,\ldots,\tilde\bxi_{D'},  \btheta^\sT\bz^{\sPG}\big) \Big] \Bigg| \leq C' \frac{L}{d^c}.
\end{align}
\end{theorem}

The proof of this theorem can be found in Appendix \ref{section:ProofOfTheoremGaussianChaosConvergenceToStandardGaussianGeneral}.

\paragraph*{Uniform CLT for the second phase (PGE to CGE).} 
In the second phase, we replace the components of the features orthogonal to the support $\bx_S$ in the degree-$1$ and degree-$2$ Hermite chaos by their Gaussian equivalents.  
Note that $\proj_{S,\perp}\bx$ is already Gaussian and independent of $\bx_S$, so no replacement is needed for this part.  
For the degree-$2$ Hermite term, we can decompose
\[
\btheta^\sT \bV_2 \proj_{S,\perp} \bh_2 (\bx)  = \<\bW_{\setminus S}^\sT \bD_{\btheta}  \bW_{\setminus S} , \bH_2(\bx_{\setminus S})\>  +  \bx_S \bW_S^\sT \bD_\btheta \bW_{\setminus S} \bx_{\setminus S},
\]
where $\bx = (\bx_S,\bx_{\setminus S}) \in \R^{s + (d-s)}$, $\bW = [\bW_S, \bW_{\setminus S}] \in \R^{p \times (s + (d-s))}$, and $\bD_{\btheta} =\diag (\btheta) \in \R^{p\times p}$. For the first term, the only nontrivial contraction condition is equivalent to requiring that 
$\|\bW_{\setminus S}^\sT\bD_\btheta\bW_{\setminus S}\|_{\op}$ tends to zero.  
For the second term, we will simply show that it has vanishing contribution.  
For the degree-$2$ chaos components in the response, Assumption~\ref{assumption:target} again guarantees a CLT.

Hence, the constrained parameter set for this phase adds conditions controlling quadratic interactions beyond those in $\bTheta^{\sPG}$, and is defined as
\begin{equation}\label{equ:thetap2}
  \bTheta_{\bW}^{\sCG}(\eps,K) := \left\{
      \btheta \in \bTheta_\bW^{\sPG}(K):
          \bigl\|\bW_{\setminus S}^{\sT}\bD_{\btheta}\,
                 \bW_{\setminus S}\bigr\|_{\op}
              \leq d^{-\eps},\;
              \bigl\|\bW_S^{\sT}\bD_{\btheta}\,
                 \bW_{\setminus S}\bigr\|_{F}
              \leq d^{-\eps}
    \right\},
\end{equation}
which ensures uniform CLT between PGE and CGE for any fixed $K,\eps>0$. As before, we will choose $\eps>0$ sufficiently small---depending only on the constants in the assumptions---so that $\hbtheta^\sPG \in \bTheta_{\bW}^{\sCG}(\eps,K)$ with high probability. For brevity, we will omit the subscripts and write $\bTheta^{\sCG}$ when the dependence on $\bW$ and $(\eps,K)$ is clear from context.

\begin{theorem}[Replacing second-order chaos with partial isotropic gaussian]\label{theorem:RecallBoundedLipschitzFunctionSupErgettt}
Let $\mathcal{L}$ denote the class of $L$-Lipschitz functions $\varphi:\R^{m+1} \to \R$. 
Then for any constants $C,K,\eps>0$, there exists constants $c,C'>0$ such that with probability at least $1-d^{-C}$ over $\bW$,
\begin{align}
\sup_{\varphi \in \mathcal{L}}\;\;
\sup_{\btheta \in \bTheta_{\bW}^{\sCG}(\eps,K)} \Bigg| \E \Big[ \varphi\big(\bx_S,\bxi_2,\bxi_3,
\ldots,\bxi_{D'}, \btheta^\sT\bz^\sPG\big) \Big] \nonumber 
- \E \Big[ \varphi\big(\bx_S, \tilde \bxi_2, \tilde \bxi_3,\ldots,\tilde\bxi_{D'},  \btheta^\sT\bz^{\sCG}\big) \Big] \Bigg| \leq C' \frac{L}{d^c}.
\end{align}
\end{theorem}

The proof of this theorem can be found in Appendix \ref{app:proof_theorem:RecallBoundedLipschitzFunctionSupErgettt}.

\subsection{Lindeberg Swapping Phase I : RF to PGE}
\label{sec:lindeberg-phase-is}

In the first phase of the Lindeberg swapping argument, we replace all chaos components of degree $k \ge 3$ in the features and responses by Gaussian vectors with matching covariance.  
The result of this step is summarized below.

\begin{theorem}[Phase I of Lindeberg Swapping]\label{sec:main-conditions-lindeberg-swap-framework}
Suppose Assumptions \ref{assumption:scaling}--\ref{assumption:activation} and \ref{ass:test_loss} hold. There exist constants $c,c',C,K_0,d_0>0$ depending only on the constants in these assumptions, such that for all $K_{\Gamma} \geq K_0$, all $d^{-c'} \leq \tau_1 \leq 1/(\log d)^{K_{\Gamma}}$, all $|\tau_2| \leq \tau_1/(\log d)^{K_{\Gamma}}$, and all $d \geq d_0$, the following holds: For any twice-differentiable function $\varphi : \R \to \R$ with $\| \varphi \|_\infty,  \| \varphi' \|_\infty, \| \varphi '' \|_\infty \leq 1$,  
\begin{equation}\label{eq:diff_Lindeberg_phase_1_main_theorem}
 \left| \E \Big[ \varphi \left( \hcR_{n,p}^* (\btau,\bZ^\RF,\boldf^\RF) \right) \Big] - \E \Big[ \varphi \left( \hcR_{n,p}^* (\btau, \bZ^\sPG,\boldf^\sPG) \right) \Big] \right| \leq  \frac{d^{-c}}{\tau_1^C}. 
\end{equation}
\end{theorem}

The proof of this theorem is given in Appendix~\ref{sec:convergence-optimizersspg} and follows the standard Lindeberg interpolation method.  
We construct a sequence of intermediate empirical risk problems by swapping one data point at a time between the two models.  
Let $(\bZ_q, \boldf_q)$ denote the dataset where the first $q$ samples are drawn from the RF model and the remaining $n-q$ samples from the PGE model.  
Thus $(\bZ_n, \boldf_n) = (\bZ^\RF, \boldf^\RF)$ and $(\bZ_0, \boldf_0) = (\bZ^\sPG, \boldf^\sPG)$.  
Define
\[
\Phi_q := \hcR_{n,p}^*(\btau, \bZ_q, \boldf_q),
\]
so that the difference in~\eqref{eq:diff_Lindeberg_phase_1_main_theorem} decomposes as
\[
\left| \E \big[ \varphi \left( \Phi_n \right) \big] - \E \big[ \varphi \left( \Phi_0 \right) \big] \right| \leq \sum_{q=1}^n \left| \E \big[ \varphi \left(\Phi_q \right) \big] - \E \big[ \varphi \left( \Phi_{q-1} \right) \big] \right|.
\]
Each term in the sum is bounded by $O(n^{-1} d^{-c} \tau_1^{-C})$.  
To obtain these bounds, we introduce a leave-one-out objective $\Phi_{\setminus q}$ and perform a second-order Taylor expansion of $\varphi(\Phi_q)$ and $\varphi(\Phi_{q-1})$ around $\Phi_{\setminus q}$, following the approach of~\cite{huUniversalityLawsHighdimensional2022}.  
We further introduce a quadratic surrogate to control the interpolation error (see Appendix~\ref{app:putting_things_together_phase_i_lindeberg} for details).

The key technical step is to show that all minimizers of the intermediate problems lie in the `good parameter' region $\bTheta_{\bW}^{\sPG} (K)$ defined in \eqref{equation:Ustarthetap}, with high probability.  
This is established in Lemmas~\ref{lemma:sublevelgeometry} (Appendix \ref{sec:sublevelsets}) and~\ref{lemma:optinThetaPG} (Appendix \ref{app:localization_Phase_1}).  
Intuitively, the bounds 
\[
|\mu_0 \1_p^\sT \hbtheta|, \quad
\|\mu_1 \bW^\sT \hbtheta\|_2 \le (\log d)^K
\]
follow from the calibrated growth condition (Assumption~\ref{assumption:polynomial-growth}):  
since the empirical risk value is bounded by $\widehat{\cR}_{n,p}(\bzero) \le (\log d)^K$ with high probability, the empirical predictions $\|\hat\by_{\cI}\|_2$ must remain at most $(\log d)^{K'}$ on a sufficiently large subset $\cI$ of samples, which in turn controls the low-dimensional spike directions. Finally, to establish the bound $\|\hbtheta\|_\infty \le (\log d)^K / d^{1/4}$, we compare the objective minimum with that of a modified problem in which the $j$-th feature $\sigma(\<\bw_j, \bx\>)$ is removed.  
By strong convexity, the difference between the two objective values controls $|\hbtheta_j|$, and a union bound over $j \in [p]$ completes the argument.

\subsection{Lindeberg Swapping Phase II : PGE to CGE}
\label{sec:second_phase_step_2_and_3}

In the second phase of the Lindeberg swapping argument, we replace with independent Gaussian vectors the degree-$2$ chaos components in the features and responses of the PGE model that lie outside the signal subspace spanned by $\bx_S$.  
The result of this step is summarized below.

\begin{theorem}[Phase II of Lindeberg Swapping]\label{sec:main-conditions-lindeberg-swap-framework2}
Suppose Assumptions \ref{assumption:scaling}--\ref{assumption:activation} and \ref{ass:test_loss} hold. There exist constants $c,c',C,K_0,d_0>0$ depending only on the constants in these assumptions, such that for all $K_{\Gamma} \geq K_0$, all $d^{-c'} \leq \tau_1 \leq 1/(\log d)^{K_{\Gamma}}$, all $|\tau_2| \leq \tau_1/(\log d)^{K_{\Gamma}}$, and all $d \geq d_0$, the following holds. For any twice-differentiable function $\varphi : \R \to \R$ with $\| \varphi \|_\infty,  \| \varphi' \|_\infty, \| \varphi '' \|_\infty \leq 1$,  
\begin{equation}\label{eq:diff_Lindeberg_phase_2_main_theorem}
 \left| \E \Big[ \varphi \left( \hcR_{n,p}^* (\btau,\bZ^\sPG,\boldf^\sPG) \right) \Big] - \E \Big[ \varphi \left( \hcR_{n,p}^* (\btau, \bZ^\sCG,\boldf^\sCG) \right) \Big] \right| \leq  \frac{d^{-c}}{\tau_1^C}. 
\end{equation}
\end{theorem}

The proof of this theorem is provided in Appendix~\ref{sec:lindeberg-phase-ii}.  
It follows the same general structure as the proof of Theorem~\ref{sec:main-conditions-lindeberg-swap-framework}, and we highlight only the new elements specific to this phase.  The main additional challenge is to ensure that the minimizers lie in the constrained parameter set $\bTheta_{\bW}^{\sCG}(K,\eps)$ defined in~\eqref{equ:thetap2}, with high probability.  
This localization result is established in Proposition~\ref{prop:abstractCLTcondition} (Appendix~\ref{sec:abstract-CLT-conditioniic}).  
The proof requires controlling the operator norm
\[
\| \bW_{\setminus S}^\sT \bD_{\hbtheta} \bW_{\setminus S} \|_\op = \sup_{\bv \in \S^{d-1}} \bv^\sT \bW_{\setminus S} ^\sT\bD_{\hbtheta} \bW_{\setminus S} \bv,
\]
which necessitates a uniform bound over all directions $\bv \in \S^{d-1}$. 

To obtain such a bound, we introduce a \emph{leave-one-direction-out} (LODO) objective.  
Specifically, for each $\bv \in \S^{d-1}$, we define modified data 
$(\bZ_{-\bv}, \boldf_{-\bv})$ by replacing
\[
\bx_i \;\mapsto\; \bx_{i,-\bv} = (\bI - \bv\bv^\sT)\bx_i,
\qquad
\bw_j \;\mapsto\; \bw_{j,-\bv} = \frac{(\bI - \bv\bv^\sT)\bw_j}{\|(\bI - \bv\bv^\sT)\bw_j\|_2},
\]
and consider the corresponding objective
\[
\cP_{n,-\bv} (\btheta) = \widehat{\cR}_{n,p} (\btheta;  \bZ_{-\bv},\boldf_{-\bv}) + \btau \cdot \bGamma^{\bW_{-\bv}} (\btheta) +d^{1/8}\| \btheta\|_\infty,
\]
where the additional regularization term enforces that the minimizer $\check{\btheta}_{-\bv}$ of $\cP_{n,-\bv}$ satisfies 
$\|\check{\btheta}_{-\bv}\|_\infty = \widetilde{O}_d(d^{-1/8})$ uniformly over $\bv$.  
A detailed explanation of this construction and its role in controlling the tensor contractions is given in Appendix~\ref{subsec:phaseII-tensor}.

\appendix

\bibliographystyle{amsalpha}
\bibliography{bibliography.bib}

\clearpage

\section{Preliminaries}
\label{app:preliminaries}

\subsection{Notation for high probability bounds.} For (deterministic or)
random variables $X,Y \geq 0$ depending on $(n,p,d)$, we will write
\[X \prec Y\]
to mean, for any constant $C>0$, there exists a constant $K \equiv K(C)>0$
such that
\[\P[X \leq (\log d)^K \cdot Y] \geq 1-d^{-C}.\]
If $X,Y$ are both deterministic, this means simply that
$X \leq (\log d)^K \cdot Y$ for a constant $K>0$. Here, $K$ may
depend also on other constant quantities that do not depend on $n,p,d$, such as
the constants of Assumptions \ref{assumption:scaling} through
\ref{ass:test_loss}.

For a possibly random set $S$ and (deterministic or) random variables
$X(\btheta),Y(\btheta) \geq 0$ indexed by $\btheta \in S$, we will likewise
write
\[X(\btheta) \prec Y(\btheta) \text{ simultaneously over } \btheta \in S\]
to mean, for any constant $C>0$, there exists a constant $K \equiv K(C)>0$
such that
\[\P[X(\btheta) \leq (\log d)^K \cdot Y(\btheta) \text{ for all }
\btheta \in S] \geq 1-d^{-C}.\]

Throughout our arguments, we will apply repeatedly the Gaussian
hypercontractivity inequality
(c.f.\ \cite[Theorem 6.7]{jansonGaussianHilbertSpaces1997})
for any standard Gaussian vector $\bg$, polynomial function $F$ of degree $D$, 
and a constant $c \equiv c(D)>0$ depending only on $D$,
$\P[|F(\bg)-\E F(\bg)| \geq t\sqrt{\Var F(\bg)}] \leq 2e^{-ct^{2/D}}$.
This implies, in particular, that if $D>0$ is a constant independent of $n,p,d$, then
\[|F(\bg)| \prec |\E F(\bg)|+\sqrt{\Var F(\bg)}.\]

\subsection{Perturbed empirical risk and stability}
\label{app:perturbation_stability_ERM}

Let $\cT_{K_{\Gamma}}(\cdot)$ denote a smooth, increasing function satisfying
\[\cT_{K_{\Gamma}} (x)=x \text{ if } |x| \leq (\log d)^{K_{\Gamma}-1},
\qquad |\cT_{K_{\Gamma}}(x)| \leq (\log d)^{K_{\Gamma}} \text{ for all } x \in \R.\]
Recall that the perturbed empirical risk is defined as
\begin{align}
\label{equ:perturbed_objectiveERM_app}
\widehat{\mathcal R}_{n,p}(\btheta;\btau,\bZ,\boldf)
=\widehat{\mathcal R}_{n,p}(\btheta;\bZ,\boldf)
+\btau \cdot \bGamma^\bW (\btheta),
\end{align}
where 
\[
\Gamma_{1}^\bW (\btheta) = \cT_{K_{\Gamma}} \left( \| \bV_+^\sT \btheta \|_2^2 \right), \qquad \Gamma^\bW_{2} (\btheta) = \cT_{K_{\Gamma}} (L_{\bW} (\btheta)), 
\]
and
\[
\bV_+ = [\mu_0\bones_p , \mu_1\bW]\in\R^{p\times (d+1)}, \qquad 
L_\bW(\btheta)=\E_{\bz^{\sPG},y^{\sPG}}
[\ell_{\test}(y^{\sPG},\langle\btheta,\bz^{\sPG}\rangle)\mid \bW].
\]
Throughout, we take $\tau_1 \geq 0$ and $\| \btau \|_\infty \leq 1/ (\log d)^{K_{\Gamma}}$, so that $|\btau \cdot \bGamma^\bW (\btheta)| \leq 2$ for all $\btheta$.

We denote
\[
\widehat{\mathcal R}_{n,p}^*(\btau,\bZ,\boldf) = \min_\btheta \widehat{\mathcal R}_{n,p}(\btheta;\btau,\bZ,\boldf), \qquad \hbtheta := \argmin_{\btheta} \widehat{\mathcal R}_{n,p}(\btheta;\btau,\bZ,\boldf).
\]
When emphasizing dependencies, we use $\hbtheta_{\btau}$ or $\hbtheta (\btau,\bZ,\boldf)$.

To prove Theorem~\ref{thm:universality_train_quadratic} (universality for the unperturbed empirical risk, $\btau=\bzero$), we first establish universality for the perturbed objective with $\tau_1>0$, and then invoke stability of the ERM to extend the result to $\tau_1=0$. Specifically, we will first prove the following theorem:

\begin{theorem}[Universality of the Perturbed Empirical Risk]\label{thm:universality_Perturbed_risk}
    Suppose Assumptions \ref{assumption:scaling}--\ref{assumption:activation} and \ref{ass:test_loss}
    hold. There exist constants $c,c',C,K_0,d_0>0$ depending only on the constants in these assumptions, such that for all $K_{\Gamma} \geq K_0$, all $d^{-c'} \leq \tau_1 \leq 1/(\log d)^{K_{\Gamma}}$, all $|\tau_2| \leq \tau_1/(\log d)^{K_{\Gamma}}$, and all $d \geq d_0$, the following holds: For any twice-differentiable function $\varphi : \R \to \R$ with $\| \varphi \|_\infty,  \| \varphi' \|_\infty, \| \varphi '' \|_\infty \leq 1$, 
    \[
    \left| \E  \Big[ \varphi \left( \hcR_{n,p}^* (\btau,\bZ^\RF,\boldf^\RF) \right) \Big] - \E  \Big[ \varphi \left( \hcR_{n,p}^* (\btau, \bZ^\sCG,\boldf^\sCG) \right) \Big] \right| \leq  \frac{d^{-c}}{\tau_1^C}.
    \]
\end{theorem}

To compare with the case $\tau_1=0$, we next establish a stability property of the empirical risk.

\begin{lemma}[Empirical Risk stability]\label{thm:ERM_stability}
    Suppose Assumptions \ref{assumption:scaling}--\ref{assumption:activation} and \ref{ass:test_loss}
    hold. Then for any $\tau_1 \geq 0$ and all 1-Lipschitz functions $\varphi : \R \to \R$,
    \begin{equation}\label{eq:ERM_tau_1_0_diff}
    \begin{aligned}
   \E  \Big[ \left| \varphi \left( \hcR_{n,p}^* (\btau,\bZ,\boldf) \right) - \varphi \left(\hcR_{n,p}^* ((0,\tau_2),\bZ,\boldf) \right) \right| \Big]  \leq&~ \tau_1 (\log d)^{K_\Gamma},
    \end{aligned}
    \end{equation}
    where $(\bZ,\boldf)\in \{(\bZ^\RF,\boldf^\RF), (\bZ^\sCG,\boldf^\sCG)\}$.
\end{lemma}
\begin{proof}  
Since $\tau_1 \cdot \Gamma_1 (\btheta)$ is nonnegative, using the definition of the perturbed risk,
    \[
    \widehat{\mathcal R}_{n,p}^{*}((0,\tau_2),\bZ,\boldf) \leq \widehat{\mathcal R}_{n,p}^{*}(\btau,\bZ,\boldf) \leq \widehat{\mathcal R}_{n,p}(\hbtheta_{(0,\tau_2)};\btau,\bZ,\boldf) = \widehat{\mathcal R}^{*}_{n,p}((0,\tau_2),\bZ,\boldf)  + \tau_1 \cT_{K_{\Gamma}} ( \| \bV_+^\sT \hbtheta_{(0,\tau_2)}\|_2^2).
    \]
    The lemma follows from applying $|\cT_{K_{\Gamma}}(\cdot)| \leq (\log d)^{K_\Gamma}$.
 
\end{proof}

Combining Theorem \ref{thm:universality_Perturbed_risk} and Lemma \ref{thm:ERM_stability} and setting $\tau_1 = d^{-\eps}$ for sufficiently small $\eps>0$ yields Theorem \ref{thm:universality_train_quadratic}. Thus, Appendices \ref{sec:convergence-optimizersspg} and \ref{sec:lindeberg-phase-ii} are devoted to establishing Theorem \ref{thm:universality_Perturbed_risk} for the perturbed empirical risk with $\tau_1>0$.

\subsection{Basic bounds for features and labels.}
\begin{lemma}\label{lemma:yzbounds}
For each of the Random Features, Partial Gaussian, and Conditional Gaussian
models, we have
\[\|\bz^\RF\|_2 \prec d, \quad \|\bz^\sPG\|_2 \prec d,
\quad \|\bz^\sCG\|_2 \prec d,
\quad |y^\RF| \prec 1, \quad |y^\sPG| \prec 1,
\quad |y^\sCG| \prec 1.\]
\end{lemma}
\begin{proof}
Expanding $\bz^\RF=\sum_{k=0}^D \mu_k\bV_k\bh_k(\bx)$ and writing
$\E_\bx$ for the expectation over $\bx$ only (i.e.\ conditional on $\bW$),
we have
\[\E_\bx\|\bz^\RF\|_2^2=\Tr \sum_{k=0}^D \mu_k^2 \bV_k\bV_k^\sT.\]
Recalling $\bV_k=[\bq_k(\bw_1),\ldots,\bq_k(\bw_p)]^\sT$, we have
\[\Tr \bV_k\bV_k^\sT=\sum_{i=1}^p \<\bq_k(\bw_i),\bq_k(\bw_i)\>
=\sum_{i=1}^p \<\bw_i^{\otimes k},\bw_i^{\otimes k}\>
=\sum_{i=1}^p (\|\bw_i\|_2^2)^k=p\]
for each $k=0,\ldots,D$. Furthermore $|\mu_k| \prec 1$
by Assumption \ref{assumption:activation}, so
$\E_\bx\|\bz^\RF\|_2^2 \prec p \asymp d^2$.
Then by Gaussian hypercontractivity over $\bx$, we have
$\|\bz^\RF\|_2 \prec d$. For $\bz^\sPG$ and $\bz^\sCG$, we have similarly
\[\E_{\bx,\bg_*}\|\bz^\sPG\|_2^2
=\E_{\bx,\bg_1,\bg_2,\bg_*}\|\bz^\sCG\|_2^2=\Tr\Big(\sum_{k=0}^2
\mu_k^2\bV_k\bV_k^\sT+\mu_{>2}^2\bI\Big)
\prec d^2\]
with expectations taken over the independent Gaussian variables
$\bx,\bg_*,\bg_1,\bg_2$ defining $\bz^\sPG,\bz^\sCG$. Then applying
hypercontractivity over $(\bx,\bg_*,\bg_1,\bg_2)$, also
$\|\bz^\sPG\|_2,\|\bz^\sCG\|_2 \prec d$.

For the labels, we have $y^\RF=\eta(\bx_S,\bbeta_2^\sT
\bh_2(\bx),\ldots,\bbeta_{D'}^\sT \bh_{D'}(\bx),\eps)$ where $\bbeta_k=[\bbeta_{k,1},\ldots,\bbeta_{k,s_k}]$. The
columns of each $\bbeta_k \in \R^{B_{d,k} \times s_k}$ have
unit norm, so $\E\|\bbeta_k^\sT \bh_k(\bx)\|_2^2=s_k$. Then
Gaussian hypercontractivity over $\bx$ implies
$\|\bbeta_k^\sT \bh_k(\bx)\|_2 \prec 1$. We have also
$|\eps| \prec 1$ by Assumption \ref{assumption:noise}.
Then by Assumption~\ref{assumption:target} for $\eta(\cdot)$,
$|y^\RF| \prec 1$. Similarly, for $y^\sPG$ and $y^\sCG$, we may apply also
$\|\bbeta_k^\sT \bg_k\|_2 \prec 1$ for each $k=2,\ldots,{D'}$ and the
Gaussian vector $\bg_k \in \R^{B_{d,k}}$ defining $y^\sPG,y^\sCG$. Then
$|y^\sPG|,|y^\sCG| \prec 1$.
\end{proof}

\subsection{Misclassification error under linear predictors}

In this section, we prove a technical lemma in the case of binary classification under Assumption \ref{assumption:polynomial-growth}(ii) and Assumption \ref{assumption:bayes-error}.

\begin{lemma}\label{lem:linear_predictor_misclassification}
Under Assumption \ref{assumption:bayes-error},
there exist constants $C,c,c_0>0$ such that with probability at least $1 - Ce^{-cn}$,
    \begin{equation}
          \inf_{\btheta \in \R^p}   \frac{1}{n} \sum_{i=1}^n \1\left(y_i^\RF={-}\sign(\btheta^\sT \1_p + \btheta^\sT \bW\bx_i)\right) \ge c_0 .
    \end{equation}
An identical statement holds for the PGE and CGE labels $y_i^\sPG$, $y_i^\sCG$ as well as the mixed data $\{(\bx_i,y_i)\}_{i=1}^n$ for the intermediary problems of the two Lindeberg swapping phases.
\end{lemma}

\begin{proof}
    The proof is identical for all data $\{(\bx_i,y_i)\}_{i=1}^n$ with $y_i \in \{y_i^\RF,y_i^\sPG,y_i^\sCG\}$: By the definition of $f_i \in \{f_i^\RF,f_i^\sPG,f_i^\sCG\}$, there exists a constant $\delta>0$ for which $\P[\|f_i\|_2 \leq 1] \geq \delta$.
    Then by Assumption \ref{assumption:bayes-error}, there exists a constant $c_0 \in (0,1/2)$ such that for all $(a,\bb) \in \R \times \R^d$,
    \[
    \P[y_i={-}\sign(a + \bb^\sT \bx_i)]=\E[\P[\eta(f_i,\eps) ={-}\sign(a + \bb^\sT \bx_i)  \mid \bx_i,f_i]] \geq c_0\,\P[\|f_i\|_2 \leq 1] \geq c_0\delta.
    \]
    Let $f_{a,\bb}(\bx,y)=\1\{y=-\sign(a+\bb^\sT \bx)\}$. The function class $\{f_{a,\bb}:(a,\bb) \in \R \times \R^d\}$ has VC-dimension at most that of the half-spaces $\{a+\bb^\sT \bx<0:(a,\bb) \in \S^d\}$ which is $d+1$. Thus (c.f.\ \cite[Theorem 8.3.23]{vershyninHighDimensionalProbabilityIntroduction2018})
    for a universal constant $C>0$, 
    \[\E \underbrace{\sup_{(a,\bb) \in \R \times \R^d}
    {-}\frac{1}{n}\sum_{i=1}^n \1\{y_i=-\sign(a+\bb^\sT \bx_i)\}}_{:=Z}
    \leq \sup_{i=1}^n
   {-}\P[y_i={-}\sign(a + \bb^\sT \bx_i)]+C\sqrt{\frac{d}{n}} \leq {-}\frac{c_0\delta}{2}.\]
    Applying Talagrand's concentration inequality \cite[Theorem 1.4]{talagrand1996new}
    $\P[Z \geq \E Z+t] \leq Ce^{-cnt\log(1+t)}$ for universal constants $C,c>0$ any $t>0$, we have with probability $1-e^{-c'n}$ that $Z \leq {-}c_0\delta/4$, i.e.\
    \[
    \inf_{(a,\bb)\in \R \times \R^d} \frac{1}{n} \sum_{i=1}^n \1\left(y_i={-}\sign(a + \bb^\sT \bx_i)\right) \geq \frac{c_0\delta}{4}.
    \]
    Applying this result to $(a,\bb) = (\bones_p^\sT \btheta, \bW^\sT \btheta)$ and adjusting $c_0$ finishes the proof.
\end{proof}

\clearpage
\section{Gaussian Equivalence Theorems}

\subsection{Preliminaries on Malliavin Calculus}
\label{sec:preliminaries_malliavin}
This section provides a brief overview of the necessary concepts from Malliavin calculus and Wiener chaos expansions that are used in our analysis. For a more comprehensive treatment, we refer the reader to the Chapter 1 of \cite{MalliavinCalculusRelated2006}.

Let $(\R^d)^{\odot k}$ be the space of symmetric tensors in
$(\R^d)^{\otimes k}$, where   $(\R^d)^{\otimes k}$ is the space of general $k$-th order tensors.  Denote by $\Sym:(\R^d)^{\otimes k} \to (\R^d)^{\odot k}$
the symmetrization map
\begin{equation}\label{eq:Sym}
\Sym(\bT)_{i_1i_2\ldots i_k}=\frac{1}{k!}\sum_{\text{permutations } \pi
\text{ of } \{1,2,\ldots,k\}} T_{i_{\pi(1)}i_{\pi(2)}\ldots i_{\pi(k)}}.
\end{equation}
For any $k,\ell \geq 1$, $r \in \{0,1,\ldots,\min(k,\ell)\}$, 
and symmetric tensors $\bS \in (\R^d)^{\odot k}$ and $\bT \in (\R^d)^{\odot
\ell}$, denote the
partial contraction $\otimes_r$ over $r$ coordinates by
\begin{align}\label{eq:partial_contraction}
(\bS \otimes_r \bT)_{i_1\ldots i_{k-r}j_1\ldots j_{\ell-r}}
=\sum_{a_1,\ldots,a_r=1}^d S_{i_1 \ldots i_{k-r} a_1\ldots a_r}
T_{j_1 \ldots j_{\ell-r} a_1\ldots a_r},\end{align}
where $\otimes_0 \equiv \otimes$ for $r=0$ is the usual tensor product.
Then define the symmetrized contractions
$\tilde \otimes_r:(\R^d)^{\odot k} \times
(\R^d)^{\odot \ell} \to (\R^d)^{\odot k+\ell-2r}$ by 
\[\bS \tilde \otimes_r \bT=\Sym(\bS \otimes_r \bT).\]

 Let $\bx \in \R^d$ denote a standard Gaussian vector
 $\bx \sim \normal(0,\id_d)$.
 Let $L^2(\cG)$ be the space of functions of $\bx$ with finite second
 moment, equipped with the Gaussian inner-product
 $\langle F,G \rangle_{L^2(\cG)}=\E[F(\bx)G(\bx)]$.
 For each $k=0,1,2,\ldots$ define the $k^\text{th}$ Wiener chaos space
 \[\cH_k=\operatorname{span}\Big\{\He_\bk(\bx):\bk \in \{0,1,2,\ldots\}^d,\,
 \|\bk\|_1=k\Big\} \subset L^2(\cG),\]
 the closed linear subspace of $L^2(\cG)$ spanned by the multivariate Hermite
 polynomials of degree $k$. For example, $\cH_0$ is the space of constant
 functions, $\cH_1$ is the space of mean-zero linear functions of $\bx$, etc.
 Let $I_0:\R \to \cH_0$ be given by $I_0(t)=t$, and let
 $I_k:(\R^d)^{\odot k} \to \cH_k$ for $k \geq 1$ be the linear map defined by
 \[I_k(\bT)=\sqrt{k!}\<\bT,\iota(\bh_k(\bx))\>\]
where $\iota:\R^{B_{d,k}} \to (\R^d)^{\odot k}$ is
the isometry defined in Appendix \ref{section:HermitePolynomials}, satisfying
$\<\iota(\bu),\iota(\bv)\>=\<\bu,\bv\>$. Then by the orthogonality of the
Hermite polynomials in $L^2(\cG)$, $I_k$ satisfies the isometric property
\begin{equation}\label{eq:Ikisometry}
\<I_k(\bS),I_k(\bT)\>_{L^2(\cG)}=k!\<\bS,\bT\>.
\end{equation}
Note that by Lemma
\ref{lemma:SymmetricTensorProductMalliavinS}, $\cH_k$ is equivalently the linear
span of $\{\He_k(\<\bw,\bx\>):\bw \in \mathbb{S}^{d-1}\}$, and
\[I_k(\bw^{\otimes k})=\sqrt{k!}\,\He_k(\<\bw,\bx\>)
\text{ for all } \bw \in \mathbb{S}^{d-1}.\]
 
 We denote by $J_k:L^2(\cG) \to \cH_k$ the orthogonal projection onto $\cH_k$.
 In particular, $J_0(F)=\E F(\bx)$ is its mean.
 For any $F \in L^2(\cG)$,
 there exists a unique symmetric tensor
 $\bT_k \in (\R^d)^{\odot k}$ for each $k=0,1,2,\ldots$ such that
 $J_k(F)=I_k(\bT_k)$, and thus
 \[F=\sum_{k=0}^\infty J_k(F)=\sum_{k=0}^\infty I_k(\bT_k).\]
 We define the Malliavin derivative $D:\dom(D) \to (L^2(\cG))^d$ by
 \begin{equation}\label{eq:malliavingradient}
 DF=(\partial_{x_1} F,\ldots, \partial_{x_d} F),
 \qquad \dom(D)=\Big\{F \in L^2(\cG):\sum_{k=0}^\infty
 k\|J_k(F)\|_{L^2(\cG)}^2<\infty\Big\}
 \end{equation}
 where, for smooth functions $F$ with compact support,
$\partial_{x_j}$ is the usual partial derivative of
$F(\bx)=F(x_1,\ldots,x_d)$ in the variable $x_j$, and this is extended by
completion to $\dom(D)$.
 We also define the Ornstein-Uhlenbeck infinitesimal generator $L:\dom(L) \to
 L^2(\cG)$ by
 \[LF=\sum_{k=0}^\infty {-}k\,J_k(F),
 \qquad \dom(L)=\Big\{F \in L^2(\cG):\sum_{k=0}^\infty
 k^2\|J_k(F)\|_{L^2(\cG)}^2<\infty\Big\},\]
 and we define its inverse $L^{-1} F=\sum_{k=1}^\infty {-}\frac{1}{k} J_k(F)$
 whenever $J_0(F)=\E F(\bx)=0$.
 
 Our computations rest on the following three identities.
 
 \begin{lemma}[\cite{nourdin2009stein}, Eq.\ (2.29)]
 \label{lemma:SymmetricTensorProductMalliavin}
 For any $k,\ell \geq 1$, $\bS \in (\R^d)^{\odot k}$,
 and $\bT \in (\R^d)^{\odot \ell}$,
 \[I_k(\bS)I_\ell(\bT)=\sum_{r=0}^{\min(k,\ell)} r! \binom{k}{r} \binom{\ell}{r}
 I_{k+\ell-2r}(\bS \tilde \otimes_r \bT).\]
 \end{lemma}
 
 \begin{lemma}[\cite{nualartCentralLimitTheorems2007}, Lemma 2]\label{lemma:DIidentity}   
 For any $k,\ell \geq 1$, $\bS \in (\R^d)^{\odot k}$,
 and $\bT \in (\R^d)^{\odot \ell}$,
 \[(DI_k(\bS))^{\sT} (DI_\ell(\bT))=k\ell\sum_{r=1}^{\min(k,\ell)}
 (r-1)!\binom{k-1}{r-1}\binom{\ell-1}{r-1}I_{k+\ell-2r}(\bS \tilde \otimes_r
 \bT).\]
 \end{lemma}
 
\begin{lemma}[Integration-by-parts]
\label{lemma:GaussianMalliavin}
Let $F \in \dom(D) \subset L^2(\cG)$ satisfy $\E[F(\bx)]=0$, and let $G \in
\dom(D)$. Then for any smooth and bounded function $f: \R \to \R$, we have
\[
\mathbb{E} \left[ F(\bx) f(G(\bx)) \right]
= \mathbb{E} \left[ \left\langle DG(\bx),\, {-}DL^{-1} F(\bx) \right\rangle f'\big(G(\bx)\big) \right].
\]
\end{lemma}
\begin{proof}
This follows as in \cite[Theorem 3.1]{nourdin2009stein}, writing
$F=LL^{-1}F={-}\delta DL^{-1}F$ where $\delta$ is the 
adjoint of $D$ satisfying $\E[F \cdot \delta v]=\E\<DF,v\>$. Hence
$\E[F f(G)]={-}\E[(\delta DL^{-1} F)f(G)]
=\E[f'(G)\<DG,{-}DL^{-1}F\>]$, where we used the chain rule $Df(G)=f'(G)DG$.
\end{proof}

In the case of $F(\bx)=\langle \bw,\bx \rangle \in \cH_1$ 
and $G(\bx)=\langle \bw',\bx \rangle \in \cH_1$ for $\bw,\bw'
 \in \R^d$, we have $-L^{-1}F=F$, $DF(\bx)=\bw$, and $DG(\bx)=\bw'$,
so this gives the usual Gaussian integration-by-parts identity
 \[\E[\langle \bw,\bx \rangle f(\langle \bw',\bx \rangle)]
 =\<\bw,\bw'\>\,\E f'(\langle \bw',\bx \rangle).\]
Lemma \ref{lemma:GaussianMalliavin} may be understood as a generalization
 of this identity to higher-order Hermite polynomial functions $F,G \in \cH_k$
and linear combinations of such functions.

\subsection{Excess kurtosis bound}
\label{app:excess-kurtosis}

\begin{lemma}\label{section:ProofOfLemmaGaussianChaosConvergenceToStandardGaussian}

Fix any $k \geq 1$, and let $Q=I_k(\bT) \in \cH_k$ for some $\bT \in
(\R^d)^{\odot k}$. Then there is a constant $C>0$ depending only on $k$ for which
\[C^{-1}\max_{r=1}^{k-1} \|\bT \otimes_r \bT\|_2^2 \leq \E[Q^4]-3\E[Q^2]^2 \leq C\max_{r=1}^{k-1} \|\bT \otimes_r \bT\|_2^2.\]
\end{lemma}
\begin{proof}An abstract version of
this result is shown in the arguments of \cite{nualart2005central}, which for
convenience we reproduce here for our case:
By Lemma \ref{lemma:SymmetricTensorProductMalliavin},
\begin{align}
Q^2&=\sum_{r=0}^k r! \binom{k}{r}^2 I_{2(k-r)}(\bT \tilde{\otimes}_r \bT).
\end{align}
Taking the expected square and using the isometric property
\eqref{eq:Ikisometry} of $\{I_k\}_{k \geq 0}$, we have
\begin{align}
\E[Q^4]=\E[(Q^2)^2]
&=\sum_{r=0}^k (r!)^2\binom{k}{r}^4 \E[I_{2(k-r)}(\bT \tilde \otimes_r
\bT)^2]\notag\\
&=\sum_{r=0}^k (r!)^2 \binom{k}{r}^4[2(k-r)]!\|\bT \tilde \otimes_r \bT\|_2^2\\
&=(k!)^2\|\bT\|_2^4+(2k)!\|\bT \tilde \otimes_0 \bT\|_2^2
+\sum_{r=1}^{k-1} (r!)^2 \binom{k}{r}^4[2(k-r)]!\|\bT \tilde \otimes_r
\bT\|_2^2,\label{eq:Q4}
\end{align}
the final equality isolating the two summands for $r=k$ and $r=0$.
For the $r=0$ term, we have
\[\|\bT \tilde \otimes_0 \bT\|_2^2
=\|\Sym(\bT \otimes \bT)\|_2^2
=\frac{1}{(2k)!^2}\sum_{\pi,\pi'} \langle \pi(\bT \otimes \bT),
\pi'(\bT \otimes \bT) \rangle\]
where the summation is over all pairs of permutations $\pi,\pi'$ of
$\{1,\ldots,2k\}$, and $\pi(\bT \otimes \bT)$ is the tensor that
permutes the indices of $\bT \otimes \bT$ according to $\pi$. Since $\bT$ is
symmetric, it may be checked that
$\langle \pi(\bT \otimes \bT),\pi'(\bT \otimes \bT) \rangle$ depends
only on the cardinality $r=|\{\pi(1),\ldots,\pi(k)\} \cap
\{\pi'(1),\ldots,\pi'(k)\}|$ and is given by
\[\langle \pi(\bT \otimes \bT),\pi'(\bT \otimes \bT) \rangle
=\|\bT \otimes_r \bT\|_2^2.\]
For each fixed $\pi$, the number of permutations $\pi'$ with
$r=|\{\pi(1),\ldots,\pi(k)\} \cap \{\pi'(1),\ldots,\pi'(k)\}|$ is
$\binom{k}{r}^2(k!)^2$, so
\[\|\bT \tilde \otimes_0 \bT\|_2^2
=\frac{1}{(2k)!}\sum_{r=0}^k \binom{k}{r}^2(k!)^2
\|\bT \otimes_r \bT\|_2^2
=2 \cdot \frac{(k!)^2}{(2k)!}\|\bT\|_2^4
+\sum_{r=1}^{k-1} \frac{(k!)^2}{(2k)!}\binom{k}{r}^2\|\bT \otimes_r \bT\|_2^2,\]
the last equality again isolating the terms for $r=k$ and $r=0$. Applying this
to \eqref{eq:Q4},
\[\E[Q^4]=3(k!)^2\|\bT\|_2^4
+\sum_{r=1}^{k-1} (r!)^2 \binom{k}{r}^4[2(k-r)]!\|\bT \tilde \otimes_r
\bT\|_2^2+\sum_{r=1}^{k-1}
(k!)^2\binom{k}{r}^2\|\bT \otimes_r \bT\|_2^2.\]
The lemma then follows from the identity
$3\E[Q^2]^2=3\E[I_k(\bT)^2]^2=3(k!)^2\|\bT\|_2^4$, and the bound
$\|\bT \tilde \otimes_r \bT\|_2^2
\leq \|\bT \otimes_r \bT\|_2^2$.
\end{proof}

Lemma \ref{section:ProofOfLemmaGaussianChaosConvergenceToStandardGaussian}
allows us to give a different formulation of the genericity condition
\[\E[\<\bbeta_{ki},\bh_k(\bx)\>^4] \leq 3+d^{-c}\]
for the unit vectors $\bbeta_{ki}$ of Assumption \ref{assumption:target}:
Let $\iota(\bbeta_{ki}) \in (\R^d)^{\odot k}$ be the symmetric tensor embedding
of $\bbeta_{ki}$, so that
\begin{align}
\<\bbeta_{ki},\bh_k(\bx)\>
=\<\iota(\bbeta_{ki}),\iota(\bh_k(\bx))\>
=I_k\left(\frac{\iota(\bbeta_{ki})}{\sqrt{k!}}\right).
\end{align}
By Lemma \ref{section:ProofOfLemmaGaussianChaosConvergenceToStandardGaussian},
the genericity condition \eqref{eq:ass_bbeta_ki_4_th_moment} is then equivalent to
the existence of a constant $c'>0$ such that for all $k=2,\ldots,D'$ and
$i=1,\ldots,s_k$,
\begin{align}
\|\bbeta_{ki}\|_2=1,
\qquad \|\iota(\bbeta_{ki}) \otimes_r \iota(\bbeta_{ki})\|_F \leq d^{-c'}
\text{ for every } 1 \leq r \leq k-1.
  \label{eq:admissibility_tensor}
\end{align}
We may check, for example, that this condition holds when
$\bbeta_{ki}$ are uniformly distributed on the unit sphere:

\begin{lemma}\label{lemma:admissibility_condition_example}
If $\bbeta_{ki} \sim \Unif(\S^{B_{d,k}-1})$, then for any constants $C>0$ and
$c' \in (0,1/2)$, $\bbeta_{ki}$ satisfies \eqref{eq:admissibility_tensor}
with probability at least $1-d^{-C}$ for all large $d$.
\end{lemma}
\begin{proof}
Let us represent $\bbeta_{ki}=\bg/\|\bg\|_2$, where $\bg \sim \normal(0,\bI) \in
\R^{B_{d,k}}$.
Write $\bG=\iota(\bg) \in (\R^d)^{\odot k}$. Fixing $r \in \{1,\ldots,k-1\}$,
we have $(\bG \otimes_r \bG)_{I,I'}=\sum_A \bG_{I,A}\bG_{I',A}$
where $I=(i_1,\ldots,i_{k-r})$, $I'=(i_1',\ldots,i_{k-r}')$, and
$A=(a_1,\ldots,a_r)$ denote tuples of indices in $[d]$. Thus
\[\|\bG \otimes_r \bG\|_F^2
=\sum_{I,I'}\Big(\sum_A \bG_{I,A}\bG_{I',A}\Big)^2
=\sum_{I,I',A,A'}\bG_{I,A}\bG_{I',A}\bG_{I,A'}\bG_{I',A'}.\]
Note that two entries $\bG_{I,A}$ and $\bG_{I',A'}$ of $\bG$ are independent
unless $(I,A)$ is a permutation of $(I',A')$. Thus
$\E[\bG_{I,A}\bG_{I',A}\bG_{I,A'}\bG_{I',A'}]=0$ unless either
$I'$ is a permutation of $I$, $A'$ is a permutation of $A$, or
$(I',A')$ is a permutation of $(I,A)$. The number of tuples $(I,I',A,A')$
satisfying each case is at most $Cd^{k-r}d^rd^r$,
$Cd^rd^{k-r}d^{k-r}$, and $Cd^k$ respectively, for a constant $C>0$ depending
only on $(k,r)$. Thus
\[\E\|\bG \otimes_r \bG\|_F^2 \prec \max(d^{k+r},d^{k+(k-r)}) \prec d^{2k-1}.\]
By Gaussian hypercontractivity, this implies
$\|\bG \otimes_r \bG\|_F \prec d^{k-1/2}$. On the other hand, since $B_{d,k}
\asymp d^k$, a standard chi-squared tail bound shows 
$1/\|\bg\|_2 \prec d^{-k/2}$. So
\[\|\iota(\bbeta_{ki}) \otimes_r \iota(\bbeta_{ki})\|_F
=\frac{1}{\|\bg\|_2^2} \cdot \|\bG \otimes_r \bG\|_F \prec d^{-1/2},\]
which implies the lemma.
\end{proof}

 \subsection{An Abstract Partial Gaussian Equivalence Theorem}
\label{sec:PartialGaussianEquivalence}
 In this section, we now present a general interpolation bound for replacing higher-order Wiener chaos components by Gaussian surrogates while preserving the lower-order ones.  

Suppose $\bx \sim \cN(0, \id_d)$, and recall the vector $\bh_l(\bx) \in \R^{B_{d,l}}$ of degree-$l$ multivariate Hermite polynomials.

            \begin{definition}[Tuple-indexed chaos vector along the interpolation path]
\label{def:tuple_interpolated}
Fix integers \(K \ge 1\) and \(K' \in \{0,\dots,K\}\).
For each $l=1,\ldots,K$, fix an integer \(k_l \ge 1\) and choose coefficient vectors \(\ba_{l1},\dots,\ba_{lk_l} \in \R^{B_{d,l}}\). Let
\[
  \cI \;:=\; \bigl\{(l,j)\,\big|\, l=1,\dots,K,\; j=1,\dots,k_l\bigr\}.
\]
For every \(l>K'\), let \(\bg_l \sim \cN(\mathbf 0,\id_{B_{d,l}})\) be independent of each other and of \(\bx\). For \(t\in[0,1]\) and each \((l,j)\in\cI\), define
\[
  s_{l,j}(t)=
  \begin{cases}
    \displaystyle \ba_{lj}^{\sT}\,\bh_l(\bx), & l\le K',\\[6pt]
    \displaystyle \ba_{lj}^{\sT}\!\Bigl(\sqrt{1-t}\,\bh_l(\bx)+\sqrt{t}\,\bg_l\Bigr), & K'<l\le K,
  \end{cases}
\]
and set $\bs(t)=(s_{l,j}(t))_{(l,j) \in \cI}$.
\end{definition}

The special cases are $\bs(0)$, which recovers the original chaos vector, and
$\bs(1)$, which replaces all orders \(l>K'\) by Gaussian surrogates.

\begin{theorem}[Partial Gaussian Equivalence Theorem]
\label{theorem:RecallBoundedLipschitzFunctionSupErgetPartial}
Let \(K \ge 1\), \(K' \in \{0,\dots,K\}\), and 
\(\bs(t) = \bigl(s_{l,j}(t)\bigr)_{(l,j)\in\cI} \in \R^{k_1+\cdots+k_K}\) be as in
Definition~\ref{def:tuple_interpolated}. Let
\[
  \cI_{\mathrm{high}} \;:=\; \{(l,j)\in\cI:\; l>K'\}
\]
denote the set of swapped coordinates.
Let \(\varphi:\R^{k_1+\cdots+k_K}\to\R\) be any twice continuously-differentiable function, with
$\| \|\nabla^2 \varphi\|_\op\|_{\infty} <\;\infty$.
Then
\[
  \bigl|\, \E\bigl[\varphi(\bs(0)) - \varphi(\bs(1))\bigr] \,\bigr|
  \;\le\;
  C\, \| \|\nabla^2 \varphi\|_\op\|_{\infty}\,
  \sum_{\alpha\in\cI_{\mathrm{high}}}\;\sum_{\beta\in\cI}
  \sqrt{\,\operatorname{Var}\bigl(\,\langle \nabla_\bx s_{\alpha}(0),\, \nabla_\bx s_{\beta}(0)\rangle \bigr)}\,,
\]
where $\nabla_\bx$ denotes the gradient in $\bx$,
and \(C>0\) is a constant depending only on \(K\), \(K'\), and \(\{k_l\}_{l=1}^K\).
\end{theorem}

\begin{proof}
Define
\[
\psi(t)=\E\big[\varphi(\bs(t))\big],\qquad t\in[0,1],
\]
so that
\begin{equation}\label{eq:GaussianMalliavingood}
  \Bigl|\E\big[\varphi(\bs(0))-\varphi(\bs(1))\big]\Bigr|
  \;=\;|\psi(1)-\psi(0)|
  \;\le\;\int_0^1 |\psi'(t)|\,\d t.
\end{equation}
By the form of \(s_{l,j}(t)\),
only the coordinates $(l,j) \in \cI_{\mathrm{high}}$ depend on \(t\). Thus
\begin{equation}\label{eq:psi-derivative}
  \psi'(t)=
  \E\Bigg[
    \sum_{\alpha=(l,j)\in\cI_{\mathrm{high}}}
      \partial_{\alpha}\varphi(\bs(t))\cdot
      \frac12\Bigl(
          \frac{1}{\sqrt{t}}\;\ba_{lj}^{\sT}\bg_l
          \;-\;
          \frac{1}{\sqrt{1-t}}\;\ba_{lj}^{\sT}\bh_l(\bx)
      \Bigr)
  \Bigg].
\end{equation}

\paragraph{Step 1: Gaussian integration by parts for the \(\bg_l\)-term.}
Fix \(\alpha=(l,j)\in\cI_{\mathrm{high}}\).
For the term involving \(\bg_l\) in \eqref{eq:psi-derivative}, we apply
 Gaussian integration-by-parts on \(\bg_l\).
The chain rule yields
\[
\nabla_{\bg_l}\big(\partial_{\alpha}\varphi(\bs(t))\big)=
  \sum_{j'=1}^{k_l}
  \partial_{\alpha,(l,j')}^{2}\varphi(\bs(t))\,
  \nabla_{\bg_l}s_{l,j'}(t)
  =\sqrt{t}\sum_{j'=1}^{k_l}
  \partial_{\alpha,(l,j')}^{2}\varphi(\bs(t))\,\ba_{lj'}.
\]
Therefore,
\[
  \E\Big[
    \partial_{\alpha}\varphi(\bs(t))\cdot \frac{1}{2\sqrt{t}}\;\ba_{lj}^{\sT}\bg_l
  \Big]=
  \frac12\,
  \E\Big[
    \big\langle \nabla_{\bg_l}\big(\partial_{\alpha}\varphi(\bs(t))\big),\,\ba_{lj}\big\rangle
  \Big]=
  \frac12\,
  \E\Big[
    \sum_{j'=1}^{k_l}\partial_{\alpha,(l,j')}^{2}\varphi(\bs(t))\;
    \ba_{lj}^{\sT}\ba_{lj'}
  \Big].
\]
 
\paragraph{Step 2: Malliavin integration by parts for the \(\bh_l(\bx)\)-term.}
For the term involving \(\bh_l(\bx)\) in \eqref{eq:psi-derivative}, we apply
the Malliavin integration-by-parts identity of Lemma
\ref{lemma:GaussianMalliavin} for \(\bx\), with the choices
\[
  G=\partial_{\alpha}\varphi(\bs(t)),\qquad
  F=\ba_{lj}^{\sT}\bh_l(\bx),\qquad
  f(u)=u.
\]
Here \(F\) lies in the \(l\)-th Wiener chaos (\(l\ge1\)), so
\(L^{-1}F = -\frac{1}{l}F\) and hence
\(-DL^{-1}F = \frac{1}{l}DF\),
where the Malliavin derivative $D$ is simply the gradient $\nabla_\bx$. Therefore
\[
  \E\!\Big[
    \partial_{\alpha}\varphi(\bs(t))\cdot \ba_{lj}^{\sT}\bh_l(\bx)
  \Big]
  \;=\;
  \frac{1}{l}\,
  \E\!\Big[
    \big\langle \nabla_\bx\partial_{\alpha}\varphi(\bs(t)),\,
              \nabla_\bx \ba_{lj}^{\sT}\bh_l(\bx)
    \big\rangle 
  \Big]
  \;=\;
  \frac{1}{l}\,
  \E\!\Big[
    \big\langle \nabla_\bx\partial_{\alpha}\varphi(\bs(t)),\,
              \nabla_\bx s_{\alpha}(0)
    \big\rangle 
  \Big].
\]
By the chain rule and the explicit form of \(s_{r,j'}(t)\),
\[
\nabla_\bx\partial_{\alpha}\varphi(\bs(t))=
  \sum_{\gamma\in\cI}
    \partial_{\gamma\alpha}^{2}\varphi(\bs(t))\; \nabla_\bx s_\gamma(t),
  \qquad
  \nabla_\bx s_{r,j'}(t)=
  \begin{cases}
    \nabla_\bx s_{r,j'}(0), & r\le K',\\[3pt]
    \sqrt{1-t}\,\nabla_\bx s_{r,j'}(0), & r> K'.
  \end{cases}
\]
Splitting the sum accordingly and inserting the prefactor \(\frac{1}{2\sqrt{1-t}}\)
from \eqref{eq:psi-derivative}, we obtain
\begin{align*}
  \E\!\Big[
    \partial_{\alpha}\varphi(\bs(t))\cdot \frac{1}{2\sqrt{1-t}}\;\ba_{lj}^{\sT}\bh_l(\bx)
  \Big]
  &= \frac{1}{2\sqrt{1-t}}\cdot \frac{1}{l}\,
     \E\!\Big[
       \sum_{\gamma=(r,j')\in\cI}\partial_{\gamma\alpha}^{2}\varphi(\bs(t))\;
       \big\langle \nabla_\bx s_{r,j'}(t),\,\nabla_\bx s_{\alpha}(0)\big\rangle
     \Big]\\
  &= \frac12\cdot \frac{1}{l}\,
     \E\!\Big[
       \sum_{\gamma\in\cI_\mathrm{high}}
       \partial_{\gamma\alpha}^{2}\varphi(\bs(t))\;
       \big\langle \nabla_\bx s_{ \gamma}(0),\,\nabla_\bx s_{ \alpha}(0)\big\rangle 
     \Big]
     \\
  &\qquad
     +\;\frac{1}{2\sqrt{1-t}}\cdot \frac{1}{l}\,
     \E\!\Big[
       \sum_{\delta\in(\cI \setminus \cI_\mathrm{high})}
       \partial_{\delta\alpha}^{2}\varphi(\bs(t))\;
       \big\langle \nabla_\bx s_{ \delta}(0),\,\nabla_\bx s_{ \alpha}(0)\big\rangle 
     \Big].
\end{align*}

\paragraph{Step 3: Centering by Lemma \ref{lemma:DIidentity}.}
For any \(l,l'\ge1\) and \(j\in\{1,\dots,k_l\}\), \(j'\in\{1,\dots,k_{l'}\}\), write
\begin{align}
  s_{l,j} (0) = I_l\Big(\frac{\iota(\ba_{lj})}{\sqrt{l!}}\Big),
  \quad
  s_{l',j'} (0) = I_{l'}\Big(\frac{\iota(\ba_{l'j'})}{\sqrt{l'!}}\Big).
\end{align}
Lemma~\ref{lemma:DIidentity} then directly implies that
\[
  \E\!\left[\frac{1}{l}\,
      \big\langle \nabla_\bx s_{l,j}(0),\,\nabla_\bx s_{l',j'}(0)\big\rangle 
  \right]
  \;=\;
  \begin{cases}
    \ba_{lj}^{\sT}\ba_{lj'}, & l=l',\\[3pt]
    0, & l\ne l'.
  \end{cases}
\]
Thus, for any $\alpha=(l,j)$, with some simple calculation,
\begin{align}
   &\left|\frac12\,
   \E\Bigl[
     \sum_{j'=1}^{k_l}
       \partial_{\alpha,(l,j')}^2\varphi(\bs(t))\;
       \ba_{lj}^{\mathsf T}\ba_{lj'}
   \Bigr]     - \frac12\,
   \E\Bigl[
\frac1l\sum_{\gamma\in\cI_{\mathrm{high}}}
       \partial_{\gamma\alpha}^2\varphi(\bs(t))\;
       \bigl\langle \nabla_\bx s_{\alpha}(0),\,\nabla_\bx s_{\gamma}(0)\bigr\rangle
   \Bigr]         \right| \\\leq & C\,
   \|\|\nabla^2 \varphi\|_\op\|_{\infty}\,
   \sum_{\gamma\in\cI_{\mathrm{high}}}
     \sqrt{\,
       \Var\bigl(\,\langle \nabla_\bx s_{ \alpha}(0),\,\nabla_\bx s_{ \gamma}(0)\rangle\bigr)
     }\;.
\end{align}
Similarly,
\begin{align}
    &\frac{1}{2\sqrt{1-t}}\,
    \E\Bigl[
      \frac1l\sum_{\delta\in (\cI \setminus \cI_{\mathrm{high}})}
        \partial_{\alpha\delta}^2\varphi(\bs(t))\;
        \bigl\langle \nabla_\bx s_{ \alpha}(0),\,\nabla_\bx s_{\delta}(0)\bigr\rangle 
    \Bigr] \\      &\leq \frac{C\|\|\nabla^2 \varphi\|_\op\|_{\infty} }{\sqrt{1-t}}
\sum_{\delta \in (\cI \setminus \cI_{\mathrm{high}})} \sqrt{\Var (\<\nabla_\bx
\bs_{\alpha}(0), \nabla_\bx \bs_{\delta}(0)\>) }.
\end{align}
Noting that $\int_{0}^{1} \frac{1}{\sqrt{1-t}}\d t = 2,$ we conclude that \begin{align}
    &\eqref{eq:GaussianMalliavingood} \leq C\,
    \|\|\nabla^2 \varphi\|_\op\|_{\infty}\,
    \sum_{\alpha\in\cI_{\mathrm{high}}}\;
    \sum_{\delta\in\cI}
      \sqrt{\,
        \Var \bigl(\,\langle \nabla_\bx s_{\alpha}(0),\,\nabla_\bx s_{\delta}(0)\rangle\bigr)
      }\;.
\end{align}
\end{proof}

  \begin{corollary}\label{corollary:RecallBoundedLipschitzFunctionSupErgetPartial}
Adopt the notation of Theorem~\ref{theorem:RecallBoundedLipschitzFunctionSupErgetPartial}.
Assume the coefficient families satisfy
\begin{align}
  &\max_{l\le K'}\;\max_{1\le j\le k_l}\|\ba_{lj}\|_2\;\le\;B,
  \label{eq:C1_assump}\\[4pt]
  &\max_{K'<l}\;\max_{1\le j\le k_l}
     \Bigl|
       \E\bigl(\ba_{lj}^{\mathsf T}\bh_l(\bx)\bigr)^{4}
       -3\|\ba_{lj}\|_2^{4}
     \Bigr|
  \leq \rho,
  \label{eq:C2_assump}
\end{align}
for some parameters \(B \equiv B(d)>0\) and \(\rho \equiv \rho(d)>0\).
Then, for any twice continuously-differentiable function
\(\varphi:\R^{k_1+\dots+k_K}\to\R\) with \(\|\|\nabla^2 \varphi\|_\op\|_{\infty}<\infty\),
\[
  \bigl|\,
    \E\bigl[\varphi(\bs_0)-\varphi(\bs_1)\bigr]
  \,\bigr|
  \;\le\;
  C\,\|\|\nabla^2 \varphi\|_\op\|_{\infty}\,
  \bigl(
      B\rho^{1/4}+\rho^{1/2}
  \bigr),
\] 
for a constant \(C>0\) depending only on \(K\), \(K'\), and the multiplicities \(\{k_l\}_{l=1}^{K}\).
\end{corollary}

  \begin{proof}
  Recall that
  \(\cI_{\mathrm{high}}=\{(l,j):l>K'\}\)
  and
  \(\cI=\bigcup_{l=1}^{K}\{(l,j):1\le j\le k_l\}\).
  
  \paragraph{Step 1: Diagonal high–high terms.}
  Fix \(\alpha=(l,j)\in\cI_{\mathrm{high}}\) and set
  \(F(\bx)=s_{\alpha}(0)=\ba_{lj}^{\mathsf T}\bh_l(\bx) =I_l(\frac{\iota(\ba_{lj})}{\sqrt{l!}})\).
  By Lemma~\ref{lemma:DIidentity},
  \begin{equation}\label{eq:DF-spectral}
    \frac1l\|\nabla_\bx F\|_2^{2}
    =\|\ba_{lj}\|_2^{2}
     +l\sum_{r=1}^{l-1}(r-1)!\binom{l-1}{r-1}^{2}\,
         I_{2l-2r}\Big(\frac{\iota(\ba_{lj})}{\sqrt{l!}}\tilde\otimes_{r}\frac{\iota(\ba_{lj})}{\sqrt{l!}}\Big),
  \end{equation}
  where we recall that \(\iota(\ba_{lj})\) is an isometry satisfying
     $\<\iota(\ba_{lj}),\iota(\ba_{lj})\> =\|\ba_{lj}\|_2^{2}$.
  Then orthogonality of chaoses gives
  \[
    \Var\bigl(\|\nabla_\bx F\|_2^{2}\bigr)
    =l^{4}\!\sum_{r=1}^{l-1}\left(\frac{(r-1)!}{l!}\right)^{2}\binom{l-1}{r-1}^{4}
        (2l-2r)! \|\iota(\ba_{lj})\tilde\otimes_{r}\iota(\ba_{lj})\|_{F}^{2}.
  \]
  Lemma~\ref{section:ProofOfLemmaGaussianChaosConvergenceToStandardGaussian} implies
  \[
    \Var\bigl(\|\nabla_\bx F\|_2^{2}\bigr)
    \leq
    C\bigl(\E[F^{4}]-3\E[F^{2}]^{2}\bigr)
    \leq C\rho,
  \]
  where we used \eqref{eq:C2_assump} for the final inequality.
  
  \paragraph{Step 2: Off-diagonal high–high terms.}
  Let \(\alpha=(l,j)\) and \(\beta=(l',j')\) be two distinct indices in
  \(\cI_{\mathrm{high}}\) with \(l\le l'\).
  Writing \(s_{\alpha}(0)=I_l(\bH)\) and \(s_{\beta}(0)=I_{l'}(\bG)\),
  where \begin{align}
    \bH &=  \frac{\iota(\ba_{lj})}{\sqrt{l!}},\quad
    \bG = \frac{\iota(\ba_{l'j'})}{\sqrt{l'!}},
  \end{align}
  the product formula Lemma~\ref{lemma:DIidentity}
  gives
  \[
    \langle \nabla_\bx s_{\alpha}(0),\nabla_\bx s_{\beta}(0)\rangle
    =ll'\!\sum_{r=1}^{l}(r-1)!
        \binom{l-1}{r-1}\binom{l'-1}{r-1}\,
        I_{l+l'-2r}(\bH\tilde\otimes_{r}\bG).
  \]
  Squaring and using orthogonality, if \(l< l'\) we have
  \begin{equation}\label{eq:HH-var}
    \Var(\langle \nabla_\bx s_{\alpha}(0),\nabla_\bx s_{\beta}(0)\rangle)
    =l^{2}l'^{2} 
       \sum_{r=1}^{l}(r-1)!^{2}
         \binom{l-1}{r-1}^{2}\binom{l'-1}{r-1}^{2}
         (l+l'-2r)!\,
         \|\bH\tilde\otimes_{r}\bG\|_{F}^{2},
  \end{equation}
  and if $l=l'$ we have as in Step 1 above
  \begin{align}\label{eq:HH-var1}
    \Var(\langle \nabla_\bx s_{\alpha}(0),\nabla_\bx s_{\beta}(0)\rangle)
    =l^4 
       \sum_{r=1}^{l-1}(r-1)!^{2}
         \binom{l-1}{r-1}^{4} 
         (2l-2r)!\,
         \|\bH\tilde\otimes_{r}\bG\|_{F}^{2}.
  \end{align}
  
  For \(1\le r<l\), observe that
$\|\bH\tilde\otimes_{r}\bG\|_F
\leq \|\bH\otimes_r\bG\|_F$, where
\((\bH\otimes_r\bG)_{I,J}=\sum_{A} \bH_{I,A}\,\bG_{J,A}\) for index tuples
\(I=(i_1,\ldots,i_{l-r})\), \(J=(j_1,\ldots,j_{l'-r})\), and
\(A=(a_1,\ldots,a_r)\). Then
\begin{align}
  \|\bH\tilde\otimes_{r}\bG\|_F^2
  &\leq \sum_{I,J}\Bigl(\sum_{A} \bH_{I,A}\,\bG_{J,A}\Bigr)^2
   = \sum_{I,J}\;\sum_{A,B} \bH_{I,A}\,\bG_{J,A}\,\bH_{I,B}\,\bG_{J,B}\notag\\
&=\sum_{A,B} (\bH\otimes_{\,l-r}\bH)_{A,B} (\bG\otimes_{\,l'-r}\bG)_{A,B}\\
  &\leq \bigl\|\,\bH\otimes_{\,l-r}\bH\,\bigr\|_{F}
     \bigl\|\,\bG\otimes_{\,l'-r}\bG\,\bigr\|_{F}
\leq \frac12\bigl(
         \|\bH\otimes_{l-r}\bH\|^{2}_F+\|\bG\otimes_{l'-r}\bG\|^{2}_F
       \bigr).\label{eq:contractioninequality}
\end{align}
Then applying Lemma~\ref{section:ProofOfLemmaGaussianChaosConvergenceToStandardGaussian} as above,
  $\|\bH\tilde\otimes_{r}\bG\|_F^{2} \leq C\rho$.  For the maximal contraction
\(r=l\) when $l<l'$, \eqref{eq:contractioninequality} implies
\[\|\bH\tilde\otimes_{l}\bG\|_F^2
\leq \|\bH\|_F^2\|\bG \otimes_{l'-l} \bG\|_F,\]
so \(\|\bH\tilde\otimes_{l}\bG\|_F^{2}\le CB^2\rho^{1/2}\) by 
Lemma~\ref{section:ProofOfLemmaGaussianChaosConvergenceToStandardGaussian}.
  Substituting these bounds into \eqref{eq:HH-var} and \eqref{eq:HH-var1} yields
  \[
    \Var(\langle \nabla_\bx s_{\alpha}(0),\nabla_\bx s_{\beta}(0)\rangle)
    \;\le\;
    C\bigl(\rho+B^2\rho^{1/2}\bigr)
  \]
  
  \paragraph{Step 3: Mixed high–low terms.}
  For \(\alpha\in\cI_{\mathrm{high}}\) and
  \(\beta\in\cI\setminus\cI_{\mathrm{high}}\),
  the same argument 
  gives the identical bound
  \[
    \Var\bigl(\langle \nabla_\bx s_{\alpha}(0), \nabla_\bx s_{\beta}(0)\rangle\bigr)
    \;\le\;
    C\bigl(\rho+B^2\rho^{1/2}\bigr).
  \]
  Substituting these into the bound of
Theorem~\ref{theorem:RecallBoundedLipschitzFunctionSupErgetPartial}
  completes the proof.
  \end{proof}
 
\subsection{Proof of Theorem \ref{theorem:RecallBoundedLipschitzFunctionSupErgetisotropic}}
\label{section:ProofOfTheoremGaussianChaosConvergenceToStandardGaussianGeneral}

Recall the decomposition
\[\bh_k(\bx)=\proj_S\bh_k(\bx)+\proj_{S,\perp}\bh_k(\bx)\]
where $\proj_S\bh_k(\bx)$ is supported on the degree-$k$ multivariate
Hermite polynomials involving only $\bx_S$, and $\proj_{S,\perp}\bh_k(\bx)$ is
supported on those that involve some coordinate of $\bx_{\setminus S}$. We
consider a further decomposition of $\proj_{S,\perp}\bh_k(\bx)$,
\[\bh_k(\bx)=\proj_S\bh_k(\bx)+\proj_{S,\perp}^{\text{mixed}}\bh_k(\bx)
+\proj_{S,\perp}^{\text{pure}}\bh_k(\bx)\]
where $\proj_{S,\perp}^{\text{pure}}\bh_k(\bx)$ is supported on those Hermite
polynomials involving only $\bx_{\setminus S}$, and
$\proj_{S,\perp}^{\text{mixed}}\bh_k(\bx)$ is supported on those depending on
both $\bx_S$ and $\bx_{\setminus S}$. We will write as shorthand
\[\bh_{k,S}(\bx)=\proj_S\bh_k(\bx)+\proj_{S,\perp}^{\text{mixed}}\bh_k(\bx),
\qquad \bh_{k,\setminus S}=\proj_{S,\perp}^{\text{pure}}\bh_k(\bx),\]
i.e.\ $\bh_{k,\setminus S}(\bx)$ is supported on the indices
$\{\bk:k_1+\ldots+k_s=0\}$
and $\bh_{k,S}(\bx)$ is supported on the indices
$\{\bk:k_1+\ldots+k_s \geq 1\}$.
For a Gaussian vector $\bg_k\sim\cN(\mathbf 0,\mathbf I_{B_{d,k}})$, we
decompose correspondingly
\[\bg_k=\bg_{k,S}+\bg_{k,\setminus S}.\]

The following lemma first controls the magnitudes of the components in Theorem
\ref{theorem:RecallBoundedLipschitzFunctionSupErgetisotropic}
corresponding to $\bh_{k,S}(\bx)$ and $\bg_{k,S}$.

\begin{lemma}\label{lemma:NoisePart}
For any fixed constant $K>0$, there exists a constant $c>0$ such that
\begin{enumerate}
  \item For the random feature components of orders $k \geq 3$,
\begin{align}
\sup_{\btheta \in \bTheta_{\bW}^{\sPG}(K)} \E_{\bx} \left| \btheta^\sT
\sum_{k=3}^{D} \mu_k \bV_k \bh_{k,S}(\bx)\right| \prec \frac{1}{d^{1/2}\mu_1} +
\frac{1}{d^{1/4}},\label{equation:NoisePart1}\\
\sup_{\btheta \in \bTheta_{\bW}^{\sPG}(K)} \E_{\bg_k} \left| \btheta^\sT
\sum_{k=3}^{D} \mu_k \bV_k \bg_{k,S}\right| \prec
\frac{1}{d^{1/2}\mu_1} + \frac{1}{d^{1/4}}.\label{equation:NoisePart1gaussian}
\end{align}
\item For the target function component of each order $k \geq 2$ and each
$i \in [s_k]$,
\begin{equation}\label{equation:NoisePart2}
    \E_{\bx} \left|\bbeta_{ki}^\top \bh_{k,S}(\bx)\right|
\prec \frac{1}{d^{c}},
\qquad \E_{\bg_k} \left|\bbeta_{ki}^\top \bg_{k,S}\right|
\prec \frac{1}{d^{c}}.
\end{equation}
\end{enumerate}
\end{lemma}
    \begin{proof} Part 1 will be proved by moment calculation, and
Part 2 will be implied by the genericity condition for
$\bbeta_{ki}$.
       \paragraph{Proof of  \eqref{equation:NoisePart1}.} 
For a multivariate polynomial $p(\bx)$, we denote $\mathcal{P}_S[p(\bx)] $ as
the summation of all monomials in $p(\bx)$ that are dependent on at least one
variable $x_1,\ldots,x_s$. Note that for any $k_1,\ldots,k_s \geq 0$
with $k_1+\ldots+k_s \geq 1$,
\[\prod_{i=1}^s \He_{k_i}(x_i)
=\mathcal{P}_S\bigg[\prod_{i=1}^s \He_{k_i}(x_i)\bigg]
-\E\,\mathcal{P}_S\bigg[\prod_{i=1}^s \He_{k_i}(x_i)\bigg],\]
where the right-most term must be the constant in
the monomial expansion of $\prod_{i=1}^s \He_{k_i}(x_i)$ by
the identity $\E \prod_{i=1}^s \He_{k_i}(x_i)=0$. Then
$\bh_{k,S}(\bx)=\mathcal{P}_S \bh_k(\bx)
-\E_{\bx_S} \mathcal{P}_S \bh_k(\bx)$, implying that
\[\btheta^\top \bV_k \bh_{k,S}(\bx)
=\mathcal{P}_S[\btheta^\top \bV_k \bh_k(\bx)]
-\E_{\bx_S}\mathcal{P}_S[\btheta^\top \bV_k \bh_k(\bx)].\]
Thus to show \eqref{equation:NoisePart1}, by Jensen's inequality, it
suffices to show
\begin{equation}\label{eq:NoisePart1reduced}
\sup_{\btheta \in \bTheta_{\bW}^{\sPG}(K)} \E_{\bx} \left|\cP_S\left[\btheta^\sT
\sum_{k=3}^{D} \mu_k \bV_k \bh_k(\bx)\right]\right| \prec
\frac{1}{d^{1/2}\mu_1} + \frac{1}{d^{1/4}}.
\end{equation}

For this, let us write $\tilde\He_j(x)=\sqrt{j!}\He_j(x)$. Then
\begin{equation}
\btheta^\sT \bV_k\bh_k(\bx)
=\sum_{i=1}^p \theta_i
            \frac{1}{\sqrt{k!}}\tilde\He_k(\<\bw_i,\bx\>).
        \end{equation}
Applying the Hermite polynomial identity \eqref{lemma:SumHermite} and selecting
those terms with non-zero exponent in $\bx_S$, we have
        \begin{equation}
           \cP_S\left[ \tilde\He_k(\<\bw_i,\bx\>) \right]
=\sum_{j=0}^{k-1}\binom{k}{j}\<\bw_{iS},\bx_S\>^{k-j}
\tilde\He_j(\<\bw_{i\setminus S},\bx_{\setminus S}\>)
        \end{equation}
where $\bw_i=(\bw_{iS},\bw_{i\setminus S})$ and
$\bx=(\bx_S,\bx_{\setminus S})$. Thus, 
        \begin{align}
            \left(\E_\bx \left|\mathcal{P}_S\left[\sum_{i=1}^{p} \theta_i 
            \tilde\He_k(\<\bw_i,\bx\>)\right]\right| \right)^2&\leq  \E_\bx \left(\mathcal{P}_S\left[\sum_{i=1}^{p} \theta_i  \tilde\He_k(\<\bw_i,\bx\>)\right]\right) ^2\\
            &= \E_\bx \left( \sum_{j=0}^{k-1}\binom{k}{j}\left[\sum_{i=1}^{p}\theta_i
\<\bw_{iS},\bx_S\>^{k-j}\tilde\He_j(\<\bw_{i\setminus S},\bx_{\setminus S}\>)\right] \right) ^2\\
            &\leq C\sum_{j=0}^{k-1} \underbrace{\E_\bx \left(\sum_{i=1}^{p}\theta_i
\<\bw_{iS},\bx_S\>^{k-j}\tilde{\He_j}(\<\bw_{i\setminus S},\bx_{\setminus
S}\>)\right)^2}_{:=A_j}.
        \end{align}
        For fixed $j$, this quantity $A_j \equiv A_j(\btheta,\bW)$ satisfies
        \begin{align}
&A_j
= \E_\bx \sum_{i,l=1}^p \theta_i  \theta_l 
\<\bw_{iS},\bx_S\>^{k-j}\<\bw_{lS},\bx_S\>^{k-j}
\tilde{\He_j}(\<\bw_{i\setminus S},\bx_{\setminus S}\>)\tilde{\He_j}(\<\bw_{l\setminus S},\bx_{\setminus S}\>)\\
            &\overset{(a)}{=} \E_\bx \sum_{i,l=1}^p \theta_i  \theta_l 
\<\bw_{iS},\bx_S\>^{k-j}\<\bw_{lS},\bx_S\>^{k-j}
\left(\sum_{q=0}^{\lfloor \frac{j}{2}\rfloor  }(-1)^q \|\bw_{i\setminus
S}\|_2^{j-2q}
\|\bw_{iS}\|_2^{2q}\binom{j}{2q}\frac{(2q)!}{2^q q!}\tilde{\He}_{j-2q}\left(\frac{\<\bw_{i\setminus
S},\bx_{\setminus S}\>}{\|\bw_{i\setminus S}\|_2}\right)\right) \\
            &
\quad\quad\quad\quad\quad\quad\quad\quad\quad\quad\quad\quad\quad\quad\quad
\cdot \left(\sum_{q=0}^{\lfloor \frac{j}{2}\rfloor  }(-1)^q \|\bw_{l\setminus
S}\|_2^{j-2l}
\|\bw_{lS}\|_2^{2q}\binom{j}{2q}\frac{(2q)!}{2^q q!}\tilde{\He}_{j-2q}\left(\frac{\<\bw_{l\setminus
S},\bx_{\setminus S}\>}{\|\bw_{l\setminus S}\|_2}\right)\right) \\ 
            &\overset{(b)}{=} \sum_{i,l=1}^p \theta_i \theta_l\,
p_{2(k-j)}(\bw_{iS},\bw_{lS})
\sum_{q=0}^{\lfloor \frac{j}{2}\rfloor  }
\|\bw_{iS}\|_2^{2q}\|\bw_{lS}\|_2^{2q}\binom{j}{2q}^2\Big(\frac{(2q)!}{2^q
q!}\Big)^2(j-2q)!\<\bw_{i\setminus S},\bw_{l\setminus S}\>^{j-2q}
\end{align}
where (a) holds by Lemma \ref{lemma:factorHermite}, and (b) holds
by independence of $\bx_S,\bx_{\setminus S}$ and
Lemma \ref{lemma:multivariate_gaussian_hermite}, for a polynomial
$p_{2(k-j)}(\bw_{iS},\bw_{lS})$ that is homogeneous of degree $2(k-j)$
representing a sum over pairings of $k-j$ copies of $\bw_{iS}$ and $k-j$ copies of
$\bw_{lS}$. Isolating the terms for $i=l$, this implies
\begin{align}
A_j &\leq C\underbrace{\sum_{i=1}^p \theta_i^2
\|\bw_{iS}\|_2^{2(k-j)}}_{:=A_j^{(1)}}
+C\sum_{q=0}^{\lfloor \frac{j}{2}\rfloor  }\underbrace{\left(\mathop{\sum_{i,l=1}^p}_{i\ne l} \theta_i \theta_l
p_{2(k-j)}(\bw_{iS},\bw_{lS})
\|\bw_{iS}\|_2^{2q}\|\bw_{lS}\|_2^{2q}\<\bw_{i\setminus S},\bw_{l\setminus
S}\>^{j-2q}\right)_+}_{:=A_{j,q}^{(2)}}.
        \end{align} 

               For the first term $A_j^{(1)}$, for any $j \in
\{0,\ldots,k-1\}$ and $\btheta \in \bTheta^{\sPG}_{\bW}(K)$, applying
$\sup_{i,k} |w_{ik}| \prec d^{-1/2}$ and the condition $\|\btheta\|_2 \prec 1$ defining $\bTheta_\bW^{\sPG}(K)$,
        \begin{equation}\label{equation:WassersteinDistanceRandomVariable3firstterm} 
            A_j^{(1)} \prec d^{-1}.
        \end{equation}
        For the second term $A_{j,q}^{(2)}$, 
        applying $\sup_{i,k} |w_{ik}| \prec d^{-1/2}$ and $\sup_{i\ne l}  \< \bw_i, \bw_l \> \prec d^{-1/2}$,
        \begin{align}
            &A_{j,q}^{(2)}
            \prec  \mathop{\sum_{i,l=1}^p}_{i\ne l} |\theta_i
\theta_l|(d^{-1/2})^{2k-j+2q}. \label{equation:WassersteinDistanceRandomVariable3bbbhighuniform}
        \end{align}
When $k \ge 4$, since $j \leq k-1$ and $q \leq \lfloor j/2 \rfloor$, the
exponent $2k-j+2q$ in
\eqref{equation:WassersteinDistanceRandomVariable3bbbhighuniform} is at least 5. Thus, we have
 \begin{equation} \label{equation:Combineoverlap1}
            A_{j,q}^{(2)} \prec \sum_{i,l}^p |\theta_i \theta_l| \frac{1}{d^{5/2}}  \prec \frac{1 }{\sqrt{d}} \sum_{i=1}^{p}\theta_i^2 \prec \frac{1}{\sqrt{d}}.
        \end{equation}
When $k=3$, the same bound holds for $j  \ne 2,$ or $q \ne 0$. The only remaining case
is $j = 2$ and $q = 0$, for which
$p_{2(k-j)}(\bw_{iS},\bw_{lS})=\<\bw_{iS},\bw_{lS}\>$ and we need to control
\begin{equation} \label{eq:rewrite_k3_quad_form}
A_{j,q}^{(2)}=\Big(\sum_{i\ne l}^p \theta_i \theta_l \<\bw_{iS},\bw_{lS}\>\<\bw_{i\setminus
S},\bw_{l\setminus S}\>^{2}\Big)_+ \leq \Big(\sum_{i,l}^p \theta_i \theta_l 
\<\bw_{iS},\bw_{lS}\>\<\bw_{i\setminus
S},\bw_{l\setminus S}\>^{2}\Big)_+.
\end{equation}
To analyze this, let $\bW_{\backslash S} \in \R^{p \times (d-1)}$ have rows
$\{\bw_{i\backslash S}\}_{i=1}^p$, and let
$\bM \in \R^{p \times p}$ and $\bU \in \mathbb{R}^{p \times s}$ have entries $\bM_{il} = (\<\bw_{i\backslash S} , \bw_{l\backslash S}\>)^2$ and
$\bU_{ik} = \theta_i w_{ik}$. Then the upper bound in
\eqref{eq:rewrite_k3_quad_form} is precisely $(\Tr\bU^\sT \bM \bU)_+$.
Recalling the identity $\bV_k \bV_k^\sT =
(\bW\bW^\sT)^{\odot k}$ and defining $\bV_{2, \backslash S} \in \mathbb{R}^{p
\times B_{d-s,2}}$ from $\bW_{\setminus S}$ in the same manner as $\bV_2$ is defined from $\bW$, we have $\bM = \bV_{2, \backslash S} \bV_{2, \backslash S}^\sT$.
Here, as in Corollary~\ref{corollary:OperatorNormF2cdecompose},
$\bV_{2, \backslash S}$ has a decomposition
\begin{equation}
    \bV_{2, \backslash S} = \bV'_{2c} + \frac{1}{d}\mathbf{1}_{p}\mathbf{e}'^\sT_c,
\end{equation}
where $\|\bV'_{2c}\|_\op \prec 1$ and
$\mathbf{e}'_c=(1,\ldots,1,0,\ldots,0)^\top$ with first $d-s$ entries non-zero.
Substituting this into the quadratic form and using the inequality $(A+B)(A+B)^\sT \preceq 2AA^\sT + 2BB^\sT$,
\begin{align}
\Tr\bU^\sT \bM \bU = \Tr\bU^\sT (\bV_{2,
\backslash S} \bV_{2, \backslash S}^\sT) \bU
    &\leq 2\Tr\bU^\sT (\bV'_{2c} \bV'_{2c}{}^\sT) \bU + 2\Tr \bU^\sT \left(
\frac{1}{d^2} \mathbf{1}_p \mathbf{e}'^\sT_c \mathbf{e}'_c \mathbf{1}_p^\sT
\right) \bU \\
    &= 2\|\bV'_{2c}{}^\sT \bU\|_F^2  + \frac{2\|\mathbf{e}'_c\|_2^2}{d^2}
\|\mathbf{1}_p^\sT \bU\|_2^2.\label{equation:WassersteinDistanceRandomVariable3bbbhighmatrix3}
\end{align}
For the first term of \eqref{equation:WassersteinDistanceRandomVariable3bbbhighmatrix3}, since $\|\bV'_{2c}\|_\op \prec 1$, we have
\begin{equation}
    2\|\bV'_{2c}{}^\sT \bU\|_F^2 \leq 2\|\bV'_{2c}{}^\sT\|_\text{op}^2
\|\bU\|_F^2 \prec \sum_{i=1}^p \theta_i^2 \|\bw_{iS}\|_2^2 \prec \frac{1}{d}.
\end{equation}
For the second term in
\eqref{equation:WassersteinDistanceRandomVariable3bbbhighmatrix3}, note that
$\|\mathbf{e}'_c\|_2^2= d-s$ and $\mathbf{1}_p^\sT \bU = \sum_i \theta_i
\bw_{iS} = \btheta^\sT \bW \be_S$, where $\be_S \in \R^d$ has first $s$ entries 1 and remaining entries 0. From the definition of
$\bTheta^{\sPG}_{\bW}(K)$, we know $\|\mu_1 \btheta^\top \bW\|_2 \prec 1$, which
implies $|\btheta^\top \bW\be_S| \prec 1/|\mu_1|$. Thus, the second term is bounded as
\begin{equation}
    \frac{2\|\mathbf{e}'_c\|_2^2}{d^2} \|\mathbf{1}_p^\sT \bU\|_2^2 \prec \frac{1}{d\mu_1^2}.
\end{equation}
        Combining these bounds gives \begin{equation}\label{equ:Combineoverlap2}
            A_{j,q}^{(2)} \prec \frac{1}{d\mu_{1}^2}.
        \end{equation}
Finally, combining \eqref{equation:WassersteinDistanceRandomVariable3firstterm},
\eqref{equation:Combineoverlap1},  and \eqref{equ:Combineoverlap2} shows $A_j
\prec (d\mu_1^2)^{-1}+d^{-1/2}$ for each $j=0,\ldots,k-1$ and $k=3,\ldots,D$,
which implies \eqref{eq:NoisePart1reduced} and thus
\eqref{equation:NoisePart1}.

\paragraph{Proof of \eqref{equation:NoisePart1gaussian}.} 
For each $k=3,\ldots,D$, recalling that $\bV_k=[\bq_k(\bw_1),\ldots,\bq_k(\bw_p)]^\top$, we have similarly
\begin{align*}
\E_{\bg_k}(\btheta^\top \bV_k\bg_{k,S})^2
&=\E_{\bg_k}\Big(\sum_{i=1}^p \theta_i \<\bq_k(\bw_i),\bg_{k,S}\>\Big)^2\\
&=\sum_{i,l=1}^p \theta_i\theta_l 
\Big(\<\bq_k(\bw_i),\bq_k(\bw_l)\>
-\<\bq_k(\bw_{i\setminus S}),\bq_k(\bw_{l\setminus S})\>\Big)\\
&=\sum_{i,l=1}^p \theta_i\theta_l
\Big(\<\bw_i,\bw_l\>^k-\<\bw_{i\setminus S},\bw_{l\setminus S}\>^k\Big)\\
&=\sum_{i,l=1}^p \theta_i\theta_l \sum_{j=0}^{k-1}
\binom{k}{j} \<\bw_{iS},\bw_{lS}\>^{k-j}\<\bw_{i\setminus S},\bw_{l\setminus
S}\>^j\\
&\leq C\sum_{j=0}^{k-1}\Bigg(\sum_{i=1}^p
\theta_i^2\|\bw_{iS}\|_2^{2(k-j)}+C\Bigg(\mathop{\sum_{i,l=1}^p}_{i \neq l}
\theta_i\theta_l \<\bw_{iS},\bw_{lS}\>^j
\<\bw_{i\setminus S},\bw_{l\setminus S}\>^{k-j}\Bigg)_+\Bigg).
\end{align*}
This is bounded in the same way as the above terms $A_j^{(1)}$ and
$A_{j,q}^{(2)}$ for $q=0$, showing \eqref{equation:NoisePart1gaussian}.

  \paragraph{Proof of  \eqref{equation:NoisePart2}.}

By orthonormality of the multivariate Hermite polynomials,
\begin{equation} \label{eq:var_as_coeff_sum}
    \E_{\bx}\left(\bbeta_{ki}^\top \bh_{k,S}(\bx) \right)^2 
=\E_{\bg_k}\left(\bbeta_{ki}^\top \bg_{k,S} \right)^2 
=\sum_{\mathbf{k}:
\|\bk\|_1=k,\,k_1+\ldots+k_s \geq 1} (\bbeta_{ki})_{\bk}^2.
\end{equation}
We relate this sum to the genericity condition for $\bbeta_{ki}$. Let $\bT = \iota(\bbeta_{ki}) \in (\mathbb{R}^d)^{\odot k}$, where $\iota$ is the isometric embedding defined in \eqref{eq:iota_entries}. 
Let $S_{\bT}$ be the sum of squares of all entries of $\bT$ with at least one
index belonging to $\{1,\ldots,s\}$:
\[S_{\bT}=\sum_{\mathbf{i}:\mathbf{i} \cap \{1,\ldots,s\} \neq \emptyset}
(T_\mathbf{i})^2\]
The quantity \eqref{eq:var_as_coeff_sum} is a weighted sum of these squared
tensor entries, with $\eqref{eq:var_as_coeff_sum} \leq C S_\bT$ for a constant
$C>0$ depending only on $k$. By the characterization \eqref{eq:admissibility_tensor} of the genericity condition, there exists a constant $c'>0$ such that
\[
    \|\bT \otimes_{k-1} \bT\|_{F} \le d^{-c'}.
\]
Then letting $\bA=\bT \otimes_{k-1} \bT$, its diagonal entries satisfy
$\bA_{rr} \leq \|\bA\|_F \leq d^{-c'}$. On the other hand, since $\bT$ is
symmetric, for each $j \in \{1,\ldots,k\}$,
\[\bA_{rr}=\sum_{a_2,\dots,a_k} (T_{r,a_2,\dots,a_k})^2
=\sum_{\mathbf{i}:i_j=r} (T_{\mathbf{i}})^2,\]
so
\[S_\bT \leq \sum_{j=1}^k \sum_{r=1}^s \sum_{\mathbf{i}:i_j=r} (T_\mathbf{i})^2
=k(\bA_{11}+\ldots+\bA_{ss}) \prec d^{-c'}.\]
    \end{proof}

The next lemma checks the condition needed in
Corollary \ref{corollary:RecallBoundedLipschitzFunctionSupErgetPartial}
to replace the components of $\bh_{k,\setminus S}(\bx)$ by Gaussian surrogates.

  \begin{lemma}\label{lemma:GaussianMalliavinVarianceBound112}
  For any constant $K>0$ and each $k=3,\ldots,D$,
  \begin{equation}\label{equation:GaussianMalliavinVarianceBound22}
      \sup_{\btheta\in\bTheta_{\bW}^{\sPG}(K)} \E_\bx
[(\btheta^\sT\bV_k\bh_{k,\setminus S}(\bx))^4] -
3[\E_\bx(\btheta^\sT\bV_k\bh_{k,\setminus S}(\bx))^2 ]^2 \prec \frac{1}{d^{1/4}}+ \frac{1}{d\mu_{1}^2}.
    \end{equation}
\end{lemma}

\begin{proof}
To ease notation, we will consider $S=\emptyset$. The proof for general $S$ 
holds verbatim, upon replacing $\bw_i$ throughout by $\bw_{i\setminus S}$.

For any integer $k \ge 3$, let $Q_k(\bx) = \btheta^\sT\bV_k\bh_k(\bx)$. We write this variable as $Q_k(\bx) = I_k(\bT_k)$, where $\bT_k$ is a symmetric $k$-th order tensor defined as
\[
    \bT_k = \frac{1}{\sqrt{k!}} \sum_{i=1}^p \theta_i \bw_i^{\otimes k}.
\]
By Lemma~\ref{section:ProofOfLemmaGaussianChaosConvergenceToStandardGaussian}, the deviation from Gaussianity is bounded by
\begin{equation} \label{eq:4th_moment_bound}
    \left| \E_\bx[Q_k(\bx)^4] - 3\left(\E_\bx[ Q_k(\bx)^2 ]\right)^2 \right| \leq C\max_{r=1, \dots, k-1} \| \bT_k \otimes_r \bT_k \|_F^2,
\end{equation}
where $C$ is a constant depending only on $k$. We proceed to bound $\| \bT_k \otimes_r \bT_k \|_F^2$ uniformly for all $\btheta \in \bTheta_{\bW}^{\sPG}(K)$ and for each possible contraction order $r \in \{1, \dots, k-1\}$.

\paragraph{Step 1: Matricization of the Tensor Contraction.}
 For a fixed contraction order $r$, we reshape the $k$-th order tensor $\bT_k \in \R^{d^{\otimes k}}$ into a matrix $\bA \in \R^{d^r \times d^{k-r}}$. Writing $I = (i_1, \dots, i_r)$ and $J = (j_1, \dots, j_{k-r})$, the entries of $\bA$ are given by
$\bA_{I,J} = (\bT_k)_{i_1, \dots, i_r, j_1, \dots, j_{k-r}}$. Then
$(\bT_k \otimes_r \bT_k)_{J,J'} = (\bA^\sT \bA)_{J,J'}$, and
\begin{equation}
    \| \bT_k \otimes_r \bT_k \|_F^2 = \|\bA^\sT \bA\|_F^2.
\end{equation}
Due to the symmetry, we can assume without loss of generality that $1 \leq r \leq k-r$.

Applying $\bA = \frac{1}{\sqrt{k!}} \sum_{i=1}^p \theta_i (\bw_i^{\otimes r}) (\bw_i^{\otimes (k-r)})^\sT$,
\begin{align*}
    \bA^\sT\bA &= 
    \frac{1}{k!} \sum_{i,j=1}^p \theta_i \theta_j \langle \bw_i^{\otimes r}, \bw_j^{\otimes r} \rangle (\bw_i^{\otimes (k-r)}) (\bw_j^{\otimes (k-r)})^\sT \\
    &= \frac{1}{k!} \sum_{i,j=1}^p \theta_i \theta_j \langle \bw_i, \bw_j \rangle^r (\bw_i^{\otimes (k-r)}) (\bw_j^{\otimes (k-r)})^\sT.
\end{align*}
Then, since $\|\bA^\sT \bA\|_F^2 = \Tr((\bA^\sT \bA)^2)$,
\begin{align*}
    \|\bA^\sT \bA\|_F^2 &=  
    \frac{1}{(k!)^2} \sum_{i,j,l,m} \theta_i\theta_j\theta_l\theta_m \langle \bw_i,\bw_j \rangle^r \langle \bw_l,\bw_m \rangle^r \Tr\left( (\bw_i^{\otimes(k-r)})(\bw_j^{\otimes(k-r)})^\sT (\bw_l^{\otimes(k-r)})(\bw_m^{\otimes(k-r)})^\sT \right) \\
    &= \frac{1}{(k!)^2} \sum_{i,j,l,m} \theta_i\theta_j\theta_l\theta_m \langle \bw_i,\bw_j \rangle^r \langle \bw_l,\bw_m \rangle^r \langle \bw_j,\bw_l \rangle^{k-r} \langle \bw_m,\bw_i \rangle^{k-r}.
\end{align*}
Define $\bM_m=\bV_m\bV_m^\top \in \R^{p \times p}$ with entries $(\bM_m)_{ij} = \langle \bw_i, \bw_j \rangle^m$. The above can be rewritten as
\begin{equation} \label{eq:quartic_form}
    \|\bA^\sT \bA\|_F^2 = \frac{1}{(k!)^2} \sum_{i,j,l,m} \theta_i\theta_j\theta_l\theta_m (\bM_r)_{ij} (\bM_r)_{lm} (\bM_{k-r})_{il} (\bM_{k-r})_{jm}.
\end{equation}
  Letting $\bH = \bM_r \bD_\btheta \bM_{k-r} \bD_\btheta \bM_r$ where $\bD_\btheta=\diag(\btheta)$, this is equivalently
\[\|\bT_k \otimes_r \bT_k\|_F^2=\|\bA^\top \bA\|_F^2
=\frac{1}{(k!)^2} \btheta^\sT (\bH \circ \bM_{k-r}) \btheta.\]

\paragraph{Step 2: Bounding Strategy via Matrix Decomposition.}
For any integer $m \ge 1$, note that
\begin{align*}
    (\bM_m)_{ij}=\sum_{a_1, \dots, a_m=1}^d (w_{ia_1} \cdots w_{ia_m}) (w_{ja_1} \cdots w_{ja_m})
\sum_{a_1, \dots, a_m=1}^d (\bw_{a_1} \odot \dots \odot \bw_{a_m})_i \cdot (\bw_{a_1} \odot \dots \odot \bw_{a_m})_j.
\end{align*}
This shows that the matrix $\bM_m$ can be decomposed as a sum of rank-1 matrices:
\[
    \bM_m = \sum_{a_1, \dots, a_m=1}^d (\bw_{a_1} \odot \dots \odot \bw_{a_m})(\bw_{a_1} \odot \dots \odot \bw_{a_m})^\sT.
\]
  By the identity $\bA \circ (\bx\by^\sT) = \mathrm{diag}(\bx) \bA \mathrm{diag}(\by)$, we then get:
\begin{align*}
    \btheta^\sT (\bH \circ \bM_{k-r}) \btheta &= \sum_{a_1, \dots, a_{k-r}} \btheta^\sT (\bH \circ [(\bw_{a_1} \odot \cdots \odot \bw_{a_{k-r}})(\bw_{a_1} \odot \cdots \odot \bw_{a_{k-r}})^\sT]) \btheta \\
    &= \sum_{a_1, \dots, a_{k-r}} \btheta^\sT \mathrm{diag}(\bw_{a_1} \odot \cdots \odot \bw_{a_{k-r}}) \bH \mathrm{diag}(\bw_{a_1} \odot \cdots \odot \bw_{a_{k-r}}) \btheta \\
    &\le \|\bH\|_{\op} \sum_{a_1, \dots, a_{k-r}} \|\mathrm{diag}(\bw_{a_1} \odot \cdots \odot \bw_{a_{k-r}}) \btheta\|_2^2\\
    &= \|\bH\|_{\op} \sum_{i=1}^p \theta_i^2 \left( \sum_{a_1, \dots, a_{k-r}} w_{ia_1}^2 \dots w_{ia_{k-r}}^2 \right)
    \leq \|\bH\|_{\op} \|\btheta\|_2^2 \prec \|\bH\|_{\op}.
\end{align*}
For $3 \leq r \leq k-r$, Lemma~\ref{lemma:OperatorNormFk} and the condition $\|\btheta\|_{\infty} \prec d^{-1/4}$ provide $\|\bH\|_{\op} \prec d^{-1/2}$.
However, special care is needed when $r=1$ or $r=2$.

\paragraph{Step 3:  Special Case Analysis ($r=1, k-r=2$).}

 We present the analysis for the case $k=3$, which implies $r=1$ and $k-r=2$.
In this case, $\|\bH\|_\op=\|\bW^\sT \bD_\btheta \bM_2 \bD_\btheta \bW\|_{\op}$, where $\bM_2 = \bV_2 \bV_2^\sT$. We use the decomposition of $\bV_2$ from Corollary \ref{corollary:OperatorNormF2cdecompose}:
\[
    \bV_2 = \bV_{2c} + \frac{1}{d}\1_p\be_c^\sT,
\]
where $\|\bV_{2c}\|_{\op} \prec 1$. This implies
\begin{align*}
    \|\bH\|_\op 
    & \prec \|\bW^\top\bD_\btheta\bV_{2c}\bV_{2c}^\sT \btheta\bW\|_\op+ \left\| \bW^\sT \bD_\btheta \left(\frac{1}{d}\1_p\1_p^\sT\right) \bD_\btheta \bW \right\|_{\op}.
\end{align*}
The spike component of $\bM_2$ contributes the following term to the operator norm:
\begin{align*}
    \left\| \bW^\sT \bD_\btheta \left(\frac{1}{d}\1_p\1_p^\sT\right) \bD_\btheta \bW \right\|_{\op} &= \frac{1}{d} \left\| (\bW^\sT \bD_\btheta \1_p)(\1_p^\sT \bD_\btheta \bW) \right\|_{\op} 
    =\frac{1}{d} \left\| \bW^\sT \btheta \right\|_2^2
    \prec \frac{1}{d\mu_1^2},
\end{align*}
the last inequality using $\|\mu_1 \btheta^\sT \bW\|_2 \prec 1$
from the definition of $\bTheta_{\bW}^{\sPG}(K)$.
The centered (bounded norm) component of $\bM_2$ satisfies
\[
    \|\bW^\sT \bD_\btheta \bV_{2c}\bV_{2c}^\sT \bD_\btheta \bW\|_{\op} \le \|\bW^\sT \bD_\btheta\|_{\op} \|\bV_{2c}\|_{\op}^2 \|\bD_\btheta \bW\|_{\op} \prec \|\bW^\sT \bD_\btheta \bD_\btheta \bW\|_{\op} = \|\bW^\sT \bD_{\btheta}^2 \bW\|_{\op}.
\]
To bound this term, we apply a $\epsilon$-net argument. For any unit vector $\bv \in \mathbb{R}^d$, consider the quadratic form
\begin{align}\label{eq:union-bound-example-sum-power-4}
    \bv^\sT \bW^\sT \bD_{\btheta}^2 \bW \bv &= \sum_{i=1}^p \theta_i^2 (\bw_i^\sT \bv)^2
    \le \sqrt{\sum_{i=1}^p \theta_i^4} \sqrt{\sum_{i=1}^p (\bw_i^\sT \bv)^4}
    \le \|\btheta\|_\infty \|\btheta\|_2 \sqrt{\sum_{i=1}^p (\bw_i^\sT \bv)^4}.
\end{align}
Since $\btheta \in \bTheta_{\bW}^{\sPG}(K)$, we have $\|\btheta\|_\infty \leq
d^{-1/4}(\log d)^K$ and $\|\btheta\|_2 \leq (\log d)^K$. For any fixed unit vector $\bv \in \R^d$, the random variables $\{\sqrt{d}\bw_i^\top \bv\}_{i=1}^p$ and i.i.d.\ and $O(1)$-subgaussian. Then a standard tail bound shows 
$\P[\sum_{i=1}^p (\sqrt{d}\bw_i^\top \bv)^4-\E(\sqrt{d}\bw_i^\top \bv)^4 \geq t]
\leq 2e^{-c\min(t^{1/2},t^2/p)}$ for a constant $c>0$. For any $C>0$, applying this bound with $t \asymp d^2 \asymp p$, this implies that there exists $C'>0$ sufficiently large
such that $\sum_{i=1}^p (\bw_i^\sT \bv)^4 \leq C'$ with probability at
least $1 - e^{-Cd}$. Then there exists a constant $K'>0$ for which
    $\bv^\sT \bW^\sT \bD_{\btheta}^2 \bW \bv \leq d^{-1/4}(\log d)^{K'}$
    with probability $1-e^{-Cd}$.
By a $\epsilon$-net argument over the unit sphere $\S^{d-1}$, this bound extends to the operator norm to show
\[
    \|\bW^\sT \bD_{\btheta}^2 \bW\|_{\op} \prec \frac{1}{d^{1/4}}.
\]
Combining these bounds for the spike and centered components, we obtain $\|\bH\|_\op \prec d^{-1/4}+(d\mu_1^2)^{-1}$. This concludes the proof of the bound
\[
    \E_\bx[Q_3(\bx)^4] - 3[\E_\bx Q_3(\bx)^2 ]^2 \prec \frac{1}{d^{1/4}} + \frac{1}{d\mu_{1}^2}
\]
for $k=3$. The remaining cases $(r=2,k-r=2)$, $(r=1, k-r \ge 3)$,
and $(r=2, k-r \ge 3)$ for $k \geq 4$ follow from similar arguments, and we omit the details for brevity.
\end{proof} 

\begin{proposition}\label{prop:NormThetaV}
For any $K>0$, every $\btheta \in \bTheta_\bW^\sPG(K)$ satisfies also
\[\|\mu_k\btheta^\sT \bV_k\|_2 \prec 1 \text{ for each } k=2,\ldots,D.\]
	\end{proposition}
	\begin{proof}
	By Lemma \ref{lemma:OperatorNormFk}, we know that $\|\bV_k\|_{\op} \prec
1$ for $k\geq 3$. Since $\btheta \in \bTheta_\bW^\sPG(K)$ implies
$\|\btheta\|_2 \prec 1$, and $\mu_k \prec 1$ by Assumption
\ref{assumption:activation}, the bounds
$\|\mu_k\btheta^\sT \bV_k\|_2 \prec 1$ for $k \geq 3$ are immediate.

	 For $k=2$, by Corollary \ref{corollary:OperatorNormF2cdecompose}, we decompose \begin{equation}
		\bV_2 = \bV_{2c} + \frac{1}{d} \1_p\be_c^\sT,
	 \end{equation}
	 where $\|\bV_{2c}\|_\op \prec 1$,
	$\1_p\in \R^p$ is the all-1's vector, and $\be_c \in \R^{B_{d,2}}$  is the vector
with 1 in the first $d$ coordinates and 0 in the rest. Then $\|\mu_2\btheta^\top
\bV_{2c}\|_2 \prec 1$ follows also from $\|\btheta\|_2 \prec 1$ and
$\mu_2 \prec 1$, while
$\|\mu_2 \btheta^\top (\frac{1}{d}\1_p\be_c^\sT)\|_2 \prec 1$ follows also from
the assumption $|\btheta^\top \1_p| \prec 1$ for
$\btheta \in \bTheta_\bW^\sPG(K)$.
\end{proof}

We now show an intermediary result for Theorem
\ref{theorem:RecallBoundedLipschitzFunctionSupErgetisotropic},
in which the Hermite features $\bh_k(\bx)$ are replaced by Gaussian surrogates
$\bg_k$ for a class of twice-differentiable functions $\varphi$ with bounded
first and second derivatives. The proof of Theorem
\ref{theorem:RecallBoundedLipschitzFunctionSupErgetisotropic} will be completed
by further showing that each
$\bV\bg_k$ for $k \geq 3$ may be replaced by an isotropic Gaussian vector,
and that the result may be extended to the class of Lipschitz functions
$\varphi$.

  \begin{lemma}[A Term-by-Term CLT for Higher-order Features]\label{theorem:RecallBoundedLipschitzFunctionSupErget}
  For $k=3,\ldots,\max(D,D')$, let
  $\bg_k\sim\cN(\mathbf 0,\mathbf I_{B_{d,k}})$
  be mutually independent Gaussian vectors also independent of $\bx$. Let
  $\bxi_k$ and $\tilde\bxi_k$ be as in Theorem
\ref{theorem:RecallBoundedLipschitzFunctionSupErgetisotropic}, and let
$\tilde\cL$ be the space of functions $\varphi:\R^{m+1}
\to \R$ that are twice continuously-differentiable
with $\|\|\nabla \varphi\|_2\|_\infty \leq L_1 \equiv L_1(d)$
and $\|\|\nabla \varphi^2\|_\op\|_\infty \leq L_2 \equiv L_2(d)$. Then
for any constant $K>0$, there exists a constant $c>0$ such that
simultaneously over $\varphi \in \tilde\cL$ and $\btheta \in
\bTheta_\bW^\sPG(K)$,
  \begin{align*}
    &\Bigl|
    \E_{\bx}\!
    \bigl[\,
    \varphi\bigl(
    \bx_S,\, \bxi_2,\bxi_3,\ldots,\bxi_{D'},\, \btheta^{\sT}\sigma(\bW\bx)
    \bigr)
    \bigr] \\
    &\qquad-
    \E_{\bx,\{\bg_k\}_{k\ge 3}}\!
    \Bigl[\,
    \varphi\Bigl(
    \bx_S,\, \bxi_2, \tilde\bxi_3,\ldots,\tilde\bxi_{D'},
    \btheta^{\sT}\! \Bigl(\sum_{j=0}^2 \mu_j \bV_j \bh_j(\bx) + \sum_{j=3}^D \mu_j \bV_j \bg_j \Bigr)
    \Bigr)
    \Bigr]
    \Bigr| \prec  \frac{L_1+L_2}{d^c}.
    \end{align*}
\end{lemma}

\begin{proof}
Denote
\[\begin{gathered}
\bxi_{k,\setminus S}=(\bbeta_{ki}^\top \bh_{k,\setminus S}(\bx))_{i=1}^{s_k},
\qquad \tilde \bxi_{k,\setminus S}=(\bbeta_{ki}^\top \bg_{k,\setminus
S})_{i=1}^{s_k}\\
\sigma_{\setminus S}(\bx)=\sum_{j=0}^2 \mu_j \bV_j\bh_j(\bx)
+\sum_{j=3}^D \mu_j \bV_j\bh_{j \setminus S}(\bx)\\
\tilde \sigma_{\setminus S}(\bx)=\sum_{j=0}^2 \mu_j \bV_j\bh_j(\bx)
+\sum_{j=3}^D \mu_j \bV_j\bg_{j \setminus S}
\end{gathered}
\]
Then Lemma \ref{lemma:NoisePart} implies
\[\|\bxi_k-\bxi_{k,\setminus S}\|_2,
\|\tilde\bxi_k-\tilde\bxi_{k,\setminus S}\|_2 \prec d^{-c}
\text{ for each } k=2,\ldots,D',\]
\[\sup_{\btheta \in \bTheta_\bW^{\sPG}(K)}
\Bigg\{
|\btheta^\top (\sigma(\bW\bx)-\sigma_{\setminus S}(\bx))|,\;
\Big|\btheta^\top \Big(\sum_{j=0}^2 \mu_j \bV_j\bh_j(\bx)
+\sum_{j=3}^D \mu_j\bV_j\bg_j-\tilde \sigma_{\setminus S}(\bx)\Big)\Big|\Bigg\}
\prec \frac{1}{d^{1/2}\mu_1}+\frac{1}{d^{1/4}}.\]
In light of the bounds $\|\|\nabla \varphi\|_2\|_\infty \leq L_1$ and
$\mu_1^{-1} \prec
d^c$ for some $c<1/2$ by Assumption \ref{assumption:activation},
it suffices to show that simultaneously over $\varphi \in \tilde \cL$
and $\btheta \in \bTheta_{\bW}^{\sPG}(K)$,
  \begin{align}\label{eq:RecallBoundedLipschitzFunctionSupErget}
    &\Bigl|
    \E_{\bx}\!
    \bigl[\,
    \varphi\bigl(
    \bx_S,\, \bxi_{2\setminus S},\bxi_{3\setminus S},\ldots,\bxi_{D'\setminus
S},\, \btheta^{\sT}\sigma_{\setminus S}(\bx)
    \bigr)
    \bigr] \\
    &\qquad-
    \E_{\bx,\{\bg_k\}_{k\ge 3}}\!
    \Bigl[\,
    \varphi\Bigl(
    \bx_S,\, \bxi_{2\setminus S}, \tilde\bxi_{3\setminus
S},\ldots,\tilde\bxi_{D' \setminus S},
    \btheta^{\sT}\tilde\sigma_{\setminus S}(\bx) \Bigr)
    \Bigr)
    \Bigr]
    \Bigr| \prec  \frac{L_2}{d^{c}}.
    \end{align}

For this, let us condition on $\bx_S$ and $\bW$ and
fix $\varphi \in \cL$ and $\btheta \in \bTheta_\bW^\sPG(K)$. We apply
Theorem~\ref{theorem:RecallBoundedLipschitzFunctionSupErgetPartial} in the form
Corollary~\ref{corollary:RecallBoundedLipschitzFunctionSupErgetPartial},
with $K'=2$ and conditional on $\bx_S$ and $\bW$, to replace
$\bh_{k,\setminus S}(\bx)$ by $\bg_{k,\setminus S}$ for all orders
$k \geq K'+1=3$. We note that
\begin{align*}
\bV_1\bh_1(\bx)&=\bV_1\proj_S\bh_1(\bx)+\bV_1\proj_{S,\perp}\bh_1(\bx),\\
\bV_2\bh_2(\bx)&=\bV_2\proj_S\bh_2(\bx)+\bV_2\proj_{S,\perp}^{\text{mixed}}\bh_2(\bx)
+\bV_2\proj_{S,\perp}^{\text{pure}}\bh_2(\bx),
\end{align*}
where $\bV_2\proj_{S,\perp}^{\text{mixed}}\bh_2(\bx)$ has the explicit form
\begin{equation}\label{eq:V2mixedterm}
\bV_2\proj_{S,\perp}^{\text{mixed}}\bh_2(\bx)
=\Big(\sum_{j \in S}\sum_{k \notin S} \sqrt{2}w_{ij}w_{ik} x_jx_k\Big)_{i=1}^p
\end{equation}
and may be understood as linear functions (order-1 chaoses) in the variables
$\bx_{\setminus S}$ conditional on $\bx_S$.
Thus, for fixed $\btheta$, $\bx_S$, and $\bW$,
\[\varphi\bigl(
    \bx_S,\, \bxi_{2\setminus S},\bxi_{3\setminus S},\ldots,\bxi_{D'\setminus
S},\, \btheta^{\sT}\sigma_{\setminus S}(\bx)
    \bigr)\]
is a function --- call it $\tilde \varphi_{\btheta,\bx_S,\bW}$ --- of the
order-1 components
\[\mu_1\btheta^\top \bV_1\proj_{S,\perp}\bh_1(\bx),
\qquad \mu_2\btheta^\top \bV_2\proj_{S,\perp}^{\text{mixed}}\bh_2(\bx)\]
of the Wiener chaos space over the variables $\bx_{\setminus S}$,
the order-2 components
\[\mu_2 \btheta^\top \bV_2\proj_{S,\perp}^\text{pure}\bh_2(\bx)
\equiv \mu_2 \btheta^\top \bV_2\bh_{2,\setminus S}(\bx), \quad
\bbeta_{2,i}^\top \bh_{2,\setminus S}(\bx),\]
and the order-3 and higher components
\[\mu_k \btheta^\top \bV_k\bh_{k,\setminus S}(\bx) \text{ for } k=3,\ldots,D,
\qquad \bbeta_{ki}^\top \bh_{k,\setminus S}(\bx)
\text{ for } k=3,\ldots,D'.\]
By the given condition for $\varphi$, this
function $\varphi_{\btheta,\bx_S,\bW}$ satisfies
$\|\|\nabla^2 \varphi_{\btheta,\bx_S,\bW}\|_\op\|_\infty \leq L_2$.
We check the conditions of Corollary
\ref{corollary:RecallBoundedLipschitzFunctionSupErgetPartial}:
By Lemma \ref{lemma:GaussianMalliavinVarianceBound112}
and the genericity assumption \eqref{eq:ass_bbeta_ki_4_th_moment}, on an event of
probability $1-d^{-C}$ over $\bW$, \eqref{eq:C2_assump} holds
with $\rho=d^{-c}$
for $\mu_k \btheta^\top \bV_k\bh_{k,\setminus S}(\bx)$ with $k
\geq 3$ and $\bbeta_{ki}^\top \bh_{k,\setminus S}(\bx)$ with $k \geq 3$.
By the condition $\|\mu_1\btheta^\top \bW\|_2 \prec 1$ for
$\btheta \in \bTheta_\bW^\sPG(K)$, the bound
$\|\mu_2 \btheta^\top \bV_2\|_2 \prec 1$
of Proposition \ref{prop:NormThetaV}, and 
$\|\bbeta_{2,i}\|_2=1$, \eqref{eq:C1_assump} holds with $B=(\log d)^{K'}$ and some constant $K'>0$ for
$\mu_1\btheta^\top \bV_1\proj_{S,\perp}\bh_1(\bx)$,
$\mu_2\btheta^\top \bV_2\proj_{S,\perp}^\text{pure}\bh_2(\bx)$, and
$\bbeta_{2,i}^\top \bh_{2,\setminus S}(\bx)$. Finally, by the bound
\begin{align}
\E_{\bx_{\setminus S}}
\|\mu_2\btheta^\top \bV_2\proj_{S,\perp}^{\text{mixed}}\bh_2(\bx)\|_2^2
&=2\mu_2^2\sum_{k \notin S}
\Big(\sum_{j \in S} \sum_{i=1}^p
\theta_i w_{ij}w_{ik}x_j\Big)^2\\
&\leq C\mu_2^2 \sum_{j \in S} x_j^2 \sum_{k \notin S}
\Big(\sum_{i=1}^p \theta_i w_{ij} w_{ik}\Big)^2
    \leq C\|\bx_S\|_2^2 \|\mu_2 \btheta^\top \bV_2\|_2^2
\end{align}
and $\|\mu_2 \btheta^\top \bV_2\|_2 \prec 1$ by Proposition
\ref{prop:NormThetaV}, 
the condition \eqref{eq:C1_assump} holds with $B=(\log d)^{K'}\|\bx_S\|_2$
for the final component
$\mu_2\btheta^\top \bV_2\proj_{S,\perp}^{\text{mixed}}\bh_2(\bx)$. Thus, 
Corollary \ref{corollary:RecallBoundedLipschitzFunctionSupErgetPartial} implies
that with probability $1-d^{-C}$ over $\bW$,
there exist constants $K',c>0$ such that
for any $\varphi \in \cL$, $\btheta \in \bTheta_\bW^{\sPG}(\bW)$, and $\bx_S \in
\R^s$, we have
\begin{align*}
&\Bigl|\E_{\bx_{\setminus S}}
    \bigl[\,
    \varphi\bigl(
    \bx_S,\, \bxi_{2\setminus S},\bxi_{3\setminus S},\ldots,\bxi_{D'\setminus
S},\, \btheta^{\sT}\sigma_{\setminus S}(\bx)
    \bigr)
    \bigr]\\
&\qquad-
    \E_{\bx_{\setminus S},\{\bg_k\}_{k\ge 3}}
    \Bigl[\,
    \varphi\Bigl(
    \bx_S,\, \bxi_{2\setminus S}, \tilde\bxi_{3\setminus
S},\ldots,\tilde\bxi_{D' \setminus S},
    \btheta^{\sT}\tilde\sigma_{\setminus S}(\bx) \Bigr)
    \Bigr)
    \Bigr]
    \Bigr| \leq \frac{(\log d)^{K'}L_2(1+\|\bx_S\|_2)}{d^c}.
\end{align*}
Taking the expectation on both sides over $\bx_S$ 
shows \eqref{eq:RecallBoundedLipschitzFunctionSupErget}, completing the proof.
\end{proof}

\begin{lemma}\label{lemma:GaussianMalliavinVarianceBound112beta}
Let $3 \leq k \leq D'$, and let
$\bbeta_{ki} \in \R^{B_{d,k}}$ be any vector satisfying $\|\bbeta_{ki}\|_2
= 1$ and the genericity condition \eqref{eq:ass_bbeta_ki_4_th_moment}. Then 
for any constant $K>0$, there exists a constant $c>0$ such that 
\begin{equation}
    \sup_{\btheta \in \bTheta_{\bW}^{\sPG}(K)} \left| \mu_k \btheta^\sT \bV_k
\bbeta_{ki} \right| \prec \frac{1}{d^{c}}.
\end{equation}
\end{lemma}

\begin{proof}
Denote $\bu = \bV_k \bbeta_{ki} \in \R^p$. Then $\mu_k \btheta^\sT \bV_k
\bbeta_{ki} = \mu_k \langle \btheta, \bu \rangle$.
Since $|\mu_k| \prec 1$ and $\|\btheta\|_2 \prec 1$
for $\btheta \in \bTheta_{\bW}^{\sPG}(K)$, it suffices to show that $\|\bu\|_2
\prec d^{-c}$.
By definition of $\bV_k = [\bq_k(\bw_1), \dots, \bq_k(\bw_p)]^\sT$, we have
\begin{equation}
    \|\bu\|_2^2 = \|\bV_k \bbeta_{ki}\|_2^2 = \sum_{i=1}^p \underbrace{(\bbeta_{ki}^\sT
\bq_k(\bw_i))^2}_{:=U_i^2}.
\end{equation}
Let $\bT=\iota(\bbeta_{ki}) \in (\R^d)^{\odot k}$ via the isometry
$\iota$ defined in \eqref{eq:iota_entries}. We have $\|\bT\|_F^2 =
\|\bbeta_{ki}\|_2^2 = 1$.
By Lemma \ref{lemma:SymmetricTensorProductMalliavinS},  we have
$\iota(\bq_k(\bw_i)) = \bw_i^{\otimes k}$. Thus,
\begin{align}
    \E[U_i^2]&= \E_{\bw_i} \left[ \sum_{j_1,\ldots,j_k=1}^d
\bT_{j_1\ldots j_k} \bw_{ij_1}\ldots \bw_{ij_k} \cdot
\sum_{l_1,\ldots,l_k=1}^d \bT_{l_1\ldots l_k} \bw_{il_1}\ldots
\bw_{il_k}\right]\\
&=\sum_{j_1,\ldots,j_k,l_1,\ldots,l_k=1}^d \bT_{j_1\ldots j_k}
\bT_{l_1\ldots l_k} \E\left[\bw_{ij_1}\ldots \bw_{ij_k}\bw_{il_1}
\ldots \bw_{il_k}\right].\label{eq:Yi2_expansionin}
\end{align}
Representing $\bw_i=\bg/\|\bg\|_2$ where $\bg \sim \mathcal{N}(0,\bI_d)$,
and using independence of $\bw_i$ and $\|\bg\|_2$, we see that
\begin{equation}
    \E_{\bw_i}[\bw_{ij_1} \ldots \bw_{ij_k}\bw_{il_1} \ldots \bw_{il_k}]
=f(k,d)\E_{\bg}[g_{j_1} \ldots g_{j_k} g_{l_1}\ldots g_{l_k}]
\end{equation}
where $f(k,d)=1/\E\|\bg\|_2^{2k} \leq Cd^{-k}$. By Wick's formula,
$\E_{\bg}[g_{j_1} \ldots g_{j_k} g_{l_1}\ldots g_{l_k}]$ is given by a sum over
pairings of $\{j_1,\ldots,j_k,l_1,\ldots,l_k\}$, which yields
\[\E[U_i^2]=f(k,d)\sum_{\text{pairings of } j_1,\ldots,j_k,l_1,\ldots,j_k}
\mathop{\sum_{j_1,\ldots,j_k,l_1,\ldots,l_k=1}^d}_{\text{indices are equal in
each pair}} \bT_{j_1\ldots j_k}\bT_{l_1\ldots l_k}.\]
We group these pairings based on the number $r \in \{0,\ldots,k\}$
of pairs connecting an index in $J=\{j_1,\ldots,j_k\}$ to an
index in $L=\{l_1,\ldots,l_k\}$. Note that the remaining indices of $J$ and $L$
are paired with themselves, so $k-r$ must be even. Since $\bT$ is symmetric,
the sum over indices $j_1,\ldots,j_k,l_1,\ldots,l_k$ for any fixed
pairing with $r=k$ is $\|\bT\|_F^2=1$. This sum for any fixed pairing with $r<k$
is $\<\bT \otimes_r \bT,\bI_d^{\otimes (k-r)}\>$ where $\bI_d$ is the $d \times
d$ identity matrix.
%
%
By the genericity condition \eqref{eq:admissibility_tensor} for $\bbeta_{ki}$,
we have $|\langle \bT \otimes_r \bT, \bI_d^{\otimes (k-r)}\rangle|
\leq d^{(k-r)/2} \|\bT \otimes_r \bT\|_F \leq d^{-c+(k-r)/2}$
for a constant $c>0$. Applying these cases and using $f(k,d) \leq Cd^{-k}$
gives, for any $k \geq 3$ and $k-r \in \{0,\ldots,k\}$ even,
\[\E[U_i^2] \leq Cd^{-2-c}.\]
Then $\E\|\bu\|_2^2=p \cdot \E[U_i^2] \leq d^{-c'}$ for a constant $c'>0$.
By hypercontractivity of the uniform distribution on the sphere for $\bw_i$, this implies
$\|\bu\|_2 \prec d^{-c'/2}$, as desired.
\end{proof}

We now conclude the proof of Theorem
\ref{theorem:RecallBoundedLipschitzFunctionSupErgetisotropic} by showing that
the terms $\{\bV_k\bg_k\}_{k \geq 3}$ in
Lemma~\ref{theorem:RecallBoundedLipschitzFunctionSupErget} may be replaced by
independent isotropic Gaussian vectors, and that the result 
of Lemma~\ref{theorem:RecallBoundedLipschitzFunctionSupErget} 
may be extended to the class of Lipschitz functions $\varphi$.
 
\begin{proof}[Proof of Theorem \ref{theorem:RecallBoundedLipschitzFunctionSupErgetisotropic}]

Let $\tilde \cL$ be the class of twice-differentiable functions in Lemma
\ref{theorem:RecallBoundedLipschitzFunctionSupErget}.
Let $\bg_k' \sim \mathcal{N}(0, \bI_{B_{d,k}})$ be independent Gaussian vectors
for $k=3,\ldots,D$, independent of all other randomness. Then $\mu_{>2}\bg_*$ is
equal in law to $\mu_3\bg_3'+\ldots+\mu_D\bg_D'$. For any $\varphi \in
\tilde\cL$, since $\|\|\nabla \varphi\|_2\|_\infty \leq L_1$, we have
\begin{align}\label{eq:wass_sum_bound}
    & \sup_{\btheta \in
\bTheta_{\bW}^{\sPG}(K)}\Bigg|\E_{\bx,\{\bg_k\}}
    \left[\,
    \varphi\left(\bx_S,\bxi_2,\tilde\bxi_3,\ldots,\tilde\bxi_{D'},
    \btheta^{\sT}\!\sum_{k=3}^D \mu_j \bV_k\bg_k \right) \right]\notag\\
&\hspace{1in}-
\E_{\bx,\{\bg_k\},\{\bg_k'\}}
    \left[\,
    \varphi\left(\bx_S,\bxi_2,\tilde\bxi_3,\ldots,\tilde\bxi_{D'},
    \btheta^{\sT}\! \sum_{k=3}^D \mu_k \bg_k'
    \right)
    \right]\Bigg| \notag \\
    &\le L_1 \sum_{k=3}^D \sup_{\btheta \in \bTheta_{\bW}^{\sPG}(K)}
W_1\Big(\underbrace{(\tilde \bxi_k, \btheta^\sT \mu_k \bV_k
\bg_k)}_{:=\bX_k},
\underbrace{(\tilde\bxi_k, \btheta^\sT \mu_k \bg_k')}_{:=\bY_k}\;\Big|\;\bW\Big),
\end{align}
where $W_1(\;\cdot \mid \bW)$
denotes the Wasserstein-1 distance between the joint laws
of the above vectors $\bX_k$ and $\bY_k$ over the randomness of
$(\bg_k,\bg_k')$ conditional on $\bW$, and we set
$\tilde \bxi_k=0$ for $k=D'+1,\ldots,D$ if $D>D'$.
Recalling that $\tilde \bxi_k=(\bbeta_{ki}^\top \bg_k)_{i=1}^{s_k}$,
the laws of $\bX_k$ and $\bY_k$ conditional on $\bW$
are multivariate Gaussian, with covariance matrices
\begin{equation}
    \bC_{\bY_k} = \begin{pmatrix}
        (\<\bbeta_{ki},\bbeta_{k,j}\>)_{i,j=1}^{s_k} & 0 \\
        0 & \mu_k^2 \|\btheta\|_2^2
    \end{pmatrix},
\qquad
    \bC_{\bX_k} = \begin{pmatrix}
(\<\bbeta_{ki},\bbeta_{k,j}\>)_{i,j=1}^{s_k}
        & (\mu_k \bbeta_{ki}^\sT \bV_k^\sT \btheta)_{i=1}^{s_k} \\
        (\mu_k \btheta^\sT \bV_k \bbeta_{ki})_{i=1}^{s_k} & \mu_k^2 \btheta^\sT
\bV_k \bV_k^\sT \btheta
    \end{pmatrix}.
\end{equation}
By Lemma~\ref{lemma:GaussianMalliavinVarianceBound112beta}, we have
$|\mu_k \btheta^\sT \bV_k \bbeta_{ki}| \prec d^{-c}$ for each $i=1,\ldots,s_k$.
By Lemma~\ref{lemma:OperatorNormFk}, we have
\[\|\bV_3\bV_3^\sT-(3/d)\bW \bW^\sT - \bI_p\|_{\op} \prec d^{-1/2},
\quad \|\bV_4\bV_4^\sT-(3/d)\1_p\1_p^\top - \bI_p\|_{\op} \prec
d^{-1/2}\]
\[\|\bV_k\bV_k^\sT-\bI_p\|_\op \prec d^{-1/2} \text{ for } k \geq 5.\]
Then applying $|\mu_k| \prec 1$, $\|\btheta\|_2 \prec 1$,
$|\btheta^\sT \1_p| \prec 1$, and
$|\btheta^\sT \bW\bW^\sT \btheta| \le \|\btheta^\sT \bW\|_2^2 \prec \mu_1^{-2}$
for all $\btheta \in \bTheta_{\bW}^{\sPG}(K)$, it follows for any $k \geq 3$
that $|\mu_k^2 \btheta^\sT \bV_k \bV_k^\sT \btheta - \mu_k^2 \|\btheta\|_2^2|
\prec d^{-1/2} + (d\mu_1^2)^{-1}$. Hence, by the condition for $\mu_1$ in Assumption \ref{assumption:activation},
\[\|\bC_{\bX_k}-\bC_{\bY_k}\|_\op \prec d^{-c}\]
for some $c>0$,
which implies that $W_1(\bX_k,\bY_k \mid \bW) \prec d^{-c'}$ for some $c'>0$. Applying this
bound to \eqref{eq:wass_sum_bound} and combining
with the result of Lemma \ref{theorem:RecallBoundedLipschitzFunctionSupErget},
we obtain the claim of Theorem
\ref{theorem:RecallBoundedLipschitzFunctionSupErgetisotropic} over the function
class $\tilde\cL$, i.e.\ for any $K>0$, there exists $c>0$ such that
simultaneously over $\varphi \in \tilde \cL$ and
$\btheta \in \bTheta_\bW^\sPG(K)$,
\begin{equation}\label{eq:CLTforC2function}
\bigg|\E_{\bx}\Big[\varphi(\bx_S,\bxi_2,\bxi_3,\ldots,\bxi_{D'},\btheta^\top
\bz^\RF)\Big]
-\E_{\bx,\{\bg_k\},\bg_*}\Big[\varphi(\bx_S,\bxi_2,\tilde\bxi_3,\ldots,\tilde\bxi_{D'},\btheta^\top
\bz^\sPG)\Big]\bigg| \prec \frac{L_1+L_2}{d^c}.
\end{equation}
To conclude the proof, we extend this to the Lipschitz class $\cL$ by a
smoothing argument: Let $\bG \sim \cN(0, \bI_{m+1})$ be a standard
Gaussian vector independent of all other randomness.
For any $\varphi \in \cL$ and
a smoothing parameter $\kappa \equiv \kappa(d)>0$, we define the smoothed
function $\varphi_\kappa:\R^{m+1}\to \R$ as
$\varphi_\kappa(\bx)=\E_{\bG}[\varphi(\bx+\kappa \bG)]$.
Since $\varphi$ is $L$-Lipschitz,
\begin{equation}\label{eq:varphikappaapprox}
|\varphi_{\kappa}(\bx)-\varphi(\bx)|
\leq \E_\bG[L\|\kappa \bG\|_2] \leq CL\kappa
\end{equation}
for a constant $C>0$ depending only on $m$.
Note that $\varphi_\kappa$ is twice continuously-differentiable, with
\begin{equation}
\|\nabla \varphi_\kappa\|_2=\|\E_{\bG}[\nabla
\varphi(\bx+\kappa \bG)]\|_2 \leq L.
\end{equation}
Applying Gaussian integration by parts, also
\begin{equation}
\|\nabla^2 \varphi_\kappa(\bx)\|_\op
=\|\E_{\bG}[\nabla^2 \varphi(\bx+\kappa \bG)]\|_\op
=\frac{1}{\kappa}\|\E_{\bG}[\nabla\varphi(\bx+\kappa \bG)
\bG^\top]\|_\op \leq \frac{CL}{\kappa}.
\end{equation}
Thus \eqref{eq:CLTforC2function} applies to $\varphi_\kappa$ with $L_1=L$ and
$L_2=CL/\kappa$. Then combining
\eqref{eq:CLTforC2function} with the approximation bound
\eqref{eq:varphikappaapprox} shows, simultaneously over $\varphi \in \cL$ and 
$\btheta \in \bTheta_\bW^\sPG(K)$,
\[\bigg|\E_{\bx}\Big[\varphi(\bx_S,\bxi_2,\bxi_3,\ldots,\bxi_{D'},\btheta^\top
\bz^\RF)\Big]
-\E_{\bx,\{\bg_k\},\bg_*}\Big[\varphi(\bx_S,\bxi_2,\tilde\bxi_3,\ldots,\tilde\bxi_{D'},\btheta^\top
\bz^\sPG)\Big]\bigg| \prec L\kappa+\frac{L}{\kappa d^c}.\]
The theorem now follows upon choosing $\kappa=d^{-c/2}$ and adjusting the value
of $c$.
\end{proof}

\subsection{Proof of Theorem \ref{theorem:RecallBoundedLipschitzFunctionSupErgettt}}
\label{app:proof_theorem:RecallBoundedLipschitzFunctionSupErgettt}

\begin{proof}
The proof uses again the general interpolation framework established in
Theorem~\ref{theorem:RecallBoundedLipschitzFunctionSupErgetPartial} and its
Corollary~\ref{corollary:RecallBoundedLipschitzFunctionSupErgetPartial}, 
and is similar to the proof of
Theorem~\ref{theorem:RecallBoundedLipschitzFunctionSupErgetisotropic}. Here, we
highlight the key differences in the argument, and omit details that are
analogous to the previous proofs.

\paragraph{Step 1: Decomposing the swapped term and controlling mixed-coordinate components.}
We first consider $\varphi \in \tilde \cL$ for the preceding twice-differentiable class
$\tilde \cL$. Denote
\[\btheta^\sT \bV_2 \proj_{S,\perp} \bh_2(\bx)
=\underbrace{\btheta^\sT \bV_2 \proj_{S,\perp}^\text{mixed} \bh_2(\bx)}
_{:=Q_{\text{mixed}}(\bx)}
+\underbrace{\btheta^\sT \bV_2 \proj_{S,\perp}^\text{pure} \bh_2(\bx)}
_{:=Q_{\text{pure}}(\bx_{\setminus S})}\]
Recalling the explicit form \eqref{eq:V2mixedterm} for
$\bV_2 \proj_{S,\perp}^\text{mixed} \bh_2(\bx)$, a
direct calculation shows
\[\E\left[Q_{\text{mixed}}(\bx)^2\right]
=\sum_{j \in S}\sum_{k \notin S}
2\Big(\sum_{i=1}^p \theta_i w_{ij}w_{ik}\Big)^2
=2\bigl\|\bW_S^{\sT}\bD_{\btheta}\,\bW_{\backslash S}\bigr\|_{F}^2.
\]
Applying the condition
$\|\bW_S^{\sT}\bD_{\btheta}\,\bW_{\backslash S}\|_F \leq
d^{-\eps}$ for $\btheta \in \bTheta_{\bW}^{\sCG}(\eps,K)$, this
shows $\E[Q_{\text{mixed}}(\bx)^2]\leq d^{-2\eps}$,
and hence $Q_{\text{mixed}}(\bx) \prec d^{-\eps}$ by Gaussian
hypercontractivity. The same calculation shows $\btheta^\sT
\bV_2\proj_{S,\perp}^\text{mixed}\bg_2 \prec d^{-\eps}$, and hence these
components of $\btheta^\top\bz^\sPG$ and $\btheta^\top\bz^\sCG$ may first be
replaced by 0. As in the proof of
Theorem~\ref{theorem:RecallBoundedLipschitzFunctionSupErgetPartial},
the components $\bbeta_{2,i}^\top(\bh_2(\bx)-\bh_{2,\setminus S}(\bx))$
of $\tilde \bxi_2$ may also be replaced by 0.

\paragraph{Step 2: Bounding the fourth moment of the purely non-signal term.}
We now apply
Corollary~\ref{corollary:RecallBoundedLipschitzFunctionSupErgetPartial} with
$K'=1$, conditional on $\bx_S$, $\bg_3,\ldots,\bg_{D'},\bg_*$, and $\bW$, 
to replace $\proj_{S,\perp}^\text{pure}\bh_2(\bx) \equiv \bh_{2,\setminus S}$
by $\bg_{k,\setminus S}$. For this, we must check the
condition \eqref{eq:C2_assump} for the order-2 chaos components
$Q_{\text{pure}}(\bx_{\setminus S})$ 
and $\bbeta_{2,i}^\top \bh_{2,\setminus S}$ over $\bx_{\setminus S}$,
and the condition \eqref{eq:C1_assump} for the order-1 component
$\btheta^\top \proj_{S,\perp} \bh_1(\bx)$.
We note that \eqref{eq:C2_assump} holds with $\rho=d^{-c}$ for
$\bbeta_{2,i}^\top \bh_{2,\setminus S}$ by the genericity assumption
(\ref{eq:admissibility_tensor}), and
\eqref{eq:C1_assump} holds with $B=(\log d)^{K'}$
for $\btheta^\top \proj_{S,\perp} \bh_1(\bx)$ by the bound $\|\btheta\|_2 \prec
1$. For $Q_{\text{pure}}(\bx_{\setminus S})$, write
$Q_{\text{pure}}(\bx_{\setminus S})=I_2(\bT_{2, \setminus S})$
where
\[\bT_{2, \setminus S}=\frac{1}{\sqrt{2}}\sum_{i=1}^p \theta_i \bw_{i\setminus
S}^{\otimes 2}, \qquad
\bT_{2, \setminus S} \otimes_1 \bT_{2, \setminus S}
=\frac{1}{2}\sum_{i=1}^p \|\bw_{i\setminus S}\|_2^2 \cdot
\theta_i\bw_{i\setminus S}\bw_{i\setminus S}^\top.\]
Then by Lemma~\ref{section:ProofOfLemmaGaussianChaosConvergenceToStandardGaussian}, we have
\begin{equation}
    \E [Q_{\text{pure}}(\bx_{\setminus S})^4] - 3\left(\E[ Q_{\text{pure}}(\bx_{\setminus S})^2 ]\right)^2 \le C\|\bT_{2, \setminus S} \otimes_1 \bT_{2, \setminus S}\|_F^2
\leq C'\left\| \bW_{\backslash S}^\sT \bD_{\btheta} \bW_{\backslash S}\right\|_{\op}^2.
\end{equation}
Now applying the condition
$\|\bW_{\backslash S}^{\sT}\bD_{\btheta}\,\bW_{\backslash S}\|_{\op}
\leq d^{-\eps}$ for
$\btheta \in \bTheta_{\bW}^{\sCG}(\eps,K)$,
this verifies \eqref{eq:C2_assump} with $\rho=Cd^{-2\eps}$
for $Q_{\text{pure}}(\bx_{\setminus S})$. Thus by
Corollary~\ref{corollary:RecallBoundedLipschitzFunctionSupErgetPartial},
we may replace $\proj_{S,\perp}^\text{pure}\bh_2(\bx) \equiv \bh_{2,\setminus
S}$ by $\bg_{k,\setminus S}$ in $\btheta^\top\bz^\sPG$ and $\bbeta_{2,i}^\top
\bh_2(\bx)$, establishing
an analogue of Theorem \ref{theorem:RecallBoundedLipschitzFunctionSupErgettt} for $\varphi \in
\tilde \cL$. The result for the Lipschitz class $\varphi \in \cL$ then follows
from the same smoothing argument as before.
\end{proof}

\clearpage

\section{Lindeberg Phase I}\label{sec:convergence-optimizersspg}

We set up notations for the first phase of our Lindeberg swapping argument,
which interpolates between the original random features model and the Partial
Gaussian Equivalent (PGE) model by replacing Hermite features of degrees
$k \ge 3$ with independent Gaussian noise.

Recall that $\{(\bz_i^\RF,y_i^\RF)\}_{i=1}^n$ denote the features and
labels in the RF model, given by
\begin{align*}
\bz_i^\RF&=\mu_0\1_p+\sum_{k=1}^D \mu_k\bV_k\bh_k(\bx_i),\\
y_i^\RF&=\eta\big(\underbrace{\bx_{iS},\bbeta_2^\sT \bh_2(\bx_i),
\bbeta_3^\sT \bh_3(\bx_i),\ldots \bbeta_{D'}^\sT \bh_{D'}(\bx_i)}_{=f_*(\bx_i) =f_i^\RF};\eps_i\big).
\end{align*}
Here $\bx_{iS}$ are the first $s$ coordinates of 
$\bx_i \in \R^d$, and each $\bbeta_k=[\bbeta_{k1},\ldots,\bbeta_{ks_k}] \in
\R^{B_{d,k} \times s_k}$ for $k=2,\ldots,D'$ represents $s_k$
projections of the degree-$k$ Hermite features $\bh_k(\bx_i)$.

We will denote by $\{(\bz_i^\sPG,y_i^\sPG)\}_{i=1}^n$ the features and target
function in the PGE model. Recall that these are given correspondingly by 
\begin{align*}
\bz_i^\sPG&=\mu_0\1_p+\mu_1\bV_1\bh_1(\bx_i)+\mu_2\bV_2\bh_2(\bx_i)
+\mu_{>2}\bg_{i*},\\
y_i^\sPG&=\eta\big(\underbrace{\bx_{iS},\bbeta_2^\sT \bh_2(\bx_i),
\bbeta_3^\sT \bg_{i3},\ldots,\bbeta_{D'}^\sT \bg_{iD'}}_{=f_i^\sPG};\eps_i).
\end{align*}
Here $\bg_{i3},\ldots,\bg_{iD'},\bg_{i*}$ are independent standard Gaussian
vectors (independent also across samples $i=1,\ldots,n$)
where $\bg_{ik} \in \R^{B_{d,k}}$ and $\bg_{i*} \in \R^p$.

For each sample index $q \in [n]$, we define three interlinked
optimization objectives:

\paragraph{(i) Leave-one-out (LOO) objective.}
Let
\[\bZ_{\setminus q}=\bigl[\bz_1^\RF,\ldots,\bz_{q-1}^\RF,\bzero,
\bz_{q+1}^\sPG,\ldots,\bz_n^\sPG\bigr], \qquad
\by_{\setminus q}=\bigl[y_1^\RF,\ldots,y_{q-1}^\RF,
0,y_{q+1}^\sPG,\ldots,y_n^\sPG\bigr],\]
and denote by $(\bz_i,y_i)$ the columns/entries of $(\bZ_{\setminus
q},\by_{\setminus q})$.
The LOO empirical risk, minimizer, and minimum risk value are defined as
\begin{align*}
\widehat{\cR}_{\setminus q}(\btheta)
&=\frac{1}{n}\sum_{i=1}^n \ell(y_i,\<\btheta,\bz_i\>)
+ \frac{\lambda}{2}\|\btheta\|_2^2+\btau\cdot \bGamma^\bW (\btheta),\\
\hat\btheta_{\setminus q}&=\argmin_{\btheta\in\R^p}
  \widehat{\cR}_{\setminus q}(\btheta),
\qquad \Phi_{\setminus q}=\widehat{\cR}_{\setminus q}(\hat\btheta_{\setminus
q}).
\end{align*}

\paragraph{(ii) Augmented objective.}
Let $(\tilde \bz_q,\tilde y_q)$ be either $(\bz_q^\RF,y_q^\RF)$
or $(\bz_q^\sPG,y_q^\sPG)$, let
\[\bZ_{\cup q}=\bigl[\bz_1^\RF,\ldots,\bz_{q-1}^\RF,\tilde
\bz_q,\bz_{q+1}^\sPG,\ldots,\bz_n^\sPG\bigr], \qquad
\by_{\cup q}=\bigl[y_1^\RF,\ldots,y_{q-1}^\RF,\tilde y_q,y_{q+1}^\sPG,
\ldots,y_n^\sPG\bigr],\]
and denote by $(\bz_i,y_i)$ the columns/entries of
$(\bZ_{\cup q},\by_{\cup q})$.
The augmented empirical risk, minimizer, and minimum risk value are defined as
\begin{align}
\widehat{\cR}_{\cup q}(\btheta)
&=\frac{1}{n}\sum_{i=1}^n \ell(y_i,\<\btheta,\bz_i\>)
+ \frac{\lambda}{2}\|\btheta\|_2^2+\btau\cdot \bGamma^\bW (\btheta) =\widehat{\cR}_{\setminus q}(\btheta)
+\frac{1}{n}\big[\ell\bigl(\tilde y_q,\,\<\btheta,\tilde\bz_q\>\bigr)
-\ell(0,0)\big],
\label{eq:augmented-objective}\\
\hat\btheta_{\cup q}&=\argmin_{\btheta\in\R^p}
\widehat{\cR}_{\cup q}(\btheta),
\qquad
\Phi_q=\widehat{\cR}_{\cup q}(\hat\btheta_{\cup q}).
\end{align}
Here, $\hcR_{\cup q}(\btheta) \equiv \hcR_{\cup q}(\btheta;\tilde \bz_q,\tilde
y_q)$, $\hat\btheta_{\cup q} \equiv \hat\btheta_{\cup q}(\tilde \bz_q,\tilde
y_q)$ and $\Phi_q \equiv \Phi_{\cup q}(\tilde \bz_q,\tilde y_q)$ depend
also on the choice of the re-inserted sample $(\tilde \bz_q,\tilde y_q)$,
although we will suppress this dependence in the notation.

\paragraph{(iii) Quadratic surrogate.}
Let $\hat\btheta_{\setminus q}$ be the minimizer of the above LOO objective, and
let $\bH_{\setminus q}$ be the Hessian of this
LOO objective evaluated at its minimizer:
\begin{align}\label{eq:familyI-Hessian}
  \bH_{\setminus q}=\nabla^2
  \widehat{\cR}_{\setminus q}(\hat \btheta_{\setminus q}).
\end{align}
We define the quadratic surrogate objective, its minimizer, and minimum risk as
\begin{align}
\widetilde\cR_{\cup q}(\btheta)
&=\widehat{\cR}_{\setminus q}\bigl(\hat\btheta_{\setminus q}\bigr)
+\frac{1}{n}\big[\ell\bigl(\tilde y_q,\,\<\btheta,\tilde\bz_q\>\bigr)
-\ell(0,0)\big]
+ \frac12\,(\btheta-\hat\btheta_{\setminus q})^{\sT}
  \bH_{\setminus q}\,(\btheta-\hat\btheta_{\setminus
q}),\label{eq:familyI-quadratic}\\
\tilde\btheta_{\cup q}
&=\argmin_{\btheta\in\R^p}
\widetilde \cR_{\cup q}(\btheta),
\qquad
\Psi_q=\widetilde \cR_{\cup q}(\tilde\btheta_{\cup q}),
\end{align}
again suppressing the notational dependence
on the choice of $(\tilde \bz_q,\tilde y_q)$.\\

We remark that the minimizers $\hat \btheta_{\setminus q}$ and
$\hat\btheta_{\cup q}$ must exist because $\hcR_{\setminus q}$ and
$\hcR_{\cup q}$ are continuous and bounded below by
$(\lambda/2)\|\btheta\|_2^2-2$. Lemma \ref{lemma:optinThetaPG}
to follow will verify that they are in fact unique on a
high-probability event where $\hcR_{\setminus q}$ and $\hcR_{\cup q}$ are
strongly convex on a subset of $\R^p$ containing these minimizers.
On the complementary event, if the minimizer $\hat \btheta_{\setminus q}$ is not
unique, $\widetilde \cR_{\cup q}$ may be defined by any choice of such a
minimizer, and the choice is inconsequential for our arguments.

\subsection{Geometry of sub-level sets of the empirical
risk}\label{sec:sublevelsets}

For each $q=1,\ldots,n$, let us define a set $\widehat S_q(K)$ by
\[\widehat S_q(K)=\left\{\btheta:\frac{1}{n}
\sum_{i:i \neq q} \ell(y_i,\<\btheta,\bz_i\>)
+\frac{\lambda}{2}\|\btheta\|_2^2<(\log d)^K\right\}\]
where $\{(\bz_i,y_i)\}_{i \neq q}$ are the columns/entries of
$(\bZ_{\setminus q},\by_{\setminus q})$ (or equivalently, of
$(\bZ_{\cup q},\by_{\cup q})$). This is a sub-level set of the empirical risk
removing the $q^\text{th}$ sample and the perturbation by $\Gamma^\bW(\btheta)$.

 Corresponding to each LOO optimization problem
$(\hcR,\bZ,\by)=(\hcR_{\setminus q},\bZ_{\setminus q},\by_{\setminus q})$
or augmented optimization problem
$(\hcR,\bZ,\by) \equiv (\hcR_{\cup q},\bZ_{\cup q},\by_{\cup q})$,
let us denote by $\bH \equiv \bH_{\setminus q}$
or $\bH \equiv \bH_{\cup q}$ the Hessian of the empirical risk, i.e.\
\[\bH(\btheta)=\nabla^2 \hcR(\btheta),\]
Let
\begin{equation}\label{eq:Vcombined}
\bV=\begin{pmatrix} \mu_0 \bV_0 & \mu_1 \bV_1 & \ldots & \mu_D \bV_D
\end{pmatrix}
\end{equation}
be the matrix defining the random features, where
\[\bV_0=\1_p, \qquad \bV_1=\bW, \qquad \bV_k=[\bq_k(\bw_1),\ldots,\bq_k(\bw_p)]^\sT \in \R^{p \times
B_{d,k}}.\]

The goal of this section is to prove the following structural lemma
about the above sub-level sets $\widehat S_q$ and the geometries of the LOO and
augmented risks restricted to these sets. In particular,
part (c) establishes an important property for our subsequent
analyses: Although $\bV$ exhibits large ``spike'' singular values
corresponding to a low-rank subspace spanned by
$(\bV_0,\bV_1)=(\1_p,\bW) \in \R^{p \times (d+1)}$,
the Hessian of the empirical risk effectively regularizes these spikes.

\begin{lemma}\label{lemma:sublevelgeometry}
For any constants $C,K>0$, there exists $K'>0$ such that with probability at least $1-d^{-C}$, the following holds: For each $q=1,\ldots,n$, we have that
\begin{enumerate}[label=(\alph*)]
\item For all $\btheta \in \widehat S_q(K)$,
\[\|\btheta\|_2 \leq (\log d)^{K'},
\qquad \|\mu_k \btheta^\sT \bV_k\|_2 \leq (\log d)^{K'} \text{ for each }
k=0,1,\ldots,D.\]
\item Letting $L_\bW(\btheta)$ be the test loss in \eqref{eq:test_loss_perturbation_def},
for all $\btheta \in \widehat S_q(K)$,
\[|L_\bW(\btheta)| \leq (\log d)^{K'},
\qquad
{-}(\log d)^{K'}\bV\bV^\sT \preceq \nabla^2 L_\bW(\btheta) \preceq 
(\log d)^{K'}\bV\bV^\sT.\]
In particular, if $K_{\Gamma}>K'$, then $\Gamma^\bW(\btheta)=(\| \bV_+^\sT \btheta\|_2^2, L_\bW(\btheta))$
for all $\btheta \in \widehat S_q(K)$.
\item Suppose $K_\Gamma>K'$, $0<\tau_1 \leq 1/(\log d)^{K_\Gamma}$, and $|\tau_2| \leq \tau_1/(\log d)^{K_\Gamma}$. For any choice of
$(\tilde \bz_q,\tilde y_q) \in \{(\bz_q^\RF,y_q^\RF),(\bz_q^\sPG,y_q^\sPG)\}$,
let $\bH \equiv \bH_{\setminus q}$
or $\bH \equiv \bH_{\cup q}$ be the corresponding Hessian. Then for all $\btheta \in \widehat S_q(K)$,
\[\|\bV^\sT \bH^{-1}(\btheta)\bV\|_\op \leq 1/\tau_1
\quad\text{ and }\quad 
\bH(\btheta) \succeq \tau_1 \bV\bV^\sT
+\frac{\lambda}{2}\,\bI.\]
\end{enumerate}
\end{lemma}

The remainder of this section proves
Lemma \ref{lemma:sublevelgeometry}. Throughout this proof,
we abbreviate $\widehat S(K) \equiv \widehat S_q(K)$, and
$(\widehat \cR,\bH,\bZ,\by)$ refers to a choice of
$(\widehat \cR_{\setminus q},\bH_{\setminus q},\bZ_{\setminus q},\by_{\setminus q})$
or $(\widehat \cR_{\cup q},\bH_{\cup q},\bZ_{\cup q},\by_{\cup q})$.
The proof is similar for all such choices, so to ease notation, we will often
focus on the setting of the original random features model
(i.e.\ $(y_i,\bz_i)=(y_i^\RF,\bz_i^\RF)$ for all $i=1,\ldots,n$), and explain
within the arguments any modifications that are needed to treat the other cases.

\begin{lemma}\label{lemma:predictionbound}
For any constants $C,K>0$, there exist $c_0,K'>0$ such that
with probability at least $1-d^{-C}$, every $\btheta \in \widehat S(K)$
satisfies
\begin{align}
\frac{1}{n}\sum_{i=1}^n \1\{|\<\btheta,\bz_i\>| \leq (\log d)^{K'}\}
&\geq c_0,\label{eq:predictionbound}
\end{align}
\end{lemma}
\begin{proof}
In the regression setting of Assumption \ref{assumption:polynomial-growth}(i), suppose
$\btheta \in \widehat S(K)$, so
\[\frac{1}{n}\sum_{i:i \neq q} \ell(y_i,\<\btheta,\bz_i\>)
\leq \frac{1}{n}\sum_{i:i \neq q} \ell(y_i,\<\btheta,\bz_i\>)
+\frac{\lambda}{2}\|\btheta\|_2^2<(\log d)^K.\]
The calibrated growth condition of
Assumption~\ref{assumption:polynomial-growth}(i) for $\ell(\cdot)$ then implies
\[
\begin{aligned}(\log d)^K>&~\frac{\sc_2}{n}\sum_{i:i \neq q} |\<\btheta,\bz_i\>-y_i|^{\sr_2} \1 \left\{|\<\btheta,\bz_i\>-y_i| \geq \sC_2 \right\} \\
\geq&~ \frac{\sc_2(\log d)^{2K}}{n}\sum_{i:i \neq q} \1\bigl\{|\<\btheta,\bz_i\>
-y_i|\ge (\log d)^{2K/\sr_2}\bigr\},
\end{aligned}\]
so
\begin{equation}\label{eq:proof_sharp_bound}
\frac{1}{n}\sum_{i:i \neq q} \1\bigl\{|\<\btheta,\bz_i\>|
\geq |y_i|+(\log d)^{2K/\sr_2}\bigr\}
\leq \frac{1}{\sc_2(\log d)^K}.
\end{equation}
By Lemma \ref{lemma:yzbounds}, we have $\|\by\|_\infty \prec 1$.
This means that for any constant $C>0$, there exists $K''>0$ such
that $\|\by\|_\infty \leq (\log d)^{K''}$ with probability $1-d^{-C}$.
On this event, the bound \eqref{eq:proof_sharp_bound}
implies that \eqref{eq:predictionbound}
holds for some constants $c_0,K'>0$ depending on $K,K'',\sc_2,\sr_2$.

Let us consider now the case of binary classification in Assumption \ref{assumption:polynomial-growth}. Let us decompose $\bV_2=\bV_{2c}+\frac{1}{d}\1_p\be_c^\top$ as in Corollary \ref{corollary:OperatorNormF2cdecompose}, and define
\[
R_i := \btheta^\sT \bz_i - \mu_0 \bones_p^\sT \btheta - \mu_1 \bx_i^\sT \bW^\sT \btheta -\frac{\mu_2}{d} (\be_c^\sT \bh_2 (\bx_i) ) \bones_p^\sT \btheta = \mu_2 \btheta^\sT \bV_{2c} \bh_2 (\bx_i) + \sum_{k=3}^{D'} \btheta^\sT \He_k (\bW \bx_i).
\]
By Lemmas \ref{lemma:BernsteinInequalityHermiteFeatures} and \ref{lemma:OperatorNormZkFk} and Corollary \ref{corollary:OperatorNormF2cdecompose}, we have
\[
\| \bR \|_2^2 \leq \left(\mu_2^2 \|\bV_{2c}\|_\op^2 \| \bh_2 (\bX) \|_\op^2 + \sum_{k =3}^{D} \mu_k^2 \| \He_k (\bW\bX^\sT) \|_\op^2\right) \| \btheta \|_2^2 \prec n\| \btheta \|_2^2.
\]
Thus, for any $C>0$, there exist $K_0,K_0'>0$ such that with probability at least $1 - d^{-C}$, for all $\btheta \in \widehat S(K)$,
\[
\frac{1}{n} \sum_{i = 1}^n \1\bigl\{ |R_i | > (\log d)^{K_0}\bigr\} \leq\frac{1}{n} \sum_{i = 1}^n \frac{|R_i|^2}{(\log d)^{2K_0}} \leq \frac{1}{(\log d)^{K_0'}}.
\]
Similarly, by Lemma \ref{lemma:BernsteinInequalityHermiteFeatures}, we have
\[
\| \mu_1 \bW^\sT \bX \|_\op^2 \prec n d^{1/2}, \qquad \| \bh_2 (\bX)^\sT \be_c\|_2^2 \prec d n,
\]
so that for some $K_1,K_1',K_2,K_2'>0$, with probability at least $1 - d^{-C}$, for all $\btheta \in \widehat S(K)$,
\[
\begin{aligned}
\frac{1}{n} \sum_{i = 1}^n \1\bigl\{ | \mu_1 \btheta^\sT \bW \bx_i | > \sqrt{d} (\log d)^{K_1}\bigr\} \leq&~ \frac{1}{(\log d)^{K_1'}}, \\ \frac{1}{n} \sum_{i = 1}^n \1\bigl\{ | \mu_2\be_c^\sT \bh_2(\bx_i) /d | > d^{-1/2} (\log d)^{K_2}\bigr\} \leq&~ \frac{1}{(\log d)^{K_2'}}.
\end{aligned}
\]

We restrict to the high-probability event where the above statements as well as the conclusion of Lemma \ref{lem:linear_predictor_misclassification} holds.
Consider any $\btheta \in \widehat{S} (K)$, fix a sufficiently large constant $K'>0$, and first suppose by contradiction that $| \mu_0 \bones_p^\sT \btheta | \geq \sqrt{d} (\log d)^{K'}$. Then for $K'>0$ large enough,
\begin{equation}\label{eq:thetazcloseto1theta}
|\btheta^\sT \bz_i - \mu_0 \bones_p^\sT \btheta | \leq |R_i| + |\mu_2\be_c^\sT \bh_2(\bx_i)/d| |\bones^\sT_p \btheta| + |\mu_1 \btheta^\sT \bW\bx_i | \leq \frac{1}{2}| \mu_0 \bones_p^\sT \btheta | 
\end{equation}
for at least a fraction $1-o_d(1)$ of the samples $i=1,\ldots,n$. On these samples, for $K'>0$ large enough, we must have $\sign (\btheta^\sT \bz_i) = \sign(\mu_0\bones_p^\sT \btheta ) = \sign (\mu_0 \bones_p^\sT \btheta + \mu_1 \bx_i \bW^\sT \btheta)$. Then by Lemma \ref{lem:linear_predictor_misclassification},
\begin{equation}\label{eq:missclassifedpoints}
\frac{1}{n} \sum_{i=1}^n \1 \{y_i={-}\sign(\<\btheta, \bz_i \>)\} \geq \frac{c_0}{2}.
\end{equation}
Using the calibrated growth assumption (Assumption \ref{assumption:polynomial-growth}(ii)), we deduce from \eqref{eq:thetazcloseto1theta} and \eqref{eq:missclassifedpoints} that
\[
\frac{1}{n} \sum_{i = 1}^n \ell (y_i, \<\btheta, \bz_i \>) \geq \frac{1}{n} \sum_{i = 1}^n \sc_2 ( - y_i \<\btheta, \bz_i \> )_+^{\sr_2} \geq c'|\mu_0 \bones_p^\sT \btheta|^{\sr_2} \geq  c'(\sqrt{d} (\log d)^{K'})^{\sr_2}
\]
for a constant $c'>0$ depending on $c_0,\sc_2,\sr_2$.
This contradicts $\btheta \in \widehat{S}(K)$ for large enough $K'>0$.

Thus, we must have $|\mu_0 \bones_p^\sT \btheta | \leq \sqrt{d} (\log d)^{K'}$. In this case, for some constant $K''>0$ and for a fraction $1-o_d(1)$ of samples $i=1,\ldots,n$, we have
\[
|\btheta^\sT \bz_i - \mu_0 \bones_p^\sT \btheta - \mu_1 \bx_i^\sT \bW^\sT \btheta |
\leq |R_i|+|\mu_2 \be_c^\top \bh_2(\bx_i)/d||\1_p^\top\btheta| \leq (\log d)^{K''}.
\]
Let $c_0>0$ be the constant
in Lemma \ref{lem:linear_predictor_misclassification}, and suppose first that there are fewer than
$n(1-c_0/2)$ samples
where $|\mu_0 \bones_p^\sT \btheta+\mu_1 \bx_i^\sT \bW^\sT \btheta|>2 (\log d)^{K''}$. Then we are done, because $nc_0/2-o_d(1)$ remaining samples then satisfy
$|\btheta^\sT \bz_i|
\leq |\btheta^\sT \bz_i - \mu_0 \bones_p^\sT \btheta - \mu_1 \bx_i^\sT \bW^\sT \btheta |
+| \mu_0 \bones_p^\sT \btheta + \mu_1 \bx_i^\sT \bW^\sT \btheta | \leq 3(\log d)^{K''}$.
On the other hand, if at least $n(1-c_0/2)$ samples satisfy
$|\mu_0 \bones_p^\sT \btheta+\mu_1 \bx_i^\sT \bW^\sT \btheta|>2 (\log d)^{K''}$, then for these samples we have
$\sign(\<\btheta,\bz_i\>)=\sign(\mu_0 \bones_p^\sT \btheta+\mu_1 \bx_i^\sT \bW^\sT \btheta)$.
Then Lemma \ref{lem:linear_predictor_misclassification} implies that at least $nc_0/2$ samples have $y_i={-}\sign(\<\btheta,\bz_i\>)$. Letting $\cI$ be this set of $nc_0/2$ misclassified samples, we have
\[
(\log d)^K \geq \frac{1}{n} \sum_{i=1}^n \ell (y_i, \<\btheta, \bz_i \>) \geq \frac{1}{n} \sum_{i \in \cI} \sc_2 |\<\btheta,\bz_i\>|^{\sr_2},
\]
and hence at least $nc_0/4$ of the samples of $\cI$ have $|\<\btheta,\bz_i\>| \leq ((4/\sc_2)(\log d)^K)^{1/\sr_2}$, concluding the proof.
\end{proof}

\begin{proof}[Proof of Lemma \ref{lemma:sublevelgeometry}(a)]
If $\btheta \in \widehat S(K)$, then (since $\ell(\cdot) \geq 0$) we have
$(\lambda/2)\|\btheta\|_2^2<(\log d)^K$.
Since $\lambda \succ 1$ by Assumption \ref{assumption:lambda}, this shows
\[\|\btheta\|_2 \prec 1 \text{ simultaneously over }\btheta \in \widehat S(K).\]

To show the bounds on $\mu_0\btheta^\sT \1_p$ and $\mu_1\btheta^\sT \bW$,
consider first the case of the original random features model
$\bZ=[\bz_1^\RF,\ldots,\bz_n^\RF]$. Let us expand
\[\btheta^\sT \bz_i^\RF
=\mu_{0} \btheta^{\sT}\1_p + \mu_{1} \btheta^{\sT}\bW \bx_i + 
\sum_{k=2}^D \mu_k \btheta^\sT \bV_k\bh_k(\bx_i)\]
For the $k=2$ term, we further
apply Corollary \ref{corollary:OperatorNormF2cdecompose} to
decompose $\bV_2 = \bV_{2c} + \frac{1}{d} \1_p\be_c^\sT$. Then, denoting
$\bh_k(\bX)=[\bh_k(\bx_1),\ldots,\bh_k(\bx_n)] \in \R^{B_{d,k} \times n}$
and $\hat \by=\btheta^\sT \bZ \in \R^{1 \times n}$, we have
\[\hat\by=\mu_0 \btheta^\sT \1_p \1_n^\sT+\mu_1\btheta^\sT \bW\bX
+\mu_2\btheta^\sT \bV_{2c}\bh_2(\bX)
+\frac{\mu_2}{d}\btheta^\sT \1_p \be_c^\sT \bh_2(\bX)
+\sum_{k=3}^D \mu_k \btheta^\sT \bV_k\bh_k(\bX).\]
Then applying $\|\btheta\|_2 \prec 1$ shown above,
$\|\bV_{2c}\|_\op \prec 1$ and $\|\be_c\|_2=\sqrt{d}$ by Corollary
\ref{corollary:OperatorNormF2cdecompose}, $\|\bh_2(\bX)\|_\op \prec d$ by
Lemma \ref{lemma:BernsteinInequalityHermiteFeatures},
$\|\bV_k\bh_k(\bX)\|_\op \prec d$ for $k \geq 3$ by Lemma \ref{lemma:OperatorNormZkFk},
and $|\mu_k| \prec 1$, we have simultaneously over
$\btheta \in \widehat S(K)$ that
\begin{equation}\label{eq:yhatbound}
\Big\|\hat\by-\mu_0 \btheta^\sT \1_p \1_n^\sT
-\mu_1\btheta^\sT \bW\bX\Big\|_2 \prec d+\sqrt{d}|\btheta^\sT \1_p|.
\end{equation}
The same bound holds for $\hat\by=\btheta^\sT \bZ$ and each case
of $\bZ \in \{\bZ_{\setminus q},\bZ_{\cup q}\}$, by expanding also
\[\btheta^\sT \bz_i^\sPG
=\mu_{0} \btheta^{\sT}\1_p + \mu_{1} \btheta^{\sT}\bW \bx_i + 
\mu_{>2} \btheta^\sT \bg_{i*}\]
and applying the operator norm bound for a Gaussian matrix
$\|[\bg_{1*},\ldots,\bg_{n*}]\|_\op \prec \sqrt{n}+\sqrt{p} \prec d$.

For a given constant $C>0$, let $c_0,K'>0$ be the constants prescribed by
Lemma \ref{lemma:predictionbound}, let
$\widehat \cI \subseteq [n]$ be the (random) set of indices
for which $|\hat y_i|=|\<\btheta,\bz_i\>| \leq (\log d)^{K'}$, and set $m=\lfloor c_0n
\rfloor$. Then
\eqref{eq:predictionbound} of Lemma \ref{lemma:predictionbound} ensures
\begin{equation}\label{eq:hatIsize}
\P[|\widehat \cI| \geq m \text{ for every } \btheta \in \widehat S(K)]
\geq 1-d^{-C}.
\end{equation}
Let $\hat \by_{\widehat \cI}$ denote the coordinates of $\hat
\by=\btheta^\sT [\bz_1,\ldots,\bz_n]$ supported on $\widehat \cI$.
Then by definition of $\widehat \cI$, we must have
$\|\hat\by_{\widehat \cI}\|_2 \leq (\log d)^{K'}\sqrt{n}$.
Then \eqref{eq:yhatbound} implies that simultaneously over $\btheta \in \widehat
S(K)$,
\begin{align}\label{equation:fallin3gauss}
    \left\|\begin{pmatrix}
        \mu_0 \btheta^\sT \1_p & \mu_{1}\btheta^\sT \bW
    \end{pmatrix} \bA(\widehat \cI)
\right\|_2 \prec d+\sqrt{d}|\btheta^\sT \1_p|,
\end{align}
where we define the matrix
\begin{align}\label{equation:MatrixAATsigma}
\bA(\cI)=\begin{pmatrix} 1 & \cdots & 1 \\ \bx_{i_1} & \cdots & \bx_{i_t}
\end{pmatrix} \text{ for } \cI=\{i_1,\ldots,i_t\}.
\end{align}
Let us bound the smallest singular value of $\bA(\widehat \cI)$:
Letting $m=\lfloor c_0 n \rfloor$ be as above, fix any
unit vector $\bw=(u,\bv) \in \S^d \subset \R^{d+1}$
and any $\delta>0$, and observe that
\begin{align}
	 &\mathbb{P}\left(  \bw^{\sT}\left(\frac{1}{n}
\bA(\widehat \cI)\bA(\widehat \cI)^\sT\right) \bw \leq
\delta \text{ and }
|\widehat \cI| \geq m \text{ for all } \btheta \in \widehat S(K)\right)\\
    &\leq \mathbb{P}\left(\inf_{\mathcal{I} \subset[n]:|\mathcal{I}|=m}
\bw^{\sT}\left(\frac{1}{n}  \bA(\mathcal{I})\bA(\mathcal{I})^\sT\right) \bw \leq \delta\right)
    =\mathbb{P}\left(\inf_{\mathcal{I} \subset[n]:|\mathcal{I}|=m}
 \sum_{i \in \mathcal{I}} (u+\<\bx_i,\bv\>)^2
\leq n\delta\right)\nonumber\\
    & \leq \mathbb{P}\left(\left|\left\{i: (u+\<\bx_i,\bv\>)^2\leq
2n\delta/m\right\}\right|\geq m/2\right).\label{equ:unif} 
\end{align}
The final inequality \eqref{equ:unif} holds because if there are fewer than
$m/2$ values of $\{(u+\<\bx_i,\bv\>)^2\}_{i=1}^n$ which are
$\leq 2n\delta/m$, then the $m/2$-th to $m$-th smallest values of
$\{(u+\<\bx_i,\bv\>)^2\}_{i=1}^n$ must all exceed $2n\delta/m$, implying their
sum exceeds $n\delta$.

Let us choose $\delta=(\log d)^{-K'}$. For any fixed unit vector
$\bw=(u,\bv)$, we have $u+\<\bx_i,\bv\> \sim \normal(u,1-u^2)$. Considering
separately the cases $|u| \leq 1/2$ and $|u|>1/2$, one may check that this
normal law has density upper bounded by a constant $C_0>0$ over
$[-\sqrt{2n\delta/m},\sqrt{2n\delta/m}]$. Hence
\begin{equation}\label{eq:anticoncentration}
\mathbb{P}\left((u+\<\bx_i,\bv\>)^2 \leq 2n\delta/m \right) \leq
2C_0\sqrt{2n\delta/m}.
\end{equation}
For the above choices $m=\lfloor c_0n \rfloor$ and
$\delta=(\log d)^{-K'}$, we have
$n \cdot 2C_0\sqrt{2n\delta/m} \leq m/4$ for all large $n,d$, so by
Hoeffding's inequality for a binomial random variable,
$\eqref{equ:unif}\leq e^{-m^2/(8n)}$.
To extend this to a union bound over all $\bw \in \S^d$, we
apply a covering net argument:
Observe that by a standard bound on the operator norm of Gaussian matrices,
$\|\frac{1}{n}\sum_{i=1}^n \bx_i\bx_i^\sT\|_\op \leq 2$ with probability
$1-2e^{-cn}$ and some constant $c>0$. On this event, we have
$\|\frac{1}{n}\bA(\cI)\bA(\cI)^\sT\|_\op \leq 6$ for every $\cI \subseteq [n]$.
Then, letting
$\mathcal{N}_\eps$ be a $\eps$-net of the sphere $\S^d \subset \R^{d+1}$
with $\eps=\delta/24$, we have
\begin{align*}
&\mathbb{P}\left(\inf_{\bw \in \S^d} \bw^{\sT}\left(\frac{1}{n}
\bA(\widehat \cI)\bA(\widehat \cI)^\sT\right) \bw \leq
\delta/2 \text{ and } |\widehat \cI| \geq m \text{ for all } \btheta \in
\widehat S(K)\right)\\
&\leq \mathbb{P}\left(\inf_{\bw \in \mathcal{N}_\eps} \bw^{\sT}\left(\frac{1}{n}
\bA(\widehat \cI)\bA(\widehat \cI)^\sT\right) \bw \leq
\delta \text{ and } |\widehat \cI| \geq m \text{ for all } \btheta \in \widehat
S(K)\right)+2e^{-cn}\\
&\leq |\mathcal{N}_\eps| \cdot e^{-m^2/(8n)}+2e^{-cn}.
\end{align*}
Since $n \asymp d^2$ and $m \asymp n$, we may choose $\mathcal{N}_\eps$ so that
$\log |\mathcal{N}_\eps| \lesssim d\log(1/\eps) 
\asymp d\log \log d \ll m^2/n$. Thus, applying this bound together with
\eqref{eq:hatIsize}, for all large $d$ we have
\begin{equation}\label{eq:uniformlowersingularvalue}
\P\left(\inf_{\bw \in \S^d} \bw^{\sT}\left(\frac{1}{n}
\bA(\widehat \cI)\bA(\widehat \cI)^\sT\right) \bw \leq
\frac{1}{2(\log d)^{K'}} \text{ for all } \btheta \in \widehat S(K)\right) \leq 2d^{-C}.
\end{equation}
This shows that the smallest singular value of
$\bA(\widehat \cI)$ is bounded as
$\sigma_{\min}(\bA(\widehat \cI)) \succ n^{1/2} \succ d$ simultaneously over
$\btheta \in \widehat S(K)$, implying that
\begin{equation}\label{eq:unif_bound_lemma_c1_aa}
\left\|\begin{pmatrix}
        \mu_0 \btheta^\sT \1_p & \mu_{1}\btheta^\sT \bW
    \end{pmatrix}\right\|_2
\prec d^{-1} \left\|\begin{pmatrix}
        \mu_0 \btheta^\sT \1_p & \mu_{1}\btheta^\sT \bW
    \end{pmatrix} \bA(\widehat \cI)\right\|_2.
    \end{equation}
Combining with \eqref{equation:fallin3gauss} shows that
$|\mu_0\btheta^\sT \1_p| \prec 1$ and $\|\mu_1 \btheta^\sT \bW\|_2 \prec 1$.
For $k=2$, applying $\|\btheta\|_2 \prec 1$, $|\mu_2| \prec 1$, and
$\bV_2=\bV_{2c}+\frac{1}{d}\1_p\be_c^\sT$
where $\|\bV_{2c}\|_\op \prec 1$ by Corollary
\ref{corollary:OperatorNormF2cdecompose}, this shows also
$\|\mu_2\btheta^\sT \bV_2\|_\op \prec 1$. Similarly, for $k \geq 3$, applying
$\|\bV_k\|_\op \prec 1$ by Lemma \ref{lemma:OperatorNormFk}, we have
$\|\mu_k\btheta^\sT \bV_k\|_\op \prec 1$.
\end{proof}



\begin{proof}[Proof of Lemma \ref{lemma:sublevelgeometry}(b)]
Recall that $L_\bW(\btheta)=\E_{\bz^{\sPG},y^{\sPG}}
[\ell_{\test}(y^{\sPG},\langle\btheta,\bz^{\sPG}\rangle) \mid \bW]$.
By Assumption \ref{ass:test_loss} for $\ell_\test(\cdot)$, there exists a
constant $C_0>0$ such that
\begin{equation}\label{eq:testlossbound}
|L_\bW(\btheta)| \leq
\E_{\bz^{\sPG},y^{\sPG}}[C_0(1+|y^\sPG|^{C_0}+|\<\btheta,\bz^\sPG\>|^{C_0})
\mid \bW]
\end{equation}
Expanding $\bz^{\sPG}$, we have
$\langle\btheta,\bz^{\sPG}\rangle = \mu_0 \btheta^\sT 
\mathbf{1}_p + \mu_1 \btheta^\sT \bW\bx+
\mu_2 \btheta^\sT \bV_2\bh_2(\bx) +
\mu_{>2}\btheta^\sT \bg_*$
where $(\bx,\bg_*)$ are independent of $\bW$. Then the bounds of
Lemma \ref{lemma:sublevelgeometry}(a) imply, simultaneously over $\btheta \in
\widehat S(K)$,
\[\E_{\bz^{\sPG}}[\langle\btheta,\bz^{\sPG}\rangle^2 \mid \bW] \prec 1.\]
By Gaussian hypercontractivity over $(\bx,\bg_*)$, we then have
(simultaneously over $\btheta \in \widehat S(K)$)
\begin{equation}\label{eq:thetazPGbound}
|\langle\btheta,\bz^{\sPG}\rangle| \prec 1.
\end{equation}
Similarly, expanding $y^{\sPG}=\eta(\bx_S,\bbeta_2^\sT \bh_2(\bx),
\bbeta_3^\sT \bg_3,\ldots,\bbeta_{D'}^\sT \bg_{D'},\eps)$ and applying
Assumption \ref{assumption:target} for $\eta(\cdot)$ and
Assumption \ref{assumption:noise} for the noise $\eps$, we have
\begin{equation}\label{eq:yPGbound}
|y^{\sPG}| \prec 1.
\end{equation}
Applying these bounds to \eqref{eq:testlossbound} shows
$|L_\bW(\btheta)| \prec 1$ simultaneously over $\btheta \in \widehat S(K)$.

For the Hessian $\nabla^2 L_\bW(\btheta)
=\E_{\bz^{\sPG},y^{\sPG}}[\ell_{\test}''(y^{\sPG},\langle\btheta,\bz^{\sPG}\rangle)\bz^{\sPG}(\bz^{\sPG})^\sT
\mid \bW]$, we have for any unit vector $\bu \in \S^{p-1} \subset
\R^p$ that
\begin{align*}
|\bu^\sT \nabla^2 L_\bW(\btheta)\bu| &\leq
\E_{\bz^{\sPG},y^{\sPG}}[C_0(1+|y^\sPG|^{C_0}+|\<\btheta,\bz^\sPG\>|^{C_0})
\<\bu,\bz^{\sPG}\>^2 \mid \bW]\\
&\leq
C_0\,\E_{\bz^{\sPG},y^{\sPG}}[(1+|y^\sPG|^{C_0}+|\<\btheta,\bz^\sPG\>|^{C_0})^2
\mid \bW]^{1/2}\,\E_{\bz^{\sPG}}[\<\bu,\bz^{\sPG}\>^4 \mid \bW]^{1/2}.
\end{align*}
Expanding again $\bz^{\sPG}$, we have
\begin{align}\label{eq:uzPGsquared}
\E_{\bz^{\sPG}}[\langle\bu,\bz^{\sPG}\rangle^2 \mid \bW]
&=\mu_0^2(\bu^\sT \1_p)^2+\mu_1^2\|\bu^\sT \bW\|_2^2
+\mu_2^2\|\bu^\sT \bV_2\|_2^2+\mu_{>2}^2\|\bu\|_2^2.
\end{align}
Recalling that $\mu_{>2}^2=\mu_3^2+\ldots+\mu_D^2$ and
$\bV=\begin{pmatrix} \mu_0\1_p & \mu_1\bW & \cdots & \mu_D\bV_D \end{pmatrix}$,
Lemma \ref{lemma:OperatorNormFk} implies that
simultaneously over all $\bu \in \S^{p-1}$,
\[\Big|\Big(\mu_0^2+\frac{3}{d^2}\mu_4^2\Big)(\bu^\sT
\1_p)^2+\Big(\mu_1^2+\frac{3}{d}\mu_3^2\Big)\|\bu^\sT \bW\|_2^2
+\mu_2^2\|\bu^\sT \bV_2\|_2^2+\mu_{>2}^2 \|\bu\|_2^2
-\|\bu^\sT \bV\|_2^2\Big| \prec d^{-1/2}\mu_{>2}^2.\]
By the conditions of Assumption \ref{assumption:activation}, we have
$\mu_4^2/d^2 \prec \mu_0^2$ and $\mu_3^2/d \prec \mu_1^2$, and hence
applying this to \eqref{eq:uzPGsquared} shows
$\E_{\bz^{\sPG}}[\langle\bu,\bz^{\sPG}\rangle^2 \mid \bW] \prec \|\bu^\sT
\bV\|_2^2$, simultaneously over all $\bu \in \S^{p-1}$.
Then again by Gaussian hypercontractivity,
$\E_{\bz^{\sPG}}[\langle\bu,\bz^{\sPG}\rangle^4 \mid \bW]^{1/2}
\prec \|\bu^\sT \bV\|_2^2$.  Combining with the bounds 
\eqref{eq:thetazPGbound} and \eqref{eq:yPGbound}, this shows
$|\bu^\sT \nabla^2 L_\bW(\btheta)\bu| \prec \|\bu^\sT \bV\|_2^2$
simultaneously over $\btheta \in \widehat S(K)$ and $\bu \in \S^{p-1}$, i.e.\
for any $C>0$, there exists $K'>0$ such that with probability $1-d^{-C}$,
\[|\bu^\sT \nabla^2 L_\bW(\btheta)\bu| \leq (\log d)^{K'}\|\bu^\sT \bV\|_2^2
\text{ for all } \btheta \in \widehat S(K) \text{ and } \bu \in \S^{p-1}.\]
On this event, we have
${-}(\log d)^{K'}\bV\bV^\sT \preceq \nabla^2 L_\bW(\btheta)
\preceq (\log d)^{K'}\bV\bV^\sT$ for all $\btheta \in \widehat S(K)$.
\end{proof}

\begin{proof}[Proof of Lemma \ref{lemma:sublevelgeometry}(c)]
Lemma \ref{lemma:sublevelgeometry}(a) and (b) implies that for $K'>0$ sufficiently large and any $K_\Gamma>K'$, with probability at least $1 -d^{-C}$,
\[
\begin{aligned}
\bH (\btheta) \succeq \lambda \bI + \tau_1 \nabla^2 \Gamma_1^\bW (\btheta) + \tau_2 \nabla^2 \Gamma_2^\bW (\btheta) =&~ \lambda \bI + 2 \tau_1 \bV_+ \bV_+^\sT + \tau_2 \nabla^2 L_\bW (\btheta) \\
\succeq&~  \lambda \bI + 2 \tau_1 \bV_+ \bV_+^\sT - |\tau_2| (\log d)^{K'}  \bV \bV^\sT.
\end{aligned}
\]
Using Corollary \ref{corollary:OperatorNormF2cdecompose} and Lemma \ref{lemma:OperatorNormFk}, and Assumption \ref{assumption:activation}, we have also with probability at least $1 - d^{-C}$,
\[
 \bV \bV^\sT = \bV_+ \bV_+^\sT + \frac{\mu_2^2}{d} \bones_p \bones_p^\sT + \mu_2^2 \bV_{2c} \bV_{2c}^\sT + \sum_{k=3}^D \mu_k^2 \bV_k \bV_k^\sT \preceq (3/2) \bV_+ \bV_+^\sT + (\log d)^{K'/2} \bI,
\]
Thus for $K'>0$ sufficiently large and $K_\Gamma>K'$, using Assumption  \ref{assumption:lambda},
$\tau_1 \leq 1/(\log d)^{K_\Gamma}$, and
$|\tau_2|\leq \tau_1/(\log d)^{K_\Gamma}$, we deduce that
\[
\bH (\btheta) \succeq \frac{4}{3} \tau_1 \bV \bV^\sT + \frac{3\lambda}{4} \bI - |\tau_2| (\log d)^{K'}  \bV \bV^\sT \succeq  \tau_1 \bV \bV^\sT + \frac{\lambda}{2} \bI.
\]
This implies
$\|\bV^\sT \bH^{-1}(\btheta)\bV\|_\op \leq 1/\tau_1$, which concludes the proof.
\end{proof}

\subsection{Inclusion of the minimizers in $\Theta_W^{\sPG}$}
\label{app:localization_Phase_1}

We now show that each of the optimizers $\hat\btheta_{\setminus q}$ 
and $\hat\btheta_{\cup q}$ is uniquely defined and belongs
to $\bTheta_\bW^\sPG(K)$ as well as the sub-level set $\widehat S_q(K)$.

\begin{lemma}\label{lemma:optinThetaPG}
For any constant $C>0$, there exist constants $K,K_\Gamma>0$ such that
with probability at least $1-d^{-C}$, the following holds. For each $q=1,\ldots,n$, choice of
$(\tilde \bz_q,\tilde y_q) \in \{(\bz_q^\RF,y_q^\RF),(\bz_q^\sPG,y_q^\sPG)\}$,
and $\hcR$ given by either $\hcR_{\setminus q}$ or $\hcR_{\cup q}$, there exists
a unique minimizer $\hat\btheta$ of $\hcR$, which furthermore satisfies
\[\hat\btheta \in \bTheta_\bW^{\sPG}(K) \cap \widehat S_q(K).\]
\end{lemma}
\begin{proof}
We abbreviate $\widehat S(K) \equiv \widehat S_q(K)$.
By Lemma \ref{lemma:yzbounds} and Assumption \ref{assumption:loss} for
$\ell(\cdot)$, for any $C>0$, there exists some $K''>0$ such that
$\frac{1}{n}\sum_{i=1}^n |\ell(y_i,0)| \leq (\log d)^{K''}$
with probability $1-d^{-C}$. On this event, since
$|\btau \cdot \bGamma^\bW(\btheta)| \leq 2$ for all $\btheta \in \R^p$
and $\ell(\cdot) \geq 0$, we must have
\begin{align*}
\hcR(\bzero)&=\frac{1}{n}\sum_{i=1}^n \ell(y_i,0)
+\btau \cdot \bGamma^\bW(\bzero) \leq (\log d)^{K''}+2,\\
\hcR(\btheta)&=\frac{1}{n}\sum_{i=1}^n \ell(y_i,\<\btheta,\bz_i\>)
+\frac{\lambda}{2}\|\btheta\|_2^2+\btau \cdot \bGamma^\bW(\btheta)
\geq (\log d)^K-2
\text{ for all } \btheta \notin \widehat S(K).
\end{align*}
Choosing $K$ large enough,
this implies that 
\begin{equation}\label{eq:sublevelRSK}
\{\btheta:\hcR(\btheta) \leq \hcR(\bzero)\} \subseteq \widehat S(K).
\end{equation}
In the remainder of the proof, we fix this value of $K$, and
restrict to the event where \eqref{eq:sublevelRSK} holds and the statements of
Lemma \ref{lemma:sublevelgeometry} also hold for $\widehat S(K)$.
By \eqref{eq:sublevelRSK},
all minimizers of $\hcR$ belong to $\widehat S(K)$.
Since $\widehat S(K)$ is a convex set, Lemma
\ref{lemma:sublevelgeometry}(c) ensures that
$\hcR(\btheta)$ is strongly convex on $\widehat S(K)$,
so the minimizer $\hat\btheta$ is unique.
Since $\hat\btheta \in \widehat S(K)$, 
Lemma \ref{lemma:sublevelgeometry}(a) and (b) ensures that
\begin{equation}\label{eq:hatthetabasicbounds}
\|\hat\btheta\|_2 \leq (\log d)^{K'}
\qquad |\hat\btheta^\sT \1_p| \leq (\log d)^{K'}
\qquad \|\mu_1\hat\btheta^\sT \bW\|_2 \leq (\log d)^{K'}
\end{equation}
for a constant $K'>0$, and for $K_{\Gamma} > 2K'$, and we have $\Gamma^\bW(\btheta)=(\| \bV_+^\sT \btheta\|_2^2, L_\bW(\btheta))$. We will denote
\begin{equation}\label{eq:notation_btau_dot_nabla_gamma}
\btau \cdot \nabla^k \bGamma^\bW(\btheta) = \tau_1 \nabla^k \Gamma_1^\bW (\btheta) + \tau_2 \nabla^k \Gamma_2^\bW (\btheta).
\end{equation}
To check that $\hat\btheta \in \bTheta_\bW^{\sPG}(K')$,
it remains to bound $\|\hat\btheta\|_\infty$.

\paragraph{Step 1: Leave the last coordinate out.}
To bound the last coordinate $\hat \theta_p$ of $\hat\btheta$, let us
define
	\begin{equation}\label{equation:DefinitionOfTheta0constrain}
		\hat \btheta_{-p} := \argmin_{\btheta\in \R^p,\,\theta_p = 0} \frac{1}{n}\sum_{i=1}^{n}\ell(y_i, \btheta^\sT \bz_i )+ \frac{\lambda}{2} \| \btheta \|_2^2+\btau\cdot\bGamma^{\bW} (\btheta)
	\end{equation}
as the minimizer of $\hcR$ constrained to the linear subspace
subspace $\mathcal{C}=\{\btheta \in \R^p \mid \theta_p=0\}$. By
\eqref{eq:sublevelRSK} and the strong convexity of
$\hcR(\btheta)$ on $\widehat S(K)$, this minimizer
$\hat\btheta_{-p}$ is also uniquely defined, and
$\hat\btheta_{-p} \in \widehat S(K)$.
Note that $\hat\btheta_{-p}$ must satisfy the first-order optimality condition
$\langle \nabla \hcR(\hat\btheta_{-p}),\btheta_{-p}
\rangle=0$ for all $\btheta_{-p} \in \mathcal{C}$, implying
\begin{align}\label{equation:KKTConditionForOptimalSolutionReduced}
\frac{1}{n}\sum_{i=1}^n \ell'(y_i, \hat{\btheta}_{-p}^\sT \bz_i)
(\btheta_{-p}-\hat{\btheta}_{-p})^\sT \bz_i + \lambda
\hat\btheta_{-p}^\sT  (\btheta_{-p} - \hat\btheta_{-p})
+[\btau \cdot\nabla\bGamma^{\bW}
(\hat\btheta_{-p})]^\sT (\btheta_{-p}-\hat\btheta_{-p})=0.
\end{align}
For any $\btheta \in \widehat S(K)$, by a second-order Taylor expansion of
$\hcR(\btheta)$ around $\hat\btheta_{-p}$, we have
\begin{align}\label{equation:TaylorExpansionOfObjectiveFunction1}
    \hcR(\btheta) &= \frac{1}{n}\sum_{i=1}^n\left\{\ell(y_i,\hat\btheta_{-p}^\sT
\bz_i) + \ell'(y_i, \hat\btheta_{-p}^\sT \bz_i) (\btheta-\hat\btheta_{-p})^\sT
\bz_i+\frac{1}{2}\ell''(y_i,\tilde \btheta^\sT \bz_i) \|
(\btheta-\hat\btheta_{-p})^\sT \bz_i\|_2^2\right\} \nonumber \\
&\qquad+\frac{\lambda}{2}\|\hat\btheta_{-p}\|_2^2
+\lambda\hat\btheta_{-p}^\sT(\btheta-\hat\btheta_{-p})
+\frac{\lambda}{2}\|\btheta-\hat\btheta_{-p}\|_2^2\\
&\qquad+\btau\cdot\bGamma^{\bW}(\hat\btheta_{-p})
+[\btau\cdot \nabla \bGamma^\bW(\hat\btheta_{-p})]^\sT(\btheta-\hat\btheta_{-p})
+\frac{1}{2}(\btheta-\hat\btheta_{-p})^\sT
[\btau \cdot \nabla^2 \bGamma^\bW(\tilde\btheta)](\btheta-\hat\btheta_{-p}),
\end{align}
where $\tilde \btheta$ is a point on the line segment
between $\hat\btheta_{-p}$ and $\btheta$. Let us decompose
$\btheta=\btheta_{-p}+\theta_p \be_p$, where $\btheta_{-p} \in \R^p$ sets the
last coordinate of $\btheta$ to 0. Then,
substituting the first-order condition
\eqref{equation:KKTConditionForOptimalSolutionReduced} for $\btheta_{-p}$
and noting $\hat\btheta_{-p}^\sT(\btheta_{-p}-\hat\btheta_{-p})
=\hat\btheta_{-p}^\sT(\btheta-\hat\btheta_{-p})$, we obtain
\begin{align}
\hcR(\btheta) &= \frac{1}{n}\sum_{i=1}^n \left\{\ell(y_i,\hat\btheta_{-p}^\sT
\bz_i) +  \ell'(y_i,\hat\btheta_{-p}^\sT \bz_i)\theta_p\be_p^\sT \bz_i  
    + \frac{1}{2}\ell''(y_i,\tilde \btheta^\sT \bz_i) \| (\btheta -
\hat\btheta_{-p})^\sT\bz_i\|_2^2 \right\} \nonumber \\
&\qquad+\frac{\lambda}{2}\|\hat\btheta_{-p}\|_2^2 +
\frac{\lambda}{2}\|\btheta-\hat\btheta_{-p}\|_2^2\\
&\qquad+\btau\cdot\bGamma^{\bW}(\hat\btheta_{-p})
+[\btau\cdot \nabla \bGamma^\bW(\hat\btheta_{-p})]^\sT \theta_p \be_p
+\frac{1}{2}(\btheta-\hat\btheta_{-p})^\sT
[\btau \cdot \nabla^2 \bGamma^\bW(\tilde\btheta)](\btheta-\hat\btheta_{-p})\\
&=\hcR(\hat\btheta_{-p})+\left\{\frac{1}{n}\sum_{i=1}^n \ell'(y_i,\hat
\btheta_{-p}^\sT  \bz_i)\bz_i^\sT\be_p+[\btau\cdot\nabla\bGamma^{\bW}(\hat\btheta_{-p})]^\sT\be_p\right\}\theta_p
+\frac{1}{2}(\btheta-\hat\btheta_{-p})^\sT \nabla^2 \hcR(\tilde
\btheta)(\btheta-\hat\btheta_{-p}).\label{equ:boundalloneKKT}
\end{align}
Since $\btheta,\btheta_{-p} \in \widehat S(K)$ which is a convex set,
we have also $\tilde\btheta \in \widehat S(K)$,
so Lemma \ref{lemma:sublevelgeometry}(c) ensures
$\nabla^2 \hcR(\tilde\btheta) \succeq (\lambda/2)\bI$.
Then the last term of \eqref{equ:boundalloneKKT} is bounded as
\[\frac{1}{2}(\btheta-\hat\btheta_{-p})^\sT \nabla^2 \hcR(\tilde
\btheta)(\btheta-\hat\btheta_{-p})
\geq \frac{\lambda}{4}\|\btheta-\hat\btheta_{-p}\|_2^2 \geq
\frac{\lambda}{4}\theta_p^2,\]
the second inequality holding since $\hat\btheta_{-p}$ has last coordinate 0.
Then, applying \eqref{equ:boundalloneKKT} with $\btheta=\hat\btheta$ being the
optimizer of the original objective $\hcR$, so that
$\hcR(\hat \btheta)-\hcR(\hat\btheta_{-p}) \leq 0$, this shows
	\begin{equation}\label{equation:KKTConditionForOptimalSolutionReduced1thetap}
	  \frac{\lambda}{4} |\hat\theta_p| \leq  \left|[\btau \cdot \nabla\bGamma^{\bW}
(\hat\btheta_{-p})]^\sT
\be_p+\frac{1}{n}\sum_{i=1}^{n}\ell'(y_i,\hat{\btheta}_{-p}^\sT \bz_i)
\bz_i^{\sT} \be_p\right|.
	\end{equation}   
We clarify that here, $\hat \theta_p$ is the last coordinate of the
unconstrained optimizer $\hat\btheta \in \R^p$, while $\hat\btheta_{-p} \in \R^p$ is the
constrained optimizer solving \eqref{equation:DefinitionOfTheta0constrain}.

Let us express the right side of
\eqref{equation:KKTConditionForOptimalSolutionReduced1thetap} more explicitly.
Consider first the case of the original
random features model where $(\bz_i,y_i)=(\bz_i^\RF,y_i^\RF)$.
Since $\hat\btheta_{-p} \in \widehat S(K)$,
by Lemma \ref{lemma:sublevelgeometry}(b) we have $\Gamma_\bW(\btheta)
=\E_{\bz^{\sPG},y^\sPG}[\ell_\test(y^{\sPG},\<\btheta,\bz^\sPG\>) \mid \bW]$
in a neighborhood of $\hat\btheta_{-p}$. 
Then differentiating $\Gamma_\bW(\btheta)$, applying the expansions
$\bz^{\sPG}=\mu_0\1_p+\mu_1\bW\bx'+\mu_2\bV_2\bh_2(\bx')+\mu_{>2} \bg_*$
and $\bz_i^\RF=\sum_{k=0}^D \mu_k \bV_k\bh_k(\bx_i)$,
and applying the identity $\be_p^\sT \bV_k\bh_k(\bx)
=\<\bq_k(\bw_p),\bh_k(\bx)\>=\He_k(\<\bw_p,\bx\>)$,
we may write \eqref{equation:KKTConditionForOptimalSolutionReduced1thetap} as
\begin{equation}\label{eq:KKTBound}
\frac{\lambda}{4} |\hat\theta_p| \leq \left| 2 \tau_1\mu_1^2 \bw_p^\sT \bW^\sT \hbtheta_{-p} + L_0+L_{\geq 1}+\sum_{k=0}^D
T_k\right|
\end{equation}
where
\begin{equation}\label{eq:KKTboundRF}
\begin{aligned}
L_0&=\tau_2\,\E_{\bx',\bg_*}[\ell_\test'(y^\sPG,\<\hat\btheta_{-p},\bz^{\sPG}\>)\mu_0] + 2 \tau_1 \mu_0^2 \<\hat\btheta_{-p},\bones_p\>  \\
L_{\geq 1}&=\tau_2\,\E_{\bx',\bg_*}[\ell_\test'(y^\sPG,\<\hat\btheta_{-p},\bz^{\sPG}\>)
(\mu_1\<\bw_p,\bx'\>+\mu_2\He_2(\<\bw_p,\bx'\>)+\mu_{>2}\bg_*^\sT \be_p)]\\
T_0&=\frac{1}{n}\sum_{i=1}^n \ell'(y_i,\hat\btheta_{-p}^\sT \bz_i)\mu_0\\
T_k&=\frac{1}{n}\sum_{i=1}^n \ell'(y_i,\hat\btheta_{-p}^\sT \bz_i)
\mu_k \He_k(\<\bw_p,\bx_i\>) \text{ for } k=1,\ldots,D.
\end{aligned}
\end{equation}
Here, $(\bx',\bg_*)$ are the independent Gaussian variables defining the test
observation $(\bz^\sPG,y^\sPG)$, and the expectations in $L_0,L_{\geq 1}$ are
over these variables only. Note that by independence of $\bw_p$ and $\bW^\sT \hbtheta_{-p}$, we have directly $|2 \tau_1\mu_1^2 \bw_p^\sT \bW^\sT \hbtheta_{-p} | \prec d^{-1/2} \| \bW^\sT \hbtheta_{-p}\|_2 \prec d^{-1/2}$.

\paragraph{Step 2: Bound $L_0+T_0$.}

We first bound $L_0+T_0$. Applying again the first-order condition
$\nabla \hcR(\hat\btheta_{-p})^\sT \btheta_{-p}=0$ for all $\btheta_{-p}$ with
last coordinate 0, we have
	\begin{equation}
		\Bigg\{\frac{1}{n}\sum_{i=1}^{n}\ell'(y_i,
\hat{\btheta}_{-p}^\sT \bz_i)\bz_i + \lambda \hat\btheta_{-p}
+\btau\cdot\nabla
\bGamma^{\bW} (\hat\btheta_{-p})\Bigg\}^\sT (\1_p-\be_p)=0.
	\end{equation}
Expanding $\bz_i=\mu_0\1_p+\sum_{k=1}^D \mu_k \bV_k\bh_k(\bx_i)$ and
rearranging,
\begin{align}\label{equation:KKTConditionForOptimalSolutionReduced2controlallone}
    &\frac{p-1}{n}\sum_{i=1}^{n}\ell'(y_i, \hat{\btheta}_{-p}^\sT \bz_i)\mu_{0}
+ \lambda \hat\btheta_{-p}^\sT \1_p +
[\btau\cdot\nabla\bGamma^{\bW}(\hat\btheta_{-p})]^\sT(\1_p -\be_p)\\
&=-\sum_{k=1}^D \left(\frac{1}{n}\sum_{i=1}^{n}\ell'(y_i, \hat{\btheta}_{-p}^\sT \bz_i)
\mu_k\bV_k\bh_k(\bx_i)\right)^\sT(\1_p-\be_p).
\end{align}
For $k \geq 3$, we have
\begin{align}
	\left|
\left(\frac{1}{n}\sum_{i=1}^{n}\ell'(y_i, \hat{\btheta}_{-p}^\sT \bz_i)
\mu_k\bV_k\bh_k(\bx_i)\right)^\sT(\1_p-\be_p)\right|
&\leq \frac{|\mu_k|}{n}
\|\bV_k\bh_k(\bX)\|_\op \left\|\begin{pmatrix}
	 \ell'(y_1, \hat{\btheta}_{-p}^\sT \bz_1)\\ 
		\vdots\\ 
	 \ell'(y_n, \hat{\btheta}_{-p}^\sT \bz_n)
	\end{pmatrix} \right\|_2 \|\1_p-\be_p\|_2\\
& \prec |\mu_k|\|\bV_k\bh_k(\bX)\|_\op \prec d,
\end{align}
where the last two inequalities apply
$\|\ell'\|_\infty \prec 1$ by Assumption \ref{assumption:loss}, $p \asymp n
\asymp d^2$, $|\mu_k| \prec 1$, and $\|\bV_k\bh_k(\bX)\|_\op \prec d$ by
Lemma \ref{lemma:OperatorNormZkFk}.
For $k=1,2$, writing $\bV_2=\bV_{2c}+\frac{1}{d}\1_p\be_c^\sT$
and applying also $\|\bW\|_\op \prec \sqrt{d}$, $\|\bV_{2c}\|_\op \prec 1$,
$\|\bX\|_\op \prec d$ and $\|\bh_2(\bX)\|_\op \prec d$ by 
Corollary \ref{corollary:OperatorNormF2cdecompose} and
Lemma \ref{lemma:BernsteinInequalityHermiteFeatures},
similar arguments show
\begin{align}
\left|\left(\frac{1}{n}\sum_{i=1}^{n}\ell'(y_i, \hat{\btheta}_{-p}^\sT
\bz_i)\mu_{1} \bV_1 \bh_1(\bx_i) \right)^{\sT} (\1_p -\be_p) \right| & \prec
|\mu_1|\|\bW\bX\|_\op \prec d^{3/2},\label{equation:BoundOfFirstTermsimilar}\\
	\left| \left( \frac{1}{n}\sum_{i=1}^{n}\ell'(y_i,
\hat{\btheta}_{-p}^\sT \bz_i)\mu_{2}\bV_{2c}\bh_2(\bx_i)\right)^{\sT} (\1_p
-\be_p) \right|  &\prec |\mu_2|\|\bV_{2c}\bh_2(\bX)\|_\op \prec d,\\
	\left| \left( \frac{1}{n}\sum_{i=1}^{n}\ell'(y_i, \hat{\btheta}_{-p}^\sT
\bz_i)\frac{\mu_2}{d}\1_p\be_c^\sT \bh_2(\bx_i) \right)^{\sT} (\1_p -\be_p)
\right| &\prec \frac{|\mu_2|}{d}\|\1_p\be_c^\sT \bh_2(\bX)\|_\op
\prec d^{3/2}.
	\end{align}
Also, we have $|\hat \btheta_{-p}^\sT \1_p| \leq \|\hat\btheta_{-p}\|_2
\|\1_p\|_2 \prec d$ by Lemma \ref{lemma:sublevelgeometry}(a).
Applying these bounds to
\eqref{equation:KKTConditionForOptimalSolutionReduced2controlallone},
\begin{align}\label{equation:KKTConditionForOptimalSolutionReduced2alt}
	\left|\frac{p-1}{n}\sum_{i=1}^{n}\ell'(y_i, \hat{\btheta}_{-p}^\sT
\bz_i)\mu_{0} +  [\btau\cdot\nabla\bGamma^{\bW} (\hat\btheta_{-p})]^\sT(\1_p -\be_p)
\right| \prec d^{3/2}.
\end{align}
Now, writing out
\begin{align}\label{equation:DecompositionOfGradientTermgamma}
&[\btau\cdot\nabla\bGamma^{\bW} (\hat\btheta_{-p})]^\sT (\1_p -\be_p)\\
&=\tau_2
\mu_0 (p-1) \E_{\bx',\bg_*}[\ell_{\test}'(y^\sPG, \<\hat\btheta_{-p},
\bz^\sPG\>)] + 2 \tau_1 (p-1) \mu_0^2 \<\hat\btheta_{-p},\bones_p\> \\
&\qquad + 2 \tau_1 \mu_1^2 (\1_p -\be_p)^\sT \bW\bW^\sT\hat\btheta_{-p}   \\
&\qquad+\tau_2\,\underbrace{\E_{\bx',\bg_*}[\ell_{\test}'(y^\sPG,\<\hat\btheta_{-p},
\bz^\sPG\>)(\mu_1\bW \bx'+\mu_2\bV_2\bh_2(\bx')+\mu_{>2}\bg_*)^\sT(\1_p
-\be_p)]}_{:=M},
\end{align}
similar arguments show that $|M| \prec d^{3/2}$. For example, for the first term
of $M$, we have
\begin{align}
&\E_{\bx',\bg_*}[\ell_{\test}'(y^\sPG,\<\hat\btheta_{-p},\bz^\sPG\>) \mu_{1}(\bW
\bx')^\sT(\1_p -\be_p)] \\
&=\E_{\bx_1',\ldots,\bx_n',\bg_{1*},\ldots,\bg_{n*}}
\left[\left(\frac{1}{n}\sum_{i=1}^{n}\ell'_{\text{test}}(y^\sPG_i,
\hat{\btheta}_{-p}^\sT \bz^\sPG_i)\mu_{1} \bW \bx_i' \right)^{\sT} (\1_p -\be_p)
\right] \prec d^{3/2},
\end{align}
where $\{\bx_i',\bg_{i*}\}_{i=1}^n$ define $n$ independent test observations
$\{y^\sPG_i,\bz_i^\sPG\}_{i=1}^n$. Recalling \eqref{eq:thetazPGbound} and
\eqref{eq:yPGbound}, on the event where $\hat\btheta_{-p} \in \widehat S(K)$,
we have $|y_i^\sPG| \prec 1$ and $|\hat{\btheta}_{-p}^\sT \bz^\sPG_i| \prec 1$,
so $|\ell'_{\text{test}}(y^\sPG_i,\hat{\btheta}_{-p}^\sT \bz^\sPG_i)| \prec 1$
by Assumption \ref{ass:test_loss} for $\ell_\test(\cdot)$.
Then this is $\prec d^{3/2}$ using the same argument as above.
Applying an analogous argument
to each term of $M$ shows $|M| \prec d^{3/2}$. The term $|2 \tau_1 \mu_1^2 (\1_p -\be_p)^\sT \bW\bW^\sT\hat\btheta_{-p}| \prec d \| \bW\|_\op \| \bW^\sT\hat\btheta_{-p} \|_\op \prec d^{3/2}$.
Then, applying these two bounds in
\eqref{equation:DecompositionOfGradientTermgamma} and substituting into
\eqref{equation:KKTConditionForOptimalSolutionReduced2alt},
\begin{align}
   &(p-1)\left|\frac{1}{n}\sum_{i=1}^{n}\left\{\ell'(y_i, \hat{\btheta}_{-p}^\sT
\bz_i)\mu_0+\tau_2 \mu_0 \E_{\bx,\bg_*}[\ell_{\test}'(y^\sPG,
\<\hat\btheta_{-p},\bz^\sPG\>)] + 2 \tau_1  \mu_0^2 \<\hat\btheta_{-p},\bones_p\> \right\} \right|
\prec d^{3/2}.
\end{align}
Recognizing that the left side is $(p-1)|L_0+T_0|$ and using $p \asymp d^2$,
this shows
\begin{equation}\label{equ:boundallone}
|L_0+T_0| \prec d^{-1/2}.
\end{equation}

\paragraph{Step 3: Bound $\E_{\bw_p}[T_k]$.}

For each $k=1,\ldots,D$, we now bound the expectation of $T_k$ over only
the randomness of $\bw_p$, conditioning on all the data $(\bX,\by)$ and also
on the projection directions $\bw_1,\ldots,\bw_{p-1}$. Importantly, by its
definition, $\hat\btheta_{-p}$ is independent of $\bw_p$. Thus
%
%
	\begin{equation}
		\E_{\bw_p}[T_k]=\frac{1}{n}\sum_{i=1}^{n}\ell'(y_i,\hat{\btheta}_{-p}^\sT
\bz_i) \mu_k \E_{\bw_p} \He_k(\<\bw_p,\bx_i\>).
	\end{equation} 
 If $k$ is odd, then $\He_k(\<\bw_p,\bx_i\>)$ is an odd function of
$\<\bw_p,\bx_i\>$, so $\E_{\bw_p} \He_k(\<\bw_p,\bx_i\>)=0$. If $k$ is even,
then $\<\bw_p,\bx_i\>$ is equal in law (conditional on $\bx_i$) to
$w_{p1}\|\bx_i\|_2$. Then by Lemma \ref{lemma:FormualHermite2}, we have
\begin{equation}
    \E_{\bw_p} \He_k(\<\bw_p,\bx_i\>) =    \E_{\bw_p} \He_k(w_{p1}\|\bx_i\|_2) =
\sqrt{k!}\sum_{i=0}^{\lfloor \frac{k}{2}\rfloor }
\frac{(-1)^i}{i!(k-2i)!} \frac{1}{2^i} \|\bx_i\|_2^{k-2i}
\E w_{p1}^{k-2i}.
\end{equation}
Note that for any even $k \geq 0$, we have $\E w_{p1}^k \leq Cd^{-k/2}$, and
from the moments of the chi-squared distribution, also
$\Var[\|\bx\|_2^k]
=2^k\frac{\Gamma(k+d/2)}{\Gamma(d/2)}
-(2^{k/2}\frac{\Gamma(k/2+d/2)}{\Gamma(d/2)})^2 \leq Cd^{k-1}$. Applying this
above shows $\Var_{\bx_i}[\E_{\bw_p} \He_k(\<\bw_p,\bx_i\>)] \leq Cd^{-1}$.
Since also $\E_{\bx_i}[\E_{\bw_p} \He_k(\<\bw_p,\bx_i\>)]=0$, Gaussian
hypercontractivity implies
\[\E_{\bw_p} \He_k(\<\bw_p,\bx_i\>) \prec d^{-1/2}\]
over the randomness of $\bx_i$.
Applying again $\|\ell'\|_\infty \leq C$ and $\mu_k \prec 1$, this shows
\begin{equation}\label{eq:ETkbound}
\E_{\bw_p}[T_k] \prec d^{-1/2}.
\end{equation}
%
%

    \paragraph{Step 4: Concentration of $T_k-\E_{\bw_p}[T_k]$.}
	Next we proceed to prove the concentration of $T_k$ around
$\E_{\bw_p}[T_k]$. We compute the variance of $T_k$ over $\bw_p$:
Squaring the above expression for the mean,
\begin{align}
\E_{\bw_p}[T_k]^2
=\frac{1}{n^2}
\sum_{i,j=1}^{n} \ell'_i\ell'_j \mu_k^2\frac{1}{k!} \sum_{l,q=0}^{\lfloor
\frac{k}{2}\rfloor} \underbrace{\frac{(-1)^{l+q}}{l!q!(k-2l)!(k-2q)!}
\frac{1}{2^{l+q}}}_{:=C(k,l,q)} \|\bx_i\|_2^{k-2l}\|\bx_j\|_2^{k-2q}\E w_{p1}^{k-2l}\E w_{p2}^{k-2q},
	\end{align}
	where we write as shorthand $\ell'_i=\ell'(y_i, \hat{\btheta}_{-p}^\sT
\bz_i)$. On the other hand, let us denote
\begin{align}
	\<\bx_i,\bx_j\> = \|\bx_i\|_2\|\bx_j\|_2\cos\theta_{ij},
\qquad \cos\theta_{ij} = \frac{\<\bx_i,\bx_j\>}{\|\bx_i\|_2\|\bx_j\|_2}.
\end{align}
where $\theta_{ij}$ is the angle between $\bx_i$ and $\bx_j$, and is independent
of $\|\bx_i\|_2$ and $\|\bx_j\|_2$. Then an analogous calculation shows
\begin{align}
	\E_{\bw_p}[T_k^2]
	&=\frac{1}{n^2}\sum_{i,j=1}^{n} \ell'_i\ell'_j \mu_k^2\frac{1}{{k!}}
\sum_{l,q=0}^{\lfloor \frac{k}{2}\rfloor }
C(k,l,q)\|\bx_i\|_2^{k-2l}\|\bx_j\|_2^{k-2q} \E_{\bw_p} w_{p1}^{k-2l}
(w_{p1}\cos \theta_{ij}+ w_{p2}\sin\theta_{ij})^{k-2q}.
	\end{align}
	The conditional variance is then
\begin{align}\label{equation:ConditionalVariancegood}
			\Var_{\bw_p}[T_k]
			&=\frac{1}{n^2}\sum_{\substack{i,j=1 \\ i \ne j}}^{n} \ell'_i\ell'_j
\mu_{k}^2\frac{1}{{k!}} \sum_{l,q=0}^{\lfloor \frac{k}{2}\rfloor }
C(k,l,q)\|\bx_i\|_2^{k-2l}\|\bx_j\|_2^{k-2q}\\
&\hspace{1in}\Big(\underbrace{\E_{\bw_p} w_{p1}^{k-2l} (w_{p1}\cos
\theta_{ij}+w_{p2}\sin\theta_{ij})^{k-2q}
- \E w_{p1}^{k-2l} \E w_{p2}^{k-2q}}_{:=W(k,l,q)}\Big).
	\end{align}
It remains to bound this quantity $W(k,l,q)$.
The intuition is that for large $d$, the angle $\theta_{ij}$ is almost
$\frac{\pi}{2}$, so $W(k,l,q)$ should be small.
We may make this intuition precise as follows:

Consider first the case where $k \geq 2$ is even.
If $k-2q=0$ or $k-2l=0$, then $W(k,l,q)=0$. For
$k-2q \geq 2$ and $k-2l \geq 2$, we use that $\E w_{p1}^i w_{p2}^j=0$ if $i,j$
are odd to expand
\begin{align*}
\E_{\bw_p}w_{p1}^{k-2l}(w_{p1}\cos \theta_{ij}+w_{p2}\sin\theta_{ij})^{k-2q}
&=\mathop{\sum_{m=0}^{k-2q}}_{\text{even}}
\binom{k-2q}{m} \E[w_{p1}^{k-2l+m}w_{p2}^{k-2q-m}]
(\cos \theta_{ij})^m(\sin \theta_{ij})^{k-2q-m}\\
&=\E[w_{p1}^{k-2l}w_{p2}^{k-2q}]+O(d^{-(k-l-q)}\cos^2 \theta_{ij}).
\end{align*}
Here, the second equality applies
$\E[w_{p1}^{k-2l+m}w_{p2}^{k-2q-m}] \leq Cd^{-(k-l-q)}$ for the summands $m \geq
2$, and applies also an expansion of $(\sin \theta_{ij})^{k-2q}=(1-\cos^2
\theta_{ij})^{(k-2q)/2}$ for the summand $m=0$. Then
\begin{align}
W(k,l,q)&=\E w_{p1}^{k-2l}w_{p2}^{k-2q}
-\E w_{p1}^{k-2l} \E w_{p2}^{k-2q}
+O(d^{-(k-l-q)}\cos^2 \theta_{ij}).
\end{align}
A standard application of Bernstein's
inequality shows $|\cos \theta_{ij}| \prec d^{-1/2}$.
Also, expanding $w_{p2}=\sqrt{1-w_{p1}^2} \cdot \tilde w_{p2}$ where
$\tilde w_{p2}$ is independent of $w_{p1}$, we have
that $|\E w_{p1}^{k-2l}w_{p2}^{k-2q}-\E w_{p1}^{k-2l} \E w_{p2}^{k-2q}|
\leq d^{-(k-l-q)-1}$. Consequently,
\[|W(k,l,q)| \prec d^{-(k-l-q)-1}.\]
Combining this with the standard concentration bounds
$\|\bx_i\|_2^{k-2l}\|\bx_j\|_2^{k-2q} \prec d^{k-l-q}$,
the bounds $\|\ell'\|_\infty \leq C$, and the assumption $\mu_k \prec 1$,
this establishes
\[\Var_{\bw_p}[T_k] \prec d^{-1}\]
in the case where $k$ is even.

Consider next the case where $k \geq 1$ is odd. Then, applying similar arguments
as above, $\E w_{p1}^{k-2l}\E w_{p2}^{k-2q}=0$ and
 \begin{align}
W(k,l,q)&=O(d^{-(k-l-q)}\cos\theta_{ij}).
\end{align}
Applying again $|\cos \theta_{ij}| \prec d^{-1/2}$,
$\|\bx_i\|_2^{k-2l}\|\bx_j\|_2^{k-2q} \prec d^{k-l-q}$,
$\|\ell'\|_\infty \leq C$, and $\mu_k \prec 1$, this shows
\[\Var_{\bw_p}[T_k] \prec d^{-1/2}\]
in the case where $k$ is odd.

Thus we have shown in all cases that $\Var_{\bw_p}[T_k] \prec d^{-1/2}$. By
hypercontractivity over the unit sphere, since $T_k$ is a polynomial of
$\bw_p$, this implies $|T_k-\E_{\bw_p}[T_k]| \prec d^{-1/4}$. Combining with
\eqref{eq:ETkbound}, this shows
\begin{equation}\label{eq:Tkbound}
|T_k| \prec d^{-1/4}.
\end{equation}

\paragraph{Step 5: Other cases of $(\bZ,\by)$} 
In the above arguments, we have assumed that $(\bz_i,y_i)=(\bz_i^\RF,y_i^\RF)$
for each $i=1,\ldots,n$.
More generally, in the case of the LOO optimization where $(\bZ,\by)=(\bZ_{\setminus q},
\by_{\setminus q})$, we may write
\eqref{equation:KKTConditionForOptimalSolutionReduced1thetap} as
\begin{equation}\label{eq:KKTBoundLOO}
\frac{\lambda}{4} |\hat\theta_p| \leq L_0+L_{\geq 1}+\sum_{j=0}^D T_k+T_*
\end{equation}
where $L_0,L_{\geq 1}$ have the same forms as in \eqref{eq:KKTboundRF},
\begin{align*}
T_k&=\frac{1}{n}\sum_{i \neq q} \ell'(y_i,\hat\btheta_{-p}^\sT \bz_i)
\mu_k \He_k(\<\bw_p,\bx_i\>) \text{ for } k=0,1,2,\\
T_k&=\frac{1}{n}\sum_{i=1}^{q-1} \ell'(y_i,\hat\btheta_{-p}^\sT \bz_i)
\mu_k \He_k(\<\bw_p,\bx_i\>) \text{ for } k=3,\ldots,D,\\
T_*&=\frac{1}{n}\sum_{i=q+1}^n \ell'(y_i,\hat\btheta_{-p}^\sT \bz_i)
\mu_{>2}\<\bg_{i*},\be_p\>.
\end{align*}
In the case of the auxiliary optimization where $(\bZ,\by)=(\bZ_{\cup q},\by_{\cup q})$,
we may write
\eqref{equation:KKTConditionForOptimalSolutionReduced1thetap} as
\eqref{eq:KKTBoundLOO} with an additional summand for $(\bz_q,y_q)=(\tilde
\bz_q,\tilde y_q)$ included into $T_k$ and/or $T_*$.

The preceding arguments for bounding $L_0+T_0$ and
$T_k$ in \eqref{eq:KKTboundRF} hold equally for bounding these quantities in
\eqref{eq:KKTBoundLOO}, where we apply also
$\|[\bg_{1*},\ldots,\bg_{n*}]\|_\op \prec \sqrt{n}+\sqrt{p} \prec d$ in addition
to $\|\bV_k\bh_k(\bX)\|_\op \prec d$ for $k \geq 3$ to obtain
\eqref{equation:KKTConditionForOptimalSolutionReduced2alt}.
To bound the additional quantity $T_*$, note that
for each $i \in \{q+1,\ldots,n\}$,
$\ell'(y_i,\hat\btheta_{-p}^\sT \bz_i)$ depends only on the first $p-1$
coordinates of $\bg_{i*}$ and is independent of the last coordinate
$(g_{i*})_p$. Then, taking the expectation and variance over only
$\{(g_{i*})_p\}_{i=q+1}^n$,
\[\E_{\{(g_{i*})_p\}_{i=q+1}^n} T_*=0,
\qquad \Var_{\{(g_{i*})_p\}_{i=q+1}^n} T_*=
\frac{1}{n^2}
\sum_{i=q+1}^n \mu_{>2}^2 \left(\ell'(y_i,\hat \btheta_{-p}^\sT \bz_i)
\right)^2 \prec d^{-2}.\]
Thus by a standard Gaussian tail bound,
\begin{equation}\label{eq:Tstarbound}
|T_*|=|T_*-\E_{\{(g_{i*})_p\}_{i=q+1}^n} T_*| \prec d^{-1}.
\end{equation}

\paragraph{Step 6: Bound $L_{\geq 1}$.}
Finally, we bound the term
\[L_{\geq 1}=\tau\,\E_{\bx',\bg_*}[\ell_\test'(y^{\sPG},\<\hat\btheta_{-p},\bz^{\sPG}\>)
(\mu_1\<\bw_p,\bx'\>+\mu_2\He_2(\<\bw_p,\bx'\>)+\mu_{>2}\bg_*^\sT\be_p)]\]
We may again re-express this as
\[L_{\geq 1}
=\tau\,\E_{\bx_1',\ldots,\bx_n',\bg_{1*},\ldots,\bg_{n*}}\left[\frac{1}{n}
\sum_{i=1}^n \ell_\test'(y_i^{\sPG},\<\hat\btheta_{-p},\bz_i^{\sPG}\>)
(\mu_1\<\bw_p,\bx_i'\>+\mu_2\He_2(\<\bw_p,\bx_i'\>)+\mu_{>2}\<\bg_{i*},\be_p\>)\right]\]
where $\{\bx_i',\bg_{i*}\}_{i=1}^n$ define $n$ independent test observations
$\{y_i^\sPG,\bz_i^\sPG\}_{i=1}^n$. Then similar arguments as those above
establishing \eqref{eq:Tkbound} for $T_1,T_2$ and \eqref{eq:Tstarbound} for
$T_*$ may be applied to show
\begin{equation}\label{equ:boundtheperturbationterm}
|L_{\geq 1}| \prec d^{-1/4}.
\end{equation}

Applying \eqref{equ:boundallone}, \eqref{eq:Tkbound}, \eqref{eq:Tstarbound},
\eqref{equ:boundtheperturbationterm}, and $\lambda^{-1} \prec 1$
to \eqref{eq:KKTBound} and \eqref{eq:KKTBoundLOO},
we obtain that $|\hat\theta_p| \prec d^{-1/4}$. This argument applies equally to
each coordinate of $\hat\btheta$, and hence $\|\hat\btheta\|_\infty \prec
d^{-1/4}$, i.e.\ for any $C>0$, there exists $K'>0$ such that with probability
$1-d^{-C}$,
\begin{equation}\label{eq:hatthetainfty}
\|\hat\btheta\|_\infty \leq (\log d)^{K'} d^{-1/4}.
\end{equation}
Combining \eqref{eq:hatthetabasicbounds} and \eqref{eq:hatthetainfty} shows
$\hat\btheta \in \bTheta_\bW^\sPG(K')$, and replacing $K$ by $\max(K,K')$
concludes the proof of the lemma.
\end{proof} 
\subsection{Analysis of the quadratic surrogate problem}
\label{sec:geometric-properties-family-i}

We now analyze the quadratic surrogate optimization
\eqref{eq:familyI-quadratic}.
We define the Moreau envelope and proximal operator of the (centered) loss by
\begin{align}
\cM_y(z;\gamma)&=\min_{x \in \R}
\left\{\ell(y,x)-\ell(0,0)+\frac{(z-x)^2}{2\gamma}\right\},\label{eq:Moreau}\\
\Prox_y(z;\gamma)&=\arg \min_{x \in
\R}\left\{\ell(y,x)-\ell(0,0)+\frac{(z-x)^2}{2\gamma}\right\},\label{eq:Prox}
\end{align}
and write as shorthand
\[\bH_{\setminus q} \equiv \bH_{\setminus q}(\hat\btheta_{\setminus q})\]
for the Hessian of the LOO empirical risk at its minimizer.

\begin{proposition}[Properties of the Moreau Envelope]\label{lemma:MathematicalExpressionLipschitzInBrAndGammaKpure}
Under Assumption \ref{assumption:loss} for the loss $\ell(\cdot)$, there is a
constant $C>0$ such that the Moreau envelope function $\cM_y(z;\gamma)$
is $C$-Lipschitz in each of the arguments $y,z,\gamma$.
\end{proposition}
\begin{proof}
Let $x^* \equiv x^*(y,z,\gamma)
=\Prox_y(z;\gamma)$ be the unique minimizer of the 
optimization \eqref{eq:Moreau} defining $\cM$. The first-order condition for
the optimality of $x^*$ is
\begin{equation} \label{eq:prox_foc}
0=\ell'(y,x^*)-\frac{z-x^*}{\gamma}.
\end{equation}
We remark that $x^*(y,z,\gamma)$
is continuously-differentiable in $(y,z,\gamma)$ by the implicit function
theorem. Then so is $\cM_y(z;\gamma)$, and the envelope theorem gives
\[\partial_z \cM_y(z;\gamma)=\frac{z-x^*}{\gamma}=\ell'(y,x^*),
\qquad
\partial_y \cM_y(z;\gamma)=\partial_y\ell(y,x^*),\]
\[\partial_\gamma \cM_y(z;\gamma)={-}\frac{(z-x^*)^2}{2\gamma^2}
=-\frac{1}{2}\ell'(y, x^*)^2\]
where we have applied \eqref{eq:prox_foc} to substitute for $z-x^*$.
Thus by Assumption \ref{assumption:loss}, these derivatives are all uniformly
bounded, showing the Lipschitz continuity of $(y,z,\gamma) \mapsto
\cM_y(z;\gamma)$.
\end{proof}

\begin{lemma}[Characterization of the Surrogate Minimizer]\label{lemma:Psi_k_function_of_r_Phi_M_k_and_gamma_kr}
For each $q=1,\ldots,n$ and choice of
$(\tilde \bz_q,\tilde y_q) \in \{(\bz_q^\RF,y_q^\RF),(\bz_q^\sPG,y_q^\sPG)\}$,
we have
\begin{equation}
\Psi_q=\Phi_{\backslash q}+\frac{1}{n}\,\cM_{\tilde y_q}(\<\tilde \bz_q,
\hat\btheta_{\backslash q}\>;\gamma_q)
\end{equation}
where the effective step size $\gamma_q \equiv \gamma_q(\tilde \bz_q)$ is
defined as
\begin{equation}\label{eq:gammaq}
\gamma_q=\frac{\tilde \bz_q^\sT \bH_{\backslash q}^{-1}\tilde \bz_q}{n}.
\end{equation}
Moreover,
\begin{equation}\label{equation:theta_k_function_of_br_hat_theta_and_H_inversepurer}
\tilde\btheta_{\cup q}=\hat\btheta_{\backslash q}-\ell'(\tilde y_q,\<\tilde \bz_q,
\tilde\btheta_{\cup q}\>)\frac{\boldsymbol{H}_{\backslash q}^{-1}
\tilde \bz_q}{n},
\qquad
\<\tilde \bz_q,\tilde \btheta_{\cup q}\>=\Prox_{\tilde y_q}(\<\tilde \bz_q,
\hat\btheta_{\backslash q}\>;\gamma_q).
\end{equation}
\end{lemma}
\begin{proof}
By definition of the quadratic surrogate objective $\widetilde \cR_{\cup
q}(\btheta)$, its minimum risk value $\Psi_q$ is given by
    \begin{equation}\label{eq:psi_q_definition}
        \Psi_q=\Phi_{\backslash q}+\min_{\btheta \in \mathbb{R}^p}
\left\{\frac{1}{n}\big[\ell(\tilde y_q,\< \tilde \bz_q ,  \btheta
\>)-\ell(0,0)\big]
+\frac{1}{2}(\btheta-\hat \btheta_{\backslash q})^\sT \bH_{\backslash
q}(\btheta-\hat \btheta_{\backslash q})\right\}.
    \end{equation}
Rearranging the first-order optimality condition
$\frac{1}{n}\ell'(\tilde y_q,\< \tilde \bz_q ,\tilde\btheta_{\cup q}\>)\tilde \bz_q +
\bH_{\backslash q}(\tilde\btheta_{\cup q}-\hat\btheta_{\backslash q})=\bzero$
shows the first statement of 
\eqref{equation:theta_k_function_of_br_hat_theta_and_H_inversepurer}.
Multiplying this statement by $\tilde\bz_q$, we get
\begin{align}
\<\tilde\bz_q, \tilde\btheta_{\cup q}\> &=
\<\tilde\bz_q,\hat\btheta_{\backslash q}\>-\ell'(\tilde y_q,\<\tilde\bz_q,
\tilde\btheta_{\cup q}\>)\gamma_q.
\end{align}
Thus $x=\<\tilde\bz_q,\tilde\btheta_{\cup q}\>$ solves the first-order condition
\begin{equation}\label{eq:proxfirstorder}
\ell'(\tilde y_q,x)+(x-\<\tilde\bz_q,\hat\btheta_{\backslash q}\>)/\gamma_q=0
\end{equation}
for the minimization problem defining $\Prox_{\tilde
y_q}(\<\tilde\bz_q,\hat\btheta_{\backslash q}\>;\gamma_q)$. Since this
minimization is strongly convex, this implies the second statement of
\eqref{equation:theta_k_function_of_br_hat_theta_and_H_inversepurer}.

Finally, to establish the form of $\Psi_q$, we evaluate the objective in
\eqref{eq:psi_q_definition} at the minimizer $\tilde\btheta_{\cup q}$:
    \begin{equation}\label{eq:psi_q_expanded}
        \Psi_q=\Phi_{\backslash q}+\frac{1}{n}\,[\ell(\tilde y_q, \<\tilde
\bz_q,\tilde \btheta_{\cup q}\>)-\ell(0,0)]+\frac{1}{2}(\tilde\btheta_{\cup
q}-\hat \btheta_{\backslash q})^\sT \bH_{\backslash q}(\tilde\btheta_{\cup q}- \hat \btheta_{\backslash q}).
\end{equation}
Substituting the form of $\tilde\btheta_{\cup
q}-\hat\btheta_{\backslash q}$ in
\eqref{equation:theta_k_function_of_br_hat_theta_and_H_inversepurer},
and then applying the first-order condition \eqref{eq:proxfirstorder}, gives
\[(\tilde\btheta_{\cup q}-\hat \btheta_{\backslash q})^\sT \bH_{\backslash
q}(\tilde\btheta_{\cup q}-\hat \btheta_{\backslash q})
=\frac{1}{n}\,\ell'(\tilde y_q,\<\tilde\bz_q,\tilde\btheta_{\cup q}\>)^2\gamma_q
=\frac{(\<\tilde\bz_q,\tilde\btheta_{\cup q}\>-
\<\tilde \bz_q,\hat\btheta_{\setminus q}\>)^2}{n\gamma_q}.\]
Plugging this back into \eqref{eq:psi_q_expanded} gives
$\Psi_q=\Phi_{\backslash q}+\frac{1}{n}\cM_{\tilde y_q}(\<\tilde
\bz_q,\hat\btheta_{\setminus q}\>,\gamma_q)$.
\end{proof}   

\begin{corollary}\label{corollary:SurrogateMinimizersInGoodSetallsamples}
For any constant $C>0$, there exists $K>0$ such that with probability at least
$1-d^{-C}$, for each $q=1,\ldots,n$ and choice of
$(\tilde \bz_q,\tilde y_q) \in \{(\bz_q^\RF,y_q^\RF),(\bz_q^\sPG,y_q^\sPG)\}$
defining $\tilde\btheta_{\cup q}$,
\begin{equation}
\tilde\btheta_{\cup q} \in \bTheta_{\bW}^\sPG(K) \cap 
\widehat S_q(K).
\end{equation}
\end{corollary}
\begin{proof}
We have $\|\bH_{\setminus q}^{-1}\|_\op \prec 1$
by Lemma \ref{lemma:sublevelgeometry}(c),
$\|\bz_q\|_2 \prec d$ by Lemma \ref{lemma:yzbounds}, and $\|\ell'\|_\infty \prec
1$ by Assumption \ref{assumption:loss}. Then applying the form of
$\tilde\btheta_{\cup q}$ in 
\eqref{equation:theta_k_function_of_br_hat_theta_and_H_inversepurer},
\begin{equation}\label{eq:surrogateLOOclose}
\|\tilde\btheta_{\cup q}-\hat\btheta_{\backslash q}\|_2 \prec d^{-1}.
\end{equation}
By Lemma \ref{lemma:optinThetaPG}, we have with probability $1-d^{-C}$ that
$\hat\btheta_{\setminus q} \in \bTheta_\bW^{\sPG}(K) \cap \widehat S_q(K)$,
meaning
\[\|\hat\btheta_{\setminus q}\|_2 \prec 1,
\quad |\hat\btheta_{\setminus q}^\sT \1_p^\sT| \prec 1,
\quad \|\mu_1\hat\btheta_{\setminus q}^\sT \bW\|_2 \prec 1,
\quad \|\hat\btheta_{\setminus q}\|_\infty \prec d^{-1/4},\]
\[\frac{1}{n}\sum_{i:i \neq q} \ell(y_i,\<\hat\btheta_{\setminus q},\bz_i\>)
+(\lambda/2)\|\hat\btheta_{\setminus q}\|_2^2 \prec 1\]
where $(\bz_i,y_i)$ are the features and labels of $(\bZ_{\setminus
q},\by_{\setminus q})$. Then, since
$\|\1_p\|_2 \asymp d$, $\|\bW\|_\op \prec \sqrt{d}$,
$\|\bz_i\|_2 \prec d$ by Lemma \ref{lemma:yzbounds}, and
$\|\ell'\|_\infty \prec 1$ by Assumption \ref{assumption:loss}, 
\eqref{eq:surrogateLOOclose} implies that the same
statements hold for $\tilde\btheta_{\cup q}$, i.e.\ for any $C>0$, there
exists (a possibly larger constant) $K>0$ such that
$\tilde \btheta_{\cup q} \in \bTheta_\bW^\sPG(K) \cap \widehat S_q(K)$.
\end{proof}

\begin{lemma}[Concentration of the Effective Step Size]\label{lemma:gamma_concentration}
For any $q=1,\ldots,n$ and choice of
$(\tilde \bz_q,\tilde y_q) \in \{(\bz_q^\RF,y_q^\RF),(\bz_q^\sPG,y_q^\sPG)\}$,
let $\E_q$ denote the expectation over $\bx_q$ defining $\bz_q^\RF$ or
$(\bx_q,\bg_{q*})$ defining $\bz_q^\sPG$ (both conditioning on $\bW$). Then
for the quantity $\gamma_q=\frac{\tilde \bz_q^\sT \bH_{\backslash q}^{-1} \tilde
\bz_q}{n}$, we have
\begin{equation}
|\gamma_q-\E_q \gamma_q| \prec d^{-1/2}.
\end{equation}
\end{lemma}
\begin{proof}
Consider first the case $\bz_q=\bz_q^\RF=\sigma(\bW\bx_q)$.
Let us write $\Var_q$ for the variance associated to $\E_q$.
By the Gaussian Poincar\'e inequality,
    \begin{equation}
    \Var_q[\gamma_q] \leq \E_q \| \nabla_{\bx_q} \gamma_q\|_2^2
=\frac{4}{n^2}\E_q \Big((\nabla_{\bx_q} \bz_q)^\sT \bH_{\backslash q}^{-1}
\bz_q\Big)^2
    \end{equation}
where $\nabla_{\bx_q} \bz_q=\diag\left(\sigma'(\bW\bx_q)\right) \bW \in \R^{p
\times d}$. Thus
\begin{align}
\Var_q[\gamma_q] &\leq 
\frac{4}{n^2} \E_q \|\bW\|_\op^2 \|\diag(\sigma'(\bW\bx_q))\|_{\op}^2
\|\bH_{\backslash q}^{-1}\|_{\op}^2 \|\bz_q\|_2^2.
\end{align}
We have $\|\bW\|_\op \prec \sqrt{d}$, $\|\bH_{\setminus q}^{-1}\|_\op \prec 1$
by Lemma \ref{lemma:sublevelgeometry}(c), and $\|\bz_q\|_2 \prec d$ by Lemma
\ref{lemma:yzbounds}. Furthermore,
\[\|\diag(\sigma'(\bW\bx_q))\|_{\op}
=\max_{j \in [p]} |\sigma'(\<\bw_j, \bx_q\>)| \prec 1\]
by the conditions for the polynomial $\sigma$ in Assumption
\ref{assumption:activation}. Applying this above shows
$\Var_q[\gamma_q] \prec d^{-1}$, and hence
$|\gamma_q-\E_q \gamma_q| \prec d^{-1/2}$ by Gaussian hypercontractivity over
$\bx_q$. In the case $\bz_q=\bz_q^\sPG=\sigma_{\leq
2}(\bW\bx_q)+\mu_{>2}\bg_{q*}$, we have instead
$\nabla_{\bx_q,\bg_{q*}} \bz_q=
[\diag\left(\sigma_{\leq 2}'(\bW\bx_q)\right) \bW,\bI]
\in \R^{p \times (d+p)}$, and the rest of the argument is the same as above.
\end{proof}

\subsection{Comparison of surrogate with LOO/auxiliary risks}
\label{sec:surrogateLOOauxcompare}

We now establish the following lemma, which compares the minimum risk values of
the surrogate, LOO, and auxiliary objectives.

\begin{lemma}\label{lemma:surrogateLOOauxcompare}
There exists a constant $K>0$ such that for each $q=1,\ldots,n$ and choice of
$(\tilde \bz_q,\tilde y_q) \in \{(\bz_q^\RF,y_q^\RF),(\bz_q^\sPG,y_q^\sPG)\}$
(defining both $\Psi_q$ and $\Phi_q$), 
\begin{align}
\E(\Psi_q-\Phi_{\setminus q})^2 &\leq \frac{(\log
d)^K}{n^2},\label{eq:PsiPhiLOOcompare}\\
\E|\Psi_q-\Phi_q| &\leq \frac{(\log d)^K}{n^{3/2} \tau_1^{9/2}}.\label{eq:PsiPhiauxcompare}
\end{align}
\end{lemma}

\begin{proof}[Proof of Lemma \ref{lemma:surrogateLOOauxcompare},
eq.\ \eqref{eq:PsiPhiLOOcompare}]

By Lemma \ref{lemma:Psi_k_function_of_r_Phi_M_k_and_gamma_kr},
\begin{equation}
\Psi_q-\Phi_{\backslash q}=\frac{1}{n}\mathcal{M}_{\tilde y_q}\left(\<\tilde
\bz_q, \hat\btheta_{\backslash q}\>;\gamma_q\right).
\end{equation}
Evaluating the optimization \eqref{eq:Moreau} defining
$\mathcal{M}_{y}(z;\gamma)$ at $x=z$, we have $\mathcal{M}_{y}(z;\gamma)
\leq \ell(y,z)-\ell(0,0)$. By \eqref{eq:Moreau}, since $\ell(\cdot) \geq 0$, we
have also $\mathcal{M}_{y}(z;\gamma) \geq {-}\ell(0,0)$. Thus
\[|\Psi_q-\Phi_{\backslash q}|
\leq \frac{1}{n}\,|\ell(\tilde y_q,\<\tilde\bz_q,\hat\btheta_{\setminus q}\>)|
+\frac{1}{n}\,|\ell(0,0)|
\leq \frac{C}{n}(1+|\tilde y_q|+|\<\tilde \bz_q,\hat\btheta_{\setminus q}\>|)\]
where the second inequality holds
by the properties of $\ell(\cdot)$ in Assumption \ref{assumption:loss}.

Importantly, $\hat\btheta_{\setminus q}$ is independent of the variables
$(\bx_q,\bg_{q*})$ defining $\tilde\bz_q$. Writing $\E_q$
for the expectation over $(\bx_q,\bg_{q*})$, we then have
$\E_q|\<\tilde \bz_q,\hat\btheta_{\setminus q}\>|^2
=\hat\btheta_{\setminus q}^\sT \E_q[\tilde \bz_q\tilde
\bz_q^\sT]\hat\btheta_{\setminus q}$.
In the case $\tilde \bz_q=\bz_q^\RF$, we have
$\E_q[\tilde \bz_q\tilde \bz_q^\sT]=\sum_{k=0}^D \mu_k^2 \bV_k\bV_k^\sT$,
while in the case $\tilde \bz_q=\bz_q^\sPG$, we have
$\E_q[\tilde \bz_q\tilde \bz_q^\sT]=\sum_{k=0}^2 \mu_k^2 \bV_k\bV_k^\sT
+\mu_{>2}^2\bI$. In both cases, applying Lemma \ref{lemma:sublevelgeometry}(a),
we get $\E_q|\<\tilde \bz_q,\hat\btheta_{\setminus q}\>|^2 \prec 1$, and hence
$|\<\tilde \bz_q,\hat\btheta_{\setminus q}\>| \prec 1$ by Gaussian
hypercontractivity over $(\bx_q,\bg_{q*})$. Applying also $|\tilde y_q| \prec 1$
from Lemma \ref{lemma:yzbounds}, this shows
\[|\Psi_q-\Phi_{\setminus q}| \prec n^{-1}.\]

This means that for any constant $C_0>0$, there exists $K>0$ for which
\begin{equation}\label{eq:goodevent}
\cE=\{(\Psi_q-\Phi_{\setminus q})^2 \leq (\log d)^K n^{-2}\}
\text{ satisfies } \P[\cE] \geq 1-d^{-C_0}.
\end{equation}
On the complementary event $\cE^c$, we note that since $|\btau \cdot
\bGamma^\bW(\btheta)| \leq 2$ and $\ell(\cdot) \geq 0$, we have by
definition
\[\Phi_{\setminus q} \geq -2,
\qquad \Phi_{\setminus q} \leq \hcR_{\setminus q}(\bzero)
\leq \frac{1}{n}\sum_{i=1}^n \ell(y_i,0)+2
\leq \frac{C}{n}\sum_{i=1}^n (1+|y_i|)+2,\]
where $\{y_i\}_{i=1}^n$ denote the labels of $\by_{\setminus q}$.
Since $\bH_{\setminus q}=
\nabla^2 \hcR_{\setminus q}(\hat\btheta_{\setminus q})$ and
$\hat\btheta_{\setminus q}$ is a minimizer of $\hcR_{\setminus q}$, we have
$\bH_{\setminus q} \succeq 0$. Then also by definition,
\[\Psi_q \geq -2,
\qquad \Psi_q \leq \hcR_{\setminus q}(\hat\btheta_{\setminus q})
\leq \Phi_{\setminus q}
+\frac{1}{n}\,\ell(\tilde y_q,\<\hat\btheta_{\setminus q},\tilde \bz_q\>)
\leq \Phi_{\setminus q}
+\frac{C}{n}(1+|\tilde y_q|+\|\hat\btheta_{\setminus q}\|_2 \cdot \|\tilde
\bz_q\|_2).\]
From the optimality condition $\hcR_{\setminus q}(\hat\btheta_{\setminus q})
\leq \hcR_{\setminus q}(\bzero)$ and bounds $|\btau \cdot \bGamma^\bW(\btheta)|
\leq 2$ and $\ell(\cdot) \geq 0$, we obtain
\[\frac{\lambda}{2}\|\hat\btheta_{\setminus q}\|_2^2
\leq \frac{1}{n}\sum_{i=1}^n \ell(y_i,0)+4
\leq \frac{C}{n}\sum_{i=1}^n (1+|y_i|)+4.\]
Combining these bounds, it may be verified that 
$\E[\Phi_{\setminus q}^4] \leq (\log d)^K$ and $\E[\Psi_q^4] \leq (\log d)^K$
for a sufficiently large constant $K>0$. Then
\[\E[(\Psi_q-\Phi_{\setminus q})^2\1\{\cE^c\}]
\leq \E[(\Psi_q-\Phi_{\setminus q})^4]^{1/2} \cdot \P[\cE^c]^{1/2}
\prec \P[\cE^c]^{1/2}.\]
Choosing $C_0=8$ in \eqref{eq:goodevent} ensures that
$\P[\cE^c]^{1/2} \prec n^{-2}$, i.e.\ there exists $K>0$
for which
\begin{equation}\label{eq:Ecomplementbound}
\E[(\Psi_q-\Phi_{\setminus q})^2\1\{\cE^c\}] \leq (\log d)^K n^{-2},
\end{equation}
and combining with \eqref{eq:goodevent} concludes the proof
of \eqref{eq:PsiPhiLOOcompare}.
\end{proof}

We next prove \eqref{eq:PsiPhiauxcompare}. For this, we establish two auxiliary
lemmas.

\begin{lemma}[Concentration of Cross-Sample Influence]\label{lemma:BernsteinInequalityHermiteFeaturesrt}

For each $q \in \{1,\ldots,n\}$, choice of
$(\tilde \bz_q,\tilde y_q) \in \{(\bz_q^\RF,y_q^\RF),(\bz_q^\sPG,y_q^\sPG)\}$,
and $i \in [n] \setminus q$, letting $\bz_i$ denote the $i^\text{th}$ column
of $\bZ_{\setminus q}$ and letting $\bz^\sPG$ denote the independent test sample defining
$L_\bW(\btheta)$, 
\begin{equation}
\frac{\bz_i^\sT \boldsymbol{H}_{\backslash q}^{-1} \tilde\bz_q}{n}
\prec \frac{1}{\tau_1^{1/2} d},
\qquad \frac{{\bz^\sPG}^\sT \boldsymbol{H}_{\backslash q}^{-1} \tilde\bz_q}{n}
\prec \frac{1}{\tau_1^{1/2}d}.
\end{equation}
\end{lemma}
\begin{proof}
Write $\E_q$ for the expectation over $(\bx_q,\bg_{q*})$ defining $\tilde
\bz_q$. Since $\bz_i$ and $\bH_{\backslash q}$ are independent of
$(\bx_q,\bg_{q*})$ we have
\[\E_q\bigg(\frac{\bz_i^\sT \boldsymbol{H}_{\backslash q}^{-1}
\tilde\bz_q}{n}\bigg)^2
=\frac{1}{n^2} \bz_i^\sT \bH_{\backslash q}^{-1}\E_q[\tilde\bz_q
\tilde\bz_q^\sT]\bH_{\backslash q}^{-1} \bz_i.\]
In the case $\tilde \bz_q=\bz_q^\RF$, we have
$\E_q[\tilde\bz_q \tilde\bz_q^\sT]=\bV\bV^\sT$, where $\bV$ is the matrix
\eqref{eq:Vcombined}. Thus
\begin{equation}\label{eq:zHzsquared}
\bz_i^\sT \bH_{\backslash q}^{-1}\E_q[\tilde\bz_q
\tilde\bz_q^\sT]\bH_{\backslash q}^{-1} \bz_i
=\|\bV^\sT \bH_{\backslash q}^{-1} \bz_i\|_2^2
\leq \|\bV^\sT \bH_{\backslash q}^{-1/2}\|_\op^2
\cdot \|\bH_{\backslash q}^{-1/2}\|_\op^2 \cdot \|\bz_i\|_2^2
\prec \tau_1^{-1} d^2,
\end{equation}
where the last inequality uses
$\|\bV^\sT \bH_{\setminus q}^{-1}\bV\|_\op \prec  \tau_1^{-1}$ 
and $\|\bH_{\setminus q}^{-1}\|_\op \prec 1$
from Lemma \ref{lemma:sublevelgeometry}(c) and $\|\bz_i\|_2 \prec d$
from Lemma \ref{lemma:yzbounds}. In the case
$\tilde \bz_q=\bz_q^\sPG$, we have
$\E_q[\tilde\bz_q \tilde\bz_q^\sT]=\bV_{\leq 2}\bV_{\leq 2}^\sT+\mu_{>2}^2\bI$,
where $\bV_{\leq 2}$ is the submatrix of $\bV$ corresponding to its components
for $k=0,1,2$. Then a similar argument shows \eqref{eq:zHzsquared}. Thus
$\E_q(\frac{\bz_i^\sT \boldsymbol{H}_{\backslash q}^{-1}
\tilde\bz_q}{n})^2 \prec \tau_1^{-1} d^{-2}$. The same argument shows
$\E_q(\frac{{\bz^\sPG}^\sT \boldsymbol{H}_{\backslash q}^{-1}
\tilde\bz_q}{n})^2 \prec\tau_1^{-1} d^{-2}$, so both statements of the lemma hold by 
hypercontractivity over $(\bx_q,\bg_{q*})$.
\end{proof}

\begin{lemma}[Moment Bound on the Surrogate Approximation Error Projection]
\label{lemma:BernsteinInequalityHermiteFeaturesrt2}

For each $q=1,\ldots,n$, choice of
$(\tilde \bz_q,\tilde y_q) \in \{(\bz_q^\RF,y_q^\RF),(\bz_q^\sPG,y_q^\sPG)\}$
(defining both $\hat\btheta_{\cup q}$ and $\tilde \btheta_{\cup q}$),
and $i \in [n] \setminus q$, letting $\bz_i$ denote the $i^\text{th}$ column
of $\bZ_{\setminus q}$ and letting $\bz^\sPG$ denote the independent test sample defining
$L_\bW(\btheta)$, 
\begin{equation}
|\bz_i^\sT (\hat{\btheta}_{\cup q}-\tilde{\btheta}_{\cup q})| \prec d^{-1} \tau_1^{-3/2},
\qquad |{\bz^\sPG}^\sT (\hat{\btheta}_{\cup q}-\tilde{\btheta}_{\cup q})| \prec d^{-1} \tau_1^{-3/2},
\end{equation}
\end{lemma}
\begin{proof}
For a smooth function $F:\R^p \to \R$, we will apply the representations
\begin{align*}
\nabla F(\by)&=\nabla F(\bx)+\Big(\int_0^1 \nabla^2 F(\bx+t(\by-\bx))
dt\Big)(\by-\bx)\\
&=\nabla F(\bx)+\nabla^2 F(\bx)(\by-\bx)
+\Big(\int_0^1 (1-t)\nabla^3 F(\bx+t(\by-\bx)) dt\Big)[\by-\bx,\by-\bx]
\end{align*}
where $\nabla^3 F(\cdot) \in \R^{p \times p \times p}$
denotes the (symmetric) third-derivative tensor of $F$,
and $\nabla^3 F(\cdot)[\by-\bx,\by-\bx] \in \R^p$
denotes its contraction along two axes with $\by-\bx$.

From the first-order condition for $\hat\btheta_{\cup q}$ to
minimize $\hcR_{\cup q}$, we have
\begin{align}
\nabla \widehat{\cR}_{\cup q}(\tilde{\btheta}_{\cup q}) &=
\nabla\widehat{\cR}_{\cup q}(\tilde{\btheta}_{\cup q})-\nabla\widehat{\cR}_{\cup
q}(\hat{\btheta}_{\cup q}) \nonumber \\
    &=\underbrace{\left(\int_0^1 \nabla^2 \widehat{\cR}_{\cup q}
\left(\hat{\btheta}_{\cup q}
+t(\tilde{\btheta}_{\cup q}-\hat{\btheta}_{\cup q})\right) dt
\right)}_{:=\tilde\bH}(\tilde{\btheta}_{\cup q}-\hat{\btheta}_{\cup q}).
\end{align}
We remark that by Lemma \ref{lemma:optinThetaPG} and
Corollary \ref{corollary:SurrogateMinimizersInGoodSetallsamples},
on an event of probability $1-d^{-C}$, there exists $K>0$ such that both
$\tilde \btheta_{\cup q},\hat\btheta_{\cup q} \in \widehat S_q(K)$.
Then $\hat{\btheta}_{\cup q}+t(\tilde{\btheta}_{\cup q}-\hat{\btheta}_{\cup q})
\in \widehat S_q(K)$ for all $t \in [0,1]$, so
Lemma \ref{lemma:sublevelgeometry}(c) holds also for the above interpolated
Hessian matrix $\tilde \bH$. In particular, $\tilde \bH$ is invertible on this
high-probability event, so rearranging the above gives
\begin{equation}\label{eq:optimization_error_integral_form}
\hat{\btheta}_{\cup q}-\tilde{\btheta}_{\cup q}
=-\tilde\bH^{-1} \nabla \widehat{\cR}_{\cup q}(\tilde{\btheta}_{\cup q}).
\end{equation}

We proceed to expand $\nabla \widehat \cR_{\cup q}(\tilde \btheta_{\cup q})$
around the LOO optimizer $\hat \btheta_{\setminus q}$.
From the first-order condition for $\hat \btheta_{\setminus q}$ as an optimizer
of $\hcR_{\setminus q}$,
\begin{align}
\nabla \widehat{\cR}_{\cup q}(\tilde{\btheta}_{\cup q})
&=\nabla \widehat{\cR}_{\cup q}(\tilde{\btheta}_{\cup q})
-\nabla\widehat{\cR}_{\backslash q}(\hat{\btheta}_{\backslash q})\\
&=\frac{1}{n}\ell'(\tilde y_q, \langle \tilde
\bz_q,\tilde{\btheta}_{\cup q}\rangle)\tilde \bz_q
+\nabla \widehat{\cR}_{\setminus q}(\tilde{\btheta}_{\cup q})
-\nabla\widehat{\cR}_{\backslash q}(\hat{\btheta}_{\backslash q})\\
&=\frac{1}{n}\ell'(\tilde y_q, \langle \tilde
\bz_q,\tilde{\btheta}_{\cup q}\rangle)\tilde \bz_q
+\nabla^2 \widehat{\cR}_{\backslash q}(\hat{\btheta}_{\backslash q})
(\tilde{\btheta}_{\cup q}-\hat{\btheta}_{\backslash q})\\
&\hspace{1in}+\left(\int_0^1 (1-t)\nabla^3 \widehat{\cR}_{\backslash q}\Big(
\hat{\btheta}_{\backslash q}
+t(\tilde\btheta_{\cup q}-\hat{\btheta}_{\backslash q})\Big) dt\right)
[\tilde{\btheta}_{\cup q}-\hat{\btheta}_{\backslash q},
\tilde{\btheta}_{\cup q}-\hat{\btheta}_{\backslash q}].
\end{align}
Let us write as shorthand
\[\bar\btheta(t)=\hat{\btheta}_{\backslash q}
+t(\tilde\btheta_{\cup q}-\hat{\btheta}_{\backslash q}).\]
Noting that
$\frac{1}{n}\ell'(\tilde y_q, \langle \tilde
\bz_q,\tilde{\btheta}_{\cup q}\rangle)\tilde \bz_q
+\nabla^2 \widehat{\cR}_{\backslash q}(\hat{\btheta}_{\backslash q})
(\tilde{\btheta}_{\cup q}-\hat{\btheta}_{\backslash q})
=\nabla \widetilde \cR_{\cup q}(\tilde \btheta_{\cup q})=0$
since $\tilde\btheta_{\cup q}$ optimizes $\widetilde \cR_{\cup q}$, the above
shows
\[\nabla \widehat{\cR}_{\cup q}(\tilde{\btheta}_{\cup q})
=\left(\int_0^1 (1-t)\nabla^3 \widehat{\cR}_{\backslash q}(\bar\btheta(t))
dt\right)[\tilde{\btheta}_{\cup q}-\hat{\btheta}_{\backslash q},
\tilde{\btheta}_{\cup q}-\hat{\btheta}_{\backslash q}].\]
Then, applying the explicit form
$\tilde{\btheta}_{\cup q}-\hat{\btheta}_{\backslash q}
={-}\frac{1}{n}\ell'(\tilde y_q,\<\tilde \bz_q,\tilde \btheta_{\cup q}\>)
\bH_{\setminus q}^{-1}\tilde\bz_q$ from
Lemma \ref{lemma:Psi_k_function_of_r_Phi_M_k_and_gamma_kr} and substituting into
\eqref{eq:optimization_error_integral_form},
\begin{align}
\bz_i^\sT(\hat{\btheta}_{\cup q}-\tilde{\btheta}_{\cup q})
={-}\frac{\ell'(\tilde y_q,\<\tilde \bz_q, \tilde{\btheta}_{\cup q}\>)^2}{n^2}
\,\bz_i^\sT \tilde \bH^{-1}\left(\int_0^1 (1-t) \nabla^3
\hcR_{\setminus q}(\bar\btheta(t))dt \right)
[\bH_{\backslash q}^{-1} \tilde\bz_q, \bH_{\backslash q}^{-1}\tilde\bz_q].
\end{align}
Decomposing (recall that we denote $\btau\cdot \nabla^3 \bGamma^\bW = \tau_1 \nabla^3 \Gamma_1^\bW + \tau_2 \nabla^3 \Gamma_2^\bW$)
\[\nabla^3 \hcR_{\setminus q}(\bar\btheta(t))
=\frac{1}{n}\sum_{j:j \neq q} \ell'''(y_j,\<\bar\btheta(t),\bz_j\>)
\bz_j^{\otimes 3}+\btau\cdot \nabla^3 \bGamma^\bW(\bar\btheta(t)),\]
we arrive finally at
\begin{align}
\bz_i^\sT(\hat{\btheta}_{\cup q}-\tilde{\btheta}_{\cup q})
&={-}\ell'(\tilde y_q,\<\tilde \bz_q, \tilde{\btheta}_{\cup q}\>)^2
\int_0^1 (1-t)\Bigg(\underbrace{\sum_{j:j \neq q} \ell'''(y_j,\<\bar\btheta(t),\bz_j\>)
\Big(\frac{\bz_j^\sT \tilde\bH^{-1} \bz_i}{n}\Big)
\Big(\frac{\bz_j^\sT \bH_{\setminus q}^{-1} \tilde \bz_q}{n}\Big)^2}_{:=A(t)}\\
&\hspace{1in}+\underbrace{\frac{1}{n^2}\{\btau\cdot\nabla^3 \bGamma^\bW(\bar\btheta(t)) \}
[\tilde \bH^{-1}\bz_i,\bH_{\backslash q}^{-1} \tilde\bz_q, \bH_{\backslash
q}^{-1}\tilde\bz_q]}_{:=B(t)}\Bigg)dt.\label{eq:zthetadiffform}
\end{align}

We proceed to bound the above terms $A(t)$ and $B(t)$ uniformly over
$t \in [0,1]$. For $A(t)$, we apply $\|\ell'''\|_\infty \prec 1$ by Assumption
\ref{assumption:loss} and $|\frac{\bz_j^\sT \bH_{\setminus q}^{-1} \tilde
\bz_q}{n}| \prec d^{-1} \tau^{-1/2}$ by Lemma
\ref{lemma:BernsteinInequalityHermiteFeaturesrt} to obtain, simultaneously over
all $t \in [0,1]$,
\[|A(t)|
\prec \frac{1}{d^2 \tau_1}\sum_{j:j \neq q}
\left|\frac{\bz_j^\sT \tilde\bH^{-1} \bz_i}{n}\right|
\prec \frac{1}{d^2\tau_1}\left(\frac{1}{n}\sum_{j:j \neq q}
(\bz_j^\sT \tilde\bH^{-1} \bz_i)^2\right)^{1/2},\]
the second equality following by Cauchy-Schwarz. In the case where
$\bz_j=\bz_j^\RF$ for each $j=1,\ldots,n$, applying 
the expansion $\bz_j=\sum_{k=0}^D \mu_k\bV_k\bh_k(\bx_j)$,
Cauchy-Schwarz, and $\|\bz_i\|_2 \prec d$ from Lemma \ref{lemma:yzbounds},
observe that
\begin{align*}
\frac{1}{n}\sum_{j=1}^n (\bz_j^\sT \tilde\bH^{-1} \bz_i)^2
&\leq \frac{\|\bz_i\|_2^2}{n} \cdot
\left\|\tilde \bH^{-1}\sum_{j=1}^n\bz_j\bz_j^\sT \tilde \bH^{-1}\right\|_\op\\
&\leq \frac{C\|\bz_i\|_2^2}{n} \cdot
\sum_{k=0}^D \mu_k^2\left\|\tilde \bH^{-1}\bV_k\sum_{j=1}^n
\bh_k(\bx_j)\bh_k(\bx_j)^\sT \bV_k^\sT \tilde \bH^{-1}\right\|_\op\\
&\prec \sum_{k=0}^D \|\mu_k\tilde\bH^{-1}\bV_k\bh_k(\bX)\|_\op^2
\end{align*}
where $\bh_k(\bX)=[\bh_k(\bx_1),\ldots,\bh_k(\bx_n)] \in \R^{B_{d,k} \times n}$.
More generally, in the case where some $\bz_j=\bz_j^\sPG$, expanding
$\bz_j=\sum_{k=0}^D \mu_k\bV_k\bh_k(\bx_j)+\mu_{>2}\bg_{j*}$, we may obtain
similarly
\[\frac{1}{n}\sum_{j=1}^n (\bz_j^\sT \tilde\bH^{-1} \bz_i)^2
\prec \sum_{k=0}^D \|\mu_k\tilde\bH^{-1}\bV_k\bh_k(\bX)\|_\op^2
+\|\mu_{>2}\tilde\bH^{-1}\bG_*\|_\op^2\]
where $\bG_*=[\bg_{1*},\ldots,\bg_{n*}] \in \R^{p \times n}$.
For $k=0,1,2$, we have by Lemmas \ref{lemma:sublevelgeometry}(c) and
\ref{lemma:BernsteinInequalityHermiteFeatures} that
\[\|\mu_k\tilde\bH^{-1}\bV_k\bh_k(\bX)\|_\op
\leq \|\tilde\bH^{-1/2}\|_\op
\|\tilde \bH^{-1/2}\bV_k\|_\op\|\bh_k(\bX)\|_\op \prec d \tau_1^{-1/2}.\]
For $k \geq 3$, we have by Lemmas \ref{lemma:sublevelgeometry}(c) and
\ref{lemma:OperatorNormZkFk} that
\[\|\mu_k\tilde\bH^{-1}\bV_k\bh_k(\bX)\|_\op
\leq |\mu_k|\|\tilde\bH^{-1}\|_\op \|\bV_k\bh_k(\bX)\|_\op \prec d.\]
For the term involving $\bG_*$, we have likewise
$\|\mu_{>2}\tilde\bH^{-1}\bG_*\|_\op
\prec \|\tilde\bH^{-1}\|_\op\|\bG_*\|_\op \prec d$, as $\|\bG_*\|_\op \prec d$.
Applying this bound in both cases, we obtain simultaneously over $t \in [0,1]$
that
\begin{equation}\label{eq:Atbound}
|A(t)| \prec d^{-1} \tau_1^{-3/2}.
\end{equation}

For the term $B(t)$, note that by Lemma \ref{lemma:optinThetaPG}
and Corollary \ref{corollary:SurrogateMinimizersInGoodSetallsamples},
with probability $1-d^{-C}$, both
$\hat{\btheta}_{\backslash q},\tilde{\btheta}_{\cup q} \in
\widehat S_q(K)$. Then also
$\bar \btheta(t) \in \widehat S_q(K)$ for each $t \in [0,1]$,
so Lemma \ref{lemma:sublevelgeometry} implies
$\bGamma^\bW(\btheta)= (\| \bV_+^\sT \btheta\|_2^2, L_\bW(\btheta))$
in a neighborhood of $\bar\btheta(t)$, so that $\btau \cdot\nabla^3 \bGamma^\bW (\btheta) = \tau_2 \nabla^3 L_\bW (\btheta)$. Then
\begin{align*}
B(t)&=n \cdot \E_{\bz^{\sPG},y^{\sPG}}\left[\ell_\test'''(y^\sPG,\<\bar\btheta(t),\bz^{\sPG}\>)\Big(\frac{(\bz^\sPG)^\sT \tilde\bH^{-1}\bz_i}{n}\Big)
\Big(\frac{(\bz^\sPG)^\sT \bH_{\setminus
q}^{-1}\tilde\bz_q}{n}\Big)^2\;\Bigg|\;\bW\right]\\
&=\E_{\bz_1^{\sPG},\ldots,\bz_n^{\sPG},y_1^{\sPG},\ldots,y_n^{\sPG}}
\left[\sum_{j=1}^n
\ell_\test'''(y_j^\sPG,\<\bar\btheta(t),\bz_j^{\sPG}\>)\Big(\frac{(\bz_j^\sPG)^\sT
\tilde\bH^{-1}\bz_i}{n}\Big)
\Big(\frac{(\bz_j^\sPG)^\sT \bH_{\setminus
q}^{-1}\tilde\bz_q}{n}\Big)^2\;\Bigg|\;\bW\right]
\end{align*}
where $\E_{\bz_1^{\sPG},\ldots,\bz_n^{\sPG},y_1^{\sPG},\ldots,y_n^{\sPG}}
[\;\cdot\mid \bW]$
denotes the expectation over $n$ independent test samples conditional on
$\bW$. This may be bounded in the same way as above, using 
\eqref{eq:thetazPGbound}, \eqref{eq:yPGbound}, and
Assumption \ref{ass:test_loss} for $\ell_\test(\cdot)$
to bound $|\ell_\test'''(y_j^\sPG,\<\bar\btheta(t),\bz_j^{\sPG}\>)| \prec 1$.
Then simultaneously over $t \in [0,1]$,
\[|B(t)| \prec d^{-1} \tau_1^{-3/2}.\]
Applying these bounds and $\|\ell'\|_\infty \prec 1$
to \eqref{eq:zthetadiffform} shows the first statement of the lemma, that
$|\bz_i^\sT(\hat\btheta_{\cup q}-\tilde\btheta_{\cup q})| \prec 1$.
The same argument holds for $\bz^\sPG$ in place of $\bz_i$,
showing the second statement of the lemma.
\end{proof}

We now conclude the proof of Lemma \ref{lemma:surrogateLOOauxcompare}. 

\begin{proof}[Proof of Lemma \ref{lemma:surrogateLOOauxcompare}, eq.\
\eqref{eq:PsiPhiauxcompare}] 
Applying $\hcR_{\cup q}(\btheta)=\frac{1}{n}\ell(\tilde y_q,\<\tilde \bz_q,\btheta\>)
+\hcR_{\setminus q}(\btheta)$ and a third-order Taylor expansion of
$\hcR_{\setminus q}(\btheta)$ around its optimizer $\hat\btheta_{\setminus q}$,
we get
 \begin{equation}
        \begin{aligned}
        \hcR_{\cup q}(\btheta)&=
\underbrace{\frac{1}{n}\ell(\tilde y_q,\<\tilde \bz_q ,  \btheta \>)
+\Phi_{\backslash q}+\frac{1}{2}(\btheta-\hat\btheta_{\backslash q})^{\sT}
\boldsymbol{H}_{\backslash q}(\btheta -\hat\btheta_{\backslash q})}_{=\widetilde
\cR_{\cup q}(\btheta)} \\
        &\hspace{0.3in}+\frac{1}{6n}\sum_{i:i \neq q} \ell'''(y_i,\<\check \btheta,\bz_i\>)[\boldsymbol{z}_i^{\sT}(\btheta-\hat\btheta_{\backslash q})]^3
+\frac{1}{6} \{ \btau \cdot \nabla^3 \bGamma^\bW(\check\btheta) \}[\btheta-\hat\btheta_{\setminus
q},\btheta-\hat\btheta_{\setminus q},\btheta-\hat\btheta_{\setminus q}]
        \end{aligned}
        \end{equation} 
for a point $\check\btheta$ between $\btheta$ and $\btheta_{\setminus q}$.
Lemma \ref{lemma:optinThetaPG} ensures with probability $1-d^{-C}$ that
$\hat\btheta_{\setminus q} \in \widehat S_q(K)$. Then, for any $\btheta
\in \widehat S_q(K)$, also $\check \btheta \in \widehat S_q(K)$, so Lemma
\ref{lemma:sublevelgeometry}(b) ensures that
$\btau \cdot \nabla^3\bGamma^\bW(\btheta)=\tau_2 \nabla^3 L_\bW(\btheta)$
in a neighborhood of $\check\btheta$. Then, on this event, the above gives
\begin{align}
\hcR_{\cup q}(\btheta)&=\widetilde \cR_{\cup q}(\btheta)
+\frac{1}{6n}\sum_{i:i \neq q} \ell'''(y_i,\<\check
\btheta,\bz_i\>)[\boldsymbol{z}_t^{\sT}(\btheta-\hat\btheta_{\backslash q})]^3\\
&\hspace{1in}+\frac{1}{6}\E_{\bz^{\sPG},y^{\sPG}}\Big[\ell_\test'''(y^{\sPG},\<\check\btheta,\bz^{\sPG}\>)[{\bz^\sPG}^\sT(\btheta-\hat\btheta_{\setminus
q})]^3 \;\Big|\; \bW\Big].
\end{align}
The difference between the minima of $\widehat \cR_{\cup q}$
and $\widetilde \cR_{\cup q}$ can be bounded by the maximum of 
$|\widehat \cR_{\cup q}-\widetilde \cR_{\cup q}|$
evaluated at the two respective minimizers,
\begin{equation}
    |\Psi_q-\Phi_q| \le \max\left\{|(\widehat \cR_{\cup q}-\widetilde \cR_{\cup
q})(\hat\btheta_{\cup q})|,\,|(\widehat \cR_{\cup q}-\widetilde \cR_{\cup
q})(\tilde\btheta_{\cup q})|\right\},
\end{equation}
Applying the above form for $\widehat \cR_{\cup q}-\widetilde \cR_{\cup q}$,
and applying also $|\hat\btheta_{\cup q}-\hat\btheta_{\setminus q}|
\leq |\hat\btheta_{\cup q}-\tilde\btheta_{\cup q}|
+|\tilde\btheta_{\cup q}-\hat\btheta_{\setminus q}|$,
\begin{align}
|\Psi_q-\Phi_q| &\prec \frac{1}{n}\sum_{i:i \neq q}
|\ell'''(y_i,\<\check\btheta,\bz_i\>)|
\left(\left|\bz_i^\sT (\hat \btheta_{\cup q}-\tilde{\btheta}_{\cup
q})\right|^3+\left|\bz_i^\sT (\tilde{\btheta}_{\cup q}-\hat\btheta_{\setminus
q})\right|^3\right)\label{eq:PsiPhidiffsquared}\\
&\qquad+\E_{\bz^{\sPG},y^{\sPG}}\Big[
|\ell_\test'''(y^{\sPG},\<\check\btheta,\bz^{\sPG}\>)|
\Big(\Big|{\bz^\sPG}^\sT(\hat\btheta_{\cup q}-\tilde\btheta_{\cup q})\Big|^3
+\Big|{\bz^\sPG}^\sT(\tilde\btheta_{\cup q}-\hat\btheta_{\setminus
q})\Big|^3\Big) \;\Big|\; \bW\Big].
        \end{align}
We have $\|\ell'''\|_\infty \prec 1$ by Assumption \ref{assumption:loss},
$|\bz_i^\sT (\hat \btheta_{\cup q}-\tilde{\btheta}_{\cup q})| \prec d^{-1}\tau_1^{-3/2}$ by
Lemma \ref{lemma:BernsteinInequalityHermiteFeaturesrt2}, and
\[|\bz_i^\sT(\tilde \btheta_{\cup q}-\hat\btheta_{\setminus q})|
=|\ell'(\tilde y_q,\<\tilde \bz_q,\tilde\btheta_{\cup q}\>)|
\cdot \frac{|\bz_i^\sT \bH_{\setminus q}^{-1}\tilde \bz_q|}{n}
\prec d^{-1} \tau_1^{-1/2}\]
by Lemma \ref{lemma:Psi_k_function_of_r_Phi_M_k_and_gamma_kr},
the bound $\|\ell'\|_\infty \prec 1$, and Lemma
\ref{lemma:BernsteinInequalityHermiteFeaturesrt}. Thus the first term of
\eqref{eq:PsiPhidiffsquared} is $\prec d^{-3}\tau_1^{-9/2}$. The second term 
of \eqref{eq:PsiPhidiffsquared} is bounded similarly, applying
Lemmas \ref{lemma:BernsteinInequalityHermiteFeaturesrt} and
 \ref{lemma:BernsteinInequalityHermiteFeaturesrt2} for $\bz^\sPG$ instead of
$\bz_i$, and \eqref{eq:thetazPGbound} and \eqref{eq:yPGbound} to bound
$|\ell_\test'''(y^\sPG,\<\check \btheta,\bz^\sPG\>)| \prec 1$. Thus
\[|\Psi_q-\Phi_q| \prec d^{-3} \prec n^{-3/2} \tau_1^{-9/2}.\]
This means that for any constant $C_0>0$, there exists $K>0$ for which
\[\cE=\{|\Psi_q-\Phi_q| \leq (\log d)^K(n^{-3/2}\tau_1^{-9/2})\} \text{ satisfies }
\P[\cE] \geq 1-d^{-C_0}.\]
The same arguments as leading to \eqref{eq:Ecomplementbound} show that
$\E[|\Psi_q-\Phi_q|\1\{\cE^c\}] \leq (\log d)^Kn^{-3/2}\tau_1^{-9/2}$ for a sufficiently
large choice of $C_0>0$, establishing \eqref{eq:PsiPhiauxcompare}.
\end{proof}

\subsection{Lindeberg replacement for the surrogate risk}
\label{sec:proof-of-theorem-vanishingquadraticsurrogate1theorem59}

\begin{lemma}\label{lemma:ExpectedValueDifferencePsiProjection}
For each $q \in \{1,\ldots,n\}$, let $\Psi_q(\bz_q^\RF,y_q^\RF)$ and
$\Psi_q(\bz_q^\sPG,y_q^\sPG)$ denote the values of $\Psi_q$ under the choices
$(\tilde \bz_q,\tilde y_q)=(\bz_q^\RF,y_q^\RF)$ and
$(\tilde \bz_q,\tilde y_q)=(\bz_q^\sPG,y_q^\sPG)$ respectively. Let $\E_q$
denote the expectation over $(\bx_q,\bg_{q*},\eps_q)$ defining
$(\bz_q^\RF,y_q^\RF)$ and $(\bz_q^\sPG,y_q^\sPG)$. Then there exists a constant
$c>0$ such that for all large $d$,
\begin{equation}  
\E\left|\mathbb{E}_q\left[\Psi_q(\bz_q^\RF,y_q^\RF)
-\Psi_q(\bz_q^\sPG,y_q^\sPG)\right]\right| \leq n^{-1}d^{-c}.
\end{equation}
\end{lemma}
\begin{proof}
By Lemma \ref{lemma:Psi_k_function_of_r_Phi_M_k_and_gamma_kr},
\begin{equation}
\Psi_q(\bz_q^\RF,y_q^\RF)-\Psi_q(\bz_q^\sPG,y_q^\sPG)
=\frac{1}{n}\mathcal{M}_{y_q^\RF}\left(\langle\bz_q^\RF,\hat\btheta_{\backslash
q}\rangle;\gamma_q(\bz_q^\RF)\right)-\frac{1}{n}\mathcal{M}_{y_q^\sPG}\left(\langle\bz_q^\sPG,
\hat\btheta_{\backslash q}\rangle; \gamma_q(\bz_q^\sPG)\right)
\end{equation}
where we denote by $\gamma_q(\bz_q^\RF)$ and $\gamma_q(\bz_q^\sPG)$ the values
of $\gamma_q$ in \eqref{eq:gammaq} for the two cases. Let us denote
\[\gamma_q^\text{const}=\E_q[\gamma_q(\bz_q^\RF)],\]
and first compare $\gamma_q(\bz_q^\RF)$ and $\gamma_q(\bz_q^\sPG)$ with
$\gamma_q^\text{const}$. By
Lemma \ref{lemma:gamma_concentration}, we have
$|\gamma_q(\bz_q^\RF)-\E_q \gamma_q(\bz_q^\RF)| \prec d^{-1/2}$ and
$|\gamma_q(\bz_q^\sPG)-\E_q \gamma_q(\bz_q^\sPG)| \prec d^{-1/2}$.
By the definition of $\gamma_q$, also
\begin{align*}
    \left|\E_q[\gamma_q(\bz_q^\RF)] - \E_q[\gamma_q(\bz_q^\sPG)]\right| &=
\frac{1}{n} \left| \Tr\left( \left(\E_q[\bz_q^\RF(\bz_q^\RF)^\sT] -
\E_q[\bz_q^\sPG(\bz_q^\sPG)^\sT]\right)\bH_{\backslash q}^{-1} \right) \right|.
\end{align*}
Using $\E_q[\bz_q^\RF(\bz_q^\RF)^\sT] = \bV\bV^\sT$ and
$\E_q[\bz_q^\sPG(\bz_q^\sPG)^\sT]=\bV_{\leq 2}\bV_{\leq 2}^\sT+\mu_{>2}^2\bI$
where $\bV_{\leq 2}$ contains the components of $\bV$ for $k=0,1,2$, this
implies
\begin{align}
\left|\E_q[\gamma_q(\bz_q^\RF)] - \E_q[\gamma_q(\bz_q^\sPG)]\right|
&\leq \frac{1}{n} \left|\bH_{\backslash q}^{-1/2}\Tr\left( (\bV\bV^\sT-(\bV_{\leq 2}\bV_{\leq
2}^\sT+\mu_{>2}^2\bI))\bH_{\backslash q}^{-1/2} \right) \right|\\
&\leq \frac{p}{n} \left\|\bH_{\backslash q}^{-1/2}\bigg(\sum_{k=3}^D
\mu_k^2(\bV_k\bV_k^\sT-\bI)\bigg)
\bH_{\backslash q}^{-1/2}\right\|_{\op}. \label{eq:trace_bound_op_norm}
\end{align}
Then applying Lemma \ref{lemma:OperatorNormFk},
\begin{align*}
\left|\E_q[\gamma_q(\bz_q^\RF)] - \E_q[\gamma_q(\bz_q^\sPG)]\right|
&\prec \left\|\bH_{\backslash q}^{-1/2} \left(\frac{\mu_3^2}{d}\bW\bW^\sT +
\frac{\mu_4^2}{d^2}\mathbf{1}_p\mathbf{1}_p^\sT\right) \bH_{\backslash q}^{-1/2}
\right\|_{\op}+\frac{1}{\sqrt{d}}\left\|\bH_{\backslash q}^{-1}\right\|_{\op}\\
&\prec \frac{\mu_3^2}{d} \left\| \bW^\sT\bH_{\backslash q}^{-1}\bW
\right\|_{\op} + \frac{\mu_4^2}{d^2} \left| \mathbf{1}_p^\sT\bH_{\backslash
q}^{-1}\mathbf{1}_p \right|+\frac{1}{\sqrt{d}}\left\|\bH_{\backslash q}^{-1}\right\|_{\op}.
\end{align*}
Applying Lemma \ref{lemma:sublevelgeometry}(c), we have
$\|\bW^\sT\bH_{\backslash q}^{-1}\bW\|_{\op},
|\mathbf{1}_p^\sT\bH_{\backslash q}^{-1}\mathbf{1}_p| \prec \tau_1^{-1}$, and $\|\bH_{\setminus q}\|_\op
\prec 1$. Thus, for $\tau_1 \geq d^{-1/2}$,
$|\E_q[\gamma_q(\bz_q^\RF)] - \E_q[\gamma_q(\bz_q^\sPG)]| \prec d^{-1/2}$, so
\[|\gamma_q(\bz_q^\RF)-\gamma_q^\text{const}| \prec d^{-1/2},
\qquad |\gamma_q(\bz_q^\sPG)-\gamma_q^\text{const}| \prec d^{-1/2}.\]

Together with the Lipschitz bound for $\gamma \mapsto \cM_y(z;\gamma)$ from
Lemma \ref{lemma:MathematicalExpressionLipschitzInBrAndGammaKpure}, this shows
\begin{align}
\left|\mathcal{M}_{y_q^\RF}\left(\langle\bz_q^\RF,\hat\btheta_{\backslash
q}\rangle;\gamma_q(\bz_q^\RF)\right)
-\mathcal{M}_{y_q^\RF}\left(\langle\bz_q^\RF,\hat\btheta_{\backslash
q}\rangle;\gamma_q^\text{const}\right)\right| &\prec d^{-1/2},
\label{eq:moreauapprox1}\\
\left|\mathcal{M}_{y_q^\sPG}\left(\langle\bz_q^\sPG,
\hat\btheta_{\backslash q}\rangle;\gamma_q(\bz_q^\sPG)\right)
-\mathcal{M}_{y_q^\sPG}\left(\langle\bz_q^\sPG,
\hat\btheta_{\backslash q}\rangle;\gamma_q^\text{const}\right)\right| &\prec
d^{-1/2}.\label{eq:moreauapprox2}
\end{align}
Finally, recall that $\bz_q^\RF=\sigma( \bW\bx_q)$ and
$y_q^\RF=\eta(\bx_{q,S},\bbeta_2^\sT \bh_2(\bx_q),\ldots,\bbeta_{D'}^\sT
\bh_{D'}(\bx_q),\eps_q)$. By the Lipschitz bound for $(y,z) \mapsto
\cM_y(z;\gamma_q^\text{const})$ from Lemma
\ref{lemma:MathematicalExpressionLipschitzInBrAndGammaKpure} and the Lipschitz
assumption for $\eta(\cdot)$ in Assumption \ref{assumption:target}, observe
that
\[\mathcal{M}_{y_q^\RF}\Big(\langle\bz_q^\RF,\btheta\rangle;
\gamma_q^\text{const}\Big)
=\varphi_{\eps_q}(\bx_{q,S},\bbeta_2^\sT \bh_2(\bx_q),
\ldots,\bbeta_D^\sT \bh_D(\bx_q),\btheta^\sT \sigma(\bW\bx_q))\]
for a (random, depending on $\eps_q$) function $\varphi_{\eps_q}(\cdot)$ that
is $C$-Lipschitz for a constant $C>0$.
Then Theorem
\ref{theorem:RecallBoundedLipschitzFunctionSupErgetisotropic} implies that for
any $K>0$,
\[\sup_{\btheta \in \bTheta_\bW^\sPG(K)}
\left|\mathcal{M}_{y_q^\RF}\Big(\langle\bz_q^\RF,\btheta\rangle;
\gamma_q^\text{const}\Big)-\mathcal{M}_{y_q^\sPG}\Big(\langle\bz_q^\sPG,\btheta\rangle;
\gamma_q^\text{const}\Big)\right|
\prec d^{-c}\]
for some constant $c>0$. Since Lemma \ref{lemma:optinThetaPG} ensures that
$\hat\btheta_{\backslash q} \in \bTheta_{\bW}^{\sPG}(K)$ with probability
$1-d^{-C}$ for any $C>0$ some $K \equiv K(C)>0$, this implies
\begin{equation}\label{eq:moreauswap}
\left|\mathcal{M}_{y_q^\RF}\Big(\langle\bz_q^\RF,\hat\btheta_{\setminus
q}\rangle;
\gamma_q^\text{const}\Big)-\mathcal{M}_{y_q^\sPG}\Big(\langle\bz_q^\sPG,\hat\btheta_{\setminus
q}\rangle;\gamma_q^\text{const}\Big)\right| \prec d^{-c}.
\end{equation}
Combining \eqref{eq:moreauapprox1}, \eqref{eq:moreauapprox2}, and
\eqref{eq:moreauswap} shows that for any $C_0>0$, there exist $K,c>0$ for which
\[\cE=\{|\mathbb{E}_q[\Psi_q(\bz_q^\RF,y_q^\RF)-\Psi_q(\bz_q^\sPG,y_q^\sPG)]|
\leq n^{-1}d^{-c}\} \text{ satisfies } \P[\cE] \geq 1-d^{-C_0}.\]
The same arguments as leading to \eqref{eq:Ecomplementbound} show that
$\E[|\Psi_q(\bz_q^\RF,y_q^\RF)-\Psi_q(\bz_q^\sPG,y_q^\sPG)|\1\{\cE^c\}]
\leq n^{-1}d^{-c}$ for a sufficiently large choice of $C_0>0$, concluding
the proof.
\end{proof}

\subsection{Proof of Theorem \ref{sec:main-conditions-lindeberg-swap-framework}}
\label{app:putting_things_together_phase_i_lindeberg}

\begin{proof}[Proof of Theorem \ref{sec:main-conditions-lindeberg-swap-framework}] For brevity, denote by $\bt_q=(\bz_q,y_q)$ the $q$-th data point.  
We interpolate between the two minimum risks by sequentially swapping one sample at a time between the RF and PGE models.  
By the standard Lindeberg telescoping argument,
\[
\begin{aligned}
    &~ \left| \E \Big[ \varphi \left( \hcR_{n,p}^* (\btau,\bZ^\RF,\boldf^\RF) \right) \Big] - \E \Big[ \varphi \left( \hcR_{n,p}^* (\btau, \bZ^\sPG,\boldf^\sPG) \right) \Big] \right| \\\leq&~ \sum_{q=1}^n \left|\E \big[ \varphi \big( \Phi_{q} ( \bt_q^\RF) \big) \big] - \E \big[ \varphi \big( \Phi_{q} ( \bt_q^\sPG) \big) \big]  \right|.
\end{aligned}
\]
For each $q$, we perform a Taylor expansion of $\varphi(\Phi_q)$ around the leave-one-out (LOO) risk $\Phi_{\setminus q}$:
\[
\varphi(\Phi_{ q} (\bt_q) ) = \varphi(\Phi_{\setminus q}) + \varphi' (\Phi_{\setminus q}) ( \Phi_{q}(\bt_q) - \Phi_{\setminus q}) + \frac{1}{2} \varphi '' (\zeta_q) ( \Phi_{ q}(\bt_q) - \Phi_{\setminus q})^2,
\]
for some $\zeta_q$ between $\Phi_q(\bt_q)$ and $\Phi_{\setminus q}$.
Applying this expansion to both $ \{\bt_q^\RF, \bt_q^\sPG \}$, and taking their difference, we obtain
\begin{equation}\label{eq:lindeberg_i_taylor_swap}
\begin{aligned}
    \Big|\E \big[ \varphi \big( \Phi_{q} ( \bt_q^\RF) \big) \big] - \E \big[ \varphi \big( \Phi_{q} ( \bt_q^\sPG) \big) &\big] \Big| \leq \| \varphi ' \|_\infty \E \left[\left| \E_{q} [  \Phi_{q} ( \bt_q^\RF) -  \Phi_{q} ( \bt_q^\sPG)] \right|\right]\\
    & + \frac{\| \varphi ''\|_\infty}{2} \left\{ \E [(\Phi_{ q} ( \bt_q^\RF) - \Phi_{\setminus q} )^2 ] + \E [(\Phi_{q} ( \bt_q^\sPG) - \Phi_{\setminus q} )^2 ]\right\},
\end{aligned}
\end{equation}
where $\E_q$ is the expectation with respect to the $q$-th sample $\bt_q$. For the first order term, we introduce the quadratic surrogate problem $\Psi_q$ and decompose
\[
\begin{aligned}
    \left| \E_{q} [  \Phi_{ q} ( \bt_q^\RF) -  \Phi_{ q} ( \bt_q^\sPG)] \right| \leq&~ \E_{q} [ |\Phi_{q} ( \bt_q^\RF) - \Psi_q ( \bt_q^\RF)| ] + \E_{q} [ |\Phi_{ q} ( \bt_q^\sPG) - \Psi_q ( \bt_q^\sPG)| ] \\
    &~+ \left| \E_q \left[\Psi_q ( \bt_q^\RF) - \Psi_q ( \bt_q^\sPG) \right] \right|.
\end{aligned}
\]
Similarly, for the second order  term,
\[
\E \left[\left(\Phi_{ q} (\bt_q) - \Phi_{\setminus q} \right)^2 \right] \leq 2 \E \left[\left(\Phi_{ q}  (\bt_q) - \Psi_{ q}  (\bt_q) \right)^2 \right] + 2\E \left[\left(\Psi_{q}  (\bt_q) - \Phi_{\setminus q}  \right)^2 \right].
\]
By Lemma \ref{lemma:surrogateLOOauxcompare}, for all $q \in [n]$ and $\bt_q \in \{ \bt^\RF_q, \bt^\sPG_q \}$,
\[
\E \left[\left(\Psi_{q}  (\bt_q) - \Phi_{\setminus q}  \right)^2 \right] \leq \frac{(\log d)^K}{n^2}, \qquad \E_{q} [ |\Phi_{q} ( \bt_q) - \Psi_q ( \bt_q)| ] 
\leq \frac{(\log d)^K}{n^{3/2} \tau_1^{9/2}},
\]
and by Lemma \ref{lemma:ExpectedValueDifferencePsiProjection},
\[
\E \left| \E_q \left[\Psi_q ( \bt_q^\RF) - \Psi_q ( \bt_q^\sPG) \right] \right| \leq \frac{1}{n d^c}.
\]
Substituting these bounds in the previous displays and combining them in \eqref{eq:lindeberg_i_taylor_swap}, we get
\[
  \left|\E \big[ \varphi \big( \Phi_{q} ( \bt_q^\RF) \big) \big] - \E \big[ \varphi \big( \Phi_{q} ( \bt_q^\sPG) \big) \big] \right| \leq \frac{1}{n d^{c'} \tau_1^{9/2}}.
\]
Summing over $q=1,\dots,n$ establishes the desired bound.
\end{proof}

\newpage

\section{Lindeberg Phase II: Swapping   Second-Order Chaos}
\label{sec:lindeberg-phase-ii}

We now consider the second phase of the Lindeberg swapping argument, which interpolates between the Partial Gaussian Equivalent (PGE) model and the Conditional Gaussian Equivalent (CGE) model, by replacing the degree-2 Hermite features that are not supported on $\bx_S$ with independent Gaussian noise.
The proof of Theorem~\ref{sec:main-conditions-lindeberg-swap-framework2} follows an identical outline as the first Lindeberg swapping, including the definitions of the leave-one-out, augmented, and surrogate objectives, as well as the technical results (Lemmas~\ref{lemma:sublevelgeometry}–\ref{lemma:ExpectedValueDifferencePsiProjection}). For brevity, we omit this redundant exposition.

We therefore focus on the key distinction between the two phases, which is proving the inclusion of the minimizers in $\bTheta^{\sCG}_\bW \subseteq \bTheta^{\sPG}_\bW$.

\begin{theorem}\label{thm:inclusion_theta_CGE}
    For any constant $C>0$, there exist constants $K,\eps,K_\Gamma>0$ such that
with probability at least $1-d^{-C}$, the following holds. For each $q=1,\ldots,n$, choice of
$(\tilde \bz_q,\tilde y_q) \in \{(\bz_q^\sPG,y_q^\sPG),(\bz_q^\sCG,y_q^\sCG)\}$,
and $\hcR$ given by either $\hcR_{\setminus q}$ or $\hcR_{\cup q}$, there exists
a unique minimizer $\hat\btheta$ of $\hcR$, which furthermore satisfies
\[\hat\btheta \in \bTheta_\bW^{\sCG}(\eps,K) \cap \widehat S_q(K).\]
\end{theorem}

Lemma \ref{lemma:optinThetaPG} (now with CGE data) already showed that $\hbtheta \in \bTheta_\bW^{\sPG}(K) \cap \widehat S_q(K)$ with this probability. Thus, it remains to show that there exists a constant $\eps >0$ such that 
\[
\bigl\|\bW_{\setminus S}^{\sT}\bD_{\hbtheta}\,
                 \bW_{\setminus S}\bigr\|_{\op}
              \leq d^{-\eps},\quad
              \bigl\|\bW_S^{\sT}\bD_{\hbtheta}\,
                 \bW_{\setminus S}\bigr\|_{F}
              \leq d^{-\eps}.
\]

\subsection{Tensor formulation and regularized LODO risk}
\label{subsec:phaseII-tensor}

It will be convenient to introduce a tensor-based formulation for the degree-2 chaos (see Section~\ref{section:HermitePolynomials}). For $k=2$, the symmetric Hermite tensor is given by $\bH_2 (\bx) = \frac{1}{\sqrt{2}} (\bx^{\otimes 2} - \bI_d) \in (\R^{d})^{\odot 2}$.  Define the order-$3$ weight tensor $\cV_2 \in \R^{p \times d \times d}$ whose $j$-th slice is $(\cV_2)_j := \bw_j^{\otimes 2} \in (\R^{d})^{\odot 2}$. The tensor contraction of $\cV_2$ with $ \bH_2 (\bx)$ is then
\[
\cV_2 \bH_2 (\bx) = ( \< (\cV_2)_j ,\bH_2 (\bx)\>)_{j \in [p]} = (\He_2 (\<\bw_j,\bx\>))_{j \in [p]} \in \R^p .
\]
Denote by $\btheta^\sT \cV_2$ the linear combination of the slices with weights $\btheta$
\[
\btheta^\sT \cV_2 = \sum_{j = 1}^p \theta_j (\cV_2)_{j} =  \sum_{j = 1}^p \theta_j \bw_j \bw_j^\sT = \bW^\sT \bD_\btheta \bW, \qquad \bD_\btheta := \diag (\theta_j)_{j \in [p]}.
\]
Let $S = \{1,\ldots , s\} \subset [d]$ be the indices of the signal support. For any order-$2$ tensor $\cA$, let $\proj_S \cA$ be the projection onto the entries whose indices both lie in $S$, and define $\proj_{S,\perp} \cA = \cA - \proj_S \cA$. Then
\[
\btheta^\sT \proj_S \cV_2 =  \begin{pmatrix}
    \bW_S^\sT \bD_\btheta \bW_S & \bzero \\
    \bzero & \bzero
\end{pmatrix}, \qquad \btheta^\sT \proj_{S,\perp} \cV_2 =  \begin{pmatrix}
   \bzero &  \bW_{S}^\sT \bD_\btheta \bW_{\setminus S} \\
    \bW_{\setminus S}^\sT \bD_\btheta \bW_{S} &  \bW_{\setminus S}^\sT \bD_\btheta \bW_{\setminus S}.
\end{pmatrix}
\]

For convenience, let $\cV_1 := \bW \in \R^{p \times d}$, and define the full design tensor $\cV$ by concatenation:
\[
  \mathcal{V}
  :=
  \bigl[
      \mu_{0}\1_{p}\;
      \bigm|\;
      \mu_{1}\mathcal{V}_{1}\;
      \bigm|\;
      \mu_{2}\mathcal{V}_{2}\;
      \bigm|\;
      \mu_{>2}\id_{p}
  \bigr].
\]
We similarly define the stacked chaos  covariates $\bh_i$ for the PGE and CGE models
\[
\bh_i^\sPG := \begin{pmatrix}
    1 \\
    \bx_i \\
    \bH_2 (\bx_i) \\
   \bg_{i*}
\end{pmatrix}, \qquad \bh_i^\sCG := \begin{pmatrix}
    1 \\
    \bx_i \\
    \proj_{S}\bH_2 (\bx_i) + \proj_{S,\perp} \bG_{i2} \\
    \bg_{i*}
\end{pmatrix},
\]
where $\bG_{i2}  = \iota (\bg_{i2})$ and $\iota(\cdot)$ is the isometry described in Section \ref{section:HermitePolynomials}. With these notations, the feature vectors can be expressed as
\[
\begin{aligned}
\bz_i^\sPG  =&~ \cV \bh_i^\sPG = \mu_0 \1_p + \mu_1\cV_1 \bx_i + \mu_2 \cV_2 \bH_2 (\bx_i) + \mu_{>2} \bg_{i*}, \\
\bz_i^\sCG  =&~ \cV \bh_i^\sCG = \mu_0 \1_p + \mu_1\cV_1 \bx_i + \mu_2 \cV_2 \proj_{S}\bH_2 (\bx_i) + \mu_2 \cV_2 \proj_{S,\perp} \bG_{i2} + \mu_{>2} \bg_{i*}, \\
\end{aligned}
\]
and the corresponding responses are
\[
\begin{aligned}
y_i^\sPG=&~\eta\big(f_i^\sPG;\eps_i), \qquad f_i^\sPG := \{\bx_{iS},\bbeta_2^\sT \bh_2(\bx_i),
\bbeta_3^\sT \bg_{i3},\ldots,\bbeta_{D'}^\sT \bg_{iD'}\}, \\
y_i^\sCG=&~\eta\big(f_i^\sCG;\eps_i), \qquad f_i^\sCG := \{\bx_{iS},\bbeta_2^\sT \bg_{i2},
\bbeta_3^\sT \bg_{i3},\ldots,\bbeta_{D'}^\sT \bg_{iD'}\},
\end{aligned}
\]
where $\bbeta_2^\sT \bh_2(\bx_i) = ( \< \bbeta_{2i}, \bh_2 (\bx_i) \>)_{i \in [s_2]} = \{ \< \bB_{2i}, \bH_2 (\bx_i) \>\}_{i \in [s_2]}$ with $\bB_{2i} := \iota (\bbeta_{2i}) \in \R^{d \times d}$. 

Finally, let $\hbtheta$ be the minimizer of the empirical risk $\widehat{\cR}_{n,p} (\btheta;\btau,\bZ,\boldf)$, where $(\bZ,\boldf)$ is a data matrix involving both PGE and CGE samples, as appears in the Lindeberg interpolation.

\paragraph{Leave-One-Direction-Out (LODO) objectives.}
Our goal is to bound 
\[
\| \bW_{\setminus S}^\sT \bD_{\hbtheta} \bW_{\setminus S} \|_\op = \sup_{\bv \in \S^{d-1}, \; \proj_S \bv = 0} \;\; \sum_{j = 1}^p \hat\theta_j \< \bv, \bw_j \>^2,
\]
where $\proj_S \bv$ is the projection onto the support coordinates $S = \{1,\ldots, s\}$. Decompose
\[
\sum_{j = 1}^p \hat\theta_j \< \bv, \bw_j \>^2 = \sum_{j = 1}^p \hat\theta_j (  \< \bv, \bw_j \>^2 - d^{-1} ) + d^{-1} \< \hat\btheta , \1_p\>,
\]
where the second term is negligible, using that $\< \hat\btheta , \1_p\> \prec 1$ by Lemma \ref{lemma:sublevelgeometry}. For the first term, if $\hbtheta$ were independent of $(\< \bv, \bw_j \>)_{j \in [p]}$, then, over the randomness in $\bW$, 
\begin{equation}\label{eq:thought-exp-bound}
\sum_{j = 1}^p \hat\theta_j (  \< \bv, \bw_j \>^2 - d^{-1} ) \prec \| \hat\btheta \|_\infty \prec d^{-1/4},
\end{equation}
since $(  \< \bv, \bw_j \>^2 - d^{-1} )$ are i.i.d.~with variance $d^{-2}$, and $ \hat\btheta \in \bTheta^\sPG_\bW$ with high probability. 

In reality, $\hbtheta$ is not independent of the $\{\<\bv,\bw_j\>\}_{j \in [p]}$, though we expect it to be nearly so. To formalize this, we compare $\hbtheta$ with $\hbtheta_{-\bv}$, the minimizer of a \emph{leave-one-direction-out} (LODO) objective, obtained by removing the direction $\bv$ from the weights matrix $\bW$ (and normalizing the obtained weights). Thus $\hbtheta_{-\bv}$ is independent of $\{\<\bv,\bw_j\>\}_{j \in [p]}$, and we can apply \eqref{eq:thought-exp-bound}. If we can bound $\| \hbtheta - \hbtheta_{-\bv} \|_2 \prec d^{-\eps}$, $\<\hbtheta_{-\bv},\1_p \> \prec 1$, and $\| \hbtheta_{-\bv}\|_\infty \prec d^{-\eps}$ then
\[
\sum_{j = 1}^p \hat\theta_j \< \bv, \bw_j \>^2 \prec d^{-1}|\<\hbtheta_{-\bv},\1_p \>| + d^{-1}|\<\hbtheta,\1_p \>| + \| \hbtheta_{-\bv} - \hbtheta\|_2 + \| \hbtheta_{-\bv}\|_\infty \prec d^{-\eps}.
\]
The main difficulty lies in controlling these quantities \emph{uniformly} over all $\bv \in \S^{d-1}$, $\proj_S\bv=0$, in particular $\| \hbtheta_{-\bv}\|_\infty$. To achieve this, we note that $\|\hbtheta\|_\infty \prec d^{-1/4}$. Hence, adding a regularization term $\kappa_d\|\cdot\|_\infty$ to the LODO objective should not significantly alter the minimizer provided $\kappa_d \ll d^{1/4}$. Choosing $\kappa_d = d^{1/8}$ ensures uniform control $\| \hbtheta_{- \bv}\|_\infty \leq (\log d)^K d^{-1/8}$ over all $\bv$.

Thus, we introduce the following regularized objective and regularized LODO objective:
\begin{align}
  \cP_{n}(\btheta)
  &:= \frac1n\sum_{i=1}^{n}
        \ell\bigl(
          y_{i},\<
          \btheta, \bz_i\> 
        \bigr)
      +\frac{\lambda}{2}\|\btheta\|_{2}^{2}
      +\btau\cdot\bGamma^{\bW}(\btheta)
      +\kappa_{d}\|\btheta\|_{\infty},
      \label{eq:Pn-main}
      \\
  \cP_{n,-\bv}(\btheta)
  &:=\frac1n\sum_{i=1}^{n}
        \ell\bigl(
          y_{i,-\bv},\<\btheta, \bz_{i,-\bv}\>
        \bigr)
      +\frac{\lambda}{2}\|\btheta\|_{2}^{2}
      +\btau\cdot\Gamma^{\bW_{-\bv}}(\btheta)
      +\kappa_{d}\|\btheta\|_{\infty},
      \label{eq:Pn-v-main}
\end{align}
where $(y_{i,-\bv},\bz_{i,-\bv})$ corresponds to data where $\bx_i,\bW$, and $\bB_{2i}$ are replaced by 
\[
\bx_{i,-\bv} := (\bI - \bv\bv^\sT) \bx_i, \qquad \bW_{-\bv} := (\bw_{j,-\bv}/\| \bw_{j,-\bv}\|_2 )_{j \in [p]}, \qquad \bB_{2i,-\bv}:= (\bI - \bv\bv^\sT)\bB_{2i} (\bI - \bv\bv^\sT),
\]
respectively, with $\bw_{j,-\bv} := (\id-\bv\bv^{\sT})\bw_j$. Let $\cV_{1,-\bv}$, $\cV_{2,-\bv}$, and $\cV_{-\bv}$ be the associated weight tensors, and let $\bh_{i,-\bv}$ be the   stacked chaos with $\bx_i$ replaced by $(\bI - \bv\bv^\sT) \bx_i$,  $\bH_2(\bx_i)$ replaced by $(\id -\bv\bv^\sT)\bH_2(\bx_i)(\id-\bv\bv^\sT)$. With slight abuse of notation, we also denote $\bH_2(\bx_{i,-\bv}) := (\id -\bv\bv^\sT)\bH_2(\bx_i)(\id-\bv\bv^\sT)$.
Note that $\cV_{-\bv} \bh_{i,-\bv} = \cV_{-\bv} \bh_i $, but  $\cV_{-\bv} \bh_{i,-\bv}  \neq \cV \bh_{i,-\bv}$ because of the normalization in $\bW_{-\bv}$. We will denote $(\bz_{i,-\bv},f_{i,-\bv})$ and $y_{i,-\bv}$ the LODO data.

Consider
\[
\check{\btheta}
  \in\argmin_{\btheta}\cP_{n}(\btheta),\qquad  \check{\btheta}_{-\bv}
  \in\argmin_{\btheta}\cP_{n,-\bv}(\btheta),
\]
some arbitrary minimizers. We will show in fact that $\check{\btheta}$ and $\check{\btheta}_{-\bv}$ are unique uniformly over $\bv$ with high probability. 
Note that because the weights
\(\mathcal{V}_{-\bv}\) are independent of
\(\{\langle\bw_j,\bv\rangle\}_{j=1}^{p}\), so is
the minimizer \(\check{\btheta}_{-\bv}\).

\subsection{Uniform bounds for the LODO objective}
\label{sec:uniform_convergencefall}

In this section, $(\bz,y)$ may refer to either PGE or CGE data.
In the proofs, unless stated otherwise, we take $(\bz,y)=(\bz^\sPG,y^\sPG)$, as this case is more involved to analyze; modifications for the CGE setting will be noted where necessary.
Let $V_S=\mathrm{span}\{\be_j : j\in S\}$ denote the signal subspace, and let $\mathcal{N}_{S^\perp}$ be a fixed $1/4$-net of the unit sphere in its orthogonal complement $V_S^\perp$, with cardinality satisfying $|\mathcal{N}_{S^\perp}|\le 9^d$.
All results in this section will be established uniformly over $\bv\in\mathcal{N}_{S^\perp}$.

    \begin{lemma}\label{lem:firstlemma}
 There exist constants $c, C,d_0 > 0$, depending only on the constants in the assumptions, such that for all $d \geq d_0$, the following holds with probability at least $1-e^{-Cd}$:
\begin{align}\label{eq:letusfirst}
    \sup_{\bv \in \mathcal{N}_{S^\perp}} \frac{1}{n} \sum_{i=1}^{n} (y_i - y_{i,-\bv})^2 \leq d^{-c},
\end{align}
Moreover, for every $k \geq 0$, there exists a constant $c_k >0$ such that
\begin{equation}\label{eq:LODO_exp_f_f_bv}
    \sup_{\bv \in \mathcal{N}_{S^\perp}} \E \left[ \left\| f - f_{-\bv}\right\|_2^k\right] \leq  d^{-c_k}.
\end{equation}
\end{lemma}

    \begin{proof}
By the Lipschitz property of $\eta$ in Assumption~\ref{assumption:target}, we have
\begin{align}\label{eq:proof_reduction_sum_revised}
    \frac{1}{n} \sum_{i=1}^{n}(y_i - y_{i,-\bv})^2 \leq C \sum_{j=1}^{s_2} \frac{1}{n} \sum_{i=1}^{n} \left( \Delta\xi_{2j}(\bx_i) \right)^2,
\end{align}
where, from the form $\bH_2(\bx)=\frac{1}{\sqrt{2}}(\bx^{\otimes 2}-\bI_d)$,
\[
\begin{aligned}
    \Delta\xi_{2,j}(\bx_i) :=&~ \langle\bB_{2j}, \bH_2(\bx_i) -  \bH_2(\bx_{i,-\bv}))\rangle 
    = \frac{1}{\sqrt{2}} \left( 2(\bx_i^\sT \bB_{2j} \bv)(\bv^\sT \bx_i) - (\bv^\sT \bB_{2j} \bv)[(\bv^\sT \bx_i)^2+1] \right).
\end{aligned}
\]
Hence,
\begin{equation}\label{eq:delta_xi_2j_bound}
\begin{aligned}
    \frac{1}{n}\sum_{i=1}^n (\Delta\xi_{2j}(\bx_i))^2 
    \leq&~  4 \bv^\sT \bB_{2j} \left( \frac{1}{n}\sum_{i=1}^n (\bv^\sT \bx_i)^2 \bx_i\bx_i^\sT \right) \bB_{2j} \bv + 2(\bv^\sT \bB_{2j} \bv)^2 \left( \frac{1}{n}\sum_{i=1}^n (\bv^\sT \bx_i)^4 +1\right)\\
    \leq&~ 4 \| \bB_{2j} \|_\op^2 \left\| \frac{1}{n}\sum_{i=1}^n (\bv^\sT \bx_i)^2 \bx_i\bx_i^\sT\right\|_\op + 2\| \bB_{2j} \|_\op^2 \left[\frac{1}{n}\sum_{i=1}^n (\bv^\sT \bx_i)^4+1\right].
\end{aligned}
\end{equation}

We first bound the second term.
A standard tail bound implies that
$\P[\sum_{i=1}^n (\bx_i^\sT \bv)^4-\E(\bx_i^\sT \bv)^4 \geq t]
\leq 2e^{-c\min(t^{1/2},t^2/n)}$ for some constant $c>0$. For any $C>0$, choosing $t \asymp n \asymp d^2$ yields a constant $C'>0$ sufficiently large
such that $\sum_{i=1}^n (\bx_i^\sT \bv)^4 \leq C' n$ with probability at
least $1 - e^{-Cd}$. By union bound, with probability at least $1 - e^{-C d}$,
\[
\sup_{\bv \in \mathcal{N}_{S^\perp}} \frac{1}{n } \sum_{i = 1}^n \<\bx_i,\bv\>^4 \leq C'.
\]
Next, consider the first term in~\eqref{eq:delta_xi_2j_bound}.
We decompose it as
\begin{equation}\label{eq:first_term_2_4_op_norm}
\begin{aligned}
\left\| \frac{1}{n}\sum_{i=1}^n (\bv^\sT \bx_i)^2 \bx_i\bx_i^\sT\right\|_\op \leq&~ 2 \left\| \frac{1}{n} \sum_{i =1}^n (\bv^\sT \bx_i)^2  \bx_{i,S} \bx_{i,S}^\sT \right\|_\op + 2 \left\| \frac{1}{n} \sum_{i =1}^n (\bv^\sT \bx_i)^2  \bx_{i,\setminus S} \bx_{i, \setminus S}^\sT \right\|_\op \\
\leq&~ \frac{2}{n} \sum_{i =1}^n (\bv^\sT \bx_i)^2  \| \bx_{i,S}\|_2^2 + 4 \sup_{\bu \in \mathcal{N}_{S^\perp}} \frac{1}{n } \sum_{i = 1}^n \<\bx_i,\bv\>^2 \<\bx_i,\bu\>^2.
\end{aligned}
\end{equation}
where the last inequality uses a standard covering argument over $1/4$-nets
(see \cite[Lemma~4.4.1]{vershyninHighDimensionalProbabilityIntroduction2018}).
Note that
\[
\sup_{\bu,\bv \in \mathcal{N}_{S^\perp}} \frac{1}{n } \sum_{i = 1}^n \<\bx_i,\bv\>^2 \<\bx_i,\bu\>^2 \leq \sup_{\bv \in \mathcal{N}_{S^\perp}} \frac{1}{n } \sum_{i = 1}^n \<\bx_i,\bv\>^4.
\]
Hence, by the same argument as above, the right-hand side of~\eqref{eq:first_term_2_4_op_norm} is bounded by $C'$ with probability at least $1 - e^{-Cd}$. Combining these estimates in \eqref{eq:delta_xi_2j_bound} and using that $\| \bB_{2j} \|_\op^2 \leq \| \bB_{2j}^2 \|_F \leq d^{-c}$ by the genericity condition in \eqref{eq:admissibility_tensor}, we obtain the desired bound \eqref{eq:letusfirst}. For the second claim, note that
\[
\E [\| f - f_{-\bv} \|_2^k] = C_k\sum_{j = 1}^{s_2} \| \bB_{2j} - (\bI- \bv\bv^\sT) \bB_{2j} (\bI- \bv\bv^\sT)\|_F^k \leq C'_k \sum_{j = 1}^{s_2} \|\bB_{2j}\|_\op^k,
\]
from which~\eqref{eq:LODO_exp_f_f_bv} follows immediately.
\end{proof}

The following lemma plays an analogous role as Lemma \ref{lemma:sublevelgeometry}.

\begin{lemma}\label{lemma:ConcentrationOfResiduallemmaAAAAAAAuniform} For any constants $C,K>0$, there exist $K'>0$ such that for all $K_\Gamma >K'$
(with $K_\Gamma$ defining $\bGamma^\bW(\btheta)$), with probability at least $1-d^{-C}$, the following holds:
\begin{itemize}
    \item[(a)] For $\btheta \in \{\hbtheta, \check{\btheta} \}$, the minimizers of $\widehat{\cR}_{n,p}$ and $\cP_{n}$ respectively, 
    \begin{equation}
\label{eq:uniform-bounds-theta-hat-check}
    \| \btheta \|_2, \;\;| \mu_0 \1_p^\sT \btheta |,\;\;  \| \mu_1 \btheta^\sT \cV_1 \|_2 , \;\; \| \mu_2 \btheta^\sT \cV_2 \|_F \; \; \leq (\log d)^{K'}.
    \end{equation}
    Furthermore for all $\bv \in \mathcal{N}_{S^\perp}$, the minimizer $\check{\btheta}_{-\bv}$ of $\cP_{n,-\bv}$ satisfies
\begin{equation}
\label{eq:uniform-bounds-theta-check-minus-v}
     \| \check{\btheta}_{-\bv}\|_2, \;\;| \mu_0 \1_p^\sT \check{\btheta}_{-\bv} |,\;\;  \| \mu_1 \check{\btheta}_{-\bv}^\sT \cV_{1,-\bv} \|_2 , \;\; \| \mu_2 \check{\btheta}_{-\bv}^\sT \cV_{2,-\bv} \|_F \; \; \leq (\log d)^{K'}.
\end{equation}

\item[(b)] We have $\bGamma^\bW (\btheta) = ( \|\bV_{+}^\sT \btheta\|_2^2, L_{\bW}(\btheta))$ for $\btheta \in \{\hbtheta, \check{\btheta} \}$. Similarly, for all $\bv \in \mathcal{N}_{S^\perp}$, 
\[
\bGamma^{\bW_{-\bv}} ( \check{\btheta}_{-\bv}) = \left( (\1_p^\sT  \check{\btheta}_{-\bv})^2 + \| \bW_{-\bv}^\sT  \check{\btheta}_{-\bv}\|_2^2, L_{\bW_{-\bv}}( \check{\btheta}_{-\bv})\right). 
\]
Furthermore, denoting $\bH_{-\bv} (\btheta) = \nabla^2 (\cP_{n,-\bv} (\btheta) - \kappa_d \| \btheta\|_\infty)$, we have
\[
 \bH_{-\bv} ( \check{\btheta}_{-\bv}) \succeq \frac{\lambda}{2} \bI, \qquad \| [\1_p \; \bW_{-\bv}] \bH_{-\bv} ( \check{\btheta}_{-\bv})  [\1_p \; \bW_{-\bv}]^\sT \|_\op \leq \tau_1^{-1}.
\]
\item[(c)] We have
\[
\| \hbtheta \|_\infty \leq \frac{(\log d)^{K'}}{d^{1/4}}, \quad \| \check{\btheta} \|_\infty \leq \frac{(\log d)^{K'}}{d^{1/4}}, \quad \sup_{\bv \in \mathcal{N}_{S^\perp}} \| \check{\btheta}_{-\bv} \|_\infty \leq \frac{(\log d)^{K'}}{d^{1/8}}.
\]
\end{itemize}

\end{lemma}

\begin{proof}[Proof of Lemma \ref{lemma:ConcentrationOfResiduallemmaAAAAAAAuniform}(a)]
    The results for $\hbtheta$ and $\check{\btheta}$ follow directly from Lemma \ref{lemma:sublevelgeometry}(a). For the LODO minimizers, we adapt the proof of Lemma \ref{lemma:sublevelgeometry}(a) to obtain a uniform bound over $\bv \in \mathcal{N}_{S^\perp}$. First, note that
    \begin{equation}\label{eq:uniform_bound_LODO_objective_trivial}
    \frac{\lambda}{2}\| \check{\btheta}_{-\bv}\|_2^2 \leq \cP_{n,-\bv} (\bzero) \leq \frac{1}{n} \sum_{i \in [n]} \ell (y_i,0) + C \left( \frac{1}{n} \sum_{i \in [n]} (y_i - y_{i,-\bv} )^2 \right)^{1/2},
    \end{equation}
    where we used the Lipschitz property of $\ell$ (Assumption \ref{assumption:loss}).  By Lemma~\ref{lem:firstlemma}, the right-hand side is uniformly bounded over $\bv \in \mathcal{N}_{S^\perp}$, implying that $\|\check{\btheta}_{-\bv}\|_2 \prec 1$ simultaneously over all $\bv \in \cN_{S^\perp}$.
    
    We now focus on the order-$2$ chaos component:
\begin{equation}
\label{eq:V2-triangle-inequality}
    \|\check{\btheta}_{-\bv}^\sT \mathcal{V}_{2,-\bv} \|_F \leq \|\check{\btheta}_{-\bv}^\sT \mathcal{V}_2(\id - \bv\bv^\sT)^{\otimes 2}\|_F + \|\check{\btheta}_{-\bv}^\sT \left(\mathcal{V}_{2,-\bv} - \mathcal{V}_2(\id - \bv\bv^\sT)^{\otimes 2}\right)\|_F.
\end{equation}
Since $(\id - \bv\bv^\sT)^{\otimes 2}$ is an orthogonal projection, we have
\begin{equation}
    \|\check{\btheta}_{-\bv}^\sT \mathcal{V}_2(\id - \bv\bv^\sT)^{\otimes 2}\|_{F} \leq \|\check{\btheta}_{-\bv}^\sT \mathcal{V}_2\|_F = \| \bV_2^\sT \check{\btheta}_{-\bv}\|_2 \leq \frac{|\mathbf{1}_p^\sT \check{\btheta}_{-\bv}| }{\sqrt{d}} +\|\bV_{2c}\|_{\op} \|\check{\btheta}_{-\bv}\|_2,
\end{equation}
where we used the decomposition of $\bV_2$ from Corollary~\ref{corollary:OperatorNormF2cdecompose}. Note that $\|\bV_{2c}\|_{\op}\|\check{\btheta}_{-\bv}\|_2 \prec 1$  simultaneously over $\bv$. For the second term in \eqref{eq:V2-triangle-inequality}, define
\begin{equation}
    \Delta(\bv) := \sum_{j=1}^p (\check{\btheta}_{-\bv})_j \left( \frac{(\bw_{j,-\bv})^{\otimes 2}}{\|\bw_{j,-\bv}\|_2^2} - (\bw_{j,-\bv})^{\otimes 2} \right) = \sum_{j=1}^p (\check{\btheta}_{-\bv})_j \left( \frac{1 - \|\bw_{j,-\bv}\|_2^2}{\|\bw_{j,-\bv}\|_2^2} \right) (\bw_{j,-\bv})^{\otimes 2}.
\end{equation}
Using $1 - \|\bw_{j,-\bv}\|_2^2 = \langle \bw_j, \bv \rangle^2$, we obtain
\begin{align}
    \|\Delta(\bv)\|_F &\leq \sum_{j=1}^p |(\check{\btheta}_{-\bv})_j| \frac{\langle \bw_j, \bv \rangle^2}{ \|\bw_{j,-\bv}\|_2^2} \|(\bw_{j,-\bv})^{\otimes 2}\|_F \nonumber  = \sum_{j=1}^p |(\check{\btheta}_{-\bv})_j| \langle \bw_j, \bv \rangle^2.
\end{align}
Applying the Cauchy–Schwarz inequality gives
\begin{equation}
    \sum_{j=1}^p |(\check{\btheta}_{-\bv})_j| \langle \bw_j, \bv \rangle^2 \leq   \|\check{\btheta}_{-\bv}\|_2 \left( \sum_{j=1}^p \langle \bw_j, \bv \rangle^4 \right)^{1/2}.
\end{equation}
From \eqref{eq:union-bound-example-sum-power-4} in the proof of Lemma \ref{lemma:GaussianMalliavinVarianceBound112}, we have $\sup_{\bv \in \mathcal{N}_{S^\perp}} \sum_{j=1}^p \langle \bw_j, \bv \rangle^4 \prec 1$. Therefore,
\begin{equation}\label{eq:bound_hermite_2_geom_lind_2}
    \sup_{\bv \in \mathcal{N}_{S^\perp}} \|\check{\btheta}_{-\bv}^\sT \mathcal{V}_{2,-\bv} \|_F \prec 1 + \frac{1}{\sqrt{d} }  \sup_{\bv \in \mathcal{N}_{S^\perp}}  | \1_p^\sT \check{\btheta}_{-\bv}|.
\end{equation}

Thus it remains to bound $| \mu_0 \1_p^\sT \check{\btheta}_{-\bv} |,  \| \mu_1 \check{\btheta}_{-\bv}^\sT \cV_{1,-\bv} \|_2$ uniformly over $\bv$. The bound \eqref{eq:uniform_bound_LODO_objective_trivial} implies that Lemma \ref{lemma:predictionbound} holds simultaneously for all $ \check{\btheta}_{-\bv}$, thus we can adapt directly the proof of Lemma~\ref{lemma:sublevelgeometry}(a) with the bound \eqref{eq:bound_hermite_2_geom_lind_2}. In particular, \eqref{equation:fallin3gauss} and \eqref{eq:unif_bound_lemma_c1_aa} hold simultaneously over $\bv \in \mathcal{N}_{S^\perp}$, with $\bW$ replaced by $\bW_{-\bv}$. For brevity, we omit the repetition of these arguments here.
\end{proof}

\begin{proof}[Proof of Lemma \ref{lemma:ConcentrationOfResiduallemmaAAAAAAAuniform}(b)]
    The proof follows the same argument as Lemma \ref{lemma:sublevelgeometry}.(b) and \ref{lemma:sublevelgeometry}.(c), using the uniform bound in Lemma \ref{lemma:ConcentrationOfResiduallemmaAAAAAAAuniform}(a).
\end{proof}

\begin{proof}[Proof of Lemma \ref{lemma:ConcentrationOfResiduallemmaAAAAAAAuniform}(c)]
First note that (recalling \eqref{eq:uniform_bound_LODO_objective_trivial} and Lemma \ref{lem:firstlemma}):
\[
\kappa_{d}\|\check{\btheta}_{-\bv}\|_{\infty} \le P_{n,-\bv}(\check{\btheta}_{-\bv}) \le \sup_{\bv \in \mathcal{N}_{S^\perp}} P_{n,-\bv} (\bzero) \prec 1.
\]
Thus, using that we chose $\kappa_{d} = d^{1/8}$, we immediately obtain $\sup_{\bv \in \mathcal{N}_{S^\perp}} \|\check{\btheta}_{-\bv}\|_{\infty} \prec d^{-1/8}$. The result for $\hbtheta$ follows from Lemma \ref{lemma:optinThetaPG}. For $\check{\btheta}$, we use the optimality of $\hbtheta,\check{\btheta}$: 
\[
\begin{aligned}
    \cP_{n} (\check{\btheta}) = \widehat{\cR}_{n,p} (\check{\btheta}) + \kappa_d \| \check{\btheta}\|_\infty \leq&~ \cP_{n} (\hbtheta) = \widehat{\cR}_{n,p} (\hbtheta) + \kappa_d \| \hbtheta\|_\infty, \qquad \widehat{\cR}_{n,p} (\hbtheta) \leq \widehat{\cR}_{n,p} (\check{\btheta}),
\end{aligned}
\]
which implies that $\| \check{\btheta} \|_\infty \leq \| \hbtheta \|_\infty \prec d^{-1/4}$.
\end{proof}

    \begin{lemma}
\label{lem:surgate1}
  We have
\begin{align}\label{eq:surgate1}
  \sup_{\bv \in \mathcal{N}_{S^\perp}} \frac{1}{n} \sum_{i=1}^{n} \left( \check{\btheta}_{-\bv}^\sT \mathcal{V}\bh_{i,-\bv} - \check{\btheta}_{-\bv}^\sT \mathcal{V}_{-\bv}\bh_{i,-\bv} \right)^2 \prec&~ d^{-1/8}, \\
  \label{eq:letussec}
    \sup_{\bv \in \mathcal{N}_S^\perp} \frac{1}{n} \sum_{i=1}^{n} \left( \check{\btheta}_{-\bv}^\sT \mathcal{V} \bh_i - \check{\btheta}_{-\bv}^\sT \mathcal{V}_{-\bv} \bh_{i,-\bv} \right)^2 \prec&~  d^{-1/8} + \mu_{1}^2  .
\end{align}
Moreover, for any $k \geq 1$, there exists a constant $c_k>0$ depending only on the constants in the assumptions such that 
    \begin{align}\label{cor:leaveoneoutstability}
        \sup_{\bv \in \mathcal{N}_S^\perp} \E \left[\left|\<\check \btheta_{-\bv}, \bz - \bz_{-\bv}\>\right|^k\mid \bW\right] \prec d^{-c_k}.
    \end{align}
    Furthermore, the bound \eqref{eq:surgate1}  holds also for $\hat \btheta$ and $\check \btheta$ in place of $\check\btheta_{-\bv}$.
\end{lemma}

\begin{proof}[Proof of Lemma \ref{lem:surgate1} Eq.~\eqref{eq:surgate1}]
    Recall $\bx_{i,-\bv}=(\id-\bv\bv^\sT)\bx_i$. We decompose the difference as
\begin{equation}
 \check{\btheta}_{-\bv}^\sT \mathcal{V}\bh_{i,-\bv} - \check{\btheta}_{-\bv}^\sT\mathcal{V}_{-\bv} \bh_{i,-\bv} = \mu_1 \check{\btheta}_{-\bv}^\sT(\mathcal{V}_1 - \mathcal{V}_{1,-\bv})\bx_{i,-\bv} + \mu_2 \check{\btheta}_{-\bv}^\sT(\mathcal{V}_2 - \mathcal{V}_{2,-\bv})\bH_2(\bx_{i,-\bv}),
\end{equation}
and bound the two contributions separately. (In the case of a sample $\bh_i^\CG$ from the CGE model, the second term is
$\mu_2 \check{\btheta}_{-\bv}^\sT(\mathcal{V}_2 - \mathcal{V}_{2,-\bv})(\proj_S \bH_2(\bx_{i,-\bv})+\proj_{S,\perp}\bG_{2i,-\bv})$, and may be treated similarly in the following.)

\paragraph*{Step 1: Linear term.} Since $\bw_j^\sT \bx_{i,-\bv} = \bw_{j,-\bv}^\sT \bx_{i,-\bv}$, 
\begin{align}
  \Delta_{i,1}(\bv) :=&~ \check{\btheta}_{-\bv}^\sT \mu_1 (\mathcal{V}_1 - \mathcal{V}_{1,-\bv}) \bx_{i,-\bv} \nonumber \\
  =&~ \mu_1 \sum_{j=1}^p (\check{\btheta}_{-\bv})_j \left( \bw_j^\sT - \frac{\bw_{j,-\bv}^\sT}{\|\bw_{j,-\bv}\|_2} \right) \bx_{i,-\bv}\\
  =&~ \mu_1 \sum_{j=1}^p (\check{\btheta}_{-\bv})_j \left( 1 - \frac{1}{\|\bw_{j,-\bv}\|_2} \right) \langle \bw_{j,-\bv}, \bx_{i,-\bv} \rangle =: \<\ba (\bv), \bx_{i,-\bv}\>, \label{eq:delta_i1_def}
\end{align}
where
\begin{equation}
  \ba(\bv) := \sum_{j=1}^p \mu_1 (\check{\btheta}_{-\bv})_j \left( 1 - \frac{1}{\|\bw_{j,-\bv}\|_2} \right) \bw_{j,-\bv}.
\end{equation}
Since $\ba(\bv)$ lies in the subspace orthogonal to $\bv$, we have $\langle \ba(\bv), \bx_{i,-\bv} \rangle = \langle \ba(\bv), \bx_i \rangle$. The average squared difference is then
\begin{equation}
  \frac{1}{n} \sum_{i=1}^n (\Delta_{i,1}(\bv))^2 = \frac{1}{n}\sum_{i=1}^n \langle \ba(\bv), \bx_i \rangle^2 \leq    \|\ba(\bv)\|_2^2 \left\| \frac{1}{n} \sum_{i=1}^n \bx_i\bx_i^\sT\right\|_{\op}
\end{equation}
Since $\{\bx_i\}_{i=1}^n$ are i.i.d.\ standard Gaussian vectors in $\R^d$, we have $\|\frac{1}{n} \sum_{i=1}^n \bx_i\bx_i^\sT\|_{\op} \prec 1$ \cite[Theorem 4.4.5]{vershyninHighDimensionalProbabilityIntroduction2018}, so
\begin{equation}
  \sup_{\bv \in \mathcal{N}_{S^\perp}}\frac{1}{n} \sum_{i=1}^n \langle \ba(\bv), \bx_i \rangle^2 \prec  \sup_{\bv \in \mathcal{N}_{S^\perp}}\|\ba(\bv)\|_2^2.
\end{equation}

We now bound $\|\ba(\bv)\|_2$ uniformly in $\bv$. Rewrite
\begin{equation}\label{eq:def_a_v_LODO}
\ba (\bv) = \mu_1 \sum_{j=1}^p (\|\bw_{j,-\bv}\|_2 - 1)(\check{\btheta}_{-\bv})_j \frac{\bw_{j,-\bv}}{\|\bw_{j,-\bv}\|_2} = \mu_1 \bW_{-\bv}^\sT  \bD_{\check{\btheta}_{-\bv}} \br (\bv),
\end{equation}
where $\br (\bv) :=  (\|\bw_{j,-\bv}\|_2 - 1)_{j \in [p]}$ and $\bD_{\check{\btheta}_{-\bv}} = \diag ((\check{\btheta}_{-\bv})_j)_{j \in [p]}.$ Thus, we simply get $\| \ba (\bv) \|_2^2 \leq \mu_1^2 \| \bW_{-\bv}^\sT  \bD_{\check{\btheta}_{-\bv}}\|_\op^2 \| \br (\bv) \|_2^2.$
Since $\|\bw_{j,-\bv}\|_2 = \sqrt{1 - \<\bw_j,\bv\>^2}$, we have $ (\|\bw_{j,-\bv}\|_2 - 1)^2 \leq \< \bw_j,\bv\>^4$, hence
\[
\sup_{\bv \in \mathcal{N}_{S^\perp}} \| \br (\bv) \|_2^2 \leq \sup_{\bv \in \mathcal{N}_{S^\perp}} \sum_{j = 1}^p \< \bw_j,\bv\>^4 \prec 1,
\]
by the same bound as in \eqref{eq:union-bound-example-sum-power-4}. For the operator norm, $\| \bW_{-\bv}^\sT  \bD_{\check{\btheta}_{-\bv}}\|_\op^2 \leq 2 \sup_{\bu \in \mathcal{N}_{S^\perp}} \| \bD_{\check{\btheta}_{-\bv}} \bW_{-\bv} \bu \|_2^2,$
and by Cauchy–Schwarz,
\[
     \| \bD_{\check{\btheta}_{-\bv}} \bW_{-\bv} \bu \|_2^2 = \sum_{j=1}^p (\check{\btheta}_{-\bv})_j^2 \left\langle \frac{\bw_{j,-\bv}}{\|\bw_{j,-\bv}\|_2}, \bu \right\rangle^2 \le \sqrt{\sum_{j=1}^p (\check{\btheta}_{-\bv})_j^4} \sqrt{\sum_{j=1}^p \left\langle \frac{\bw_{j,-\bv}}{\|\bw_{j,-\bv}\|_2}, \bu \right\rangle^4}.
\]
Using Lemmas \ref{lemma:ConcentrationOfResiduallemmaAAAAAAAuniform}(a) and \ref{lemma:ConcentrationOfResiduallemmaAAAAAAAuniform}(c),
 $\sup_{\bv \in \mathcal{N}_{S^\perp}} \| \check{\btheta}_{-\bv} \|_4^2 \leq  \sup_{\bv \in \mathcal{N}_{S^\perp}} \|\check{\btheta}_{-\bv}\|_\infty \| \check{\btheta}_{-\bv}\|_2 \prec d^{-1/8},$
and, by the same union bound as in \eqref{eq:union-bound-example-sum-power-4},
\[
\sup_{\bu,\bv \in \mathcal{N}_{S^\perp}} \sum_{j=1}^p \left\langle \frac{\bw_{j,-\bv}}{\|\bw_{j,-\bv}\|_2}, \bu \right\rangle^4 \prec 1.
\]
Combining these estimates yields
\begin{equation}\label{eq:bound_lin_term_v_feature_diff}
\sup_{\bv \in \mathcal{N}_{S^\perp}} \frac{1}{n} \sum_{i = 1}^n  \Delta_{i,1}(\bv)^2 
\prec \sup_{\bv \in \cN_{S^\perp}} \|\ba(\bv)\|_2^2 \prec d^{-1/8}.
\end{equation}

\paragraph*{Step 2: Quadratic term.} Denote 
\[
\Delta_{i,2}(\bv) := \check{\btheta}_{-\bv}^\sT \mu_2 (\mathcal{V}_2 - \mathcal{V}_{2,-\bv})   \bH_2(\bx_{i,-\bv})\, \quad \text{and} \quad \Delta \cV_2 (\bv) := \mathcal{V}_2 (\id-\bv\bv^\sT)^{\otimes 2} - \mathcal{V}_{2,-\bv}.
\]
Here $\mathcal{V}_2 (\id-\bv\bv^\sT)^{\otimes 2}$ denotes an order-3 weight tensor in $\R^{p\times d\times d}$ whose $j$-th slice is $\bw_{j,-v}^{\otimes 2}=\left[(\id-\bv\bv^\sT)\bw_j\right]^{\otimes 2}$.
Let $\bH_2(\bX_{-\bv})=[\bH_2(\bx_{1,-\bv}), \ldots, \bH_2 (\bx_{n,-\bv})]^\sT \in \R^{n \times d \times d}$, and let $\bh_2(\bX_{-\bv}) \in \R^{n \times B_{d,2}}$ denote its preimage under the isometry $\iota(\cdot)$. Then
\begin{equation}
\frac{1}{n} \sum_{i=1}^n (\Delta_{i,2}(\bv))^2 = \frac{\mu_2^2}{n} \left\| \check{\btheta}_{-\bv}^\sT \Delta\mathcal{V}_2(\bv)  \bH_2(\bX_{-\bv}) \right\|_F^2 \leq \frac{\mu_2^2}{n} \left\| \check{\btheta}_{-\bv}^\sT \Delta\mathcal{V}_2(\bv) \right\|_F^2 \left\| \bh_2(\bX_{-\bv}) \right\|_{\text{op}}^2.
\end{equation}
By Lemma \ref{lemma:BernsteinInequalityHermiteFeatures}, $\|\bh_2(\bX_{-\bv})\|_{\text{op}}^2 \leq \|\bh_2(\bX)\|_{\text{op}}^2 \prec d^2$. To bound $\left\| \check{\btheta}_{-\bv}^\sT \Delta\mathcal{V}_2(\bv) \right\|_F^2 $, use $\frac{x^2}{1-x^2} = x^2 + \frac{x^4}{1-x^2}$ to write
\begin{align}
    \check{\btheta}_{-\bv}^\sT \Delta\mathcal{V}_2(\bv) &=
    \sum_{j=1}^p (\check{\btheta}_{-\bv})_j\Big(\bw_{j,-\bv}^{\otimes 2}
    -\frac{\bw_{j,-\bv}^{\otimes 2}}
    {\|\bw_{j,-\bv}\|_2^2}\Big)
    ={-}\sum_{j=1}^p (\check{\btheta}_{-\bv})_j\Big(\frac{\<\bw_j,\bv\>^2}{1-\<\bw_j,\bv\>^2}\Big)
    \bw_{j,-\bv}^{\otimes 2}\\
    &={-}\underbrace{\sum_{j=1}^p (\check{\btheta}_{-\bv})_j \langle \bw_j, \bv \rangle^2 \bw_{j,-\bv}^{\otimes 2}}_{=: \bA(\bv)} - \underbrace{\sum_{j=1}^p (\check{\btheta}_{-\bv})_j \frac{\langle \bw_j, \bv \rangle^4}{1-\langle \bw_j, \bv \rangle^2} \bw_{j,-\bv}^{\otimes 2}}_{=: \bB(\bv)}.
\end{align}
 
For $\bA(\bv)$, define $\bc(\bv):= \left( (\check{\btheta}_{-\bv})_j \langle \bw_j, \bv \rangle^2\right)_{j \in [p]}$. Then 
\[
\| \bA(\bv) \|_F^2 = \left\| (\bI - \bv \bv^\sT) \left( \sum_{j=1}^p c_j (\bv) \bw_j \bw_j^\sT \right)  (\bI - \bv \bv^\sT) \right\|_F^2 \leq \| \bc(\bv)^\sT \cV_2 \|_F^2 = \| \bV_2^\sT \bc(\bv) \|_2^2.
\]
Using the decomposition from Corollary~\ref{corollary:OperatorNormF2cdecompose},
\begin{equation}
    \|\bV_2^\sT \bc(\bv)\|_2^2 = \left\| \left(\frac{1}{d}\be_c \mathbf{1}_p^\sT + \bV_{2c}\right)^\sT \bc(\bv) \right\|_2^2 \le \frac{2}{d^2} \|\be_c\|_2^2 (\mathbf{1}_p^\sT \bc(\bv))^2 + 2\|\bV_{2c}^\sT \bc(\bv)\|_2^2.
\end{equation}
For the first term, 
\[
\sup_{\bv \in \mathcal{N}_{S^\perp}} (\mathbf{1}_p^\sT \bc(\bv))^2  = \sup_{\bv \in \mathcal{N}_{S^\perp}}  \left| \sum_{j=1}^p (\check{\btheta}_{-\bv})_j \langle \bw_j, \bv \rangle^2 \right|^2 \leq  \sup_{\bv \in \mathcal{N}_{S^\perp}} \|\check{\btheta}_{-\bv}\|_2^2 \cdot \sum_{j=1}^p \langle \bw_j, \bv \rangle^4 \prec 1.
\]
For the second term, $\|\bV_{2c}^\sT \bc(\bv)\|_2^2 \leq \| \bV_{2c} \|_\op^2 \| \bc(\bv) \|_2^2 \prec \| \bc (\bv)\|_2^2$, and
\[
\sup_{\bv \in \mathcal{N}_{S^\perp}} \| \bc (\bv)\|_2^2 = \sup_{\bv \in \mathcal{N}_{S^\perp}} \sum_{j=1}^p (\check{\btheta}_{-\bv})_j^2 \langle \bw_j, \bv \rangle^4 \leq \sup_{\bv \in \mathcal{N}_{S^\perp}}  \|\check{\btheta}_{-\bv}\|_\infty^2 \sum_{j=1}^p \langle \bw_j, \bv \rangle^4 \prec d^{-1/4},
\]
where we used Lemma \ref{lemma:ConcentrationOfResiduallemmaAAAAAAAuniform}(c).
Hence, putting the above estimates together, we get 
\begin{equation}\label{eq:bound_A_v_unif}
\sup_{\bv \in \mathcal{N}_{S^\perp}} \| \bA(\bv) \|_F^2 \prec d^{-1/4}.
\end{equation}
For $\bB(\bv)$, note that
    \begin{align}
        \|\bB(\bv)\|_F &\leq \sum_{j=1}^p |(\check{\btheta}_{-\bv})_j| \frac{\langle \bw_j, \bv \rangle^4}{1-\langle \bw_j, \bv \rangle^2} \|\bw_{j,-\bv}\|_2^2 \leq \|\check{\btheta}_{-\bv}\|_\infty \sum_{j=1}^p \langle \bw_j, \bv \rangle^4,
    \end{align}
since  $\|\bw_{j,-\bv}\|_2^2 = 1-\langle \bw_j, \bv \rangle^2$. Arguing as above, we immediately obtain
\begin{equation}\label{eq:bound_B_v_unif}
\sup_{\bv \in \mathcal{N}_{S^\perp}} \| \bB(\bv) \|_F^2 \prec d^{-1/4}.
\end{equation}
Combining \eqref{eq:bound_A_v_unif} and \eqref{eq:bound_B_v_unif} yields the bound on the quadratic part:
\begin{equation}\label{eq:bound_quad_term_v_feature_diff}
    \sup_{\bv \in \mathcal{N}_{S^\perp}} \frac{1}{n} \sum_{i=1}^n \Delta_{i,2}(\bv)^2 \prec d^{-1/4}.
\end{equation}
Finally, \eqref{eq:surgate1} follows by combining \eqref{eq:bound_lin_term_v_feature_diff} and \eqref{eq:bound_quad_term_v_feature_diff}. We note that the preceding arguments hold equally with $\hat\btheta$ or $\check\btheta$ in place of $\check\btheta_{-\bv}$.
\end{proof}

\begin{proof}[Proof of Lemma \ref{lem:surgate1} Eq.~\eqref{eq:letussec}]
    First, note that
\begin{align}
    \frac{1}{n} \sum_{i=1}^{n} \left( \check{\btheta}_{-\bv}^\sT (\mathcal{V} \bh_i - \mathcal{V}_{-\bv} \bh_{i,-\bv}) \right)^2
    \leq \frac{2}{n} \sum_{i=1}^{n} \left( \check{\btheta}_{-\bv}^\sT \mathcal{V} (\bh_i - \bh_{i,-\bv}) \right)^2 + \frac{2}{n} \sum_{i=1}^{n} \left( \check{\btheta}_{-\bv}^\sT ( \mathcal{V} -  \mathcal{V}_{-\bv}) \bh_{i,-\bv} \right)^2.
\end{align}
The second term is exactly \eqref{eq:surgate1} and is uniformly $\prec d^{-1/8}$. For the first term, decompose
\[
\check{\btheta}_{-\bv}^\sT \mathcal{V} (\bh_i - \bh_{i,-\bv}) = \underbrace{\mu_1 \check{\btheta}_{-\bv}^\sT \mathcal{V}_1 \left(\bx_i - \bx_{i,-\bv} \right)}_{:= \Delta_{i,1} (\bv)} + \underbrace{\mu_2 \check{\btheta}_{-\bv}^\sT \mathcal{V}_2 \left(\bH_2 (\bx_i) - \bH_2 ( \bx_{i,-\bv}) \right)}_{:= \Delta_{i,2} (\bv)}.
\]
(In the case of a sample $\bh_i^\CG$ from the CGE model, the second term is $\Delta_{i,2}(\bv)=\mu_2 \check{\btheta}_{-\bv}^\sT \mathcal{V}_2 \proj_{S,\perp} \bG_{i2,-\bv}$ and again may be treated similarly in the following.)

\paragraph*{Step 1: Linear term.} Since $\bx_i-\bx_{i,-\bv}=\<\bx_i,\bv\>\bv$, this term is $\Delta_{i,1} (\bv) =\mu_1 \langle \bx_i, \bv \rangle \left( \check{\btheta}_{-\bv}^\sT \bW \bv \right).$
Hence,
\begin{align*}
    \frac{1}{n}\sum_{i=1}^n \Delta_{i,1}(\bv)^2 = \mu_1^2 \left( \check{\btheta}_{-\bv}^\sT \bW \bv \right)^2 \left( \frac{1}{n}\sum_{i=1}^n \langle \bx_i, \bv \rangle^2 \right).
\end{align*}
By a chi-squared tail bound and union bound, $ \sup_{\bv \in \mathcal{N}_{S^\perp}} n^{-1}\sum_{i=1}^n \langle \bx_i, \bv \rangle^2\prec 1$. For fixed $\bv$, note that $\<\bw_j,\bv\>$ is independent of 
$\bw_{j,-\bv}/\|\bw_{j,-\bv}\|_2$, so
$\check{\btheta}_{-\bv}$ is independent of  $\{ \< \bw_j ,\bv\> \}_{j \in [p]}$. Then conditional on $\check\btheta_{-\bv}$, $\check{\btheta}_{-\bv}^\sT \bW \bv$ is $(\|\check{\btheta}_{-\bv}\|_2^2/d)$-subgaussian, so $\P[|\check{\btheta}_{-\bv}^\sT \bW \bv| \geq t \mid \check{\btheta}_{-\bv}] \leq 2e^{-ct^2d/\|\check\btheta_{-\bv}\|_2^2}$.
Using
$\sup_{\bv}\|\check{\btheta}_{-\bv}\|_2\prec 1$,
choosing $t \asymp (\log d)^K$ for large enough $K$, and taking a union bound over $\cN_{S^\perp}$,
\begin{equation}
\sup_{\bv \in \mathcal{N}_{S^\perp}}  \frac{1}{n}\sum_{i=1}^n \Delta_{i,1}(\bv)^2 \prec \mu_1^2  \cdot \sup_{\bv \in \mathcal{N}_{S^\perp}}( \check{\btheta}_{-\bv}^\sT \bW \bv)^2 \prec \mu_1^2.\label{eq:linear-term-bound}
\end{equation}

\paragraph*{Step 2: Quadratic term.} Using $\bH_2(\bx) = \frac{1}{\sqrt{2}}(\bx\bx^\top - \id)$ and $\bH_2(\bx_{i,-\bv})=(\bI-\bv\bv^\top)\bH_2(\bx_i)(\bI-\bv\bv^\top)$,
\begin{align*}
    \bH_2(\bx_i) - \bH_2(\bx_{i,-\bv}) &=  \frac{1}{\sqrt{2}} \left( \langle \bx_i, \bv \rangle (\bx_{i}\bv^\sT + \bv\bx_{i}^\sT) - \left[\langle \bx_i, \bv \rangle^2+1\right] \bv\bv^\sT \right).
\end{align*}
Thus we can decompose the quadratic term as
\begin{align*}
  \frac{1}{n} \sum_{i=1}^n \Delta_{i,2}(\bv)^2 &\leq  2\mu_2^2 \underbrace{ \frac{1}{n}\sum_{i=1}^n \left( \langle \bx_i, \bv \rangle^2 \sum_{j=1}^p (\check{\btheta}_{-\bv})_j \langle \bw_j, \bv \rangle^2 \right)^2}_{=:a(\bv)}+
  2\mu_2^2\underbrace{\left(\sum_{j=1}^p (\check{\btheta}_{-\bv})_j\<\bw_j,\bv\>^2\right)^2}_{:=S(\bv)^2}\\&+ 4 \mu_2^2  \underbrace{\frac{1}{n}\sum_{i=1}^n \left(\langle \bx_i, \bv \rangle \sum_{j=1}^p (\check{\btheta}_{-\bv})_j \langle \bw_j, \bx_{i} \rangle \langle \bw_j, \bv \rangle \right)^2}_{=:b(\bv)}.
\end{align*} 
The first term can be rewritten as 
\[a (\bv) =  \left(\sum_{j=1}^p (\check{\btheta}_{-\bv})_j \langle \bw_j, \bv \rangle^2 \right)^2 \left(\frac{1}{n}\sum_{i=1}^n \langle \bx_i, \bv \rangle^4\right)= S(\bv)^2 \left(\frac{1}{n}\sum_{i=1}^n \langle \bx_i, \bv \rangle^4\right).
\]
Write
\[
S(\bv) = d^{-1} \< \1_p, \check{\btheta}_{-\bv} \> + \sum_{j=1}^p(\check{\btheta}_{-\bv})_j (  \langle \bw_j, \bv \rangle^2 - d^{-1} ).
\]
By Lemma~\ref{lemma:ConcentrationOfResiduallemmaAAAAAAAuniform}(a), $\sup_{\bv \in \mathcal{N}_{S^\perp}}  |\< \1_p, \check{\btheta}_{-\bv} \>| \prec 1$. 
Since $\check{\btheta}_{-\bv}$ is independent of $\{\<\bw_j,\bv\>\}_{j \in [p]}$, Bernstein’s inequality \cite[Theorem 2.8.1]{vershyninHighDimensionalProbabilityIntroduction2018} 
yields, for each $\bv \in \cN_{S^\perp}$,
\[\P\bigg[\bigg|\sum_{j=1}^p (\check{\btheta}_{-\bv})_j (  \langle \bw_j, \bv \rangle^2 - d^{-1})| \geq t \;\bigg|\; \check{\btheta}_{-\bv} \bigg]
\leq 2e^{-c\min(\frac{t^2d^2}{\|\check\btheta_{-\bv}\|_2^2},\frac{td}{\|\check\btheta_{-\bv}\|_\infty})}.\]
Then applying
$\sup_{\bv}\|\check{\btheta}_{-\bv}\|_2\prec 1$ and
$\sup_{\bv}\|\check{\btheta}_{-\bv}\|_\infty\prec d^{-1/8}$
(Lemma~\ref{lemma:ConcentrationOfResiduallemmaAAAAAAAuniform}(c)) and taking a union bound over $\cN_{S^\perp}$ yields
\begin{equation}
\label{equ:fix1}
\sup_{\bv \in \mathcal{N}_{S^\perp}} \left| \sum_{j=1}^p(\check{\btheta}_{-\bv})_j (  \langle \bw_j, \bv \rangle^2 - d^{-1} ) \right|  \prec  d^{-1/8}.
\end{equation}
Thus $\sup_{\bv} S(\bv)^2 \prec d^{-1/4}$.
Moreover $\sup_{\bv} n^{-1}\sum_i\<\bx_i,\bv\>^4\prec 1$, so
$\sup_{\bv} a(\bv)\prec d^{-1/4}$.

For $b(\bv)$, introduce the matrix $\bC_\bv = \operatorname{diag}(\langle\bw_1, \bv\rangle, \dots, \langle\bw_p, \bv\rangle)$ and write
\[
    b(\bv)=\check{\btheta}_{-\bv}^\sT \bC_\bv \bW \left(\frac{1}{n}\sum_{i=1}^{n}\langle\bx_i,\bv\rangle^2 \bx_i\bx_i^\sT \right) \bW^\sT \bC_\bv\check{\btheta}_{-\bv}.
\]
From \eqref{eq:first_term_2_4_op_norm}, we have $\sup_{\bv \in \mathcal{N}_{S^\perp}} \left\| n^{-1} \sum_{i=1}^{n}\langle\bx_i,\bv\rangle^2 \bx_i\bx_i^\sT\right\|_\op \prec 1$. Thus it suffices to bound 
\[
\|  \bW^\sT \bC_\bv \check{\btheta}_{-\bv}\|_2 \leq 2\sup_{\bu \in \cN_{S^\perp}} |\bu^\sT \bW^\sT \bC_\bv \check{\btheta}_{-\bv}|.
\]
Let $\bu_{-\bv} = \bu - \<\bv,\bu\>\bv$. Then,
\begin{align}
 \bu^\sT \bW^\sT \bC_\bv \check{\btheta}_{-\bv}&=
 \sum_{j=1}^p (\check{\btheta}_{-\bv})_j
 \<\bu,\bw_j\>\<\bw_j,\bv\>
 \\ &= \<\bu,\bv\> \underbrace{\sum_{j =1}^p(\check{\btheta}_{-\bv})_j\<\bw_j,\bv\>^2}_{=S(\bv)}  + \| \bu_{-\bv} \|_2 \underbrace{\sum_{j =1}^p(\check{\btheta}_{-\bv})_j \<\bw_j,\bv\> \< \bw_j, \bu_{-\bv}/\| \bu_{-\bv} \|_2 \>}_{:=T(\bv,\bu)}.
 \end{align}
 The first term is bounded by $\sup_\bv |S(\bv)| \prec d^{-1/8}$ above. For the second term, let $\tilde \bu_{-\bv}=\bu_{-\bv}/\|\bu_{-\bv}\|_2$ and $\bw_{j,-\bv}=\bw-\<\bv,\bw\>\bv$. Then
 \begin{equation}\label{eq:partial_sum_check_theta_j_k}
 T(\bv,\bu)= \sum_{j =1}^p \left[ (\check{\btheta}_{-\bv})_j\frac{\<\bw_j,\tilde\bu_{-\bv}\>}{\|\bw_{j,-\bv}\|_2} \right] \cdot \| \bw_{j,-\bv} \|_2 \<\bw_j,\bv\>.
 \end{equation}
For $\bv,\bu$ fixed, $\| \bw_{j,-\bv} \|_2 \<\bw_j,\bv\>$ is independent of $(\check{\btheta}_{-\bv})_j \frac{\<\bw_j,\tilde\bu_{-\bv}\>}{\|\bw_{j,-\bv}\|_2}$, as $\tilde\bu_{-\bv}$ is a unit vector orthogonal to $\bv$ so the latter depends on $\bw_j$ only via $\bw_{j,-\bv}/\|\bw_{j,-\bv}\|_2$. Thus, conditional on $\bw_{j,-\bv}/\|\bw_{j,-\bv}\|_2$, $T(\bv,\bu)$ is a subgaussian random variable with subgaussian constant 
\[
\frac{1}{d}\sum_{j = 1}^p (\check{\btheta}_{-\bv})_j^2  \frac{\<\bw_j,\tilde\bu_{-\bv}\>^2}{\| \bw_{j,-\bv} \|_2^2} \leq d^{-1} \| \check{\btheta}_{-\bv} \|_4^2 \sqrt{\sum_{j = 1}^p \frac{\<\bw_j,\tilde\bu_{-\bv}\>^4}{\| \bw_{j,-\bv} \|_2^4}}
\]
We have $\sup_{\bv \in \cN_{S^\perp}} \|\check\btheta_{-\bv}\|_4^2 \leq \sup_{\bv \in \cN_{S^\perp}} \|\check\btheta_{-\bv}\|_\infty \|\check\btheta_{-\bv}\|_2
\prec d^{-1/8}$ by Lemma \ref{lemma:ConcentrationOfResiduallemmaAAAAAAAuniform}(a) and (c), and
$\sup_{\bv,\bu \in \cN_{S^\perp}} \sum_{j=1}^p \frac{\<\bw_j,\tilde\bu_{-\bv}\>^4}{\| \bw_{j,-\bv} \|_2^4} \prec 1$ by the same argument as in \eqref{eq:union-bound-example-sum-power-4} applied to the i.i.d.\ vectors $\bw_{j,-\bv}/\|\bw_{j,-\bv}\|_2$ on the sphere of dimension $d-2$ instead of
$\bw_j \in \S^{d-1}$. Then the above subgaussian constant is $\prec d^{-9/8}$ uniformly over
$\bu,\bv \in \cN_{S^\perp}$. Applying this subgaussian tail bound for $T(\bv,\bu)$ and taking a union bound over $\bv,\bu \in \cN_{S^\perp}$, we obtain $\sup_{\bv,\bu \in \cN_{S^\perp}} |T(\bv,\bu)| \prec d^{-1/16}$.
We conclude that $
\sup_{\bv \in \mathcal{N}_{S^\perp}} b(\bv) \prec d^{-1/8}.$ Combining this bound with the bounds on $a(\bv)$ and $S(\bv)$ yields that
\[\sup_{\bv \in \cN_{S^\perp}} \frac{1}{n}
\sum_{i=1}^n \Delta_{i,2}(\bv)^2 \prec d^{-1/8}.\]

The claim \eqref{eq:letussec} follows by combining these bounds on the linear and quadratic terms, and the bound \eqref{eq:surgate1} for the second summand in the initial decomposition.
\end{proof}

\begin{proof}[Proof of Lemma \ref{lem:surgate1} Eq.~\eqref{cor:leaveoneoutstability}]
Noting that $\bz-\bz_{-\bv}=\cV\bh-\cV_{-\bv}\bh_{-\bv}$, this follows by arguments similar to \eqref{eq:letussec} and the assumption $|\mu_1| \prec d^{-c}$. We omit the details for brevity.
\end{proof}

\begin{lemma}\label{lem:samereanson}
  There exists a constant $c>0$ only depending on the constants in the assumptions, such that
\begin{align}\label{eq:leaveoneoutstabilityreas}
    \sup_{\bv\in \mathcal{N}_S^\perp} \left| \frac{1}{n}\sum_{i=1}^n \nabla \ell \left(y_{i,-\bv}, \check{\btheta}_{-\bv}^\sT \mathcal{V}_{-\bv} \bh_{i,-\bv} \right)^\sT
    \begin{pmatrix}
      y_i - y_{i,-\bv} \\
      \check{\btheta}_{-\bv}^\sT \mathcal{V} \bh_i - \check{\btheta}_{-\bv}^\sT \mathcal{V}_{-\bv} \bh_{i,-\bv}
    \end{pmatrix} \right|
    \prec d^{-c}.
  \end{align}
  The same bound holds if we replace $\check{\btheta}_{-\bv}$  in $  \check{\btheta}_{-\bv}^\sT (\mathcal{V} \bh_i - \mathcal{V}_{-\bv} \bh_{i,-\bv})$ with $\hat{\btheta}$ or $\check{\btheta}$.
\end{lemma}

\begin{proof}
We bound \eqref{eq:leaveoneoutstabilityreas} by uniformly bounding each of the two contributions:
\begin{align}
    T_1 (\bv) &:= \frac{1}{n}\sum_{i=1}^n \partial_y \ell\left(y_{i,-\bv}, \check{\btheta}_{-\bv}^\sT \mathcal{V}_{-\bv}\bh_{i,-\bv}\right)(y_i - y_{i,-\bv}), \\
    T_2 (\bv) &:= \frac{1}{n}\sum_{i=1}^n \partial_{\hat y}\ell\left(y_{i,-\bv}, \check{\btheta}_{-\bv}^\sT \mathcal{V}_{-\bv} \bh_{i,-\bv}\right) \left( \check{\btheta}_{-\bv}^\sT \mathcal{V} \bh_i - \check{\btheta}_{-\bv}^\sT \mathcal{V}_{-\bv} \bh_{i,-\bv} \right).
\end{align}
For convenience, we denote $\bt_{i,-\bv} := (y_{i,-\bv}, \check{\btheta}_{-\bv}^\sT \mathcal{V}_{-\bv} \bh_{i,-\bv})$.

By Cauchy-Schwarz inequality, the Lipschitz property of $\ell$ (Assumption \ref{assumption:loss}), and Lemma \ref{lem:firstlemma},
\begin{align}
    \sup_{\bv \in \mathcal{N}_{S^\perp}} |T_1 (\bv) | \leq \sup_{\bv \in \mathcal{N}_{S^\perp}} \left( \frac{1}{n}\sum_{i=1}^n \partial_y \ell(\bt_{i,-\bv})^2 \right)^{1/2} \left( \frac{1}{n}\sum_{i=1}^n (y_i - y_{i,-\bv})^2 \right)^{1/2} \prec d^{-c}.
\end{align}
For $T_2 (\bv)$, the bound \eqref{eq:letussec} of  Lemma \ref{lem:surgate1}, together with Cauchy-Schwarz inequality and the Lipschitzness of $\ell$ again, gives:
\begin{align}
    \sup_{\bv \in \mathcal{N}_{S^\perp}} |T_2 (\bv) | \leq \sup_{\bv \in \mathcal{N}_{S^\perp}} \left( \frac{1}{n}\sum_{i=1}^n \partial_y \ell(\bt_{i,-\bv})^2 \right)^{1/2} \left( \frac{1}{n}\sum_{i=1}^n  \left( \check{\btheta}_{-\bv}^\sT \mathcal{V} \bh_i - \check{\btheta}_{-\bv}^\sT \mathcal{V}_{-\bv} \bh_{i,-\bv} \right)^2 \right)^{1/2} \prec d^{-c}.
\end{align}
However, the bound \eqref{eq:letussec}   does not apply for the optimizers $\hat \btheta$ and $\check \btheta$. Here, we show how to prove \eqref{eq:leaveoneoutstabilityreas} with the optimizer $\check \btheta$; the same argument works for $\hat \btheta$.

We introduce the intermediate term $\check{\btheta} ^\sT \mathcal{V} \bh_{i,-\bv}$ and decompose $T_2 = T_{2a} + T_{2b}$, where
\begin{align}
    T_{2a} (\bv) &:= \frac{1}{n}\sum_{i=1}^n \partial_{\hat y}\ell(\bt_{i,-\bv}) \left( \check{\btheta} ^\sT \mathcal{V} \bh_i - \check{\btheta} ^\sT \mathcal{V} \bh_{i,-\bv} \right), \\
    T_{2b} (\bv) &:= \frac{1}{n}\sum_{i=1}^n \partial_{\hat y}\ell(\bt_{i,-\bv}) \left( \check{\btheta} ^\sT \mathcal{V} \bh_{i,-\bv} - \check{\btheta} ^\sT \mathcal{V}_{-\bv} \bh_{i,-\bv} \right).
\end{align}
Cauchy-Schwarz inequality and the bound \eqref {eq:surgate1} of Lemma \ref{lem:surgate1} directly give
\begin{align}
   \sup_{\bv \in \mathcal{N}_{S^\perp}}  |T_{2b} (\bv)| &\leq  \sup_{\bv \in \mathcal{N}_{S^\perp}} \sqrt{\frac{1}{n}\sum_{i=1}^n \partial_{\hat y}\ell(\bt_{i,-\bv})^2} \sqrt{\frac{1}{n}\sum_{i=1}^n \left( \check{\btheta} ^\sT \mathcal{V}\bh_{i,-\bv} - \check{\btheta} ^\sT \mathcal{V}_{-\bv} \bh_{i,-\bv} \right)^2} \prec d^{-1/16}.
\end{align}

  For $T_{2a}(\bv)$,   denote $u_1 (\bv) = \mu_1\check{\btheta}^\sT \mathcal{V}_1 \bv$ and $\bU_2 = \mu_2\check{\btheta}^\sT \mathcal{V}_2 $, decompose
\begin{align}
    \check{\btheta}^\sT \mathcal{V} (\bh_i-\bh_{i,-\bv}) &= u_1 (\bv) \< \bv,\bx_i\> + \frac{\bv^\sT \bU_2 \bv }{\sqrt{2}}(\langle \bx_i, \bv \rangle^2 - 1) + \sqrt{2} \langle \bx_i, \bv \rangle   \bx_{i,-\bv}^\sT \bU_2  \bv .
\end{align}
For the first two terms we use that $u_1 (\bv) \leq \| \mu_1 \bW^\sT \check{\btheta} \|_2 \prec 1$ and $\bv^\sT \bU_2 \bv \prec 1$ uniformly over $\bv$ by Lemma \ref{lemma:ConcentrationOfResiduallemmaAAAAAAAuniform}(a). For fixed $\bv$, let $g_i := \<\bx_i,\bv\>$. Then $g_i$ is independent of $\bt_{i,-\bv}$.  By Lipschitzness of $\ell$, $  |\frac{1}{n} \sum_{i=1}^n \partial_{\hat y}\ell(\bt_{i,-\bv}) g_i|$ is upper bounded by $\frac{C}{\sqrt{n}}$ with probability at least $1-2e^{-cn}$ for appropriate constants $C,c>0$. By Bernstein's inequality, 
\begin{align}
  \left| \frac{1}{n} \sum_{i=1}^n \partial_{\hat y}\ell(\bt_{i,-\bv}) (g_i^2 - 1)\right| \ge \frac{C}{\sqrt{d}},
\end{align}
with probability at least $1-2e^{-Cd}$ for a sufficiently large constant $C>0$. For the third term, it contributes to the following term inside $T_{2a}(\bv)$:
\begin{align}
   \frac{1}{n} \sum_{i=1}^n \partial_{\hat y}\ell(\bt_{i,-\bv}) g_i \bx_{i,-\bv}^\sT \bU_2(\bv) \bv   \leq \left\|  \frac{1}{n} \sum_{i=1}^n \partial_{\hat y}\ell(\bt_{i,-\bv}) g_i \bx_{i,-\bv}^\sT\right\|_2 \| \bU_2 \bv\|_2.
\end{align}
By Lemma \ref{lemma:ConcentrationOfResiduallemmaAAAAAAAuniform}(a), we know that $\|\bU_2\bv\|_2 \prec 1$.
Construct a matrix $\bM := (\bM_1,\cdots,\bM_n)$ by defining its column vector as \begin{align}
    \bM_i := \partial_{\hat y}\ell(\bt_{i,-\bv}) \bx_{i,-\bv}.
\end{align}
Then it suffices to upper bound $\|\frac{1}{n}\bM \bg\|_2,$ where $\bg = (g_1,\ldots,g_n).$
Conditional on $\bM$, this term follows the distribution $\cN(0,\frac{1}{n^2}\bM\bM^\sT)$, where $\frac{1}{n^2}\bM\bM^\sT$ has rank $d$. Then by standard concentration inequality of Gaussian random vectors, we have \begin{align}
    &\left\|  \frac{1}{n} \sum_{i=1}^n \partial_{\hat y}\ell(\bt_{i,-\bv}) g_i \bx_{i,-\bv}^\sT\right\|_2^2 = \left\|\frac{1}{n}\bM\bg\right\|_2^2  
    \leq  \frac{1}{n^2}\|\bM\bM^\sT\|_{op}(d+dt),
\end{align}
with probability at least $1-2e^{-cd\min(t^2,t)}. $ By Assumption \ref{assumption:loss}, we have \begin{align}
    \|\bM\bM^\sT\|_{op} \leq Cn,
\end{align}
with probability at least $1-2e^{-cn}$. Consequently, for any $C>0$, there exists $C'>0$ such that $\|\frac{1}{n}\bM\bg\|_2 \leq \frac{C'}{\sqrt{d}}$ with probability at least $1-2e^{-Cd}$. Taking a union bound over $\cN_{S^\perp}$ and combining with the analyses of the first two terms, $\sup_{\bv \in \mathcal{N}_{S^\perp}} |T_{2a}| \prec d^{-1/4}$. This concludes the proof.
\end{proof}

\subsection{Proof of Theorem \ref{thm:inclusion_theta_CGE}}
\label{sec:abstract-CLT-conditioniic}

Lemma \ref{lemma:optinThetaPG}  shows that $\hbtheta \in \bTheta_\bW^{\sPG}(K) \cap \widehat S_q(K)$ with probability at least $1 - d^{-C}$. Thus, it remains to control the two conditions associated with order-$2$ chaos CLT.

\begin{proposition}\label{prop:abstractCLTcondition}
 For any constant $C>0$, there exist constants $K,\eps,K_\Gamma>0$ such that
with probability at least $1-d^{-C}$, the following holds. For each $q=1,\ldots,n$, choice of
$(\tilde \bz_q,\tilde y_q) \in \{(\bz_q^\sPG,y_q^\sPG),(\bz_q^\sCG,y_q^\sCG)\}$,
and $\hcR$ given by either $\hcR_{\setminus q}$ or $\hcR_{\cup q}$, the (unique) minimizer $\hat\btheta$ of $\hcR$ satisfies
\begin{align} \label{eq:op_norm_CGE_loc}
\bigl\|\bW_{\setminus S}^{\sT}\bD_{\hbtheta}\,
                 \bW_{\setminus S}\bigr\|_{\op}
              \leq&~ d^{-\eps}, \\
    \bigl\|\bW_S^{\sT}\bD_{\hbtheta}\,
                 \bW_{\setminus S}\bigr\|_{F}
              \leq&~ d^{-\eps}. \label{eq:fob_norm_CGE_loc}        
\end{align}
\end{proposition}

    \begin{proof}[Proof of Proposition \ref{prop:abstractCLTcondition} Equation \eqref{eq:op_norm_CGE_loc}] Recall that we defined $\mathcal{N}_{S^{\perp}}$ a $1/4$-net on the unit sphere orthogonal to the signal coordinates $S$. A standard covering argument (e.g., \cite[Lemma 4.4.1]{vershyninHighDimensionalProbabilityIntroduction2018}) gives  
\begin{equation}\label{eq:net-reduction}
  \bigl\|
  \proj_{S,\perp}  \bW^{\sT}\bD_{\hat{\btheta}}\bW\proj_{S,\perp}
  \bigr\|_{\op}
  \;\le\;
  2 \cdot \max_{\bv \in \mathcal{N}_{S^{\perp}}}
    \bigl|\,
      \bv^{\sT} \bW^{\sT} \bD_{\hat{\btheta}} \bW \bv
    \bigr|.
\end{equation}
Further recall that with high probability uniformly over $\{ \hbtheta,\check{\btheta},\check{\btheta}_{-\bv}\} $, we can remove the truncation in $\bGamma^\bW (\btheta)$ and $\bGamma^{\bW_{-\bv}} (\btheta)$, thus for the remainder of the proof we will assume that $\cT_{K_{\Gamma}} = {\rm id}$ by Lemma \ref{lemma:ConcentrationOfResiduallemmaAAAAAAAuniform} (b).

\paragraph*{Step 1: Bounding the distance between optimizers.} The optimality conditions for $\hat{\btheta}$ and $\check{\btheta}$ implies that $\hcR_{n,p}(\hat{\btheta}) \le \hcR_{n,p}(\check{\btheta})$ and $\cP_n (\check{\btheta}) \leq \cP_n (\hbtheta)$. Thus,
\begin{align*}
 \hcR_{n,p}(\check{\btheta}) + \kappa_d \|\check{\btheta}\|_{\infty} &\le  \hcR_{n,p}(\hat{\btheta}) + \kappa_d \|\hat{\btheta}\|_{\infty} 
\implies \hcR_{n,p}(\check{\btheta}) - \hcR_{n,p}(\hat{\btheta}) \le \kappa_d ( \|\hat{\btheta}\|_{\infty} - \|\check{\btheta}\|_{\infty} ).
\end{align*}
 By Lemma \ref{lemma:ConcentrationOfResiduallemmaAAAAAAAuniform}(c), $ \hcR_{n,p}(\check{\btheta}) - \hcR_{n,p}(\hat{\btheta})\prec \kappa_d d^{-1/4} = d^{-1/8}$. 
Next, we relate this difference in objective values to the squared $\ell_2$-distance between the optimizers. A second-order Taylor expansion of $\hcR_{n,p}$ around $\hat{\btheta}$ yields:
\[
\hcR_{n,p}(\check{\btheta}) - \hcR_{n,p}(\hat{\btheta}) =  \frac{1}{2} \int_0^1 (1-t) \< \nabla^2 \hcR_{n,p}(t \check{\btheta} + (1-t)\hbtheta), (\check{\btheta}-\hat{\btheta})^{\otimes 2}\> \de t \geq \frac{\lambda}{4} \| \check{\btheta}-\hat{\btheta}\|_2^2
\]
where we used the first order condition and Lemma \ref{lemma:sublevelgeometry}(c). We deduce that 
\begin{equation}\label{equ:step2bound}
  \|\check{\btheta}-\hat{\btheta}\|_2 \prec d^{-1/16} .
\end{equation}

\paragraph*{Step 2: Bounding $\cP_n (\check{\btheta}_{-\bv}) - \cP_{n,-\bv}(\check{\btheta}_{-\bv})$.}
Let us establish a bound on the difference between $\cP_n (\check{\btheta}_{-\bv})$ and the LODO objective $\cP_{n,-\bv}(\check{\btheta}_{-\bv})$ uniformly over $\bv \in \mathcal{N}_{S^\perp}$. Decompose
\begin{equation}
\label{eq:Pn_diff_decomp}
\cP_n(\check{\btheta}_{-\bv}) - \cP_{n,-\bv}(\check{\btheta}_{-\bv}) = \Delta_{\ell}(\check{\btheta}_{-\bv}) + \tau_1 \Delta_{\bGamma,1}(\check{\btheta}_{-\bv}) + \tau_2 \Delta_{\bGamma,2}(\check{\btheta}_{-\bv}),
\end{equation}
where
\begin{align*}
\Delta_{\ell}(\check{\btheta}_{-\bv}) &:= \frac{1}{n}\sum_{i=1}^n \left[ \ell(y_i, \check{\btheta}_{-\bv}^{\sT}\bz_i) - \ell(y_{i,-\bv}, \check{\btheta}_{-\bv}^{\sT}\bz_{i,-\bv}) \right], \\
\Delta_{\bGamma,s}(\check{\btheta}_{-\bv}) &:= \Gamma_s^{\bW}(\check{\btheta}_{-\bv}) - \Gamma_s^{\bW_{-\bv}}(\check{\btheta}_{-\bv}).
\end{align*}
By Lipschitz property of  $\ell$ (Assumption \ref{assumption:loss}), we get
\begin{align*}
|\Delta_{\ell}(\check{\btheta}_{-\bv})| 
&\le \frac{C}{n} \sum_{i=1}^n \left( |y_i - y_{i,-v}| + |\check{\btheta}_{-\bv}^{\sT}(\bz_i - \bz_{i,-v})| \right) \\
&\le C \left( \frac{1}{n}\sum_{i=1}^n (y_i - y_{i,-v})^2 \right)^{1/2} + C \left( \frac{1}{n}\sum_{i=1}^n \left( \check{\btheta}_{-\bv}^{\sT}(\bz_i - \bz_{i,-v}) \right)^2 \right)^{1/2}.
\end{align*}
Thus, by Lemma \ref{lem:firstlemma}, Lemma \ref{lem:surgate1}, and the assumption $|\mu_1| \prec d^{-c}$, there exists a constant $c>0$ such that
\[
\sup_{\bv \in \mathcal{N}_{S^\perp}} |\Delta_{\ell}(\check{\btheta}_{-\bv})|  \prec d^{-c}.
\]
For $\Delta_{\bGamma,1}$, we decompose the difference as  
\begin{align}\label{eq:decompo-delta_gamma-1}
\Delta_{\bGamma,1}(\check{\btheta}_{-\bv}) =&~ \mu_1^2 \check{\btheta}_{-\bv}^\sT \left( \bW\bW^\sT - \bW_{-\bv}\bW_{-\bv}^\sT \right)\check{\btheta}_{-\bv} \\
=&~ \mu_1^2\< \bv , \bW^\sT\check{\btheta}_{-\bv}\>^2 + 2 \mu_1\<\ba (\bv) , \bW_{-\bv}^\sT\check{\btheta}_{-\bv}\> + \| \ba (\bv) \|_2^2,
\end{align}
where $\ba (\bv)=\mu_1\check\btheta_{-\bv}^\top(\bW-\bW_{-\bv})$ is as defined preceding \eqref{eq:def_a_v_LODO}. Thus, from the proof of Lemma \ref{lem:surgate1} and by Lemma \ref{lemma:ConcentrationOfResiduallemmaAAAAAAAuniform}(a) and bound \eqref{eq:linear-term-bound}, we get
\[
\sup_{\bv \in \mathcal{N}_{S^\perp}}  |\Delta_{\bGamma,1}(\check{\btheta}_{-\bv})| \prec d^{-c}.
\]
For $\Delta_{\bGamma,2}$, using the pseudo-Lipschitz property of $\ell_\test$ (Assumption \ref{ass:test_loss}), the Lipschitz property of $\eta$, and Cauchy-Schwarz, 
\[
\sup_{\bv \in \mathcal{N}_{S^\perp}} | \Delta_{\bGamma,2} | \prec\sup_{\bv \in \mathcal{N}_{S^\perp}} \left\{ \E [| f -f_{-\bv}|^C] + \E [| \<\check{\btheta}_{-\bv} , \bz -\bz_{-\bv}\>|^C] \right\} \prec d^{-c},
\]
where we used \eqref{eq:LODO_exp_f_f_bv} in Lemma \ref{lem:firstlemma} and \eqref{cor:leaveoneoutstability} in Lemma \ref{lem:surgate1}. We deduce that
\begin{equation}\label{eq:objective_diff_bound_at_theta_v}
\sup_{\bv \in \mathcal{N}_{S}^{\perp}} \left| \cP_n(\check{\btheta}_{-\bv}) - \cP_{n,-\bv}(\check{\btheta}_{-\bv}) \right| \prec d^{-c}.
\end{equation}

\paragraph{Step 3: Control $\cP_{n,-\bv}(\check{\btheta}_{-\bv}) - \cP_{n}(\check{\btheta})$.} By convexity of $\ell$ with respect to its second argument,
\[
\ell(y_i, \check{\btheta}^{\sT}\bz_i) - \ell(y_{i,-\bv}, \check{\btheta}_{-\bv}^{\sT}\bz_{i,-\bv}) \geq \partial_{\hat y} \ell(y_{i,-\bv}, \check{\btheta}_{-\bv}^{\sT}\bz_{i,-\bv}) (\check{\btheta}^{\sT}\bz_i - \check{\btheta}_{-\bv}^{\sT}\bz_{i,-\bv}) + r_i(\bv),
\]
where $|r_i (\bv)| = | \ell(y_i, \check{\btheta}^{\sT}\bz_i) - \ell(y_{i,-\bv},  \check{\btheta}^{\sT}\bz_i)| \leq C |y_i - y_{i,-\bv}|$ by Assumption \ref{assumption:loss}. Further write $\|\check{\btheta}\|_2^2 - \|\check{\btheta}_{-\bv}\|_2^2 = \|\check{\btheta} - \check{\btheta}_{-\bv}\|_2^2 + 2 \check{\btheta}_{-\bv}^{\sT}(\check{\btheta} - \check{\btheta}_{-\bv})$. Thus we can decompose the difference as
\begin{equation}\label{eq:decompo_diff_cp_n_v}
\begin{aligned}
    \cP_n(\check{\btheta}) - \cP_{n,-\bv}(\check{\btheta}_{-\bv}) \geq &~ \frac{1}{n}\sum_{i=1}^n \partial_{\hat y} \ell(y_{i,-\bv}, \check{\btheta}_{-\bv}^{\sT}\bz_{i,-\bv}) (\check{\btheta}^{\sT}\bz_i - \check{\btheta}_{-\bv}^{\sT}\bz_{i,-\bv}) + R_n (\bv) \\
&~ + \frac{\lambda}{2}\|\check{\btheta} - \check{\btheta}_{-\bv}\|_2^2 + \lambda \check{\btheta}_{-\bv}^{\sT}(\check{\btheta} - \check{\btheta}_{-\bv}) + \kappa_d \left( \|\check{\btheta}\|_\infty - \|\check{\btheta}_{-\bv}\|_\infty \right) \\
&~+ \btau \cdot \left( \bGamma^{\bW}(\check{\btheta}) - \bGamma^{\bW_{-\bv}}(\check{\btheta}_{-\bv}) \right),
\end{aligned}
\end{equation}
where $R_n (\bv) = \frac{1}{n} \sum_{i =1}^n r_i (\bv)$. The first-order condition for optimality at $\check{\btheta}_{-\bv}$ is given by the subgradient inclusion:
\[
\mathbf{0} \in \partial \cP_{n,-\bv}(\check{\btheta}_{-\bv}) = \nabla \hcR_{n,-\bv}(\check{\btheta}_{-\bv}) + \kappa_d \partial \|\check{\btheta}_{-\bv}\|_\infty.
\]
This implies that there exists a subgradient vector $\bg_{-\bv} \in \partial \|\check{\btheta}_{-\bv}\|_\infty$ such that \begin{equation} \nabla \hcR_{n,-\bv}(\check{\btheta}_{-\bv}) = -\kappa_d \bg_{-\bv}. \label{eq:subgradient_condition}\end{equation}
By definition of the subdifferential of the $\ell_\infty$-norm, the vector $\bg_{-\bv}$ satisfies $\|\bg_{-\bv}\|_1 \le 1$ and, crucially, $\bg_{-\bv}^{\sT}\check{\btheta}_{-\bv} = \|\check{\btheta}_{-\bv}\|_\infty$. Taking the inner product of \eqref{eq:subgradient_condition} with the vector $(\check{\btheta} - \check{\btheta}_{-\bv})$ gives:
\begin{equation} \label{eq:kkt_inner_product}
\nabla \hcR_{n,-\bv}(\check{\btheta}_{-\bv})^{\sT}(\check{\btheta} - \check{\btheta}_{-\bv}) = -\kappa_d \bg_{-\bv}^{\sT}(\check{\btheta} - \check{\btheta}_{-\bv}).
\end{equation}
Injecting this identity in \eqref{eq:decompo_diff_cp_n_v} simplifies the expression of the lower bound:
\begin{equation}
\begin{aligned}
    \cP_n(\check{\btheta}) - \cP_{n,-\bv}(\check{\btheta}_{-\bv}) \geq &~ \frac{1}{n}\sum_{i=1}^n \partial_{\hat y} \ell(y_{i,-\bv}, \check{\btheta}_{-\bv}^{\sT}\bz_{i,-\bv})(
\check{\btheta}^{\sT}(\bz_i -  \bz_{i,-\bv})) + R_n (\bv) \\
&~ + \frac{\lambda}{2}\|\check{\btheta} - \check{\btheta}_{-\bv}\|_2^2 +  \kappa_d \left( \|\check{\btheta}\|_\infty -\bg_{-\bv}^\sT \check{\btheta} \right) \\
&~+ \btau \cdot \left( \bGamma^{\bW}(\check{\btheta}) - \bGamma^{\bW_{-\bv}}(\check{\btheta}_{-\bv}) \right) - [\btau \cdot \nabla_\btheta \bGamma^{\bW_{-\bv}}(\check{\btheta}_{-\bv})]^\sT ( \check{\btheta} - \check{\btheta}_{-\bv}).
\end{aligned}
\end{equation}
The first line is  $\prec d^{-c}$ uniformly over $\bv$ by Lemma \ref{lem:firstlemma} and Lemma \ref{lem:samereanson}. For the third line,
\[
\begin{aligned}
    &~\Gamma_1^\bW (\check{\btheta}) -  \Gamma^{\bW_{-\bv}}_1 (\check{\btheta}_{-\bv}) - \nabla_\btheta \Gamma_1^{\bW_{-\bv}}(\check{\btheta}_{-\bv})^\sT ( \check{\btheta} - \check{\btheta}_{-\bv}) \\
    =&~ \check{\btheta}^\sT (\bLambda - \bLambda_{-\bv}) \check{\btheta} + (\check{\btheta} - \check{\btheta}_{-\bv})^\sT \bLambda_{-\bv} (\check{\btheta} - \check{\btheta}_{-\bv}),
\end{aligned}
\]
where $\bLambda := \mu_0^2 \1 \1^\sT + \mu_1^2 \bW\bW^\sT$ and $\bLambda_{-\bv} := \mu_0^2 \1 \1^\sT + \mu_1^2 \bW_{-\bv}\bW_{-\bv}^\sT$. For the first term, we use the decomposition \eqref{eq:decompo-delta_gamma-1} and write
\[
\check{\btheta}^\sT (\bLambda - \bLambda_{-\bv}) \check{\btheta} = \mu_1^2\< \bv , \bW^\sT\check{\btheta}\>^2 + 2 \mu_1\<\ba (\bv) , \bW_{-\bv}^\sT\check{\btheta}\> + \| \ba (\bv) \|_2^2 \geq - 2 \| \mu_1 \bW^\sT \check{\btheta} \|_2\| \ba (\bv) \|_2,
\]
which is lower bounded by $-O_{\prec}(d^{-1/16})$ uniformly over $\bv$.
For $\Gamma_2^{\bW}$, we Taylor expand the test loss around the LODO data:
\begin{equation}\label{equ:taylorexpansionone}
\begin{aligned}
     \Delta \Gamma_2 (\bv) := &~\Gamma_2^\bW (\check{\btheta}) -  \Gamma^{\bW_{-\bv}}_2 (\check{\btheta}_{-\bv}) - \nabla_\btheta \Gamma_2^{\bW_{-\bv}}(\check{\btheta}_{-\bv})^\sT ( \check{\btheta} - \check{\btheta}_{-\bv}) \\
     =&~ \E \left[ \nabla \ell_{\test} \left(y^\sPG_{-\bv}, \<\check{\btheta}_{-\bv} , \bz^\sPG_{-\bv} \>\right) \begin{pmatrix}
y^\sPG - y_{-\bv}^\sPG \\
\<\check{\btheta} , \bz^\sPG - \bz_{-\bv}^\sPG \>
\end{pmatrix} \right] \\
&~ + \frac{1}{2} \E\left[ \left\<\nabla^2 \ell_{\test} (\tilde y, \tilde u), \begin{pmatrix}
y^\sPG - y_{-\bv}^\sPG \\
\<\check{\btheta} , \bz^\sPG \> - \<\check{\btheta}_{-\bv} , \bz_{-\bv}^\sPG \>
\end{pmatrix}^{\otimes 2} \right\>\right].
\end{aligned}
\end{equation}
By the same argument as in Lemma \ref{lem:samereanson}, the first order term is $\prec d^{-c}$ uniformly over $\bv$. For the second order term, by pseudo-Lipschitz property of the second derivative and a truncation argument, it is dominated uniformly over $\bv$ by
\[
\begin{aligned}
\E \left[ ( y^\sPG - y_{-\bv}^\sPG)^2\right] + \E[(\<\check{\btheta}  - \check{\btheta}_{-\bv}, \bz^\sPG \>^2] + \E[(\<\check{\btheta}_{-\bv} , \bz^\sPG - \bz_{-\bv}^\sPG \>)^2].
\end{aligned}
\]
The first and third terms are $\prec d^{-c}$ uniformly over $\bv$ by Lemma \ref{lem:firstlemma} and Lemma \ref{lem:surgate1}. The second term is bounded by
\[
\E[(\<\check{\btheta}  - \check{\btheta}_{-\bv}, \bz^\sPG \>^2] \prec  (\check{\btheta}  - \check{\btheta}_{-\bv})^\sT (\bLambda +\bI ) (\check{\btheta}  - \check{\btheta}_{-\bv}),
\]
where we used Corollary \ref{corollary:OperatorNormF2cdecompose}. Further note that we assume $|\tau_2| \leq \tau_1 / \log^{K_{\Gamma}} d$ and $|\mu_1|\prec d^{-c}$. Thus, 
\[
\begin{aligned}
    &~\frac{\lambda}{2}\|\check{\btheta} - \check{\btheta}_{-\bv}\|_2^2 + \btau \cdot \left( \bGamma^{\bW}(\check{\btheta}) - \bGamma^{\bW_{-\bv}}(\check{\btheta}_{-\bv}) \right) - [\btau \cdot \nabla_\btheta \bGamma^{\bW_{-\bv}}(\check{\btheta}_{-\bv})]^\sT ( \check{\btheta} - \check{\btheta}_{-\bv}) \\
    \succeq&~ - d^{-c} +\frac{\lambda}{4}\|\check{\btheta} - \check{\btheta}_{-\bv}\|_2^2 + \frac{\tau_1}{2} (\check{\btheta} - \check{\btheta}_{-\bv})^\sT \bLambda (\check{\btheta} - \check{\btheta}_{-\bv}).
\end{aligned}
\]
Putting the above bounds together, we obtain that uniformly over $\bv$,
\begin{align}\label{equ:controloperator}
      \cP_n(\check{\btheta}) - \cP_{n,-\bv}(\check{\btheta}_{-\bv})  \succ  \|\check{\btheta} - \check{\btheta}_{-\bv}\|_2^2 - d^{-c}.
    \end{align}

    \paragraph{Step 4: Bound on $\|\check{\btheta} - \check{\btheta}_{-\bv}\|_2$.} By optimality condition,
    \[
    \begin{aligned}
    0 \geq \cP_n(\check{\btheta}) - \cP_n(\check{\btheta}_{-\bv}) =&~  \cP_n(\check{\btheta}) - \cP_{n,-\bv}(\check{\btheta}_{-\bv}) + \cP_{n,-\bv}(\check{\btheta}_{-\bv})- \cP_n(\check{\btheta}_{-\bv})  \\
    \succ &~ \|\check{\btheta} - \check{\btheta}_{-\bv}\|_2^2 - d^{-c},
    \end{aligned}
    \]
    where on the second line, we used \eqref{equ:controloperator} from Step 3 and \eqref{eq:objective_diff_bound_at_theta_v} from Step 2. We deduce that
     \begin{equation}\label{eq:uniform_control_finalsas}
       \sup_{\bv \in \mathcal{N}_{S}^\perp} \|\hat \btheta - \check{\btheta}_{-\bv}\|_2^2 \prec d^{-c}.
      \end{equation}

 \paragraph*{Step 5: Concluding.} Decompose
 \begin{equation}\label{eq:final_sum_decomposition}
 \left| \sum_{j=1}^p \hat\theta_j \<\bw_j,\bv\>^2 \right| \leq  \left| \sum_{j=1}^p (\check{\btheta}_{-\bv})_j \<\bw_j,\bv\>^2 \right| +  \left| \sum_{j=1}^p (\hat\theta_j - (\check{\btheta}_{-\bv})_j) \<\bw_j,\bv\>^2 \right| 
 \end{equation}
 The first term corresponds to $S(\bv)$ and was bounded in the proof of \eqref{eq:letussec} in Lemma \ref{lem:surgate1} by $d^{-c}$ uniformly over $\bv$. For the second term, we use that 
 \[
\sup_{\bv \in \mathcal{N}_{S}^\perp}  \left| \sum_{j=1}^p (\hat\theta_j - (\check{\theta}_{-\bv})_j) \<\bw_j,\bv\>^2 \right|  \leq \sup_{\bv \in \mathcal{N}_{S}^\perp}  \| \hbtheta -\check{\btheta}_{-\bv}\|_2 \cdot \sup_{\bv \in \mathcal{N}_{S}^\perp} \sqrt{\sum_{j \in [p]} \<\bw_j,\bv\>^4}.
 \]
 The second factor is $\prec 1$ (e.g., see \eqref{eq:union-bound-example-sum-power-4}). For the first factor, combining \eqref{equ:step2bound} in Step 1 and \eqref{eq:uniform_control_finalsas} in Step 4, we obtain $\prec d^{-c}$. Thus,
 \[
 \| \bW_{\setminus S}^\sT \bD_{\hbtheta} \bW_{\setminus S}\|_\op \leq 2 \sup_{\bv \in \mathcal{N}_{S}^\perp} \left| \sum_{j=1}^p \hat\theta_j \<\bw_j,\bv\>^2 \right| \prec d^{-c},
 \]
which concludes the proof.
    \end{proof}

\begin{proof}[Proof of Proposition \ref{prop:abstractCLTcondition} Equation \eqref{eq:fob_norm_CGE_loc}] It suffices to bound for each $k \in [S]$, $\| \be_k^\sT \bW^\sT \bD_{\hbtheta} \bW_{\setminus S} \|_2$. Without loss of generality, fix $k = 1$. Then
\[
\| \be_1^\sT \bW^\sT \bD_{\hbtheta} \bW_{\setminus S}\|_2 \leq 2 \sup_{\bv \in \mathcal{N}_{S}^\perp} \left| \sum_{j = 1}^p \hat \theta_j w_{j,1} \<\bw_j,\bv\> \right|.
\]
Proceeding as above, we decompose
\[
\left| \sum_{j = 1}^p \hat \theta_j w_{j,1} \<\bw_j,\bv\> \right| \leq \left| \sum_{j = 1}^p (\check{\theta}_{-\bv})_j w_{j,1} \<\bw_j,\bv\> \right| + \left| \sum_{j = 1}^p (\hat \theta_j-(\check{\theta}_{-\bv})_j) w_{j,1} \<\bw_j,\bv\> \right|.
\]
The first term is bounded similarly as \eqref{eq:partial_sum_check_theta_j_k}, while the second term is bounded using Cauchy-Schwarz and the uniform bound on $\| \hbtheta - \check{\btheta}_{-\bv} \|_2$. We omit these repetitive details.
\end{proof}

\clearpage
\section{Universality of the Test Error}
\label{sec:TestError}

This section is dedicated to the proof of Theorem~\ref{thm:universality_test_error}, which establishes the universality of the test error. Denote
\[
\begin{aligned}
\hbtheta_\btau =&~ \argmin_{\btheta} \widehat{\cR}_{n,p} (\btheta; \btau,\bZ,\boldf), \qquad &&\widehat{\cR}^*_{n,p} ( \btau,\bZ,\boldf) = \widehat{\cR}^*_{n,p} (\hbtheta_\btau; \btau,\bZ,\boldf).
\end{aligned}
\]

\begin{proof}[Proof of Theorem \ref{thm:universality_test_error}]
    Consider $\btau$ as prescribed in Assumption \ref{ass:test_error}. First, note that
    \[
    \begin{aligned}
    &~ \left| \frac{\widehat{\cR}^*_{n,p} ( \btau,\bZ^\RF,\boldf^\RF) - \widehat{\cR}^*_{n,p} ( (\tau_1,0),\bZ^\RF,\boldf^\RF)}{\tau_2} - \rho \right| \\
    \leq &~ \left| \frac{\widehat{\cR}^*_{n,p} ( \btau,\bZ^\RF,\boldf^\RF) - \widehat{\cR}^*_{n,p} ( \btau,\bZ^\sCG,\boldf^\sCG)}{\tau_2}  \right| + \left| \frac{\widehat{\cR}^*_{n,p} ( (\tau_1,0),\bZ^\RF,\boldf^\RF) - \widehat{\cR}^*_{n,p} ( (\tau_1,0),\bZ^\sCG,\boldf^\sCG)}{\tau_2}  \right| \\
    &~ + \left| \frac{\widehat{\cR}^*_{n,p} ( \btau,\bZ^\sCG,\boldf^\sCG) - \widehat{\cR}^*_{n,p} ( (\tau_1,0),\bZ^\sCG,\boldf^\sCG)}{\tau_2} - \rho \right| \\
    \overset{\P}{\to}&~ 0,
    \end{aligned}
    \]
    where we used Theorem \ref{thm:universality_Perturbed_risk} and Assumption \ref{ass:test_error}.

    Moreover,
    \[
    \widehat{\cR}^*_{n,p} ( \btau,\bZ^\RF,\boldf^\RF) \leq \widehat{\cR}^*_{n,p} (\hbtheta^\RF_{(\tau_1,0)}; \btau,\bZ^\RF,\boldf^\RF) = \widehat{\cR}^*_{n,p} ((\tau_1,0),\bZ^\RF,\boldf^\RF) + \tau_2 \Gamma_2^\bW (\hbtheta^\RF_{(\tau_1,0)}).
    \]
    Thus, taking $\tau_2 \in \{\pm (\log d)^{-2K_\Gamma}\}$, we get
    \[
    | \Gamma_2^\bW (\hbtheta^\RF_{(\tau_1,0)}) - \rho | \leq \left| \frac{\widehat{\cR}^*_{n,p} ( \btau,\bZ^\RF,\boldf^\RF) - \widehat{\cR}^*_{n,p} ( (\tau_1,0),\bZ^\RF,\boldf^\RF)}{\tau_2} - \rho \right|.
    \]
    The same holds for the CGE model, and we deduce that
    \[
    \Gamma^\bW_2 (\hbtheta^\RF_{(\tau_1,0)}) \overset{\P}{\to} \rho, \qquad \Gamma^\bW_2 (\hbtheta^\sCG_{(\tau_1,0)}) \overset{\P}{\to} \rho.
    \]
    Let's show that 
    \[
    \Gamma^\bW_2 (\hbtheta^\RF_{(\tau_1,0)}) - \cR_\test (\hbtheta^\RF_{(0,0)}; \P_{\bz^\RF,f^\RF}) \overset{\P}{\to} 0.
    \]
    An identical proof yields $ \Gamma^\bW_2 (\hbtheta^\sCG_{(\tau_1,0)}) - \cR_\test (\hbtheta^\sCG_{(0,0)}; \P_{\bz^\sCG,f^\sCG}) \overset{\P}{\to} 0$, which concludes the proof.

    For simplicity, denote $\hbtheta^\RF_{\tau_1} := \hbtheta^\RF_{(\tau_1,0)}$. By Assumption \ref{ass:test_loss} and Lemma \ref{lemma:sublevelgeometry}, we have $\Gamma^\bW_2 (\hbtheta^\RF_{\tau_1}) - L_\bW ( \hbtheta^\RF_{\tau_1})\overset{\P}{\to} 0$ for $K_\Gamma$ chosen large enough but independent of $d$. Furthermore, by the CLT applied to $\hbtheta^\RF_{\tau_1}$ (Theorem \ref{theorem:RecallBoundedLipschitzFunctionSupErgetisotropic}) with a standard truncation argument, we obtain
    \[
     \underbrace{\E \left[ \ell_{\test} \left(y^\sPG, \<  \hbtheta^\RF_{\tau_1}  , \bz^\sPG \>\right)\Big| \bW\right]}_{=L_{\bW}(\hat\btheta_{\tau_1}^\RF)}  -  \underbrace{\E \left[ \ell_{\test} \left(y^\RF, \<  \hbtheta^\RF_{\tau_1}  , \bz^\RF \>\right)\Big| \bW\right]}_{\cR_\test (\hbtheta^\RF_{\tau_1}; \P_{\bz^\RF,f^\RF})}
      \overset{\mathbb P}{\longrightarrow}  0
    \]
    (where $(y^\sPG,\bz^\sPG)$ and $(y^\RF,\bz^\RF)$ denote a test sample independent of $\hat\btheta_{\tau_1}^\RF$ conditional on $\bW$).
    It remains to show that the test error at $\hbtheta^\RF_{\tau_1}$ converges in probability to the test error at $\hbtheta^\RF_0$. Using the pseudo-Lipschitzness assumption on the test loss (Assumption \ref{ass:test_loss}),
    \[
    \begin{aligned}
    &~\left| \E \left[ \ell_{\test} \left(y^\RF, \<  \hbtheta^\RF_{\tau_1}  , \bz^\RF \>\right)\Big| \bW\right] - \E \left[ \ell_{\test} \left(y^\RF, \<  \hbtheta^\RF_{0}  , \bz^\RF \>\right)\Big| \bW\right]\right| \\
    \leq&~ C \E\left[1 + |y^\RF|^C +| \<  \hbtheta^\RF_{\tau_1}  , \bz^\RF \>|^{C} + | \<  \hbtheta^\RF_{0}  , \bz^\RF \>|^{C} \Big|\bW \right]^{1/2} \E\left[ \<\bz^\RF, \hbtheta^\RF_{\tau_1}- \hbtheta^\RF_0\>^2 \Big|\bW \right]^{1/2}
    \end{aligned}
    \]
    Using Lemma \ref{lemma:sublevelgeometry}, for some constant $K'>0$ not depending on $\tau_1$ and $K_\Gamma$, the first term is bounded with probability at least $1 - d^{-1}$ by $(\log d)^{K'}$. For the second term, by Corollary \ref{corollary:OperatorNormF2cdecompose}, the bounds $|\mu_0 \bones_p^\sT \hbtheta^\RF_{\tau_1}| \prec 1$ and $|\mu_0 \bones_p^\sT \hbtheta^\RF_{0}| \prec 1$, and Lemma \ref{lemma:OperatorNormFk}, we have likewise with probability at least $1 -d^{-1}$ that
    \begin{equation}\label{eq:lipsch_RF_tau_1_0}
     \E\left[ \<\bz^\RF, \hbtheta^\RF_{\tau_1}- \hbtheta^\RF_0\>^2 \Big|\bW\right] \leq \|\bV_+^\sT (\hbtheta^\RF_{\tau_1}- \hbtheta^\RF_0)\|_2^2 + (\log d)^{K'} \| \hbtheta^\RF_{\tau_1}- \hbtheta^\RF_0\|_2^2.
    \end{equation}

    Denote $\hat\btheta^t := t \hbtheta^\RF_{\tau_1} + (1-t) \hbtheta^\RF_0$ and $\widehat{\cR}_{n,p} (\btheta;\tau_1) :=\widehat{\cR}_{n,p} (\btheta;(\tau_1,0),\bZ,\boldf)$. Using the first order optimality  condition, the second order Taylor expansion of the risk gives
    \[
    \widehat{\cR}_{n,p} ( \hbtheta_0^\RF; \tau_1) = \widehat{\cR}_{n,p} ( \hbtheta_{\tau_1}^\RF; \tau_1) + \int_0^1 (1-t) (\hbtheta^\RF_{\tau_1}- \hbtheta^\RF_0)^\sT \nabla^2 \widehat{\cR}_{n,p} ( \hbtheta^t; \tau_1) (\hbtheta^\RF_{\tau_1}- \hbtheta^\RF_0) \de t.
    \]
    Note that with probability at least $1-d^{-1}$, we have $\hbtheta^t \in \widehat{S} (K)$ for all $t \in [0,1]$ for a constant $K>0$ sufficiently large, and thus by Lemma \ref{lem:strong_convex-landscape} to follow, we have
    \[
    \int_0^1\nabla^2 \widehat{\cR}_{n,p} ( \hbtheta^t; \tau_1) \de t \succeq (\log d)^{-K'} ( \bV_+ \bV_+^\sT + \bI_p )
    \]
    for a constant $K'>0$ not depending on $\tau_1$ and $K_\Gamma$.
    We deduce that 
    \begin{equation}\label{eq:upper-bound_tau_1-0}
    \|\bV_+^\sT (\hbtheta^\RF_{\tau_1}- \hbtheta^\RF_0)\|_2^2 + \| \hbtheta^\RF_{\tau_1}- \hbtheta^\RF_0\|_2^2  \leq (\log d)^{K'} \left( \widehat{\cR}_{n,p} ( \hbtheta_0^\RF; \tau_1) - \widehat{\cR}_{n,p} ( \hbtheta_{\tau_1}^\RF; \tau_1) \right).
    \end{equation}
    Furthermore,
    \[
    \begin{aligned}
    \widehat{\cR}_{n,p} ( \hbtheta_0^\RF; \tau_1) =&~ \widehat{\cR}_{n,p} ( \hbtheta^\RF_0; 0) + \tau_1 \cT_{K_\Gamma} (\| \bV_+ \hbtheta_0^\RF\|_2^2) \leq \widehat{\cR}_{n,p} ( \hbtheta^\RF_{\tau_1}; 0) + \tau_1 \cT_{K_\Gamma} (\| \bV_+ \hbtheta_0^\RF\|_2^2) \\
    =&~ \widehat{\cR}_{n,p} ( \hbtheta^\RF_{\tau_1}; \tau_1) + \tau_1 \left( \cT_{K_\Gamma} (\| \bV_+ \hbtheta_0^\RF\|_2^2)- \cT_{K_\Gamma} (\| \bV_+ \hbtheta_{\tau_1}^\RF\|_2^2)\right)\\ 
    \leq&~  \widehat{\cR}_{n,p} ( \hbtheta^\RF_{\tau_1}; \tau_1) + \tau_1 (\log d)^{K'},
    \end{aligned}
    \]
    where the last line holds with probability at least $1 - d^{-1}$ for some constant $K'>0$ and any $K_\Gamma>K'$, by Lemma \ref{lemma:sublevelgeometry}. Combining the above display with \eqref{eq:upper-bound_tau_1-0} and \eqref{eq:lipsch_RF_tau_1_0}, we deduce that there exists a constant $K''>0$ not depending on $\tau_1$ or $K_\Gamma$ such that with probability at least $1 - Cd^{-1}$,
    \[
    \left| \E \left[ \ell_{\test} \left(y^\RF, \<  \hbtheta^\RF_{\tau_1}  , \bz^\RF \>\right)\Big| \bW\right] - \E \left[ \ell_{\test} \left(y^\RF, \<  \hbtheta^\RF_{0}  , \bz^\RF \>\right)\Big| \bW\right] \right| \leq \tau_1(\log d)^{K''}.
    \]
    Taking $K_\Gamma>K''$ and recalling that $\tau_1 = (\log d)^{-K_\Gamma}$ concludes the proof.
\end{proof}

The following lemma shows a version of the Hessian lower bound in Lemma \ref{lemma:sublevelgeometry}, with a leading factor $(\log d)^{-K'}$ for $\bV_+\bV_+^\top$ that does not depend on $\tau_1$ or $K_\Gamma$, using the additional condition of Assumption \ref{assumption:lsc}.

\begin{lemma}\label{lem:strong_convex-landscape}
Under Assumptions \ref{assumption:scaling}--\ref{assumption:activation} and \ref{assumption:lsc}, the following holds. Consider the sub-level set 
\[
\widehat{S} (K) = \left\{ \btheta : \frac{1}{n} \sum_{i = 1}^n \ell (y_i,\<\btheta,\bz_i\>) + \frac{\lambda}{2}\|\btheta \|_2^2 < (\log d)^K \right\},
\]
for data from either the RF or CG models. Let
$\widehat{\cR}_{n,p}(\btheta;(\tau_1,0))$ denote the empirical risk setting $\tau_2=0$.
Then for any constants $C,K>0$, there exists $K'>0$ such that for any sufficiently large constant $K_\Gamma>0$ and all $0\leq \tau_1\leq (\log d)^{-K_\Gamma}$, with probability at least $1 - d^{-C}$,
\[
\bH (\btheta) := \nabla^2 \widehat{\cR}_{n,p} (\btheta;(\tau_1,0)) \succeq \frac{1}{(\log d)^{K'}} \bV_+ \bV_+^\sT + \frac{\lambda}{2} \bI, \qquad \text{for all $\btheta \in \widehat{S} (K)$.}
\]
\end{lemma}

\begin{proof} Consider the RF model with $\bz_i := \bz_i^\RF$. The CG model follows a similar argument.
By Lemma \ref{lemma:sublevelgeometry}, for $K_{\Gamma}$ bigger than a constant, $\Gamma_1^\bW (\btheta) =   \| \bV_+^\sT \btheta \|_2^2$ for all $\btheta \in \widehat{S} (K)$ with probability $1 - d^{-C}$. Then, since the loss $\ell(\cdot)$ is convex,
we have $\bH(\btheta) \succeq \lambda\bI$. Noting that $\mu_0 \neq 0$ by Assumption \ref{assumption:activation}, we may
write 
\[\bz_i=\begin{pmatrix} \bV_+ & \bI_p \end{pmatrix} \begin{pmatrix} \bu_i \\ \bv_i
\end{pmatrix}\]
where
\begin{align*}
\bu_i&=\begin{pmatrix}
1+\frac{\mu_2}{d\mu_0}\sum_{j=1}^d \frac{x_{ij}^2-1}{\sqrt{2}}\\ 
          \bx_i
\end{pmatrix},
\quad \bv_i=\mu_2 \bV_{2c}\bh_2(\bx_i) + \mu_3\bV_3\bh_3(\bx_i)+\ldots
+\mu_D\bV_D\bh_D(\bx_i).
\end{align*}
Then
\begin{align}
 \bH(\btheta)
&\succeq\frac{1}{n}\sum_{i=1}^n \ell''(y_i,\<\btheta,\bz_i\>)
\bz_i\bz_i^\top+\lambda\bI\\
&=\begin{pmatrix} \bV_+ & \bI \end{pmatrix}
\underbrace{\left(\frac{1}{n}\sum_{i=1}^n
\ell''(y_i,\<\btheta,\bz_i\>)\begin{pmatrix} \bu_i \\ \bv_i \end{pmatrix}
\begin{pmatrix} \bu_i \\ \bv_i \end{pmatrix}^\top
+\begin{pmatrix} \bzero & \bzero \\
\bzero & \frac{\lambda}{2}\bI \end{pmatrix}\right)}_{:=\bA}
\begin{pmatrix}\bV_+ & \bI \end{pmatrix}^\top+\frac{\lambda}{2} \bI.\label{eq:Hlowerbound}
\end{align}

It remains to bound the smallest eigenvalue of $\bA$. Let us write
\[\bA=\begin{pmatrix} \bA_{11} & \bA_{12} \\ \bA_{21} &
\bA_{22}\end{pmatrix},
\qquad \bA_{11}=\frac{1}{n}\sum_{i=1}^n
\ell''(y_i,\<\btheta,\bz_i\>)\bu_i\bu_i^\top.\]
Then
\[\bA=\begin{pmatrix} \bA_{11} & \bA_{12} \\ \bA_{21} &
\bA_{22}\end{pmatrix}
=\underbrace{\begin{pmatrix} \bI & \bzero \\ \bA_{21}\bA_{11}^{-1} & \bI
\end{pmatrix}}_{:=\bL}
\begin{pmatrix} \bA_{11} & \boldsymbol{0} \\ \boldsymbol{0} & \bS_{11}
\end{pmatrix} \underbrace{\begin{pmatrix} \bI & \bA_{11}^{-1}\bA_{12} \\
\bzero & \bI \end{pmatrix}}_{:=\bL^\top}\]
where $\bS_{11}$ is the Schur-complement of the upper left block,
$\bS_{11}=\bA_{22}-\bA_{21}\bA_{11}^{-1}\bA_{12}$. Then
\begin{equation}\label{eq:lambdaminschur}
\lambda_{\min}(\bA) \ge \min\left\{\lambda_{\min}(\bA_{11}),
\lambda_{\min}(\bS_{11})\right\} \cdot \sigma_{\min}(\bL)^2,
\end{equation}
where $\lambda_{\min}(\cdot)$ denotes the smallest eigenvalue, and
$\sigma_{\min}(\bL)$ is the smallest singular value of $\bL$.

By Lemma \ref{lemma:predictionbound},
the bound $\|\by\|_\infty \prec 1$, Lemma \ref{lemma:yzbounds}, and
the local strong convexity condition of Assumption \ref{assumption:lsc}, we have for any constants $C,K>0$, there exist $c_0,K'>0$ such that
with probability at least $1-d^{-C}$, every $\btheta \in \widehat S(K)$
satisfies
\begin{equation}\label{eq:losscurvaturebound}
\frac{1}{n}\sum_{i=1}^n \1\Big\{\ell''(y_i,\<\btheta,\bz_i\>)
\geq \frac{1}{(\log d)^{K'}}\Big\} \geq c_0.
\end{equation}
To bound $\lambda_{\min}(\bA_{11})$, given any constant $C>0$, let
$c_0,K'>0$ be the constants of \eqref{eq:losscurvaturebound},
and let $\widehat \cI \subseteq [n]$ be those indices for
which $\ell''(y_i,\<\btheta,\bz_i\>) \geq 1/(\log d)^{K'}$. Then
\[\bA_{11} \succeq
\frac{1}{n(\log d)^{K'}}\sum_{i \in \widehat \cI} \bu_i\bu_i^\top.\]
and \eqref{eq:losscurvaturebound} ensures
\[\P[|\widehat \cI| \geq c_0n \text{ for all } \btheta \in \widehat S(K)] \geq
1-d^{-C}.\]
Furthermore, we have $|\mu_2| \prec 1$ and $|\mu_0|^{-1} \prec 1$ by Assumption
\ref{assumption:activation}, and $\sum_{j=1}^d (x_{ij}^2-1)/d \prec d^{-1/2}$.
Then, setting $m=\lfloor c_0 n \rfloor$ and $\delta=(\log d)^{-K'}$,
for any fixed unit vector $\bw \in \S^d \subset \R^{d+1}$,
we may check analogously to
\eqref{eq:anticoncentration} that $\P[\<\bw,\bu_i\>^2 \leq 2n\delta/m]
\leq 2C_0\sqrt{2n\delta/m}$ for a constant $C_0>0$.
Then the same argument as leading to \eqref{eq:uniformlowersingularvalue}
shows that for all large $d$, we have
\[\P\left[\inf_{\bw \in \S^d} \bw^\top\left(\frac{1}{n}
\sum_{i \in \widehat \cI}
\bu_i\bu_i^\top \right)\bw \leq \frac{1}{2(\log d)^{K'}}
\text{ for all } \btheta \in \widehat S(K)\right]
\leq 2d^{-C}.\]
Combining with the above, $\lambda_{\min}(\bA_{11})$ is lower bounded as
\begin{equation}\label{eq:lambdaminA11}
\lambda_{\min}(\bA_{11}) \succ 1 \text{ simultaneously over }
\btheta \in \widehat S(K).
\end{equation}

To bound $\lambda_{\min}(\bS_{11})$, note
that if $\bA \succeq \bB$ and $\bB$ is strictly positive-definite,
then $\bA^{-1} \preceq \bB^{-1}$,
so $[\bA^{-1}]_{22} \preceq [\bB^{-1}]_{22}$ and hence
$\bA_{22}-\bA_{21}\bA_{11}^{-1}\bA_{12}
=([\bA^{-1}]_{22})^{-1} \succeq ([\bB^{-1}]_{22})^{-1}
=\bB_{22}-\bB_{21}\bB_{11}^{-1}\bB_{12}$. This shows that $\bS_{11} \succeq
(\lambda/2)\bI$, so
\begin{equation}\label{eq:lambdaminS11}
\lambda_{\min}(\bS_{11}) \geq \lambda/2.
\end{equation}

Finally, to bound $\sigma_{\min}(\bL)$, observe that
\[\bA_{11}^{-1}\bA_{12}
=\left(\frac{1}{n}\sum_{i=1}^n \ell''(y_i,\<\btheta,\bz_i\>)
\bu_i\bu_i^\top\right)^{-1}
\left(\frac{1}{n}\sum_{i=1}^n \ell''(y_i,\<\btheta,\bz_i\>)
\bu_i\bv_i^\top\right)
:=(\bU_1\bD\bU_1^\top)^{-1}(\bU_1\bD\bU_2)\]
where we set $\bU_1=n^{-1/2}[\bu_1,\ldots,\bu_n]$,
$\bU_2=n^{-1/2}[\bv_1,\ldots,\bv_n]$, and
$\bD=\diag(\{\ell''(y_i,\<\btheta,\bz_i\>)\}_{i=1}^n)$.
Then $\|\bA_{11}^{-1}\bA_{12}\|_\op
\leq \|(\bU_1\bD\bU_1^\top)^{-1}\bU_1\bD^{1/2}\|_\op
\|\bD\|_\op^{1/2}\|\bU_2\|_\op$.
We have $\|(\bU_1\bD\bU_1^\top)^{-1}\bU_1\bD^{1/2}\|_\op
=[\lambda_{\min}(\bU_1\bD\bU_1^\top)]^{-1/2}=[\lambda_{\min}(\bA_{11})]^{-1/2}
\prec 1$ by \eqref{eq:lambdaminA11}.
By Assumption \ref{assumption:loss} for $\ell(\cdot)$, we have $\|\bD\|_\op
\prec 1$. By Corollary \ref{corollary:OperatorNormF2cdecompose},
Lemma \ref{lemma:BernsteinInequalityHermiteFeatures}, and
Lemma \ref{lemma:OperatorNormZkFk} which show
$\|\bV_{2c}\|_\op \prec 1$, $\|\bh_2(\bX)\|_\op \prec d$, and
$\|\bV_k\bh_k(\bX)\|_\op \prec d$ for $k \geq 3$, we have $\|\bU_2\|_\op \prec
1$. This implies that simultaneously over $\btheta \in \widehat S(K)$,
\begin{equation}\label{eq:A11invA12}
\|\bA_{11}^{-1}\bA_{12}\|_\op \prec 1.
\end{equation}
For any unit vector $\bw=(\bw_1,\bw_2)$, note that $\|\bL^\top \bw\|_2^2
=\|\bw_1+\bA_{11}^{-1}\bA_{12}\bw_2\|_2^2+\|\bw_2\|_2^2$.
Supposing that
$\|\bA_{11}^{-1}\bA_{12}\|_\op \leq (\log d)^{K'}$, 
if $\|\bw_2\|_2 \leq 1/2(\log d)^{K'}$ then
$\|\bL^\top \bw\|_2^2 \geq
\|\bw_1+\bA_{11}^{-1}\bA_{12}\bw_2\|_2^2 \geq c$ for a constant $c>0$,
while if $\|\bw_2\|_2 \geq 1/2(\log d)^{K'}$ then
$\|\bL^\top \bw\|_2^2 \geq \|\bw_2\|_2^2 \geq 1/4(\log d)^{2K'}$.
Thus, \eqref{eq:A11invA12} implies
\begin{equation}\label{eq:sigmaminL}
\sigma_{\min}(\bL) \succ 1 \text{ simultaneously over } 
\btheta \in \widehat S(K).
\end{equation}
Applying \eqref{eq:lambdaminA11}, \eqref{eq:lambdaminS11}, and
\eqref{eq:sigmaminL} to \eqref{eq:lambdaminschur} shows our desired lower bound
the smallest eigenvalue for $\bA$,
\begin{equation}\label{eq:lambdaminA}
\lambda_{\min}(\bA) \succ 1 \text{ simultaneously over } 
\btheta \in \widehat S(K).
\end{equation}
The lemma follows by applying this to \eqref{eq:Hlowerbound}.


\end{proof}

\newpage

\section{Technical Background}
\label{sec:TechnicalBackground}

In this section, we provide additional background on Hermite polynomials and prove the technical results on random matrix concentration that are used thoughout the proofs.
\subsection{Hermite Polynomials}\label{section:HermitePolynomials}


We denote the orthonormal
Hermite polynomials of a single variable $x \in \R$ by
\begin{equation}
\He_k(x)=\frac{(-1)^k}{\sqrt{k!}} e^{\frac{x^2}{2}}
\frac{d^k}{dx^k}e^{-\frac{x^2}{2}} \text{ for } k=0,1,2,\ldots
\end{equation}
These have the following basic properties:
\begin{itemize}
\item $\He_k(x)$ is a polynomial of degree $k$, with leading coefficient of
$x^k$ equal to $1/\sqrt{k!}$.
\item (orthonormality) If $X \sim \normal(0,1)$ then $\E[\He_k(X)\He_\ell(X)]=\1\{k=\ell\}$.
\item (sign symmetry) $\He_k(-x) = (-1)^k\He_k(x)$.
\item (differentiation rule) $\He_k'(x) = \sqrt{k}\,\He_{k-1}(x)$.
\item (three-term recurrence) $\sqrt{k+1}\,\He_{k+1}(x) = x\,\He_k(x) - \sqrt{k}\,\He_{k-1}(x)$.
\end{itemize}
We denote the corresponding multivariate Hermite polynomial of $\bx \in \R^d$ by
\begin{equation}
\He_{\bk}(\bx) = \prod_{i=1}^d \He_{k_i}(x_i)
\text{ for } \bk \in \{0,1,2,\ldots\}^d.
\end{equation}
The total degree of $\He_\bk(\bx)$ is given by $\|\bk\|_1:=k_1+\ldots+k_d$. The
above orthonormality of $\{\He_k\}_{k \geq 0}$ implies also the orthonormality
of $\{\He_{\bk}\}_{\bk \in \{0,1,2,\ldots\}^d}$, i.e.\ if
$\bx \sim \normal(0,\id)$, then
\[\E[\He_\bk(\bx)\He_{\pmb{\ell}}(\bx)]=\1\{\bk=\pmb{\ell}\}.\]
The number of multi-indices with $\|\bk\|_1=k$ is
$B_{d,k}=\binom{d+k-1}{k}$.
We collect the corresponding degree-$k$ multivariate Hermite polynomials
into the vector
\begin{equation}
        \label{eq:HermiteVectorunit}
\bh_k(\bx)=(\He_\bk(\bx))_{\|\bk\|_1=k} \in \R^{B_{d,k}}, 
\end{equation}
where the ordering of its coordinates will be clear from context.

Let $(\R^d)^{\odot k}$ denote the space of symmetric tensors in $(\R^d)^{\otimes
k}$. For some arguments, it will be convenient to identify $\bh_k(\bx)$ as an
element of $(\R^d)^{\odot k}$ via the following isometry
$\iota:\R^{B_{d,k}} \to (\R^d)^{\odot k}$: For any index tuple
$\bi=(i_1,\cdots,i_k) \in [d]^k$, let
\[\operatorname{type}(\bi) \in \{\bk \in \{0,1,2,\ldots\}^d:\|\bk\|_1=k\}\]
count the number of occurrences of each distinct index $i \in [d]$, e.g.\ for
$k=3$ and $\bi=(1,3,1)$, $\operatorname{type}(\bi)=(2,0,1,0,\ldots,0)$. Then
for any $\bv \in \R^{B_{d,k}}$ (indexed by $\{\bk:\|\bk\|_1=k\}$), its image
$\iota(\bv) \in (\R^d)^{\odot k}$ is the symmetric tensor with entries
\begin{align}\label{eq:iota_entries}
\iota(\bv)_\bi=\binom{k}{k_1,\ldots,k_d}^{-1/2} v_{\bk} \text{ if }
\operatorname{type}(\bi)=\bk=(k_1,\ldots,k_d).
\end{align}
Note that for each fixed $\bk=(k_1,\ldots,k_d)$, we have
$|\{\bi \in [d]^k:\operatorname{type}(\bi)=\bk\}|=\binom{k}{k_1,\ldots,k_d}$,
so this normalization of entries of $\iota(\bv)$ gives the isometric property
$\langle \iota(\bv),\iota(\bu) \rangle=\langle \bv,\bu \rangle$.

We define
\[\bH_k(\bx)=\iota(\bh_k(\bx)) \in (\R^d)^{\odot k}.\]
Entries of $\bH_k(\bx)$ are given by $1/\sqrt{k!}$ times the \emph{monic}
Hermite polynomials of degree $k$, e.g.\
\begin{equation}\label{eq:H2form}
\bH_2(\bx)=\frac{1}{\sqrt{2}}\begin{pmatrix} x_1^2-1 & x_1x_2 & \cdots &
x_1x_d \\ x_1x_2 & x_2^2-1 & \cdots & x_2x_d \\
\vdots & \vdots & \ddots & \vdots \\
x_1x_d & x_2x_d & \cdots & x_d^2-1 \end{pmatrix} \in (\R^d)^{\odot 2}.
\end{equation}

\subsubsection{Hermite polynomial identities}
\begin{lemma}[Hermite explicit, translation, and scaling formulas]\label{lem:hermite-formulas}
For every $k\ge 0$ and all $x,y,\gamma\in\R$ the following hold.
\begin{enumerate}
\item[\textup{(i)}] \textbf{Explicit formula.} \cite[Chapter V]{szeg1975orthogonal}
\begin{equation}\label{lemma:FormualHermite2}
  \He_k(x)=\sqrt{k!}\sum_{i=0}^{\lfloor k/2\rfloor}\frac{(-1)^i}{2^i\,i!\,(k-2i)!}\,x^{k-2i}.
\end{equation}
\item[\textup{(ii)}] \textbf{Translation  formula.} \cite[Chapter V]{szeg1975orthogonal}
\begin{equation}\label{lemma:SumHermite}
  \sqrt{k!}\,\He_k(x+y)=\sum_{i=0}^k \binom{k}{i}\,\sqrt{i!}\,\He_i(x)\,y^{\,k-i}.
\end{equation}
\item[\textup{(iii)}] \textbf{Multiplication (scaling) formula.} \cite[\href{https://dlmf.nist.gov/18.18.E13}{§18.18(iii)}]{NIST:DLMF}.
\begin{equation}\label{lemma:factorHermite}
  \sqrt{k!}\,\He_k(\gamma x)
  =\sum_{i=0}^{\lfloor k/2\rfloor}\gamma^{\,k-2i}(\gamma^2-1)^i
   \binom{k}{2i}\frac{(2i)!}{i!\,2^{\,i}}\,
   \sqrt{(k-2i)!}\,\He_{k-2i}(x).
\end{equation}
\end{enumerate}
\end{lemma}
 
 
\begin{lemma}\label{lemma:SymmetricTensorProductMalliavinS}
For any $k \geq 1$ and any unit vector $\bw \in \mathbb{S}^{d-1}$,
\[\He_k(\<\bw,\bx\>)=\<\bq_k(\bw),\bh_k(\bx)\>
=\<\bw^{\otimes k},\bH_k(\bx)\>\]
where we define
\[\bq_k(\bw)=\bigg(\binom{k}{k_1,\ldots,k_d}^{1/2}\prod_{j \in [d]}
w_j^{k_j}\bigg)_{\|\bk\|_1=k} \in \R^{B_{d,k}}.\]
Consequently, $\iota(q_k(\bw))=\bw^{\otimes k}$.
\end{lemma}

\begin{proof}
This identity is the content of \cite[Proposition
1.1.4]{MalliavinCalculusRelated2006} which is proven in a setting of multiple
Wiener integrals.
\end{proof}

  
    \begin{lemma}\cite{CARAMELLINO2024110239,jansonGaussianHilbertSpaces1997}
        \label{lemma:multivariate_gaussian_hermite}
        Let $\bx \sim \normal(0,\id_d)$. Consider any 
$k_1,\ldots,k_m \geq 1$ and any deterministic
unit vectors $\bw_1,\ldots,\bw_m \in \mathbb{S}^{d-1}$.

        Let $\mathcal{F}$ be the set of all multigraphs on $[m]$ without
self-loops and having the degree sequence $k_1,\ldots,k_m$. For any multigraph
$G \in \mathcal{F}$ and pair of distinct vertices $i,j \in [m]$, let
$\nu_G(i,j) \geq 0$ denote the multiplicity of the edge $(i,j)$ in $G$. Define
        \[
        \operatorname{val}(G) := \prod_{1 \leq i<j \leq m} \frac{1}{\nu_G(i,j)!}
\langle \bw_i, \bw_j \rangle^{\nu_G(i,j)}.
\]
Then
        \[
        \mathbb{E}\left[ \prod_{i=1}^m \He_{k_i}(\<\bw_i,\bx\>) \right] = \left(
\prod_{i=1}^m \sqrt{k_i!} \right) \sum_{G \in \mathcal{F}} \operatorname{val}(G). 
        \]
        \end{lemma}
\subsection{Concentration inequalities}
\label{sec:ConcentrationInequality}

\subsubsection{Concentration inequalities for Hermite features}

\begin{lemma}\label{lemma:ConcentrationHermiteFeatures}
Let $\bx \sim \normal(0,\id_d)$. Fix any $k \geq 1$. Then there exist constants
$C,C',c>0$ depending only on $k$ such that with probability at least $1-Ce^{-cd}$,
    \begin{equation}
\|\bh_k(\bx)\|_2^2 \leq C'd^k.
    \end{equation} 
\end{lemma}

\begin{proof}
Applying Lemma \ref{lemma:SymmetricTensorProductMalliavinS} and the three-term
recurrence for $\He_k$, for any $k \geq 2$ and $\bw \in \mathbb{S}^{d-1}$,
\begin{align*}
\sqrt{k}\<\bw^{\otimes k},\bH_k(\bx)\>
&=\sqrt{k}\,\He_k(\<\bw,\bx\>)\\
&=\< \bw,\bx\>\He_{k-1}(\<\bw,\bx\>)-\sqrt{k-1}\,\He_{k-2}(\<\bw,\bx\>)\\
&=\< \bw,\bx\>\<\bw^{\otimes k-1},\bH_{k-1}(\bx)\>
-\sqrt{k-1}\<\bw^{\otimes 2},\id\>\<\bw^{\otimes k-2},\bH_{k-2}(\bx)\>\\
&=\Big\<\bw^{\otimes k},\bx \otimes \bH_{k-1}(\bx)-\sqrt{k-1}\,\id \otimes
\bH_{k-2}(\bx)\Big\>\\
&=\Big\<\bw^{\otimes k},\Sym\Big(\bx \otimes \bH_{k-1}(\bx)-\sqrt{k-1}\,\id \otimes
\bH_{k-2}(\bx)\Big)\Big\>
\end{align*}
Since $\{\bw^{\otimes k}:\bw \in \mathbb{S}^{d-1}\}$ spans the symmetric
tensor space $(\R^d)^{\odot k}$, this implies the identity
\[\bH_k(\bx)=\frac{1}{\sqrt{k}}\,
\Sym\Big(\bx \otimes \bH_{k-1}(\bx)-\sqrt{k-1}\,\id \otimes
\bH_{k-2}(\bx)\Big).\]
Iterating this identity for $\bH_{k-1}(\bx),\bH_{k-2}(\bx)$, and applying
the base case $\bH_1(\bx)=\bx$, this shows that for some coefficients
$\{c_{k,a}\}_{a=0,\ldots,\lfloor k/2 \rfloor}$ depending only on $k$,
\[\bH_k(\bx)=\sum_{a=0}^{\lfloor k/2 \rfloor}
c_{k,a}\,\Sym(\id^{\otimes a} \otimes \bx^{\otimes k-2a})
=\sum_{a=0}^{\lfloor k/2 \rfloor} \frac{c_{k,a}}{k!}
\sum_{\text{permutations } \pi \text{ of } [k]}
\pi(\id^{\otimes a} \otimes \bx^{\otimes k-2a})\]
where $\pi(\id^{\otimes a} \otimes \bx^{\otimes k-2a})$ denotes the tensor
$\id^{\otimes a} \otimes \bx^{\otimes k-2a} \in (\R^d)^{\otimes k}$ with indices
permuted by $\pi$. For any two permutations $\pi,\pi'$ of $[k]$ and indices
$a,a' \in \{0,\ldots,\lfloor k/2 \rfloor\}$, denote
\[f_{\pi,\pi',a,a'}(\bx)=\<\pi(\id^{\otimes a} \otimes \bx^{\otimes k-2a}),
\pi'(\id^{\otimes a'} \otimes \bx^{\otimes k-2a'})\>.\]
Then, since $\|\bh_k\|_2^2=\|\bH_k\|_F^2$,
to establish the lemma it suffices to show
\begin{equation}\label{eq:Hermitequadconc}
\P\Big[\big|f_{\pi,\pi',a,a'}(\bx)\big| \geq C'd^k\Big] \leq Ce^{-cd}
\end{equation}
for some $k$-dependent constants $C,C',c>0$ and each fixed $(\pi,\pi',a,a')$.

We claim that for each $(\pi,\pi',a,a')$,
\begin{equation}\label{eq:Hermitequaddecomp}
f_{\pi,\pi',a,a'}(\bx)=\|\bx\|_2^{2b}d^{b'}
\text{ for some integers } b,b' \geq 0 \text{ with } b+2b' \leq k.
\end{equation}
To see this, consider the multi-graph on $[k]$ with a red edge connecting each
pair $(\pi(1),\pi(2))$, $(\pi(3),\pi(4))$, $\ldots$, $(\pi(2a-1),\pi(2a))$
and a blue edge connecting each pair
$(\pi'(1),\pi'(2))$, $\ldots$, $(\pi'(2a'-1),\pi'(2a'))$. Then
$f_{\pi,\pi',a,a'}(\bx)$ factorizes as a product over connected components of
this multi-graph, where each component is either an isolated vertex, a linear
chain, or a simple cycle since each vertex has degree at most 2. 
Each isolated vertex and linear chain contributes a factor $\|\bx\|_2^2$ to
$f_{\pi,\pi',a,a'}(\bx)$, while each simple cycle contributes a factor $d$.
Letting $b$ be the number of isolated vertices/linear chains and $b'$ the number
of cycles, we must have $b+2b' \leq k$, and this shows
\eqref{eq:Hermitequaddecomp}. In particular, $b \leq k-b'$. Thus,
\[\P\Big[\big|f_{\pi,\pi',a,a'}(\bx)\big| \geq C'd^k\Big]
=\P\Big[\|\bx\|_2^{2b} \geq C'd^{k-b'}\Big]
\leq \P\Big[\|\bx\|_2^{2b} \geq C'd^b\Big],\]
and \eqref{eq:Hermitequadconc} now follows from a standard chi-squared tail
bound for $\|\bx\|_2^2$.
\end{proof}

\begin{lemma}
    \label{lemma:BernsteinInequalityHermiteFeatures}
Let $\bx_1,\cdots,\bx_n \overset{iid}{\sim} \normal(0,\id_d)$. Fix any $k \geq
1$, and consider
\begin{align}
        \bh_k(\bX) = \begin{pmatrix} \bh_k(\bx_1)^\top  \\ \bh_k(\bx_2)^\top  \\ \vdots
\\ \bh_k(\bx_n)^\top  \end{pmatrix} \in \R^{n \times B_{d,k}}.
        \end{align}
        Then for any $D>0$, there exist constants $C,K,c>0$ depending only on
$k,D$ such that with probability at least $1-Cd^{-D}-Cne^{-cd}$,
        \begin{equation}
            \|\bh_k(\bX)\|_{\mathrm{op}}^2 \leq (\log d)^K \max(n,d^k).
        \end{equation} 
\end{lemma}

\begin{proof}
We bound $n^{-1}\bh_k(\bX)^\top\bh_k(\bX)$
by first truncating the Hermite polynomial features using Lemma
\ref{lemma:ConcentrationHermiteFeatures}, and then applying the matrix Bernstein
inequality \cite[Theorem 5.4.1]{vershyninHighDimensionalProbabilityIntroduction2018} to the truncated matrix.

Fixing a large enough constant $C_0>0$, define
\[\bM_i=\bh_k(\bx_i)\bh_k(\bx_i)^\top \1\{\|\bh_k(\bx_i)\|_2^2 \leq
C_0d^k\}.\]
Then
\begin{align}
	\frac{1}{n}\bh_k(\bX)^\top\bh_k(\bX)-\E \bM_i&= \frac{1}{n}\sum_{i =
1}^n(\bh_k(\bx_i)\bh_k(\bx_i)^\sT  - \E \bM_i)\\
&=\underbrace{\frac{1}{n}\sum_{i = 1}^n(\bM_i  - \E\bM_i)}_{:=\mathrm{I}}
+\underbrace{\frac{1}{n}\sum_{i = 1}^n\bh_k(\bx_i)\bh_k(\bx_i)^\sT
\1_{\{\|\bh_k(\bx_i)\|_2^2>C_0d^k\}}}_{:=\mathrm{II}}.
\end{align}
Choosing $C_0$ large enough so that $B_{d,k}+d^k \leq C_0d^k$,
Lemma \ref{lemma:ConcentrationHermiteFeatures} implies $\P[\mathrm{II} \neq 0]
\leq ne^{-cd}$. We bound $\mathrm{I}$ using the matrix Bernstein inequality,
noting that $\|n^{-1}\bM_i\|_\op \leq C_0d^k/n$ and
that the matrix variance satisfies
\begin{align}
	\left\|\frac{1}{n^2}\sum_{i = 1}^n
\E(\bM_i-\E\bM_i)^2 \right\|_{\mathrm{op}}
\leq  \frac{1}{n}\|\E \bM_i^2\|_{\mathrm{op}} 
\leq \frac{C_0d^k}{n}\big\|\E \bh_k(\bx_i)\bh_k(\bx_i)^\top\big\|_\op
=\frac{C_0d^k}{n}.
\end{align}
Thus matrix Bernstein's inequality yields
	$\P[\|\mathrm{I}\|_{\mathrm{op}} \ge t] \leq CB_{d,k}
\exp(-c(n/d^k)\min(t^2,t))$. Combining these bounds for $\mathrm{I}$ and
$\mathrm{II}$,
\[\P\Big[\Big\|n^{-1}\bh_k(\bX)^\top\bh_k(\bX)-\E\bM_i\Big\|_\op \geq t\Big]
\leq CB_{d,k}\exp\Big({-}\frac{cn}{d^k}\min(t^2,t)\Big)
+Cn\exp(-cd).\]
Thus, using $B_{d,k} \leq Cd^k$,
$n^{-1}\|\bh_k(\bX)\|_\mathrm{op}^2
\leq \|n^{-1}\bh_k(\bX)^\top \bh_k(\bX)-\E\bM_i\|_{\mathrm{op}}+\|\E\bM_i\|_{\mathrm{op}}$,
and $\|\E\bM_i\|_\op \leq \|\E \bh_k(\bx_i)\bh_k(\bx_i)^\top \|_\op=1$, we have
\[\P\Big[\|\bh_k(\bX)\|_\op^2 \geq (1+t)n\Big]
\leq Cd^k\exp\Big({-}\frac{cn}{d^k}\min(t^2,t)\Big)+Cn\exp(-cd).\]
The lemma follows from taking $t=(\log d)^K\max(\frac{d^k}{n},1)$
for a sufficiently large $K=K(D)$.
\end{proof}

\subsubsection{Concentration inequalities over the unit sphere}

\begin{lemma}\label{lemma:OperatorNormF2c}
Let $\bw_1,\ldots,\bw_p \overset{iid}{\sim} \Unif(\S^{d-1})$, and define
\[\bV_2=\bq_2(\bW)=[\bq_2(\bw_1),\ldots,\bq_2(\bw_p)]^\sT \in \R^{p \times
B_{d,2}}.\]
Then for any $D>0$, there exist constants $C,K>0$ depending only on $D$ such
that with probability at least $1-Cd^{-D}$,
\[\|\bV_2-\E\bV_2\|_\mathrm{op}^2 \leq (\log d)^K\max(1,p/d^2).\]
\end{lemma}
\begin{proof}
Denote $\boldf=\bq_2(\bw)-\E\bq_2(\bw)$ for $\bw \sim \Unif(\S^{d-1})$,
$\boldf_i=\bq_2(\bw_i)-\E\bq_2(\bw_i)$, and $\bF=\bV_2-\E\bV_2$.
We apply the matrix Bernstein inequality to bound
\[\bF^\top \bF=\sum_{i=1}^p \boldf_i\boldf_i^\top.\] 
Order the entries of $\bq_2(\bw)$ so that the first $d$ are
$w_1^2,\ldots,w_d^2$, and the remaining $\binom{d}{2}$ are
$\{\sqrt{2}\,w_iw_j\}_{i<j}$. Applying
$\E w_1^2=\frac{1}{d}$, $\E w_1^4=\frac{3}{d(d+2)}$, and $\E
w_1^2w_2^2=\frac{1}{d(d+2)}$, we have
$\boldf=(w_1^2-\frac{1}{d},\ldots,w_d^2-\frac{1}{d},\sqrt{2}w_1w_2,
\ldots,\sqrt{2}w_{d-1}w_d)$, and hence
\begin{equation}
\E \boldf\boldf^\top = {-}\frac{2}{d^2(d+2)}\be_c\be_c^\top
+\frac{2}{d(d+2)}\id
\end{equation}
where $\be_c \in \R^{B_{d,2}}$ is the vector with first $d$ entries being 1
and the rest 0. Then
$\|\E\boldf \boldf^\top \|_{\mathrm{op}} \leq \frac{4}{d^2}$, and also
$\|\boldf\boldf^\top \|_{\mathrm{op}}=\boldf^\top \boldf
=\|\bq_2(\bw)\|_2^2-\|\E \bq_2(\bw)\|_2^2=1-\frac{1}{d}$ with probability 1.
This gives for the matrix variance
\begin{equation}
	\bigg\|\sum_{i=1}^{p} \E(\boldf_i \boldf_i^\top  - \E\boldf\boldf^\top
)^2\bigg\|_{\mathrm{op}} \leq \bigg\|\sum_{i=1}^{p}
\E(\boldf_i\boldf_i^\sT)^2\bigg\|_{\mathrm{op}} \leq \frac{4p}{d^2}.
\end{equation}
So the matrix Bernstein inequality gives
\begin{equation}
	\P[\|\bF^\top \bF-\E \bF^\top\bF\|_{\mathrm{op}} \ge t]  \leq 2B_{d,2}
\exp(-c\min(t^2d^2/p,t)).
\end{equation}
Applying $B_{d,2} \leq Cd^2$,
$\|\bF\|_\mathrm{\op}^2 \leq \|\bF^\top\bF-\E\bF^\top\bF\|_\mathrm{op}$,
and $\|\E\bF^\top\bF\|_\mathrm{op}=p\|\E\boldf\boldf^\top\|_{\mathrm{op}}
\leq \frac{4p}{d^2}$, this shows
\[\P[\|\bF\|_\op^2 \geq (t+4)p/d^2] \leq Cd^2
\exp(-c(p/d^2)\min(t^2,t)),\]
and the lemma follows from taking $t=(\log d)^K \max(1,d^2/p)$.
\end{proof}
\begin{corollary}
\label{corollary:OperatorNormF2cdecompose}
Writing out
\begin{equation}
    \bV_2 = \begin{pmatrix}
      w_{1,1}^2 &  w_{1,2}^2 & \cdots & w_{1,d}^2 & \sqrt{2}w_{1,1}w_{1,2} & \cdots & \sqrt{2}w_{1,d-1}w_{1,d} \\
      w_{2,1}^2 &  w_{2,2}^2 & \cdots & w_{2,d}^2 & \sqrt{2}w_{2,1}w_{2,2} & \cdots & \sqrt{2}w_{2,d-1}w_{2,d} \\
        \vdots & \vdots & \ddots & \vdots & \vdots & \ddots & \vdots \\
        w_{p,1}^2 &  w_{p,2}^2 & \cdots & w_{p,d}^2 & \sqrt{2}w_{p,1}w_{p,2} & \cdots & \sqrt{2}w_{p,d-1}w_{p,d}
    \end{pmatrix} \, .
\end{equation} 
we have
\begin{equation}\label{equation:OperatorNormF2cdecompose}
    \bV_2 = \frac{1}{d}\1_{p}\be_c^\sT + \bV_{2c},
\end{equation}
where $\be_c \in \R^{B_{d,2}}$ has first $d$ entries 1 and remaining entries 0,
$\1_p \in \R^p$ is the all-1's vector, and $\bV_{2c}$ is a mean-zero matrix
satisfying, with high probability, $\|\bV_{2c}\|_\op \leq (\log
d)^K\max(1,\sqrt{p/d^2})$.
\end{corollary}

\begin{lemma}\label{lemma:OperatorNormFk}
Let $\bw_1,\ldots,\bw_p \overset{iid}{\sim} \Unif(\S^{d-1})$, fix $k \geq 3$,
and define
\[\bV_k=\bq_k(\bW)=[\bq_k(\bw_1),\ldots,\bq_k(\bw_p)]^\sT \in \R^{p \times
B_{d,k}}.\]
Then for any $C_0,D>0$, there exist constants $K,C>0$ depending only on
$k,C_0,D$ such that for any $d \geq C$ and $p \leq C_0d^2$,
with probability at least $1-p^{-D}$,
\begin{equation}\label{equ:vkequationbound}
\begin{aligned}
\|\bV_3\bV_3^\top-\bI_p-\frac{3}{d}\bW\bW^\top\|_\op &\leq (\log
d)^K/\sqrt{d} \text{ if } k=3,\\
\|\bV_4\bV_4^\top-\bI_p-\frac{3}{d^2}\1_p\1_p^\top\|_\op &\leq (\log
d)^K/\sqrt{d} \text{ if } k=4,\\
\|\bV_k\bV_k^\top-\bI_p\|_\op &\leq (\log d)^K/\sqrt{d} \text{ if } k \geq
5.
\end{aligned}
\end{equation}
In particular, $\|\bV_k\|_\op \leq C'$ for a constant $C'>0$
depending only on $k,C_0,D$.
%
\end{lemma}
\begin{proof}
Applying the isometry $\<\bq_k(\bw_i),\bq_k(\bw_j)\>=\<\bw_i^{\otimes
k},\bw_j^{\otimes k}\>=\<\bw_i,\bw_j\>^k$, we have
\[\bV_k\bV_k^\top = (\bW\bW^\top)^{\odot k}\]
For $k = 3$ and $4$, let us decompose each entry $\< \bw_i,\bw_j\>^k$ of this
matrix in the basis
$\{Q_k\}_{k\geq 0}$ of Gegenbauer polynomials that are orthonormal with respect
to the first coordinate $w_1$ of a vector $\bw \sim \Unif(\S^{d-1})$. Then
\begin{align*}
\<\bw_i,\bw_j\>^3 &= c_{3,1} Q_1 (\<\bw_i,\bw_j\>)+ c_{3,3} Q_3
(\<\bw_i,\bw_j\>),\\
\<\bw_i,\bw_j\>^4 &= c_{4,0} + c_{4,2} Q_2 (\<\bw_i,\bw_j\>)+ c_{4,4} Q_4
(\<\bw_i,\bw_j\>),
\end{align*}
where $Q_0(x)=1$, $Q_1(x)=\sqrt{d}\,x$, $Q_2(x)=c_2(x^2-1/d)$ for a quantity
$c_2 \asymp d$, so
\[
\begin{aligned}
  c_{3,1}&=\E[\<\bw_i,\bw_j\>^3 Q_1(\<\bw_i,\bw_j\>)]
=\sqrt{d}\,\E[\<\bw_i,\bw_j\>^4] = \frac{3}{\sqrt{d} (d+2)},\\
c_{4,0}&=\E[\<\bw_i,\bw_j\>^4] = \frac{3}{d(d+2)},\\
|c_{4,2}|&=|\E[\<\bw_i,\bw_j\>^4 Q_2(\<\bw_i,\bw_j\>)]|
\leq \frac{C}{d^2}.
\end{aligned}
\]
By \cite[Proposition 13]{misiakiewicz2024non}, for any $D>0$, there exist
constants $K,C>0$ such that for any $d \geq C$ and $p \leq B_{d,3}/2 \asymp
d^3$, with probability at least $1-p^{-D}$,
\begin{align}
\|(\bW\bW^\top)^{\odot 3}-c_{3,1}Q_1(\bW\bW^\top)-(1-c_{3,1}Q_1(1))\bI_p\|_\op
&\leq (\log d)^K\sqrt{p/d^3},\label{eq:MS24a3rd}\\
\|(\bW\bW^\top)^{\odot
4}-c_{4,0}\bI_p-c_{4,2}Q_2(\bW\bW^\top)-(1-c_{4,0}-c_{4,2}Q_2(1))\bI_p\|_\op
&\leq (\log d)^K\sqrt{p/d^3},\label{eq:MS24a4th}
\end{align}
where $Q_1(\cdot),Q_2(\cdot)$ are applied to $\bW\bW^\top$ entrywise. For $p
\leq C_0d^2$, a standard covering net argument shows $\|\bW\|_\op \leq
C\sqrt{d}$ with probability at least $1-e^{-cd} \geq 1-d^{-D}$ for large enough
$d$, so
\[\Big\|{-}c_{3,1}Q_1(1)\bI_p
+c_{3,1}Q_1(\bW\bW^\top)-\frac{3}{d}\bW\bW^\top\Big\|_\op
=\Big\|{-}\frac{3}{d+2}\bI_p
+\Big(\frac{3}{d+2}-\frac{3}{d}\Big)\bW\bW^\top\Big\|_\op \leq
\frac{C}{d}.\]
Combining these bounds with \eqref{eq:MS24a3rd} implies the desired statements
for $\bV_3$. Applying the decomposition $\bV_2=d^{-1}\1_p\be_c+\bV_{2c}$ in 
Corollary \ref{corollary:OperatorNormF2cdecompose}, we have also
\begin{align*}
\Big\|{-}c_{4,2}Q_2(1)\bI_p+c_{4,2}Q_2(\bW\bW^\top)\Big\|_\op
&\leq \frac{C}{d}
+\frac{C}{d}\|(\bW\bW^\top)^{\otimes 2}-d^{-1}\1_p\1_p^\top\|_\op\\
&=\frac{C}{d}+\frac{C}{d}\|\bV_2\bV_2^\top-d^{-1}\1_p\1_p^\top\|_\op\\
&=\frac{C}{d}+\frac{C}{d}\|\bV_{2c}\bV_{2c}^\top+d^{-1}\1_p\be_c^\top\bV_{2c}^\top+d^{-1}\bV_{2c}\be_c\1_p^\top\|_\op\\
&\leq (\log d)^K/\sqrt{d},
\end{align*}
the last inequality holding for $p \leq C_0d^2$ with probability $1-p^{-D}$, for
some $K=K(C_0,D)>0$. Combining this bound with \eqref{eq:MS24a4th} shows the
desired statements for $\bV_4$.

Finally, the statement for $k \geq 5$ follows from the simple entrywise bound
$\sup_{i \neq j} |\<\bw_i,\bw_j\>| \leq (\log d)^K/\sqrt{d}$ with probability
$1-d^{-D}$ and some $K=K(D)>0$, hence
\[\|(\bW\bW^\top)^{\odot k}-\bI_p\|_\op^2 \leq
\|(\bW\bW^\top)^{\odot k}-\bI_p\|_F^2 = \sum_{i\neq j} \<\bw_i,\bw_j \>^{2k}
\leq p^2[(\log d)^K/\sqrt{d}]^{2k} \leq (\log d)^{K'}/d.\]
\end{proof}
\begin{lemma}\label{lemma:OperatorNormZkFk}
   Fix any $k \ge 3$, 
let $\bw_1,\cdots,\bw_p \overset{iid}{\sim} \Unif(\S^{d-1})$ and
$\bx_1,\cdots,\bx_n \overset{iid}{\sim} \normal(0,\bI_d)$ be independent,
and consider
\begin{equation}
        \bZ_k = \begin{pmatrix}
             \He_k(\bw_1^{\sT}\bx_1) & \He_k(\bw_1^{\sT}\bx_2) & \cdots & \He_k(\bw_1^{\sT}\bx_n) \\ 
                \He_k(\bw_2^{\sT}\bx_1) & \He_k(\bw_2^{\sT}\bx_2) & \cdots & \He_k(\bw_2^{\sT}\bx_n) \\ 
                \vdots & \vdots & \ddots & \vdots \\ 
                \He_k(\bw_p^{\sT}\bx_1) & \He_k(\bw_p^{\sT}\bx_2) & \cdots & \He_k(\bw_p^{\sT}\bx_n)
        \end{pmatrix}.
    \end{equation}
For any constants
$C_0,D>0$, there exist $C,K>0$ depending only on $k,C_0,D$ such that for any $d
\geq C$ and $n,p \leq C_0d^2$, with probability at least $1-d^{-D}$,
\begin{equation}
        \|\bZ_k\|_{\op} \leq (\log d)^K d.
\end{equation}
\end{lemma}
\begin{proof}
Denote the columns of $\bZ_k$ as $\bz_1,\ldots,\bz_n$, where
$\bz_i=\bV_k\bh(\bx_i)$.
Note that by Gaussian hypercontractivity and a union bound,
for any $D>0$, there exists $K=K(C_0,D)>0$
such that with probability at least $1-d^{-D}$, we have
$\sup_{i,j} |\He_k(\bw_j^\top \bx_i)| \leq (\log d)^K$, and hence
$\sup_i \|\bz_i\|_2^2 \leq n(\log d)^{2K}$. Thus, defining
\[\bM_i=\bz_i\bz_i^\top\1_{\{\|\bz_i\|_2^2 \leq n(\log d)^{2K}\}}\]
we have with probability at least $1-d^{-D}$ that
\begin{equation}\label{eq:ZkZkequalssumM}
\frac{1}{n}\bZ_k\bZ_k^\top=\frac{1}{n} \sum_{i=1}^n \bM_i.
\end{equation}
We apply the matrix Bernstein inequality conditional on $\bW$
to bound $n^{-1}\sum_i (\bM_i-\E[\bM_i \mid \bW])$,
noting that $\|n^{-1}(\bM_i-\E[\bM_i \mid \bW])\|_\op \leq 2(\log d)^{2K}$,
and the matrix variance is bounded as
    \begin{align}
\Big\|\sum_{i = 1}^n\frac{1}{n^2}\E[(\bM_i-\E[\bM_i \mid \bW])^2 \mid \bW]\Big\|
&\leq \Big\|\sum_{i = 1}^n\frac{1}{n^2}\E[\bM_i^2 \mid \bW]\Big\|_\op\\
&\leq (\log d)^{2K}\|\E[\bz_i\bz_i^\top \mid \bW]\|_\op
=(\log d)^{2K}\|\bV_k\bV_k^\top\|_\op.
\end{align}
    Then by the matrix Bernstein inequality,
\[\P\bigg[\Big\|\frac{1}{n}\sum_{i=1}^n(\bM_i-\E[\bM_i \mid
\bW])\Big\|_\op \ge t\;\bigg|\;\bW\bigg] \leq Cn\exp\bigg(-\frac{t^2/2}{(\log
d)^{2K}\|\bV_k\bV_k^\sT\|_\op+{2t(\log d)^{2K}}}\bigg).\]
Applying $\|\bV_k\|_\op \leq (\log d)^K$ with probability $1-d^{-2D}$ by
Lemma \ref{lemma:OperatorNormFk}, and
taking $t=(\log d)^{K'}$ for some sufficiently large $K'=K'(D)>0$, this implies
with probability at least $1-d^{-D}$ that
\begin{equation}\label{eq:Mmatrixbernstein}
\Big\|\frac{1}{n}\sum_{i=1}^n(\bM_i-\E[\bM_i \mid \bW])\Big\|_\op
\leq (\log d)^{K'}.
\end{equation}
Finally, note that with probability $1-d^{-D}$, for every $i=1,\ldots,n$,
\begin{equation}\label{eq:Mmatrixbernsteinmean}
\|\E[\bM_i \mid \bW]\|_\op \leq \|\E[\bz_i\bz_i^\top \mid \bW]\|_\op
=\|\bV_k\bV_k\|_\op \leq (\log d)^{2K},
\end{equation}
the last inequality applying again $\|\bV_k\|_\op \leq (\log d)^K$
by Lemma \ref{lemma:OperatorNormFk}. Combining \eqref{eq:ZkZkequalssumM},
\eqref{eq:Mmatrixbernstein}, and \eqref{eq:Mmatrixbernsteinmean} shows
$n^{-1}\|\bZ_k\|_\op^2 \leq (\log d)^{K''}$ with probability $1-d^{-D}$
for a constant $K''=K''(D)>0$,
which implies the lemma in light of the assumption $n \leq C_0d^2$.
\end{proof}

\end{document}